\documentclass[9pt]{article}
\usepackage{amssymb}
\usepackage{tikz}
\usetikzlibrary{arrows, calc, decorations.pathmorphing, backgrounds, positioning, fit, petri, automata}
\definecolor{grey1}{rgb}{0.5,0.5,0.5}
\usepackage[top=2.54cm, bottom=2.54cm, left=2.2cm, right=2.2cm]{geometry}
\usepackage{algorithmicx}
\usepackage{algcompatible}
\usepackage{algorithm}

\usepackage{changebar}
\usepackage{diagbox}
\usepackage{pdflscape}
\usepackage{afterpage}
\usepackage{capt-of}

\usepackage{hyperref}
\hypersetup{ 
  colorlinks,
  citecolor=blue,
  linkcolor=red,
  urlcolor=magenta}

\usepackage{mathtools}
\usepackage{graphicx}
\usepackage{subfigure}
\usepackage{amsmath,amsthm,bm,bbm}
\usepackage{rotating}
\usepackage{mathrsfs}
\usepackage{chngcntr}
\usepackage{apptools}
\usepackage[titletoc,title]{appendix}
\usepackage{comment}

\usepackage{xcolor}         
\definecolor{grau}{rgb}{0.8,0.8,0.8}

\newcommand{\chen}[1]{\color{orange}}
\usepackage{setspace,lscape,longtable}
\usepackage{amsmath,amsthm,bm}
\usepackage{fullpage}  
\usepackage[shortlabels]{enumitem}
\usepackage{titlesec}
\usepackage[english]{babel}
\numberwithin{equation}{section}
\numberwithin{equation}{section}

\newtheorem{theorem}{Theorem}[section]

\newtheorem{lemma}[theorem]{Lemma}

\theoremstyle{remark}

\newtheorem{assumption}{Assumption}

\DeclareMathOperator*{\argmin}{arg\,min}

\newcommand{\prob}{{\mathbb{P}}}
\newcommand{\var}{{\mathrm{var}}}
\newcommand{\expect}{\mathbb{E}}

\newcommand{\Optilde}{{\widetilde{O}_{p}}}
\newcommand{\optilde}{{\widetilde{o}_{p}}}

\newcommand{\transpose}{^{\mathrm{T}}}

\newcommand{\calC}{{\mathcal{C}}}

\newcommand{\calE}{{\mathcal{E}}}
\newcommand{\calF}{{\mathcal{F}}}

\newcommand{\calH}{{\mathcal{H}}}
\newcommand{\calI}{{\mathcal{I}}}

\newcommand{\calL}{{\mathcal{L}}}

\newcommand{\calP}{{\mathcal{P}}}

\newcommand{\calS}{{\mathcal{S}}}
\newcommand{\calT}{{\mathcal{T}}}

\newcommand{\bT}{{\mathbf{T}}}
\newcommand{\bU}{{\mathbf{U}}}
\newcommand{\bP}{{\mathbf{P}}}

\newcommand{\bY}{{\mathbf{Y}}}
\newcommand{\bv}{{\mathbf{v}}}
\newcommand{\bx}{{\mathbf{x}}}
\newcommand{\bX}{{\mathbf{X}}}
\newcommand{\bA}{{\mathbf{A}}}
\newcommand{\bB}{{\mathbf{B}}}
\newcommand{\bC}{{\mathbf{C}}}

\newcommand{\bE}{{\mathbf{E}}}
\newcommand{\bF}{{\mathbf{F}}}
\newcommand{\bG}{{\mathbf{G}}}
\newcommand{\bH}{{\mathbf{H}}}
\newcommand{\bM}{{\mathbf{M}}}

\newcommand{\bL}{{\mathbf{L}}}

\newcommand{\bR}{{\mathbf{R}}}
\newcommand{\bQ}{{\mathbf{Q}}}
\newcommand{\bS}{{\mathbf{S}}}
\newcommand{\bW}{{\mathbf{W}}}
\newcommand{\bZ}{{\mathbf{Z}}}

\newcommand{\bw}{{\mathbf{w}}}

\newcommand{\by}{{\mathbf{y}}}
\newcommand{\bz}{{\mathbf{z}}}

\newcommand{\be}{{\mathbf{e}}}

\newcommand{\bV}{{\mathbf{V}}}

\newcommand{\bGamma}{{\bm{\Gamma}}}

\newcommand{\eye}{{\mathbf{I}}}
\newcommand{\one}{{\mathbf{1}}}

\newcommand{\btheta}{{\bm{\theta}}}

\newcommand{\bmu}{{\bm{\mu}}}

\newcommand{\zero}{{\bm{0}}}
\newcommand{\eps}{\epsilon}

\newcommand{\keywords}[1]{\par\addvspace\baselineskip\noindent\enspace\ignorespaces \textbf{Keywords: }#1}

\author{Fangzheng Xie\thanks{Department of Statistics, Indiana University}
\footnotemark[1] \thanks{Correspondence should be addressed to Fangzheng Xie (fxie@iu.edu)}
}

\title{Bias-Corrected Joint Spectral Embedding for Multilayer Networks with Invariant Subspace: Entrywise Eigenvector Perturbation and Inference}

\begin{document}

\allowdisplaybreaks
\maketitle

\begin{abstract}
In this paper, we propose to estimate the invariant subspace across heterogeneous multiple networks using a novel bias-corrected joint spectral embedding algorithm. The proposed algorithm recursively calibrates the diagonal bias of the sum of squared network adjacency matrices by leveraging the closed-form bias formula and iteratively updates the subspace estimator using the most recent estimated bias. Correspondingly, we establish a complete recipe for the entrywise subspace estimation theory for the proposed algorithm, including a sharp entrywise subspace perturbation bound and the entrywise eigenvector central limit theorem. Leveraging these results, we settle two multiple network inference problems: the exact community detection in multilayer stochastic block models and the hypothesis testing of the equality of membership profiles in multilayer mixed membership models. Our proof relies on delicate leave-one-out and leave-two-out analyses that are specifically tailored to block-wise symmetric random matrices and a martingale argument that is of fundamental interest for the entrywise eigenvector central limit theorem. 
\end{abstract}


\section{Introduction}
\subsection{Background}
\label{sub:background}

Network data are a convenient data form for representing relational structures among multiple entities. They are pervasive across a broad spectrum of application domains, including social science \cite{Girvan7821,wasserman1994social,young2007random}, biology \cite{Girvan7821,8570772}, and computer science \cite{6623779,7745482}, to name a selected few. Statistical network analysis has also been attracting attention recently, where random graph models serve as the underlying infrastructure. In a random graph, a network-valued random variable is generated by treating the vertices as deterministic and the edges connecting different vertices as random variables. Various popular random graph models have been proposed and studied in the literature, such as the renowned stochastic block models \cite{HOLLAND1983109} and their variants \cite{airoldi2008mixed,PhysRevE.83.016107,7769223}, the (generalized) random dot product graph model \cite{rubin2017statistical,young2007random}, the latent space model \cite{doi:10.1198/016214502388618906}, and exchangeable random graphs \cite{doi:10.1111/rssb.12233,10.1214/20-AOS1976,lovasz2012large}, among others. Correspondingly, there has also been substantial development regarding statistical inference for random graph models, including community detection \cite{abbe2017community,7298436,lei2015,sussman2012consistent}, vertex classification \cite{6565321,tang2013}, and network hypothesis testing \cite{10.1214/15-AOS1370,doi:10.1080/10618600.2016.1193505,tang2017}, to name a selected few. The readers are referred to survey papers \cite{abbe2017community,gao2021minimax,rdpg-survey} for the details on the recent advances in statistical network analysis.

Recently, there has been a growing interest in analyzing heterogeneous multiple networks, where the data consists of a collection of vertex-aligned networks, and each network is referred to as a \emph{layer}. These multiple networks, also known as \emph{multilayer networks} \cite{10.1093/comnet/cnu016}, arise naturally in a variety of application domains where network data are pervasive. For example, in trade networks, countries or districts are represented by vertices, and trade activities among them are modeled as edges. In these trade networks, vertices corresponding to countries or districts are aligned across layers, but different items and times lead to heterogeneous network patterns. Other examples include fMRI studies \cite{openscienceresourcefunctionalconnectomics}, protein networks \cite{10.1371/journal.pcbi.1000807}, traffic networks \cite{jacobini2020bureau}, gene coexpression networks \cite{bakken2016comprehensive}, social networks \cite{6265414}, and mobile phone networks \cite{Kiukkonen:150599}, among others. The heterogeneity of multiple network data brings extra challenges compared to their single network counterparts and requires substantially different methodologies and theories. 

\subsection{Overview}
\label{sub:overview}

Consider a collection of vertex-aligned network adjacency matrices $(\bA_t)_{t = 1}^m$ with low-rank edge probability matrices $(\bP_t)_{t = 1}^m$. Suppose $\mathrm{rank}(\bP_t) = d$ for all $t\in\{1,\ldots,m\}$ and they share the same leading eigenspace, \emph{i.e.}, the column spaces of $(\bP_t)_{t = 1}^m$ are identical. This multiple network model is called the common subspace independent edge model (COSIE) \cite{arroyo2021inference}, the formal definition of which is introduced in Section \ref{sub:COSIE}. The COSIE model is a quite flexible yet architecturally simple multiple network model that captures both the shared subspace structure across layers and the layer-wise heterogeneity pattern by allowing different eigenvalues. The COSIE model includes the popular multilayer stochastic block model and the multilayer mixed membership model as important special examples. 

The common subspace structure of the COSIE model allows us to write $\bP_t = \bU\bB_t\bU\transpose$ for some $n\times d$ matrix $\bU$ with orthonormal columns for each layer $t$, and the column space of $\bU$ remains invariant across layers, whereas the layer-specific pattern is captured by a score matrix $\bB_t$. Our goal is to estimate $\bU$ up to the right multiplication of a $d\times d$ orthogonal matrix, particularly under the challenging regime where the layer-wise signal-to-noise ratio is insufficient to consistently estimate $\bU$ using single network methods. Successful estimation of $\bU$ is of fundamental interest to several downstream network inference tasks, including community detection in multilayer stochastic block models and the hypothesis testing of membership profiles in multilayer mixed membership models. The major contribution of our work is threefold:
\begin{itemize}
        \item We design a novel bias-corrected joint spectral embedding algorithm (see Section \ref{sub:JSE} and Algorithm \ref{alg:BCJSE} there for details) for estimating $\bU$. The proposed algorithm carefully calibrates the bias due to $\expect(\sum_{t = 1}^m\bA_t)\neq \sum_{t= 1}^m\bP_t^2$ recursively by leveraging the closed form bias formula and updates the subspace estimator iteratively. Compared to some of the existing de-biasing algorithms, such as the heteroskedastic PCA \cite{zhang2018heteroskedastic}, the proposed algorithm is computationally efficient, numerically more stable, and only requires $O(1)$ number of iterations to achieve rate optimality under mild conditions.

        \item We establish the entrywise subspace estimation theory for the proposed bias-corrected joint spectral embedding estimator in the context of the COSIE model. Specifically, we establish a sharp two-to-infinity norm subspace perturbation bound and the entrywise central limit theorem for our estimator. Our proof relies on non-trivial and sharp leave-one-out and leave-two-out devices that are specifically designed for multilayer network models and a martingale argument. These technical tools may be of independent interest.

        \item Leveraging the entrywise subspace estimation theory, we further settle two multilayer network inference tasks, namely, the exact community detection in multilayer stochastic block models and the hypothesis testing of the equality of the membership profiles in multilayer mixed membership models. These problems are non-trivial extensions of their single network counterparts because we do not require the layer-wise average network expected degree to diverge, as long as the aggregated network signal-to-noise ratio is sufficient. 
\end{itemize} 

\subsection{Related work}
\label{sub:related_work}

The recent decade has witnessed substantial progress in the theoretical and methodological development for statistical analyses of single networks, where spectral methods have been playing a pivotal role not only thanks to their computational convenience but also because they directly provide insight into a variety of subsequent network inference algorithms. One of the most famous examples is the spectral clustering in stochastic block models \cite{sussman2012consistent,lei2015}, where the leading eigenvectors of the adjacency matrix or its normalized Laplacian matrix are used as the input for $K$-means clustering algorithm to recover the vertex community memberships. Also see \cite{xie2019optimal,xie2019efficient,xie2021entrywise,xie-wu-2022-ea-spn,athreya2016limit,athreya2018estimation,doi:10.1080/10618600.2022.2081576,levin2019bootstrapping,JMLR:v22:19-852,10.1093/biomet/asaa006,lyzinski2014,rubin2017statistical,https://doi.org/10.1111/rssb.12245,6565321,sussman2012consistent,doi:10.1080/10618600.2016.1193505,tang2017,tang2018,young2007random} and the survey paper \cite{MAL-079} for an incomplete list of reference. Entrywise eigenvector estimation theory for network models has been attracting growing interest recently because it provides fine-grained characterization of some downstream network inference tasks, such as the exact community detection in stochastic block models \cite{abbe-fan-wang-zhong-2020}, the fundamental limit of spectral clustering in stochastic block models \cite{zhang2023fundamental}, and the inference for membership profiles in mixed membership stochastic block model \cite{fan2019simple,xie2021entrywise}. A collection of entrywise eigenvector perturbation bounds under the notion of the two-to-infinity norm have been developed by \cite{doi:10.1137/1.9781611977912.34,cape2019signal,cape2017two,pmlr-v83-eldridge18a,doi:10.1080/01621459.2020.1751645,lyzinski2014,abbe-fan-wang-zhong-2020,xia2019sup} in the context of single network models and beyond, while the authors of \cite{athreya2016limit,tang2018,cape2019signal,xie2021entrywise} have further established the entrywise eigenvector central limit theorem beyond the perturbation error bounds. Although the above works are designed for the analysis of single networks and single symmetric random matrices, the ideas and the technical tools developed there are inspiring for the analysis of heterogeneous multiple networks, as will be made clear in Section \ref{sub:entrywise_eigenvector_CLT} below. 

The analysis of heterogeneous multiple networks has primarily focused on the context of multilayer stochastic block models and the corresponding community detection problem. Existing methodologies can be roughly classified into the following categories: spectral-based methods \cite{agterberg2022joint,bhattacharyya2018spectral,bhattacharyya2020consistent,doi:10.1080/10618600.2022.2134874,doi:10.1080/01621459.2022.2054817,zheng2022limit,arroyo2021inference,7990044,6265414,doi:10.1073/pnas.1718449115}, matrix and tensor factorization methods \cite{paul2020spectral,5360349,10.1214/21-AOS2079,10.1093/biomet/asz068}, likelihood and modularity-based methods \cite{PhysRevE.92.042807,PhysRevResearch.2.023100,PhysRevE.95.042317,chen2022global,PhysRevX.6.031005,10.1111/rssb.12200,paul2016consistent,6758385,doi:10.1080/01621459.2016.1260465}, and approximate message passing algorithms \cite{10023524}. A recently posted preprint \cite{lei2023computational} discussed the computational and statistical thresholds of community detection in multilayer stochastic block models under the so-called low-degree polynomial conjecture. Most of these works, however, did not provide the entrywise subspace estimation theory, and most of them only established the weak consistency of community detection in multilayer stochastic block models, \emph{i.e.}, a vanishing fraction of the vertices are incorrectly clustered, except for \cite{paul2016consistent,chen2022global,zheng2022limit}. The exact community detection in \cite{paul2016consistent} is achieved through the minimax rates of community detection, but their proof only guarantees the existence of such a minimax-optimal algorithm. In \cite{chen2022global}, the authors proposed a two-stage algorithm that achieves the exact community detection, but they require a strong condition on the positivity of the eigenvalues of the block probability matrices. In \cite{zheng2022limit}, the authors established the exact community detection of a multiple adjacency spectral embedding method under the much stronger regime where the layer-wise signal-to-noise ratio is sufficient. We defer the detailed comparison between our results with those from some of the above-cited work to Section \ref{sec:main_results}. For other multiple network models encompassing the COSIE model, see \cite{bhattacharyya2020consistent,8889404,draves2020bias,wang2023multilayer,pensky2021clustering,pensky2019spectral} for an incomplete list of reference. 

Our problem can be equivalently reformulated as estimating the left singular subspace of the low-rank concatenated matrix $[\bP_1,\ldots,\bP_m]$ using the noisy rectangular concatenated network adjacency matrix $[\bA_1,\ldots,\bA_m]$. There has also been growing interest in the entrywise singular subspace estimation theory for rectangular random matrices with low expected ranks \cite{agterberg2021entrywise,cai2021subspace,yan2021inference,10.1214/22-AOS2196}. From the technical perspective, the papers \cite{agterberg2021entrywise,yan2021inference} are the most similar ones to our work, but they considered rectangular random matrices with completely independent entries and studied the so-called heteroskedastic PCA estimator for the singular subspace. In contrast, the rectangular matrix of interest in our work has a block-wise symmetric structure, and we consider a different bias-corrected joint spectral embedding algorithm. Consequently, neither their results nor their proof techniques can be applied to the COSIE model context. We also defer the detailed comparison between our work and those obtained in \cite{agterberg2021entrywise,cai2021subspace,yan2021inference,10.1214/22-AOS2196} in Section \ref{sub:entrywise_eigenvector_CLT}. 

\subsection{Paper organization}
\label{sub:organization}

The rest of this paper is organized as follows. Section \ref{sec:COSIE_BCJSE} sets the stage of the COSIE model and introduces the proposed bias-corrected joint spectral embedding algorithm. The main results are elaborated in Section \ref{sec:main_results}, including the two-to-infinity norm subspace perturbation bound and the entrywise eigenvector central limit theorem of the bias-corrected joint spectral embedding estimator. In this section, we also settle the exact community detection in multilayer stochastic block models and the hypothesis testing of the equality of membership profiles for any two given vertices in multilayer mixed membership models. Section \ref{sec:numerical} illustrates the practical performance of the proposed bias-corrected joint spectral embedding algorithm and validates the entrywise subspace estimation theory via numerical experiments empirically. We conclude the paper with a discussion in Section \ref{sec:discussion}. The technical proofs of our main results are deferred to the appendices.

\subsection{Notations}
\label{sub:notation}

The symbol $:=$ is used to assign mathematical definitions throughout. Given a positive integer $n$, we let $[n]$ denote $\{1,\ldots,n\}$. For any two real numbers $a, b > 0$, let $a\vee b := \max(a, b)$ and $a\wedge b:=\min(a, b)$. For any two nonnegative sequences $(a_n)_{n = 1}^\infty, (b_n)_{n = 1}^\infty$, we write $a_n\lesssim b_n$ or $b_n\gtrsim a_n$ or $a_n = O(b_n)$ or $b_n = \Omega(a_n)$, if there exists a constant $C > 0$, such that $a_n\leq Cb_n$ for all $n$. We write $a_n\asymp b_n$ or $a_n = \Theta(b_n)$ if $a_n = O(b_n)$ and $a_n = \Omega(b_n)$. We use $C, C_1, C_2, c, c_1, c_2$, etc to denote positive constants that may vary from line to line but do not change with $n$ throughout. A sequence of events $(\calE_n)_{n = 1}^\infty$ indexed by $n$ is said to occur with high probability (w.h.p.), if for any constant $c > 0$, there exists a $c$-dependent constant $N_c > 0$, such that $\prob(\calE_n)\geq 1 - O((m + n)^{-c})$ for any $n\geq N_c$, where we assume that the number of layers $m$ depends on the number of vertices $n$ in this work. Given two sequences of nonnegative random variables $(X_n)_{n = 1}^\infty$, $(Y_n)_{n = 1}^\infty$, we say that $X_n$ is bounded by $Y_n$ w.h.p., denoted by $X_n = \Optilde(Y_n)$, if for any constant $c > 0$, there exist $c$-dependent constants $N_c, C_c > 0$, such that $\prob(X_n\leq C_cY_n)\geq 1 - O((m + n)^{-c})$ for any $n\geq N_c$, and we write $X_n = \optilde(Y_n)$ if $X_n = \Optilde(\eps_nY_n)$ for some nonnegative sequence $(\eps_n)_{n = 1}^\infty$ converging to $0$. Note that our $\Optilde(\cdot)$ and $\optilde(\cdot)$ notations are stronger than the conventional $O_p$ and $o_p$ notations in the literature of probability and statistics. 

Given positive integers $n,d$ with $n\geq d$, let $\eye_d$ denote the $d\times d$ identity matrix, $\zero_d$ denote the $d$-dimensional zero vector, $\one_d$ denote the $d$-dimensional vector of all ones, $\zero_{n\times d}$ the $n\times d$ zero matrix, $\mathbb{O}(n, d) := \{\bU\in\mathbb{R}^{n\times d}:\bU\transpose\bU\}$, and $\mathbb{O}(d):=\mathbb{O}(d, d)$. Let $\mathrm{Span}(\bX)$ denote the column space of $\bX$ for any rectangular matrix $\bX$.
For any positive integer $i$, we let $\be_i$ denote the $i$th standard basis vector whose $i$th entry is $1$ and the remaining entries are $0$ if the dimension of $\be_i$ is clear from the context. Given $a_1,\ldots,a_n\in\mathbb{R}$, we denote by $\mathrm{diag}(a_1,\ldots,a_n)$ the diagonal matrix whose diagonal entries are $a_1,\ldots,a_n$. Given a square matrix $\bM\in\mathbb{R}^{n\times n}$, let $\mathrm{diag}(\bM)$ be the diagonal matrix obtained by extracting the diagonal entries of $\bM$. If $\bM\in\mathbb{R}^{n\times n}$ is symmetric, we then denote by $\lambda_k(\bM)$ the $k$th largest eigenvalue of $\bM$, namely, $\lambda_1(\bM)\geq\ldots\geq\lambda_n(\bM)$. If $\bM$ is a positive semidefinite matrix, we follow the usual linear algebra notation and denote by $\bM^{1/2}$ as applying the square root operation to the eigenvalues of $\bM$. Let $\bM^{-1/2}:=(\bM^{-1})^{1/2}$ if $\bM$ is positive definite. Given a $p_1\times p_2$ rectangular matrix $\bA = [A_{ij}]_{p_1\times p_2}$, we let $\sigma_k(\bA)$ denote its $k$th largest singular value, namely, $\sigma_1(\bA)\geq\ldots\geq\sigma_{p_1\wedge p_2}(\bA)$, denote its spectral norm by $\|\bA\|_2:=\lambda_1^{1/2}(\bA\transpose\bA)$, its Frobenius norm by $\|\bA\|_{\mathrm{F}} = (\sum_{i = 1}^{p_1}\sum_{j = 1}^{p_2}A_{ij}^2)^{1/2}$, its infinity norm by $\|\bA\|_\infty = \max_{i\in[p_1]}\sum_{j = 1}^{p_2}|A_{ij}|$, its two-to-infinity norm by $\|\bA\|_{2\to\infty} = \max_{i\in[p_1]}(\sum_{j = 1}^{p_2}A_{ij}^2)^{1/2}$, and its entrywise maximum norm by $\|\bA\|_{\max} = \max_{i\in[p_1],j\in[p_2]}|A_{ij}|$. 
For any $\bU,\bV\in\mathbb{O}(n, d)$, let $\sin\Theta(\bU, \bV) := \mathrm{diag}\{\sigma_1(\bU\transpose\bV),\ldots,\sigma_d(\bU\transpose\bV)\}$. 
For any matrix $\bH\in\mathbb{O}(d)$, we define its matrix sign $\mathrm{sgn}(\bH)$ as follows: Suppose it has the singular value decomposition (SVD) 
$\bH = \bW_1\mathrm{diag}\{\sigma_1(\bH),\ldots,\sigma_d(\bH)\}\bW_2\transpose$,
where $\bW_1,\bW_2\in\mathbb{O}(d)$. The matrix sign of $\bH$ is defined as $\mathrm{sgn}(\bH) = \bW_1\bW_2\transpose$. Let $\calH:\mathbb{R}^{n\times n}\to\mathbb{R}^{n\times n}$ denote the hollowing operator defined by $\calH(\bA) = \bA - \mathrm{diag}(\bA) = \bA - \sum_{i = 1}^n\be_i\be_i\transpose\bA\be_i\be_i\transpose$ for any $\bA\in\mathbb{R}^{n\times n}$ \cite{10.1214/22-AOS2196}. Lastly, given a symmetric matrix $\bA\in\mathbb{R}^{n\times n}$ and a positive integer $d\leq n$, we let $\texttt{eigs}(\bA; d)$ denote the $n\times d$ eigenvector matrix of $\bA$ with orthonormal columns associated eigenvalues $\lambda_1(\bA)\geq\ldots\geq\lambda_d(\bA)$. 

\section{Model Setup and Methodology}
\label{sec:COSIE_BCJSE}

\subsection{Multilayer networks with invariant subspace}
\label{sub:COSIE}

Consider $n\times n$ symmetric random matrices $\bA_1,\ldots,\bA_m$, where $(\bA_t)_{t = 1}^m$ have independent upper triangular entries. Each network $\bA_t$ is referred to as a layer, and we model the multilayer networks $(\bA_t)_{t = 1}^m$ through the common subspace independent edge (COSIE) model \cite{arroyo2021inference,zheng2022limit} defined as follows: There exists an $n\times d$ orthonormal matrix $\bU\in\mathbb{O}(n, d)$, where $d\leq n$, and a collection of $d\times d$ symmetric score matrices $(\bB_t)_{t = 1}^m\subset\mathbb{R}^{d\times d}$, such that $\bP_t = \expect\bA_t = \bU\bB_t\bU\transpose$, $t\in[m]$, $(A_{tij} - P_{tij}:t\in[m],i,j\in[n],i\leq j)$ are independent centered Bernoulli random variables, and $A_{tij} = A_{tji}$ if $i > j$, where $A_{tij}$ and $P_{tij}$ denote the $(i, j)$th entry of $\bA_t$ and $\bP_t$, respectively. In the COSIE model, different network layers share the same principal subspace but the heterogeneity across layers is captured by $(\bB_t)_{t = 1}^m$. 

The COSIE model is flexible enough to encompass several popular heterogeneous multiple network models yet enjoys a simple architecture. Two important examples are the multilayer stochastic block model (MLSBM) and the multilayer mixed membership (MLMM) model. In MLSBM, the vertices share the same set of community structures across layers, but the layer-wise block probabilities can be different. Formally, given $n$ vertices labeled as $[n]$ and number of communities $K$, we say that random adjacency matrices $\bA_1,\ldots,\bA_m$ follow MLSBM with community assignment function $\tau:[n]\to[K]$ and symmetric block probability matrices $(\bC_t)_{t = 1}^m\in[0, 1]^{K\times K}$, denoted by $\bA_1,\ldots,\bA_m\sim\mathrm{MLSBM}(\tau;\bC_1,\ldots,\bC_m)$, if $A_{tij}\sim\mathrm{Bernoulli}(C_{t\tau(i)\tau(j)})$ independently for all $t\in[m]$, $i,j\in[n]$, $i\leq j$, where $C_{tab}$ is the $(a, b)$th entries of $\bC_t$. The community assignment function can be equivalently represented by a matrix $\bZ\in\mathbb{R}^{n\times K}$ with its $(i, k)$th entry being one if $\tau(i) = k$ and the remaining entries being zero. The MLMM models generalize the MLSBM by allowing the matrix $\bZ$ to have entries between $0$ and $1$, provided that $\bZ\one_d = \one_n$, and its $i$th row is a probability vector describing how the membership weights of the $i$th vertex are assigned to $K$ communities. 
In \cite{arroyo2021inference}, the authors showed that both the MLSBM and the MLMM models can be represented by $\mathrm{COSIE}(\bU; \bB_1,\ldots,\bB_m)$ with $\bU = \bZ(\bZ\transpose\bZ)^{-1/2}$ and $\bB_t = (\bZ\transpose\bZ)^{1/2}\bC_t(\bZ\transpose\bZ)^{1/2}$ for all $t\in[m]$, where $\bZ\in\mathbb{R}^{n\times K}$ is the community assignment or membership profile matrix, provided that $\bZ\transpose\bZ$ is invertible. 

\subsection{Joint Spectral Embedding and Bias Correction}
\label{sub:JSE}
Our goal is to estimate the invariant subspace spanned by the columns of $\bU$ by aggregating the network information across layers. Note that $\bU$ is only identifiable up to a $d\times d$ orthogonal transformation. Given any generic estimator $\widehat{\bU}$, it is necessary to consider an orthogonal matrix $\mathrm{sgn}(\widehat{\bU}\transpose\bU)$ such that $\widehat{\bU}\mathrm{sgn}(\widehat{\bU}\transpose\bU)$ is aligned with $\bU$. The matrix sign is often used as the orthogonal alignment matrix between subspaces that are comparable (see, for example, \cite{MAL-079,5714248,cape2020orthogonal}). 

Let $\bE_t = [E_{tij}]_{n\times n} = \bA_t - \bP_t$, $\bA = [\bA_1,\ldots,\bA_m]\in\mathbb{R}^{n\times mn}$, $\bP = [\bP_1,\ldots,\bP_m]\in\mathbb{R}^{n\times mn}$, $\bE = \bA - \bP$, and $\bB = [\bB_1,\ldots,\bB_m]\in\mathbb{R}^{d\times md}$. A naive approach is to use $\texttt{eigs}((1/n)\sum_{t = 1}^m\bA_t; d)$ as an estimator for $\bU$, but this method can lead to the cancellation of signals when $(\bP_t)_{t = 1}^m$ has negative eigenvalues \cite{doi:10.1080/01621459.2022.2054817}. Alternatively, observe that $\mathrm{Span}(\bU)$ can be equivalently viewed as the leading eigen-space of the sum-of-squared matrix $\bP\bP\transpose$ corresponding to its $d$-largest eigenvalues. It is thus reasonable to consider the eigenvector matrix of $\bA\bA\transpose$ corresponding to its $d$-largest eigenvalues as the sample counterpart of $\bU$. This naive approach, nonetheless, can generate bias when $m$ is large, as well observed in \cite{doi:10.1080/01621459.2022.2054817} because $\expect \bA\bA\transpose = \bP\bP\transpose + \expect \bE\bE\transpose$ and $\expect \bE\bE\transpose$ is a nonzero diagonal matrix. There are two existing approaches that attempt to account for this bias term. One strategy is computing the leading eigenvector matrix associated with the $d$-largest eigenvalues of $\calH(\bA\bA\transpose)$. This approach has been studied in \cite{doi:10.1080/01621459.2022.2054817,cai2021subspace,10.1214/22-AOS2196}, but we show in Theorem \ref{thm:entrywise_eigenvector_perturbation_bound} below that the estimation error bound is sub-optimal. The fundamental reason for such sub-optimality is because $\expect\calH(\bA\bA\transpose)\neq\bP\bP\transpose$. Another method is to iteratively correct the bias in $\expect\bE\bE\transpose$ using the diagonal entries of a low-rank approximation to $\bA\bA\transpose$. This approach is referred to as the HeteroPCA algorithm and has been studied in \cite{10.1214/21-AOS2074,yan2021inference,agterberg2021entrywise}. Neither approach is the focus of this work, although our proposed approach below shares some similarities with HeteroPCA in spirit. 

The starting point of our proposed bias correction procedure is the observation that $\expect\calH(\bA\bA\transpose) = \bP\bP\transpose + \bM$, where 
\begin{align}
\label{eqn:M_matrix}
\bM = -\sum_{i = 1}^n\be_i\be_i\transpose\sum_{t = 1}^m\sum_{j = 1}^nP_{tij}^2.
\end{align}
Formally, the bias-corrected joint spectral embedding (BCJSE) can be described in Algorithm \ref{alg:BCJSE} below. 
\begin{algorithm}[H]
\caption{Bias-corrected joint spectral embedding.}\label{alg:BCJSE}
\begin{algorithmic}
\STATE 
\STATE \textbf{Require:} Matrix $\bA\bA\transpose$, rank $d$, number of steps for bias calibration $S$, number of iterations $R$.
\STATE \textbf{Initialize:} Set $\widehat{\bU}_{0S} = \zero_{n\times d}$. 
\STATE \hspace{0.5cm} \textbf{For} $r = 1,2,\ldots,R$:
\STATE \hspace{1cm} Set $\widehat{\bM}_{r0} = \zero_{n\times n}$.
\STATE \hspace{1cm} \textbf{For} $s = 1,2,\ldots,S$:
\STATE \hspace{1.5cm} Set
    \[
    \widehat{\bM}_{rs} = -\sum_{i = 1}^n\be_i\be_i\transpose\widehat{\bU}_{(r - 1)S}\widehat{\bU}_{(r - 1)S}\transpose\left\{\calH(\bA\bA\transpose) - \widehat{\bM}_{r(s - 1)}\right\}\widehat{\bU}_{(r - 1)S}\widehat{\bU}_{(r - 1)S}\transpose\be_i\be_i\transpose.
    \]
\STATE \hspace{1cm} \textbf{End For}
\STATE \hspace{1cm} Set $\widehat{\bU}_{rS} = \texttt{eigs}(\calH(\bA\bA\transpose) - \widehat{\bM}_{rS};d)$. 
\STATE \hspace{0.5cm} \textbf{End For}
\STATE \textbf{return}  $\widehat{\bU} = \widehat{\bU}_{RS}$.
\end{algorithmic}
\label{alg1}
\end{algorithm}
The BCJSE algorithm consists of two-level iterations. At the $r$th step of the outer iteration, the inner iteration calibrates the bias $\bM$ recursively since $\widehat{\bM}_{rs}$ can be viewed as a plug-in estimator of $\bM$ using the most updated value of $\widehat{\bU}_{(r - 1)S}$ and $\widehat{\bM}_{r(s - 1)}$. Then, the outer iteration updates the estimator for $\bU$ iteratively using the most updated estimated bias $\widehat{\bM}_{rS}$. The BCJSE algorithm also generalizes the hollowed spectral decomposition method discussed in \cite{doi:10.1080/01621459.2022.2054817,10.1214/22-AOS2196,cai2021subspace} since when $R = 1$, Algorithm \ref{alg:BCJSE} returns the eigenvector matrix corresponding to the $d$-largest eigenvalues of $\calH(\bA\bA\transpose)$. Compared to HeteroPCA, BCJSE can be viewed as a ``parametric'' version of HeteroPCA because it takes advantage of the parametric form of $\bM$. As will be seen in Section \ref{sec:main_results}, the BCJSE algorithm requires only $O(1)$ number of iterations to achieve sharp estimation error bounds under mild conditions. 

\section{Main Results}
\label{sec:main_results}

\subsection{Entrywise Eigenvector Perturbation Bound}
\label{sub:entrywise_eigenvector_perturbation_bound}

We first introduce several necessary assumptions for our main results and discuss their implications. 

\begin{assumption}[Eigenvector delocalization]
\label{assumption:eigenvector_delocalization}
Define the incoherence parameter $\mu = (n/d)\|\bU\|_{2\to\infty}^2$. Then $\mu = O(1)$. 
\end{assumption}

\begin{assumption}[Signal-to-noise ratio]
\label{assumption:condition_number}
Let $\Delta_n = \min_{t\in[m]}\sigma_d(\bB_t)$ and $\kappa = \max_{t\in[m]}\sigma_1(\bB_t)/\Delta_n$. Then $\kappa,d = O(1)$, and there exists an $n$-dependent quantity $\rho_n\in(0, 1]$, such that $\max_{t\in[m],i,j\in[n]}\expect(A_{tij} - P_{tij})^2\leq \rho_n$, $\|\bP\|_{\max} = \Theta(\rho_n)$, $\Delta_n = \Theta(n\rho_n)$, $m^{1/2}n\rho_n = \omega(\log n)$, and $m = O(n^\alpha)$ for some constant $\alpha > 0$.  
\end{assumption}

Assumption \ref{assumption:eigenvector_delocalization} requires the invariant subspace basis matrix $\bU$ to be delocalized. It is also known as the incoherence condition in the literature of matrix completion \cite{candes2009exact,5466511,jain2013low} and random matrix theory \cite{erdos2013}. It requires that the entries of the invariant subspace basis matrix $\bU$  cannot be too ``spiky'', namely, the columns of $\bU$ are significantly different from the standard basis vectors. It is also a common condition from the perspective of network models. For example, when the underlying COSIE model is an MLSBM, Assumption \ref{assumption:eigenvector_delocalization} requires that the community sizes are balanced, namely, the number of vertices in each community is $\Theta(n/K)$. When each network layer is generated from a random dot product graph, Assumption \ref{assumption:eigenvector_delocalization} is also satisfied with probability approaching one if the underlying latent positions are independently generated from an underlying distribution. 

Assumption \ref{assumption:condition_number} characterizes the signal-to-noise ratio of the COSIE model through the smallest nonzero singular values of the score matrices and the so-called network sparsity factor $\rho_n$. In Assumption \ref{assumption:condition_number}, the requirement $\kappa, d = O(1)$ is a mild condition. In the context of MLSBM, this amounts to requiring that the number of communities is bounded and the condition numbers of the block probability matrices are also bounded. The sparsity factor $\rho_n$ fundamentally controls the average network expected degree through $n\rho_n$. When $m = 1$, the COSIE model reduces to the generalized random dot product graph (GRDPG) and Assumption \ref{assumption:condition_number} requires that $n\rho_n = \omega(\log n)$, which almost agrees with the standard condition for single network problems (see, for example, \cite{abbe-fan-wang-zhong-2020,xie2021entrywise}). Notably, when $m$ is allowed to increase with $n$, Assumption \ref{assumption:condition_number} only requires that $m^{1/2}n\rho_n = \omega(\log n)$ and $n\rho_n$ does not need to diverge as $n$ increases. This is because, in multilayer networks, the overall network signal can be obtained by aggregating the information from each layer, so that the requirement for the signal strength from each layer is substantially weaker than that for single network problems. A similar phenomenon has also been observed in the literature of multilayer network analysis (see, for example, \cite{10023524,10.1093/biomet/asz068,10.5555/3045118.3045279,bhattacharyya2018spectral,paul2020spectral,chen2022global,10.1214/21-AOS2079,paul2016consistent,lei2023computational})

Below, we present our first main result regarding the row-wise perturbation bound of the BCJSE through a two-to-infinity norm error estimate. 
\begin{theorem}\label{thm:entrywise_eigenvector_perturbation_bound}
Suppose Assumptions \ref{assumption:eigenvector_delocalization}--\ref{assumption:condition_number} hold. Let $\widehat{\bU}$ be the output of Algorithm \ref{alg:BCJSE} with $R, S = O(1)$. Then
\begin{align*}
\|\widehat{\bU}\mathrm{sgn}(\widehat{\bU}\transpose\bU) - \bU\|_{2\to\infty}
 = \Optilde\bigg(
 \frac{\varepsilon_n^{(\mathsf{op})} + \varepsilon_n^{(\mathsf{bc})}}{\sqrt{n}}
 \bigg),\quad
\|\widehat{\bU}\mathrm{sgn}(\widehat{\bU}\transpose\bU) - \bU\|_2
 = \Optilde(\varepsilon_n^{(\mathsf{op})} + \varepsilon_n^{(\mathsf{bc})}),
\end{align*}
where
\begin{align*}
\varepsilon_n^{(\mathsf{op})} = \frac{\log n}{m^{1/2}n\rho_n} + \frac{(\log n)^{1/2}}{(mn\rho_n)^{1/2}}\quad\text{and}\quad
\varepsilon_n^{(\mathsf{bc})} &= \frac{1}{n^{S + 1}} + \frac{1}{n^R}.
\end{align*}
\end{theorem}

In Theorem \ref{thm:entrywise_eigenvector_perturbation_bound}, each of the error bounds consists of two terms. The noise effect term $\varepsilon_n^{(\mathsf{op})}$ is determined by the inverse signal-to-noise ratio, whereas the bias correction term $\varepsilon_n^{(\mathsf{bc})}$ depends on the number of iterations $R$ and the number of bias calibration steps $S$ of Algorithm \ref{alg:BCJSE}. The two-to-infinity norm perturbation bound can be $n^{-1/2}$ smaller than the spectral norm perturbation bound when the invariant subspace basis matrix $\bU$ is delocalized under Assumption \ref{assumption:eigenvector_delocalization}, which is a well-observed phenomenon in a broad collection of low-rank random matrix models (see \cite{abbe-fan-wang-zhong-2020,cape2019signal,cape2017two,pmlr-v83-eldridge18a,10.1214/22-AOS2196,fan2016eigenvector,cai2021subspace,agterberg2021entrywise}, to name a selected few). The bias term also suggests the following practical approach to select $(R, S)$ in Algorithm \ref{alg:BCJSE} when $m = O(n^\alpha)$ for some $\alpha > 0$. Since $\rho_n\leq1$, then $\varepsilon_n^{(\mathsf{bc})} \leq \varepsilon_n^{(\mathsf{op})}$ if $m^{1/2}n \leq \max(n^{S + 1}, n^{R})$, and hence, the bias term is dominated by the noise effect term as long as $R\geq (1/2)(\log m)/(\log n) + 1$ and $S\geq (1/2)(\log m)/(\log n)$. 

One immediate consequence of Theorem \ref{thm:entrywise_eigenvector_perturbation_bound} is the exact community detection in MLSBM. We formally state this result in the theorem below.  
\begin{theorem}\label{thm:spectral_clustering_MLSBM}
Let $\bA = [\bA_1,\ldots,\bA_m]\sim\mathrm{MLSBM}(\tau;\bC_1,\ldots,\bC_m)$. Suppose that the number of communities $K$ is fixed and there exists an $n$-dependent sparsity factor $\rho_n\in(0, 1]$, such that the following conditions hold:
\begin{enumerate}
\item[(i)] $\mathrm{rank}(\bC_t) = K$ and $\sigma_k(\bC_t) = \Theta(\rho_n)$ for all $k\in[K]$, $t\in[m]$. 
\item[(ii)] $\max_{t\in[m]}\|\bC_t\|_{\max}\leq \rho_n$ and $\sum_{i = 1}^n\mathbbm{1}\{\tau(i) = k\} = \Theta(n)$ for all $k\in [K]$.
\item[(iii)] $m^{1/2}n\rho_n = \omega(\log n)$ and $m = O(n^\alpha)$ for some constant $\alpha > 0$.
\end{enumerate}
Denote by $\calL(n, K)$ the collection of all $n\times K$ matrices with $K$ unique rows and suppose $\widehat{\bL}$ solves the $K$-means clustering problem, \emph{i.e.}, $\widehat{\bL}  = \argmin_{\bL\in\calL(n, K)}\|\widehat{\bU}_{RS} - \bL\|_{\mathrm{F}}^2$. Let $\widehat{\tau}:[n]\to [K]$ be the community assignment estimator based on $\widehat{\bL}$ and $\calP_K$ be the set of all permutations over $[K]$. Then $\min_{\sigma\in\calP_K}\sum_{i = 1}^n\mathbbm{1}\{\widehat{\tau}(i) = \sigma(\tau(i))\} = 0$ w.h.p.. 
\end{theorem}

Since the concatenated adjacency matrix $\bA = [\bA_1,\ldots,\bA_m]$ can be viewed as a rectangular low-rank signal-plus-noise matrix, it is natural to draw some comparison between Theorem \ref{thm:entrywise_eigenvector_perturbation_bound} and prior works in this regard. First, we remark that the error bounds in Theorem \ref{thm:entrywise_eigenvector_perturbation_bound} are similar to those in Theorem 3.1 in \cite{cai2021subspace} in some flavor. Similar results can also be found in \cite{yan2021inference,agterberg2021entrywise}. The first major difference is that in \cite{cai2021subspace}, the authors focus on singular subspace estimation for general rectangular matrices with completely independent entries. Although the COSIE model generates a rectangular matrix after concatenating the adjacency matrices in a layer-by-layer fashion, matrix $\bA$ also exhibits a unique block-wise symmetric pattern. As will be made clear in the proofs, such a block-wise symmetric structure introduces additional technical challenges and requires slightly different technical tools. The second and perhaps most important difference between Theorem \ref{thm:entrywise_eigenvector_perturbation_bound} and Theorem 3.1 in \cite{cai2021subspace} is the bias correction term. Indeed, the BCJSE algorithm degenerates to the hallowed spectral embedding introduced in \cite{10.1214/22-AOS2196,cai2021subspace} when $S = R = 1$, in which case $\varepsilon_n^{(\mathsf{bc})}$ coincides with the diagonal deletion effect in \cite{cai2021subspace}. Nevertheless, the BCJSE algorithm is substantially more general than the hallowed spectral embedding and the bias correction term decreases when $R$ and $S$ increase. 

Considering that both the HeteroPCA algorithm \cite{zhang2018heteroskedastic} and the BCJSE algorithm are designed to completely remove the bias effect, we also draw some comparisons here. The entrywise subspace estimation theory of HeteroPCA is well studied in \cite{agterberg2021entrywise,yan2021inference}. The theory there works when the number of iterations grows with $n$. In contrast, our BCJSE algorithm takes advantage of the closed-form formula of the bias term $\bM$ and only requires $R, S = O(1)$ to remove the bias effect (\emph{i.e.}, $\varepsilon_n^{(\mathsf{bc})}$ is dominated by the noise effect term) under the mild condition that $m = O(n^\alpha)$ for some $\alpha > 0$. We also found in numerical studies that HeteroPCA occasionally requires large numbers of iterations to achieve the same level of accuracy as BCJSE, whereas the numbers of iterations $(R, S)$ of BCJSE are always small and can be determined before starting the algorithm. See Section \ref{sec:numerical} for further details.

We next compare our results with related literature in COSIE model analyses. 
In \cite{arroyo2021inference,zheng2022limit}, the authors developed a multiple adjacency spectral embedding (MASE) by computing 
\[\widehat{\bU}^{(\mathsf{MASE})} = \texttt{eigs}(\sum_{t = 1}^m\texttt{eigs}(\bA_t;d)\texttt{eigs}(\bA_t;d)\transpose; d).
\] In particular, the authors of \cite{zheng2022limit} established the two-to-infinity norm perturbation bound 
\[
\|\widehat{\bU}^{(\mathsf{MASE})}\mathrm{sgn}(\widehat{\bU}^{(\mathsf{MASE})\mathrm{T}}\bU) - \bU\|_{2\to\infty} = \Optilde\bigg\{\frac{(\log n)^{1/2}}{(n\rho_n^{1/2})}\bigg\}
\]
and the exact community detection result in MLSBM based on a MASE-based clustering under the condition that $n\rho_n = \omega(\log n)$ and $m$ is fixed. In contrast, Assumption \ref{assumption:condition_number} only requires that $m^{1/2}n\rho_n = \omega(\log n)$ and allows $n\rho_n = O(1)$. The MASE technique is unable to handle the regime where $n\rho_n = O(1)$, and their two-to-infinity norm perturbation bound is not sharp compared to our error bound in Theorem \ref{thm:entrywise_eigenvector_perturbation_bound}. In \cite{agterberg2022joint}, the authors studied a degree-corrected version of MASE (DC-MASE) for the multilayer degree-corrected stochastic block model and obtained a similar two-to-infinity norm perturbation bound for DC-MASE, but they still require that $n\rho_n\to\infty$ and fails to handle the regime where $n\rho_n = O (1)$. 

Given the close connection between MLSBM and the COSIE model, we also briefly compare the above results with the existing literature in spectral clustering for MLSBM. We focus on a selection of closely relevant results in \cite{doi:10.1080/01621459.2022.2054817,10.1214/21-AOS2079,paul2020spectral,chen2022global} that connect to Theorem \ref{thm:entrywise_eigenvector_perturbation_bound} and Theorem \ref{thm:spectral_clustering_MLSBM}. In \cite{doi:10.1080/01621459.2022.2054817}, the authors proposed to compute $\widehat{\bU}^{(\mathsf{BASE})} = \texttt{eigs}(\calH(\bA\bA\transpose); d)$ (\emph{i.e.}, running Algorithm \ref{alg:BCJSE} with $S = R = 1$) and used it for community detection in MLSBM. Note that the assumptions in \cite{doi:10.1080/01621459.2022.2054817} are that $n\rho_n = O(1)$ and $m^{1/2}n\rho_n = \Omega(\log^{1/2}(m + n))$, which are slightly different than Assumption \ref{assumption:condition_number}. The proof of Theorem 1 there implies the spectral norm perturbation bound 
\[
\min_{\bW\in\mathbb{O}(d)}\|\widehat{\bU}^{(\mathsf{BASE})}\bW - \bU\|_2 = \Optilde\bigg\{\frac{1}{n} + \frac{\log^{1/2}(m + n)}{m^{1/2}n\rho_n}\bigg\},
\] where the $n^{-1}$ term is the remaining bias effect due to diagonal deletion of $\bA\bA\transpose$ and coincides with $\varepsilon_n^{(\mathsf{bc})}$ when $R = 1$, and the $\log^{1/2}(m + n)/(m^{1/2}n\rho_n)$ term is similar to $\varepsilon_n^{(\mathsf{op})}$. In \cite{paul2020spectral}, under the condition that $mn\rho_n \geq C\log n$ for some large constant $C > 0$, the authors studied the co-regularized spectral clustering (co-reg) method \cite{NeurIPS2011_31839b03} and established a sub-optimal Frobenius norm perturbation bound 
\[
\min_{\bW\in\mathbb{O}(d)}\|\widehat{\bU}^{(\mathsf{coreg})}\bW - \bU\|_{\mathrm{F}} = \Optilde\bigg\{\frac{(\log n)^{1/4}}{(mn\rho_n)^{1/4}}\bigg\}.
\] 
See the proof of Theorem 2 there for details. In \cite{10.1214/21-AOS2079,chen2022global}, the authors studied a slightly more general inhomogeneous MLSBM where the layer-wise community assignments can be viewed as noisy versions of a global community assignment. The authors of \cite{10.1214/21-AOS2079} designed a regularized tensor decomposition approach called TWIST and showed that 
\[
\min_{\bW\in\mathbb{O}(d)}\|\widehat{\bU}^{(\mathsf{TWIST})}\bW - \bU\|_2 = \Optilde\bigg\{\frac{(\log n)^{1/2}}{(mn\rho_n)^{1/2}}\bigg\}
\] 
under the condition that $mn\rho_n\geq C(\log n)^4$ for some large constant $C > 0$, which is slightly stronger than our Assumption \ref{assumption:condition_number}. See Corollary 1 there for details. The authors of \cite{chen2022global} focused on the inhomogeneous MLSBM with two communities and balanced community sizes and showed that 
\[
\min_{\bW\in\mathbb{O}(d)}\|\texttt{eigs}(\bar{\bA}; d)\bW - \bU\|_{\mathrm{F}} = \Optilde\bigg\{\frac{1}{(mn\rho_n)^{1/2}}\bigg\},
\] where $\bar{\bA} = \sum_{t = 1}^m\omega_t\bA_t$ is a weighted average network adjacency matrix and $(\omega_t)_{t = 1}^m$ are the weights. Nevertheless, they require $\lambda_2(\expect\bar{\bA}) > 0$ and it is not entirely clear how to choose the weights when layer-wise block probability matrices contain negative eigenvalues. For the ease of readers' reference, we summarize the comparison of the above results in Table \ref{tab:comparison_perturbation_clustering} under Assumption \ref{assumption:eigenvector_delocalization}. In Table \ref{tab:comparison_perturbation_clustering}, the term ``rectangular model'' means rectangular random matrices with completely independent entries and low expected ranks, $\widehat{\bU}$ is a generic estimator depending on the context, and $\bW$ is the matrix sign $\mathrm{sgn}(\widehat{\bU}\transpose\bU)$. Note that the setups in \cite{yan2021inference,agterberg2021entrywise,10.1214/22-AOS2196} are designed for rectangular matrices with independent entries so that their results do not directly apply to the COSIE model. Therefore, we only list their generic error bounds by viewing $\bA$ as a rectangular low-rank random matrix in Table \ref{tab:comparison_perturbation_clustering}. A similar comment also applies to \cite{10.1214/21-AOS2079,chen2022global}.

It is also worth remarking that most of the aforementioned existing results did not provide entrywise eigenvector perturbation bounds and only addressed the weak recovery of community membership, \emph{i.e.}, the proportion of misclustered vertices goes to zero with probability approaching one. In contrast, Theorem \ref{thm:spectral_clustering_MLSBM} establishes the exact community detection. The only exceptions in Table \ref{tab:comparison_perturbation_clustering} are \cite{chen2022global} and \cite{zheng2022limit}. In \cite{chen2022global}, the authors designed a two-stage algorithm that achieves the information limit of community detection. However, their strong theoretical results require that the nonzero eigenvalues of layer-wise edge probability matrices are positive, whereas our theory drops such a restrictive assumption. In \cite{zheng2022limit}, the authors require that $m = O(1)$ and $n\rho_n = \omega(\log n)$ for the exact recovery of MASE-based spectral clustering, and their conditions are substantially stronger than Assumption \ref{assumption:condition_number}.

\begin{table}[!t]
\tiny
\caption{Comparison of perturbation bound results under Assumption \ref{assumption:eigenvector_delocalization}, where $\widehat{\bU}$ is a generic estimator for $\bU$ and $\bW = \mathrm{sgn}(\widehat{\bU}\transpose\bU)$ \label{tab:comparison_perturbation_clustering}}
\centering
\begin{tabular}{|c|c|c|c|}
\hline
 Method & Model & Assumption & Perturbation bound\\
\hline
BCJSE & COSIE & $\frac{\log n}{m^{1/2}n\rho_n}$, $\log m = O(\log n)$ & $\|\widehat{\bU}\bW - \bU\|_{2\to\infty}\lesssim \frac{\log n}{m^{1/2}n^{3/2}\rho_n} + \frac{(\log n)^{1/2}}{(mn^2\rho_n)^{1/2}}$\\
\hline
$\texttt{eigs}(\calH(\bA\bA\transpose);d)$ \cite{doi:10.1080/01621459.2022.2054817} & COSIE & $\frac{\log^{1/2}(m + n)}{m^{1/2}n\rho_n}\to0$, $n\rho_n = O(1)$ & $\|\widehat{\bU}\bW - \bU\|_{\mathrm{F}}\lesssim \frac{\log^{1/2}(m + n)}{(mn\rho_n)^{1/2}} + \frac{1}{n} $ \\
\hline
$\texttt{eigs}(\calH(\bA\bA\transpose);d)$ \cite{cai2021subspace} & Rectangular & $\frac{\log(m + n)}{m^{1/2}n\rho_n}\to0$ & $\|\widehat{\bU}\bW - \bU\|_{2\to\infty}\lesssim \frac{\log(mn)}{m^{1/2}n^{3/2}\rho_n} + \frac{\log^{1/2}(mn)}{(mn^2\rho_n)^{1/2}} + \frac{1}{n^{3/2}} $ \\
\hline
HeteroPCA \cite{yan2021inference} & Rectangular & $\frac{\log^2(mn)}{m^{1/2}n\rho_n} = O(1)$ & $\|\widehat{\bU}\bW - \bU\|_{2\to\infty}\lesssim \frac{\log(mn)}{m^{1/2}n\rho_n} + \frac{\log^{1/2}(mn)}{(mn\rho_n)^{1/2}} $ \\
\hline
HeteroPCA \cite{agterberg2021entrywise} & Rectangular & $\frac{\log(mn)}{n\rho_n} = O(1)$ & $\|\widehat{\bU}\bW - \bU\|_{2\to\infty}\lesssim \frac{\log(mn)}{m^{1/2}n\rho_n} + \frac{\log^{1/2}(mn)}{(mn\rho_n)^{1/2}} $ \\
\hline
MASE \cite{zheng2022limit} & COSIE & $\frac{\log n}{n\rho_n}\to0$, $m = O(1)$ & $\|\widehat{\bU}\bW - \bU\|_{2\to\infty}\lesssim \frac{\log^{1/2}(m + n)}{n\rho_n^{1/2}}$ \\
\hline
Co-reg \cite{paul2020spectral} & COSIE & $\frac{\log(m + n)}{mn\rho_n}\to0$ & $\|\widehat{\bU}\bW - \bU\|_{\mathrm{F}}\lesssim \frac{\log^{1/4}(m + n)}{(mn\rho_n)^{1/4}}$ \\
\hline
TWIST \cite{jin2022improvements} & IMLSBM & $\frac{(\log n)^4}{mn\rho_n}\to0$ & $\|\widehat{\bU}\bW - \bU\|_2\lesssim \frac{\log^{1/2}(m + n)}{(mn\rho_n)^{1/2}}$ \\
\hline
$\texttt{eigs}(\bar{\bA}; d)$ \cite{chen2022global}& IMLSBM & Balanced two-community, $\lambda_2(\bB_t) > 0$ & $\|\widehat{\bU}\bW - \bU\|_2\lesssim \frac{1}{(mn\rho_n)^{1/2}}$ \\
\hline
\end{tabular}
\end{table}

\subsection{Entrywise Eigenvector Limit Theorem}
\label{sub:entrywise_eigenvector_CLT}
In this subsection, we present the entrywise eigenvector limit theorem for BCJSE, which is our second main result. 
\begin{theorem}\label{thm:Rowwise_CLT}
Suppose Assumptions \ref{assumption:eigenvector_delocalization}--\ref{assumption:condition_number} hold. 
For each $i\in[n]$, define
\begin{align*}
\bF_i & = \sum_{t = 1}^m\bB_t\bU\transpose\mathrm{diag}(\sigma_{ti1}^2,\ldots,\sigma_{tin}^2)\bU\bB_t,\;
\bG_{i} = \bU\transpose\mathrm{diag}\bigg(\sum_{t = 1}^m\sum_{j = 1}^n\sigma_{tij}^2\sigma_{tj1}^2,\ldots,\sum_{t = 1}^m\sum_{j = 1}^n\sigma_{tij}^2\sigma_{tjn}^2\bigg)\bU,
\end{align*}
and $\bGamma_i = (\bB\bB\transpose)^{-1}(\bF_i + \bG_i)(\bB\bB\transpose)^{-1}$, where $\sigma_{tij}^2:=\var(E_{tij})$. Let $\theta_n = (n\rho_n\wedge 1)^{1/2}$ and further, assume the following conditions hold:
\begin{enumerate}
\item[(i)] $m^{1/2}n\rho_n = \omega(\theta_n(\log n)^{3/2})$, $mn\rho_n = \omega((\theta_n\log n)^2)$, and $m^{1/2}(n\rho_n)^{3/2} = \omega(\theta_n(\log n)^2)$.
\item[(ii)] $\lambda_d(\bF_i) = \Omega(mn^2\rho_n^3)$ and $\lambda_d(\bG_i) = \Omega(mn\rho_n^2)$.
\end{enumerate}
If $R > (\alpha + 1)/2$, $S > (\alpha - 1)/2$, and $R,S = O(1)$, then for any fixed vector $\bz\in\mathbb{R}^d$, 
\begin{align}
\label{eqn:rowwise_CLT_expansion}
\bz\transpose\bGamma_i^{-1/2}(\widehat{\bU}\mathrm{sgn}(\widehat{\bU}\transpose\bU) - \bU)\transpose\be_i
& = \be_i\transpose\calH(\bE\bE\transpose)
\bU(\bB\bB\transpose)^{-1}
\bGamma_i^{-1/2}\bz
\\&\quad
 + \sum_{t = 1}^m\sum_{j = 1}^nE_{tij}\be_j\transpose\bP_t
 \bU(\bB\bB\transpose)^{-1}
 \bGamma_i^{-1/2}\bz + \optilde(1)\nonumber.
\end{align}
Furthermore, $\bGamma_i^{-1/2}(\widehat{\bU}\mathrm{sgn}(\widehat{\bU}\transpose\bU) - \bU)\transpose\be_i\overset{\calL}{\to}\mathrm{N}_d(\zero_d, \eye_d)$. 
\end{theorem}
Theorem \ref{thm:Rowwise_CLT} establishes that the rows of the BCJSE $\widehat{\bU}$ are approximately Gaussian under slightly stronger conditions than those required in Theorem \ref{thm:entrywise_eigenvector_perturbation_bound}. In the asymptotic expansion \eqref{eqn:rowwise_CLT_expansion}, the leading term contains two parts: the second term is a sum of independent mean-zero random variables, and the first term is a more involved quadratic function of $\bE$. As will be made clear in the proof, the first term is asymptotically negligible when $n\rho_n\to\infty$ and the second term determines the asymptotic distribution of the $i$th row of $\widehat{\bU}$ through the Lyapunov central limit theorem (see, for example, \cite{chung2001course}), in which case $\bGamma_i\approx (\bB\bB\transpose)^{-1}\bF_i(\bB\bB\transpose)^{-1}$. The challenging regime occurs when $n\rho_n = O(1)$ and both the first term and the second term on the right-hand side of expansion \eqref{eqn:rowwise_CLT_expansion} jointly determine the asymptotic distribution of the $i$th row of $\widehat{\bU}$. Note that the sum of these two terms is no longer a sum of independent mean-zero random variables, and Lyapunov central limit theorem is not applicable. The key observation is that the leading term in \eqref{eqn:rowwise_CLT_expansion} can be written as a martingale, and we apply the martingale central limit theorem to establish the desired asymptotic normality by carefully calculating the related martingale moment bounds. See Lemma \ref{lemma:martingale_construction} for further details.

Prior work on the entrywise eigenvector limit results for multilayer networks is slightly narrower than those on the general perturbation bounds and community detection. Given that the COSIE model connects to both rectangular low-rank signal-plus-noise matrix models and multilayer networks, we draw some comparison between Theorem \ref{thm:Rowwise_CLT} and existing limit results along these two directions, particularly those in \cite{zheng2022limit,agterberg2021entrywise,yan2021inference}. In \cite{zheng2022limit}, the authors established the asymptotic normality for the rows of MASE under the stronger condition $n\rho_n = \omega(\log n)$ and $m = O(1)$. As mentioned earlier, MASE fails to handle the case where $m\to \infty$ and $n\rho_n = O(1)$, whereas the condition of Theorem \ref{thm:Rowwise_CLT} allows for $m\to\infty$ and $n\rho_n = O(1)$. In \cite{agterberg2021entrywise,yan2021inference}, the authors established the asymptotic normality for the rows of HeteroPCA for general rectangular low-rank signal-plus-noise matrices. The difference between  \cite{yan2021inference} and \cite{agterberg2021entrywise} is that the authors of \cite{agterberg2021entrywise} handled heteroskedasticity and dependence, whereas the authors of \cite{yan2021inference} allowed missing data. However, as mentioned earlier, the rectangular model considered in \cite{yan2021inference,agterberg2021entrywise} does not contain a block-wise symmetric structure, and the results there do not apply directly to our undirected multilayer network setup. Furthermore, the underlying proof techniques are fundamentally different: in \cite{zheng2022limit,agterberg2021entrywise,yan2021inference}, the leading terms in the entrywise eigenvector expansion are sums of independent mean-zero random variables, and their proofs are based on Lyapunov central limit theorem for sum of independent random variables. Nevertheless, as mentioned earlier, in the multilayer undirected network setup with $n\rho_n = O(1)$, the leading term in the entrywise eigenvector expansion formula \eqref{eqn:rowwise_CLT_expansion} is much more involved and requires the construction of a martingale sequence, together with the application of the martingale central limit theorem.  

Below, we provide an immediate application of Theorem \ref{thm:Rowwise_CLT} to the vertex membership inference in MLMM models. The authors of \cite{fan2019simple} proposed to test the equality of membership profiles of any given two vertices in a single-layer mixed membership model. This inference task for pairwise vertex comparison may be of interest in many practical applications, such as stock market investment studies and legislation. The multilayer version of this problem, which has not been explored before to our limited best knowledge, can be described as follows. Recall that, an MLMM model with a nonnegative membership profile matrix $\bZ$ satisfying $\bZ\one_n = \one_d$ and block probability matrices $(\bC_t)_{t = 1}^m$ can be equivalently represented by $\mathrm{COSIE}(\bU;\bB_1,\ldots,\bB_m)$, where $\bU = \bZ(\bZ\transpose\bZ)^{-1/2}$ and $\bB_t = (\bZ\transpose\bZ)^{1/2}\bC_t(\bZ\transpose\bZ)^{1/2}$, if $\bZ\transpose\bZ$ is invertible. Given any two vertices $i_1,i_2\in[n]$, $i_1\neq i_2$, we now consider the hypothesis testing problem $H_0:\bz_{i_1} = \bz_{i_2}$ versus $H_A:\bz_{i_1}\neq\bz_{i_2}$, where $\bz_i$ is the $i$th row of $\bZ$. It is straightforward to see that $\be_i\transpose\bU = \be_j\transpose\bU$ if $\bz_i = \bz_j$ (see, for example, \cite{fan2019simple}), so it is conceivable from Theorem \ref{thm:Rowwise_CLT} that under $H_0$, $(\bGamma_{i_1} + \bGamma_{i_2})^{-1/2}\mathrm{sgn}(\bU\transpose\widehat{\bU})(\widehat{\btheta}_{i_1} - \widehat{\btheta}_{i_2})$ approximately follows the $d$-dimensional standard normal distribution, where we use $\widehat{\btheta}_i$ to denote the $i$th row of $\widehat{\bU}$, and $\widehat{\bU}$ is the output of Algorithm \ref{alg:BCJSE}. Equivalently, this suggests that the quadratic form $(\widehat{\btheta}_{i_1} - \widehat{\btheta}_{i_2})\mathrm{sgn}(\bU\transpose\widehat{\bU})\transpose(\bGamma_{i_1} + \bGamma_{i_2})^{-1}\mathrm{sgn}(\bU\transpose\widehat{\bU})(\widehat{\btheta}_{i_1} - \widehat{\btheta}_{i_2})$ approximately follows the chi-squared distribution with $d$ degree of freedom. Because $\bGamma_{i_1}$, $\bGamma_{i_2}$ are not known and need to be estimated, we now leverage the above intuition and design a test statistic using the plug-in principle. For any $t\in[m]$, $i, j\in[n]$, we define the following plug-in estimators:
\begin{align*}
\widehat{\bB}_t& = \widehat{\bU}\transpose\bA_t\widehat{\bU},\;
\widehat{\bB}    = [\widehat{\bB}_1,\ldots,\widehat{\bB}_m],\;
\widehat{\bP}_t  = \widehat{\bU}\widehat{\bB}_t\widehat{\bU}\transpose,\;
\widehat{P}_{tij} = \be_i\transpose\widehat{\bP}_t\be_j,\;
\widehat{\sigma}_{tij}^2 = \widehat{P}_{tij}(1 - \widehat{P}_{tij}),\\
\widehat{\bF}_i  & = \sum_{t = 1}^m\widehat{\bB}_t\widehat{\bU}\transpose\mathrm{diag}(\widehat{\sigma}_{ti1}^2,\ldots,\widehat{\sigma}_{tin}^2)\widehat{\bU}\widehat{\bB}_t,\;
\widehat{\bG}_{i} = \widehat{\bU}\transpose\mathrm{diag}\bigg(\sum_{t = 1}^m\sum_{j = 1}^n\widehat{\sigma}_{tij}^2\widehat{\sigma}_{tj1}^2,\ldots,\sum_{t = 1}^m\sum_{j = 1}^n\widehat{\sigma}_{tij}^2\widehat{\sigma}_{tjn}^2\bigg)\widehat{\bU},\\
\widehat{\bGamma}_i &= (\widehat{\bB}\widehat{\bB}\transpose)^{-1}(\widehat{\bF}_i + \widehat{\bG}_i)(\widehat{\bB}\widehat{\bB}\transpose)^{-1},\;
T_{i_1i_2} = (\widehat{\btheta}_{i_1} - \widehat{\btheta}_{i_2})\transpose(\widehat{\bGamma}_{i_1} + \widehat{\bGamma}_{i_2})^{-1}(\widehat{\btheta}_{i_1} - \widehat{\btheta}_{i_2}).
\end{align*}
where $\widehat{\btheta}_i:=\widehat{\bU}\transpose\be_i$. Here, $T_{i_1i_2}$ is our test statistic for testing $H_0:\bz_{i_1} = \bz_{i_2}$ versus $H_A:\bz_{i_1}\neq \bz_{i_2}$. 

\begin{theorem}\label{thm:testing_membership_profiles}
Let $\bA = [\bA_1,\ldots,\bA_m]\sim\mathrm{MLMM}(\bZ; \bC_1,\ldots,\bC_m)$. Suppose that $d$ is fixed and there exists an $n$-dependent sparsity factor $\rho_n\in(0, 1]$, such that the following conditions hold:
\begin{enumerate}
    \item[(i)] Let $\theta_n = (n\rho_n\wedge 1)^{1/2}$. Then $n^2\rho_n = \omega(\log n)$, $m^{1/2}n\rho_n = \omega(\theta_n(\log n)^{3/2})$, $mn\rho_n = \omega((\theta_n\log n)^2)$, $m^{1/2}(n\rho_n)^{3/2} = \omega(\theta_n(\log n)^2 + \log n)$, and $m^{1/2}(n\rho_n)^2 = \omega((\log n)^{3/2})$.

    \item[(ii)] $R,S = O(1)$, $R > (\alpha + 1)/2$, and $S > (\alpha - 1)/2$. 

    \item[(iii)] $\lambda_k(\bZ\transpose\bZ) = \Theta(n)$ for all $k\in[d]$, $\min_{t\in[m]}\sigma_d(\bC_t) = \Omega(\rho_n)$, $\max_{t\in[m]}\sigma_1(\bC_t) = O(\rho_n)$, 

    $\min_{t\in[m],i,j\in[n]}\bz_i\transpose\bC_t\bz_j = \Omega(\rho_n)$, $\max_{t\in[m],i,j\in[n]}(1 - \bz_i\transpose\bC_t\bz_j) = \Theta(1)$, and $\min_{t\in[m],i,j\in[n]}(1 - \bz_i\transpose\bC_t\bz_j) = \Theta(1)$, where $\bz_i\in\mathbb{R}^d$ denotes the $i$th row of $\bZ$. 

\end{enumerate}
Then:
\begin{enumerate}
    \item[(a)] Under the null hypothesis $H_0:\bz_{i_1} = \bz_{i_2}$, where $i_1\neq i_2$, we have $T_{i_1i_2}\overset{\calL}{\to}\chi^2_d$ as $n\to\infty$.

    \item[(b)] Under the contiguous alternative hypothesis $H_A:\bz_{i_1}\neq\bz_{i_2}$ but $\|\bz_{i_1} - \bz_{i_2}\|_2 = \omega((mn\rho_n)^{-1/2}\theta_n^{-1})$, for any arbitrarily large constant $C > 0$, we have $\prob(T_{i_1i_2} > C)\to 1$ as $n\to\infty$.

    \item[(c)] Under the alternative hypothesis $H_A:\bz_{i_1}\neq\bz_{i_2}$ but $(\bGamma_{i_1} + \bGamma_{i_2})^{-1/2}(\bZ\transpose\bZ)^{-1/2}(\bz_{i_1} - \bz_{i_2})\to\bmu$, we have $T_{i_1i_2}\overset{\calL}{\to}\chi_d^2(\|\bmu\|_2^2)$, where $\chi_d^2(\|\bmu\|_2^2)$ denotes the non-central chi-squared distribution with $d$ degree of freedom and non-central parameter $\|\bmu\|_2^2$. 
\end{enumerate}
\end{theorem}

\subsection{Proof architecture}
\label{sub:proof_architecture}

We now briefly discuss the proof architecture of the main results. Let $\widehat{\bU}_{RS}$ be the output of the BCJSE algorithm. By definition, $\widehat{\bU}_{RS} = \texttt{eigs}(\calH(\bA\bA\transpose) - \widehat{\bM}_{RS}; d)$, and let $\widehat{\bS}_{RS}$ be the diagonal matrix of the associated eigenvalues of $\calH(\bA\bA\transpose) - \widehat{\bM}_{RS}$. Denote by $\bW_{RS} = \mathrm{sgn}(\widehat{\bU}_{RS}\transpose\bV)$. Let $\bQ_1 = \texttt{eigs}(\bB\bB\transpose; d)$ and $\bV = \bU\bQ_1$. Then $\bV$ is the eigenvector matrix of $\bP\bP\transpose$ corresponding to the nonzero eigenvalues, and we take $\bS$ as the diagonal matrix of these eigenvalues, such that $\bP\bP\transpose = \bV\bS\bV\transpose$. The proof is centered around the following keystone decomposition motivated by \cite{xie2021entrywise}: 
\begin{align}\label{eqn:eigenvector_decomposition_main}
\widehat{\bU}_{RS}\bW_{RS} - \bV = 
\{\calH(\bA\bA\transpose) - \bP\bP\transpose - \bM\}\bV\bS^{-1} + \bT^{(RS)},
\end{align}
where $\bT^{(RS)} = \bR_1^{(RS)} + \bR_2^{(RS)} + \bR_3^{(RS)} + \bR_4^{(RS)} + \bR_5^{(RS)} + \bR_6^{(RS)} + \bR_7^{(RS)}$,
\begin{equation}\label{eqn:eigenvector_decomposition}
\begin{aligned}
\bR_1^{(RS)} & = 
\{\calH(\bA\bA\transpose) - \bP\bP\transpose - \bM\}
(\widehat{\bU}_{RS} - \bV\bW_{RS}\transpose)\widehat{\bS}_{RS}^{-1}\bW_{RS},\quad
\bR_2^{(RS)} =  
(\bM - \widehat{\bM}_{RS})\bV\bS^{-1},\\
\bR_3^{(RS)} &= (\bM - \widehat{\bM}_{RS})(\widehat{\bU}_{RS} - \bV\bW_{RS}\transpose)\widehat{\bS}^{-1}_{RS}\bW_{RS},\quad
\bR_4^{(RS)} = \bV(\bS\bV\transpose\widehat{\bU}_{RS}  - \bV\transpose\widehat{\bU}_{RS}\widehat{\bS}_{RS})\widehat{\bS}^{-1}_{RS}\bW_{RS},\\
\bR_5^{(RS)} & = \bV(\bV\transpose\widehat{\bU}_{RS} - \bW_{RS}\transpose)\bW_{RS},\quad
\bR_6^{(RS)} = 
\{\calH(\bA\bA\transpose) - \bP\bP\transpose - \bM\}
\bV(\bW_{RS}\transpose\widehat{\bS}_{RS}^{-1} - \bS^{-1}\bW_{RS}\transpose)\bW_{RS},\\
\bR_7^{(RS)} & = (\bM - \widehat{\bM}_{RS})\bV(\bW_{RS}\transpose\widehat{\bS}_{RS}^{-1} - \bS^{-1}\bW_{RS}\transpose)\bW_{RS}.
\end{aligned}
\end{equation}
The above decomposition can be easily verified by invoking the definitions of $\bR_1^{(RS)}$ through $\bR_7^{(RS)}$.
The proof of Theorem \ref{thm:entrywise_eigenvector_perturbation_bound} is based on a recursive relation between $\|\widehat{\bU}_{RS}\bW_{RS} - \bV\|_{2\to\infty}$ and $\|\widehat{\bU}_{(R - 1)S}\bW_{(R - 1)S} - \bV\|_{2\to\infty}$, which is formally stated in terms of Lemma \ref{lemma:recursive_two_to_infinity_error_bound} below. Theorem \ref{thm:entrywise_eigenvector_perturbation_bound} then follows from an induction over $R$.
\begin{lemma}\label{lemma:recursive_two_to_infinity_error_bound}
Suppose Assumptions \ref{assumption:eigenvector_delocalization} and \ref{assumption:condition_number} hold and $R, S = O(1)$. Then 
\begin{align*}
\|\widehat{\bU}_{RS}\bW_{RS} - \bV\|_{2\to\infty} &= \Optilde\bigg(\frac{\varepsilon_n^{(\mathsf{op})}}{\sqrt{n}}\bigg) + \Optilde\bigg(\frac{1}{mn^{5/2}\rho_n^2}\bigg)\|\bM - \widehat{\bM}_{RS}\|_2,\\
\frac{1}{mn^{5/2}\rho_n^2}\|\bM - \widehat{\bM}_{RS}\|_2 &= \Optilde\bigg(\frac{1}{n}\bigg)\|\widehat{\bU}_{(R - 1)S}\bW_{(R - 1)S} - \bV\|_{2\to\infty} + \Optilde\bigg(\frac{1}{n^{S + 3/2}} + \frac{\varepsilon_n^{(\mathsf{op})}}{n^2}\bigg).
\end{align*}
\end{lemma}
The proof Lemma \ref{lemma:recursive_two_to_infinity_error_bound} breaks down into establishing error bounds for $\|\bR_1^{(RS)}\|_{2\to\infty}$ through $\|\bR^{(RS)}_7\|_{2\to\infty}$. Among these terms, the analyses of $\|\bR_2^{(RS)}\|_{2\to\infty}$ through $\|\bR^{(RS)}_7\|_{2\to\infty}$ are relatively straightforward based on the classical matrix perturbation tools \cite{doi:10.1137/0707001}. See Appendix \ref{sec:simple_remainder_analyses} for further details. 
The analysis of $\|\bR^{(RS)}_1\|_{2\to\infty}$, which is formally stated in Lemma \ref{lemma:R1_remainder} below and will be proved in Appendix \ref{sec:LOO_analysis}, is more involved and requires decoupling arguments based on the elegant leave-one-out and leave-two-out analyses (see \cite{xie2021entrywise,10.1214/19-AOS1854,javanmard2015biasing,10.1214/22-AOS2196,cai2021subspace,agterberg2021entrywise,agterberg2022joint,yan2021inference}). 
\begin{lemma}\label{lemma:R1_remainder}
Suppose Assumptions \ref{assumption:eigenvector_delocalization} and \ref{assumption:condition_number} hold, and $R,S = O(1)$. Then
\[
\|\bR_1^{(RS)}\|_{2\to\infty} = \Optilde\bigg(\frac{\varepsilon_n^{(\mathsf{op})}}{\sqrt{n}}
\bigg) + \Optilde\bigg(\frac{1}{mn^{5/2}\rho_n^2}\bigg)\|\bM - \widehat{\bM}_{RS}\|_2.
\]
\end{lemma} 
Note that because of the unique block-wise symmetric structure of the concatenated network adjacency matrix $\bA = [\bA_1,\ldots,\bA_m]$, where $\bA_t$'s are symmetric, the leave-one-out and leave-two-out analyses cause extra complication compared to those appearing in the literature of rectangular random matrices with completely independent entries. In particular, our analyses show the emergence of complicated polynomials of $\bE$ that are unique to our setting, and the error control of these quadratic forms is established by delicate analyses of their higher-order moment bounds. See Section \ref{sec:LOO_analysis} for further details. 

Regarding the proof of Theorem \ref{thm:Rowwise_CLT}, the first step is to obtain the following sharpened error bound for $\|\bT^{(RS)}\|_{2\to\infty}$. 
\begin{lemma}\label{lemma:sharpened_remainder}
Suppose the conditions of Theorem \ref{thm:Rowwise_CLT} hold. Then 
\[
\|\bT^{(RS)}\|_{2\to\infty} = \optilde\bigg(\frac{1}{m^{1/2}n\rho_n^{1/2}\theta_n}\bigg) + \Optilde\bigg(\frac{1}{n^{S + 3/2}} + \frac{1}{n^{R + 1/2}}\bigg).
\]
\end{lemma}
The proof of Lemma \ref{lemma:sharpened_remainder}, which is deferred to Appendix \ref{sub:proof_of_eigenvector_CLT}, is similar to that of Lemma \ref{lemma:R1_remainder}. The key difference is that we take advantage of the sharp error bounds in Theorem \ref{thm:entrywise_eigenvector_perturbation_bound} to obtain the necessary refinement for $\|\bR_1^{(RS)}\|_{2\to\infty}$. The remaining part of the proof of Theorem \ref{thm:Rowwise_CLT} is to show the asymptotic normality of $\be_i\transpose\{\calH(\bA\bA\transpose) - \bP\bP\transpose - \bM\}\bV\bS^{-1}$. As mentioned before, this is a nontrivial task because $\be_i\transpose\{\calH(\bA\bA\transpose) - \bP\bP\transpose - \bM\}\bV\bS^{-1}$ cannot be written as a sum of independent mean-zero random variables plus some negligible term. We borrow the idea from \cite{fan-fan-han-lv-2020} and rewrite it as a sum of  martingale difference sequence, so that the martingale central limit theorem can be applied. This is formally stated in Lemma \ref{lemma:martingale_construction} below.
\begin{lemma}\label{lemma:martingale_construction}
Let $e_{ij}$ denote the $j$th entry of the standard basis vector $\be_i$, \emph{i.e.}, $e_{ij} = \be_i\transpose\be_j = \mathbbm{1}(i = j)$. For any deterministic vector $\bz\in\mathbb{R}^d$, $t\in[m]$, $i,j,j_1,j_2\in[n]$, $j_1\leq j_2$, denote by
\begin{align}
\label{eqn:gamma_xi}
\gamma_{ij}(\bz)& = \be_j\transpose\bV\bS^{-1}\bQ_1\transpose\bGamma_i^{-1/2}\bz,\quad
\xi_{tij}(\bz) = \be_j\transpose\bP_t\bV\bS^{-1}\bQ_1\transpose\bGamma_i^{-1/2}\bz,\\
\label{eqn:btj1j2}
b_{tij_1j_2}(\bz)& = \iota_{j_1j_2}
\sum_{j_3 = 1}^{j_1 - 1}
E_{tj_2j_3}\left\{e_{ij_1}
\gamma_{ij_3}(\bz) + 
e_{ij_3}\gamma_{ij_1}(\bz)\right\}
 + 
\iota_{j_1j_2}\sum_{j_3 = 1}^{j_2 - 1}
E_{tj_1j_3}\left\{e_{ij_2}
\gamma_{ij_3}(\bz) + 
e_{ij_3}\gamma_{ij_2}(\bz)\right\},\\
\label{eqn:ctj1j2}
c_{tij_1j_2}(\bz)& = \iota_{j_1j_2}\{e_{ij_1}\xi_{tij_2}(\bz) + e_{ij_2}\xi_{tij_1}(\bz)\},
\end{align}
where $\iota_{j_1j_2} = \mathbbm{1}(j_1 < j_2) + \mathbbm{1}(j_1 = j_2)/2$. For any $(t, j_1, j_2)\in[m]\times [n]\times [n]$ with $j_1\leq j_2$, define the relabeling function 
\begin{align}
\label{eqn:relabeling_function}
\alpha(t, j_1, j_2) = \frac{1}{2}(t - 1)n(n + 1) + j_1 + \frac{1}{2}j_2(j_2 - 1).
\end{align}
Let $N_n = (1/2)mn(n + 1)$. For any $\alpha\in[N_n]$, define $\sigma$-field $\calF_{n\alpha} = \sigma(\{E_{tj_1j_2}:\alpha(t, j_1, j_2)\leq \alpha\})$ and random variable $Y_{n\alpha} = \sum_{t = 1}^m\sum_{j_1,j_2\in[n],j_1\leq j_2}\mathbbm{1}\{\alpha(t, j_1, j_2) = \alpha\}E_{tj_1j_2}(b_{tj_1j_2} + c_{tj_1j_2})$. Let $\calF_{n0} = \varnothing$. Then:
\begin{itemize}
\item[(a)] $\alpha(\cdot,\cdot,\cdot)$ is a one-to-one function from $[m]\times\{(j_1,j_2)\in[n]\times[n]:j_1\leq j_2\}$ to $[N_n]$.
\item [(b)] For any $t\in[m],j_1,j_2\in[n],j_1\leq j_2$, $E_{tj_1j_2}$ is $\calF_{n\alpha(t, j_1, j_2)}$-measurable and $b_{tj_1j_2}$ is $\calF_{n,\alpha(t, j_1, j_2) - 1}$-measurable.
\item [(c)] $(\sum_{\beta = 1}^\alpha Y_{n\beta})_{\alpha = 1}^{N_n}$ is a martingale with respect to the filtration $(\calF_{n\alpha})_{\alpha = 0}^{N_n - 1}$, and 
\[
\be_i\transpose\{\calH(\bA\bA\transpose) - \bP\bP\transpose - \bM\}\bV\bS^{-1}\bQ_1\transpose\bGamma_i^{-1/2} = \sum_{\alpha = 1}^{N_n}Y_{n\alpha} + \optilde(1).
\]
\end{itemize}
\end{lemma}

\section{Numerical Experiments}
\label{sec:numerical}

This section illustrates the practical performance of the proposed BCJSE algorithm via numerical experiments. Specifically, Section \ref{sub:numerical_subspace_estimation} focuses on the invariant subspace estimation performance, and Section \ref{sub:numerical_testing_MLMM} complements the theory of the hypothesis testing procedure for comparing vertex membership profiles in MLMM established in Theorem \ref{thm:testing_membership_profiles}. 

\subsection{Subspace estimation performance}
\label{sub:numerical_subspace_estimation}

Let $(\bC_t)_{t = 1}^m$ be a collection of $2\times 2$ matrices defined as follows:
\begin{align}\label{eqn:numerical_block_probability_matrix}
\bC_1 = \ldots = \bC_{m/2} = \rho_n\begin{bmatrix}
a & b\\b & a
\end{bmatrix},\;
\bC_{m/2 + 1} = \ldots = \bC_{m} = \rho_n\begin{bmatrix}
b & a\\a & b
\end{bmatrix},
\end{align}
where $a > b > 0$, $m$ is a positive even integer, and $\rho_n\in(0, 1]$ is the sparsity factor. Let $n$ be a positive even integer, $0 = t_1 < \ldots < t_{n/2} = 1$ be equidistant points over $[0, 1]$, and for any $i\in[n/2]$, define $x_{i1} = \sin(\pi t_i/2)$, $x_{(n/2 + i)1} = \cos(\pi t_i/2)$, $x_{i2} = \cos(\pi t_i/2)$, $x_{(n/2 + i)2} = \sin(\pi t_i/2)$, $\bx_1 = [x_{11},\ldots,x_{n1}]\transpose$, $\bx_2 = [x_{12},\ldots,x_{n2}]\transpose$, and $\bX = [\bx_1,\bx_2]\in\mathbb{R}^{n\times 2}$. Define $\bU = \bX(\bX\transpose\bX)^{-1/2}$, $\bB_t = (\bX\transpose\bX)^{-1/2}\bC_t(\bX\transpose\bX)^{-1/2}$, and consider a collection of vertex-aligned network adjacency matrices generated as $\bA_1,\ldots,\bA_m\sim\mathrm{COSIE}(\bU; \bB_1,\ldots,\bB_m)$. Clearly, $\lambda_2(\sum_{t = 1}^m\expect\bA_t) = 0$ but $\lambda_2(\expect\bA_t) < 0$ for all $t\in[m]$, so that directly taking $\texttt{eigs}(\sum_{t = 1}^m\bA_t; 2)$ leads to the cancellation of signals. In this experiment, we set $n = 80$, $m = 6400$, $a = 0.8$, $b = 0.6$, and let $\rho_n$ varies over $\{0.1, 0.2, 0.3, 0.4, 0.5, 0.6\}$. Given $(\bA_t)_{t = 1}^m$ generated from the above COSIE model, we compute $\widehat{\bU}^{(\mathsf{BCJSE})}$ using Algorithm \ref{alg:BCJSE} with $R = 2$ and $S = 1$. For comparison, we also consider the following competitors: Sum-of-squared (SoS) spectral embedding defined by $\texttt{eigs}(\bA\bA\transpose; d)$, the multiple adjacency spectral embedding (MASE) \cite{arroyo2021inference,zheng2022limit}, the bias-adjusted spectral embedding (BASE) \cite{doi:10.1080/01621459.2022.2054817,10.1214/22-AOS2196,cai2021subspace}, and the HeteroPCA (HPCA) \cite{zhang2018heteroskedastic,yan2021inference,agterberg2021entrywise}. The same experiment is repeated for $500$ independent Monte Carlo replicates. Given a generic embedding estimate $\widehat{\bU}$ for $\bU$, we take $\|\widehat{\bU}\mathrm{sgn}(\widehat{\bU}\transpose\bU) - \bU\|_{2\to\infty}$ as the criterion for measuring the accuracy of subspace estimation. 
\begin{figure}[!t]
\centering
\includegraphics[width=1\textwidth]{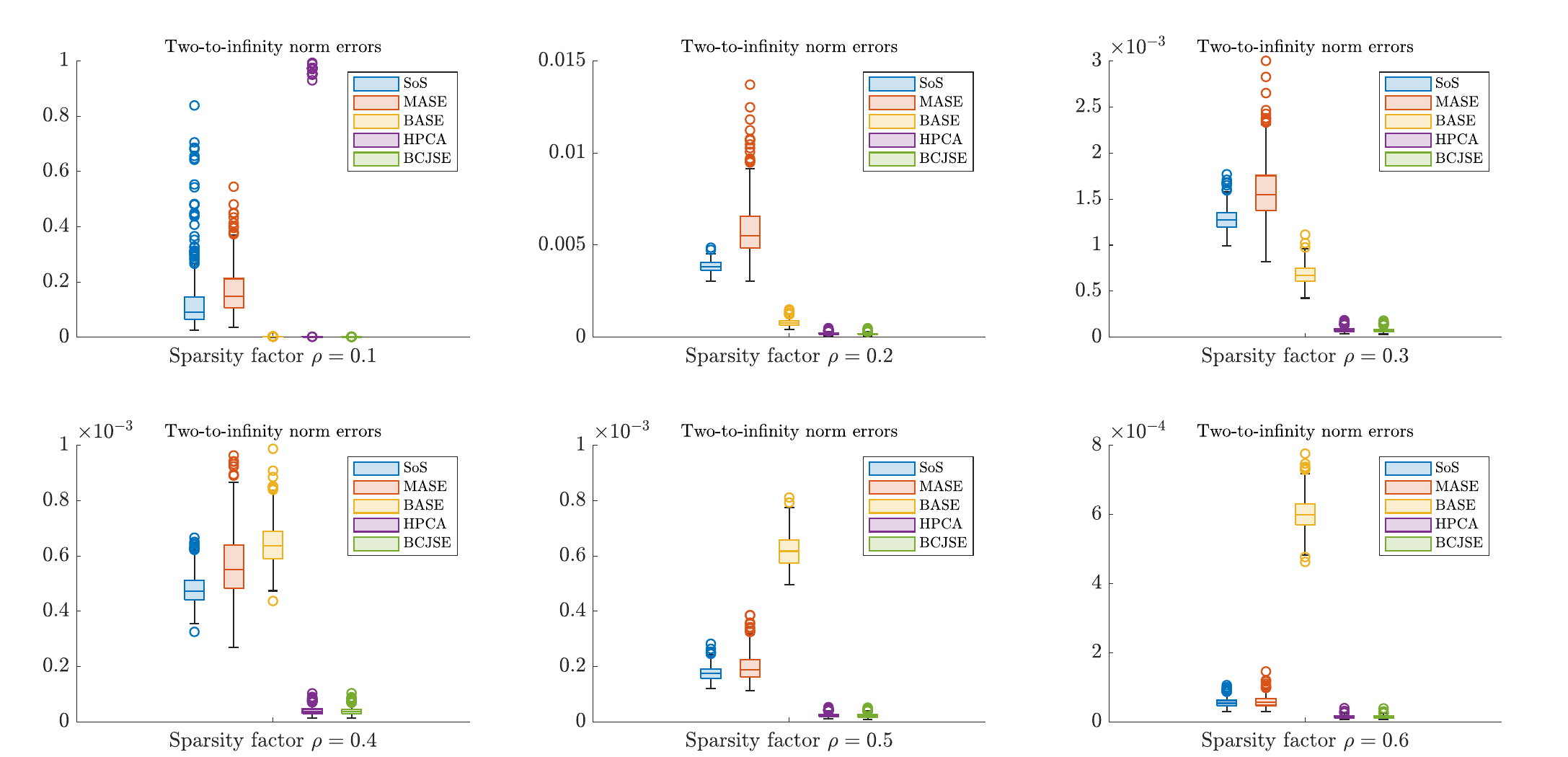}
\caption{Numerical experiment results for Section \ref{sub:numerical_subspace_estimation}: Boxplots of the two-to-infinity norm subspace estimation errors of different estimates (SoS, MASE, BASE, HPCA, and BCJSE) under $\rho_n\in\{0.1, 0.2, 0.3, 0.4, 0.5, 0.6\}$ across $500$ independent Monte Carlo replicates. }
\label{fig:COSIE_twotoinf_error_boxplot}
\end{figure}

Figure \ref{fig:COSIE_twotoinf_error_boxplot} visualizes the boxplots of the two-to-infinity norm subspace estimation errors of the aforementioned estimates under different sparsity regimes ($\rho_n\in \{0.1,0.2,0.3,0.4,0.5,0.6\}$) across $500$ independent Monte Carlo replicates. 
It is clear that when $\rho_n \in\{0.2,\ldots,0.6\}$, HPCA and BCJSE have significantly smaller two-to-infinity norm subspace estimation errors compared to the SoS spectral embedding, MASE, and BASE. When $\rho_n = 0.1$, the average performance of HPCA, BASE, and BCJSE are similar and they outperform the SoS spectral embedding and MASE, but HPCA becomes numerically less stable and causes large estimation errors occasionally. When $\rho_n\in\{0.5,0.6\}$, the bias effect caused by the diagonal deletion operation of BASE becomes transparent and leads to significantly larger subspace estimation errors compared to the remaining competitors. Also, HPCA and BCJSE have quite comparable performance in terms of the subspace estimation error across different sparsity regimes. Nevertheless, HPCA is slightly computationally more costly than BCJSE and is less stable with larger standard deviations when $\rho_n = 0.1$. Indeed, we found that in such a comparatively low signal-to-noise ratio regime, HPCA occasionally requires significantly large numbers of iterations across repeated experiments. The average number of iterations required by HPCA is $88.76$ and the corresponding standard deviation is $301.77$ across $500$ independent Monte Carlo replicates, whereas the number of iterations required by BCJSE is always $R = 2$ and $S = 1$ in this experiment.

\subsection{Hypothesis testing in MLMM}
\label{sub:numerical_testing_MLMM}

Consider a membership profile matrix $\bZ = [\bz_1,\ldots,\bz_n]\transpose\in\mathbb{R}^{n\times 2}$ defined as follows: $\bz_i = [1, 0]\transpose$ if $i\in[n_0]$, $\bz_i = [0, 1]\transpose$ if $i\in[2n_0]\backslash\{n_0\}$, $0.1 = t_0 < \ldots < t_{n - 2n_0} = 0.9$ are equidistant points over $[0.1, 0.9]$, and $\bz_{2n_0 + i} = [t_i, 1 - t_i]\transpose$ for all $i\in[n - 2n_0]$, where $n_0, n$ are positive integers satisfying $2n_0 < n$. Here, $n_0$ is the number of the so-called pure nodes (\emph{i.e.}, the membership profile vector assigns $1$ to one of the communities) in each community for the sake of identifiability of mixed membership models \cite{doi:10.1080/01621459.2020.1751645}. The same membership profile matrix has also been considered in \cite{xie2021entrywise}. Let $(\bC_t)_{t = 1}^m$ be matrices defined in \eqref{eqn:numerical_block_probability_matrix}, where $m$ is a positive even integer. Define $\bU = \bZ(\bZ\transpose\bZ)^{-1/2}$, $\bB_t = (\bZ\transpose\bZ)^{1/2}\bC_t(\bZ\transpose\bZ)^{1/2}$ for all $t\in[m]$, and we consider vertex-aligned multilayer network adjacency matrices $\bA_1,\ldots,\bA_m\sim\mathrm{COSIE}(\bU; \bB_1,\ldots,\bB_m)$. Finally, we set $a = 0.9$, $b = 0.1$, $n = 500$, $n_0 = 100$, $m = 200$, and let $\rho_n$ vary in $\{0.01, 0.02, 0.03, 0.04\}$. 

We generate $1000$ independent Monte Carlo replicates of $(\bA_t)_{t = 1}^m$ from the above MLMM and investigate the performance of the vertex membership profile hypothesis testing procedure described in Theorem \ref{thm:testing_membership_profiles}. Specifically, we let $i_1 = 1$, $i_2$ take values in $\{100, 400, 410, \ldots, 500\}$, and consider the hypothesis testing problem $H_0:\bz_{i_1} = \bz_{i_2}$ versus $H_A:\bz_{i_1}\neq\bz_{i_2}$ using the test statistic $T_{i_1i_2}$ defined in Section \ref{sub:entrywise_eigenvector_CLT}. We compute the empirical power of the proposed hypothesis testing procedure across the aforementioned $1000$ independent Monte Carlo experiments and tabulate them in Table \ref{tab:MLMM_testing_power} below. It is clear that the empirical power increases when $\|\bz_{i_1} - \bz_{i_2}\|_2$ increases under different sparsity regimes, and the empirical size is approximately $0.05$ (corresponding to the case where $\|\bz_{i_1} - \bz_{i_2}\|_2 = 0$). In addition, Figure \ref{fig:BCJSE_MLMM_inference} visualizes the null distribution of $T_{i_1i_2}$ (with $i_1 = 1$ and $i_2 = n_0$) and the distribution approximation to a randomly selected row of $\widehat{\bU}\bW - \bU$ across $1000$ repeated Monte Carlo replicates, where $\widehat{\bU}$ is the BCJSE given by Algorithm \ref{alg:BCJSE} with $R = 2$, $S = 1$, and $\bW = \mathrm{sgn}(\widehat{\bU}\transpose\bU)$. The first row of Figure \ref{fig:BCJSE_MLMM_inference} presents the histograms of $T_{i_1i_2}$ across $1000$ independent Monte Carlo replicates and they closely align with the asymptotic null distribution of $T_{i_1i_2}$ ($\chi_2^2$ distribution) when $\bz_{i_1} = \bz_{i_2}$ under different values of $\rho_n$. The second row of Figure \ref{fig:BCJSE_MLMM_inference} visualizes the scatter plots of a randomly selected row of $\widehat{\bU}\bW - \bU$ across $1000$ independent Monte Carlo replicates, where the solid and dashed curves correspond to the $95\%$ empirical and theoretical confidence ellipses. These visualizations well justify the theory established in Theorem \ref{thm:Rowwise_CLT} and Theorem \ref{thm:testing_membership_profiles}.
\begin{table}[!t]
\tiny
\caption{Empirical powers versus different $\|\bz_{i_1} - \bz_{i_2}\|_2$ and different $\rho_n$ for Section \ref{sub:numerical_testing_MLMM} \label{tab:MLMM_testing_power}}
\centering
\begin{tabular}{|c||c|c|c|c|c|c|c|c|c|c|c|c|}
\hline
\diagbox{Sparsity factor}{$\|\bz_{i_1} - \bz_{i_2}\|_2$} & $0$ & 0.071 & 0.113 & 0.156 & 0.198 & 0.241 & 0.284 & 0.326 & 0.369 & 0.411 & 0.454 & 0.496\\
\hline
$\rho_n = 0.01$ & 0.058 & 0.074 & 0.110 & 0.190 & 0.283 & 0.369 & 0.517 & 0.627 & 0.737 & 0.841 & 0.902 & 0.968\\
\hline
$\rho_n = 0.02$ & 0.049 & 0.144 & 0.288 & 0.520 & 0.739 & 0.902 & 0.966 & 0.987 & 0.999 & 1.000 & 1.000 & 1.000\\
\hline            
$\rho_n = 0.03$ & 0.048 & 0.231 & 0.532 & 0.802 & 0.958 & 0.997 & 1.000 & 1.000 & 1.000 & 1.000 & 1.000 & 1.000\\
\hline
$\rho_n = 0.04$ & 0.056 & 0.355 & 0.740 & 0.941 & 0.994 & 0.999 & 1.000 & 1.000 & 1.000 & 1.000 & 1.000 & 1.000\\
\hline                
\end{tabular}
\end{table}

\begin{figure}[!t]
\centering
\includegraphics[width=1\textwidth]{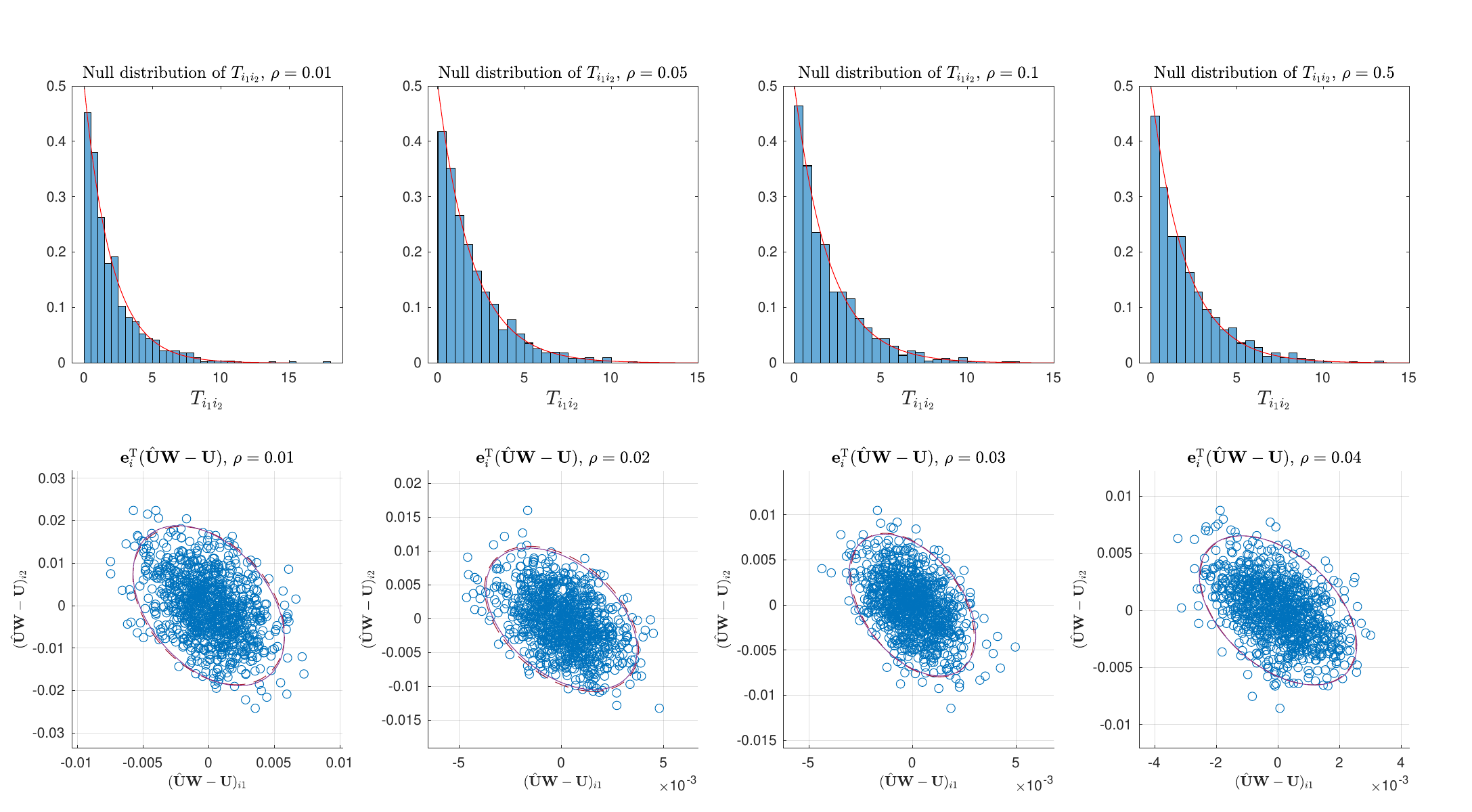}
\caption{Numerical experiment results for Section \ref{sub:numerical_testing_MLMM}: The top row visualizes the histograms of $T_{i_1i_2}$ (with $i_1 = 1$ and $i_2 = n_0$) across $1000$ repeated experiments under $\rho_n\in\{0.01, 0.02, 0.03, 0.04\}$, where the red curves represent the probability density function of $\chi_2^2$ distribution; The bottom row visualizes the scatter points of $\be_i\transpose(\widehat{\bU}\bW - \bU)$ for a randomly selected $i$ across $1000$ repeated experiments under  $\rho_n\in\{0.01, 0.02, 0.03, 0.04\}$, where the dashed curves are the $95\%$ theoretical ellipses and the solid curves are the $95\%$ empirical ellipses. }
\label{fig:BCJSE_MLMM_inference}
\end{figure}

\section{Discussion}
\label{sec:discussion}

In this paper, we design a bias-corrected joint spectral embedding algorithm for estimating the invariant subspace in the COSIE model and establish the accompanying entrywise subspace estimation theory. Our theory does not require the layer-wise average network expected degree to diverge, as long as the aggregate signal strength is sufficient. We settle the exact community detection in MLSBM and the hypothesis testing of the equality of membership profiles of any two given vertices in MLMM by leveraging the entrywise subspace estimation theory.  

In this work, we require that $m^{1/2}n\rho_n = \omega(\log n)$ for the exact community detection of the BCJSE algorithm, but it is not immediately clear whether this corresponds to the computational or information threshold of the exact community detection in MLSBM. In the case of balanced two-block MLSBM, the authors of \cite{paul2016consistent} have established the minimax rates of community detection, but later, the authors of \cite{lei2023computational} have established that there is a gap between the computational threshold and the information threshold for community detection in MLSBM under the low-degree polynomial conjecture. This gap comes from the fact that in an MLSBM, some layers may be assortative mixing (\emph{i.e.}, the nonzero eigenvalues of $\bP_t$ are positive) and some other layers may be disassortative mixing (\emph{i.e.}, $\bP_t$ contains negative eigenvalues), and identifying which layers are assortative mixing and which layers are disassortative mixing requires extra computational cost. In particular, Theorem 2.2 in \cite{lei2023computational} roughly asserts that if $m^{1/2}n\rho_n = \Omega(n^a)$ for some constant $a > 0$, then there exists a polynomial-time algorithm to achieve weak consistency community detection, and if $m^{1/2}n\rho_n\to 0$, no polynomial-time algorithm can achieve weak consistency. Theorem \ref{thm:spectral_clustering_MLSBM} sharpens the recoverable regime to $m^{1/2}n\rho_n = \omega(\log n)$ and achieves strong consistency, but it is still unclear whether this corresponds to the computational threshold of strong consistency. Meanwhile, the fundamental limit of the BCJSE-based clustering is unknown when the exact consistency is not achievable but the weak consistency is achievable. We defer these interesting research directions to future work. 

\section*{Acknowledgments}
This research was supported in part by Lilly Endowment, Inc., through its support for the Indiana University Pervasive Technology Institute.

{\appendices 
\section{Preliminary Results}
\label{sec:preliminary_results}

We warm up the technical proofs of our main results by introducing a collection of basic concentration inequalities regarding matrices $(\bE_t)_{t = 1}^m$. These concentration results are based on Bernstein's inequality and  matrix Bernstein's inequality in the form of Theorem 3 in \cite{doi:10.1080/01621459.2022.2054817}. 
\begin{lemma}
\label{lemma:norm_concentration_Et}
Let $(\bE_t)_{t = 1}^m$ be the matrices described in Section \ref{sub:JSE} and $\max_{t\in[m],i,j\in[n]}\expect E_{tij}^2\leq \rho_n$. Then
\begin{align*}
\max_{i\in[n]}\sum_{t = 1}^m\sum_{j = 1}^nE_{tij}^2 = \Optilde(mn\rho_n)\quad\mbox{and}\quad
\max_{t\in[m]}\|\bE_t\|_2 = \Optilde\left\{\log^{1/2}(m + n) + (n\rho_n)^{1/2}\right\}.
\end{align*}
\end{lemma}

\begin{proof}
For the first assertion, by Bernstein's inequality and a union bound over $i\in[n]$, we have
\begin{align*}
\max_{i\in[n]}\sum_{t = 1}^m\sum_{j = 1}^n{E}_{tij}^2&
\leq\max_{i\in[n]}\sum_{t = 1}^m\sum_{j = 1}^n\expect{E}_{tij}^2 + \max_{i\in[n]}\bigg|\sum_{t = 1}^m\sum_{j = 1}^n({E}_{tij}^2 - \expect{E}_{tij}^2)\bigg| = \Optilde(mn\rho_n).
\end{align*}
The second assertion from Remark 3.13 in \cite{10.1214/15-AOP1025} and a union bound over $t\in[m]$.
\end{proof}

\begin{lemma}\label{lemma:quadratic_form_Et}
Let $(\bE_t)_{t = 1}^m$ be the matrices described in Section \ref{sub:JSE} and suppose Assumption \ref{assumption:condition_number} hold. Then there exists a constant $C > 0$, such that for fixed matrices $(\bX_t)_{t = 1}^m,(\bY_t)_{t = 1}^n\in\mathbb{R}^{n\times d}$, for any $\tau > 0$, 
\begin{align*}
\bigg\|\sum_{t = 1}^m\bX_t\transpose\bE_t\bY_t\bigg\|_2
& \leq \max_{t\in[m]}\|\bX_t\|_{2\to\infty}\|\bY_t\|_{2\to\infty}\tau^2 + C_c\rho_n^{1/2}\tau\bigg(\sum_{t = 1}^m\|\bX_t\|_2^2\|\bY_t\|_2^2\bigg)^{1/2}
\end{align*}
with probability at least $1 - O(e^{-\tau^2})$.
\end{lemma}

\begin{proof}
We apply a classical discretization trick and follow the roadmap paved in the proof of Lemma S2.3 in \cite{xie2021entrywise}. Let $\calS^{d - 1} = \{\bv\in\mathbb{R}^d:\|\bv\|_2 = 1\}$ be the unit sphere in $\mathbb{R}^d$. For any $\eps > 0$, let $\calS_\eps^{d - 1}$ be an $\eps$-net of $\calS^{d - 1}$. Clearly, for any $\bv_1,\bv_2\in\calS^{d - 1}$, there exists vectors $\bw_1(\bv_1), \bw_2(\bv_2)\in\calS_\eps^{d - 1}$, such that $\|\bv_1 - \bw_1(\bv_1)\|_2\vee\|\bv_2 - \bw_2(\bv_2)\|_2 < \eps$, so that
\begin{align*}
\bigg\|\sum_{t = 1}^m\bX_t\transpose\bE_t\bY_t\bigg\|_2
&
 = \max_{\bv_1,\bv_2\in\calS^{d - 1}}\bigg|\{\bv_1 - \bw_1(\bv_1) + \bw_1(\bv_1)\}\transpose\sum_{t = 1}^m\bX_t\transpose\bE_t\bY_t\{\bv_2 - \bw_2(\bv_2) + \bw_2(\bv_2)\}\bigg|\\
&\leq (\eps^2 + 2\eps)\bigg\|\sum_{t = 1}^m\bX_t\transpose\bE_t\bY_t\bigg\|_2 + 
\max_{\bw_1,\bw_2\in\calS_\eps^{d - 1}}\bigg|\bw_1\transpose\sum_{t = 1}^m\bX_t\transpose\bE_t\bY_t\bw_2\bigg|_2.
\end{align*}
Setting $\eps = 1/3$ yields $\|\sum_{t = 1}^m\bX_t\transpose\bE_t\bY_t\|_2\leq (9/2)
\max_{\bw_1,\bw_2\in\calS_\eps^{d - 1}}|\bw_1\transpose\sum_{t = 1}^m\bX_t\transpose\bE_t\bY_t\bw_2|$. Furthermore, we can select $\calS_{1/3}^{d - 1}$ in such a way that $|\calS_{1/3}^{d - 1}|\leq 18^d$ \cite{pollard1990empirical}, where $|\calS_{1/3}^{d - 1}|$ denotes the cardinality of $\calS_{1/3}^{d - 1}$. Now for fixed vectors $\bw_1,\bw_2\in\calS_{1/3}^{d - 1}$, let $\bx_t = \bX_t\bw_1 = [x_{t1},\ldots,x_{tn}]\transpose$, $\by_t = \bY_t\bw_2 = [y_{t1},\ldots,y_{tn}]\transpose$. Clearly, 
\begin{align*}
\bigg|\bw_1\transpose\sum_{t = 1}^m\bX_t\transpose\bE_t\bY_t\bw_2\bigg|_2
& = \bigg|\sum_{t = 1}^m\bx_t\transpose\bE_t\by_t\bigg|_2
 = \bigg|\sum_{t = 1}^m\sum_{i < j}(x_{ti}y_{tj} + x_{tj}y_{ti})E_{tij} + \sum_{t = 1}^m\sum_{j = 1}^nx_{tj}y_{tj}E_{tjj}\bigg|_2
 \\&
 = \bigg|\sum_{t = 1}^m\sum_{i\leq j}c_{tij}E_{tij}\bigg|_2,
\end{align*}
where $c_{tij} = x_{ti}y_{tj} + x_{tj}y_{ti}$ if $i < j$, and $c_{tjj} = x_{tj}y_{tj}$. Observe that
\begin{align*}
\sum_{t = 1}^m\sum_{i\leq j}c_{tij}^2&\leq \sum_{t = 1}^m\sum_{i < j}(2x_{ti}^2y_{tj}^2 + 2x_{tj}^2y_{ti}^2) + \sum_{t = 1}^m\sum_{j = 1}^nx_{tj}^2y_{tj}^2 \leq 2\sum_{t = 1}^m\sum_{i = 1}^n\sum_{j = 1}^nx_{ti}^2y_{tj}^2 = 2\sum_{t = 1}^m\|\bx_t\|_2^2\|\by_t\|_2^2
\\&
\leq 2\sum_{t = 1}^m\|\bX_t\|_2^2\|\bY_t\|_2^2
\end{align*}
and $\max_{t\in[m],i,j\in[n]}|c_{tij}|\leq 2\max_{t\in[m]}\|\bx_t\|_\infty\|\by_t\|_\infty$.
Then by Bernstein's inequality, there exists an absolute constant $C > 0$, such that for any $\tau > 0$, 
$
|\bw_1\transpose\sum_{t = 1}^m\bX_t\transpose\bE_t\bY_t\bw_2|_2\leq C
\max_{t\in[m]}\|\bX_t\|_{2\to\infty}\|\bY_t\|_{2\to\infty}\tau^2 + C\rho_n^{1/2}\tau(\sum_{t = 1}^m\|\bX_t\|_2^2\|\bY_t\|_2^2)^{1/2}
$
with probability at least $1 - 2e^{-\tau^2}$. The proof is completed by a union bound over $\bw_1,\bw_2\in\calS_{1/3}^{d - 1}$.
\end{proof}

\begin{lemma}\label{lemma:rowwise_concentration_linear_term}
Let $(\bE_t)_{t = 1}^m$ be the matrices described in Section \ref{sub:JSE}, $\max_{t\in[m],i,j\in[n]}\expect E_{tij}^2\leq \rho_n$, and suppose $mn\rho_n = \Omega(\log(m + n))$. Then  
\begin{align*}
\bigg\|\sum_{t = 1}^m\bE_t\bU\bB_t\bigg\|_{2\to\infty}
 &= \Optilde\{m^{1/2}n\rho_n^{3/2}\log^{1/2} (m + n)\}.
\end{align*}
\end{lemma}
\begin{proof}
By Bernstein's inequality, there exists a constant $C > 0$, such that for any $i\in[n]$ and any $\tau > 0$,
\begin{align*}
\bigg\|\sum_{t = 1}^m\be_i\transpose\bE_t\bU\bB_t\bigg\|_2 
&\leq C\{\tau^2 + (mn\rho_n)^{1/2}\tau\}
\|\bU\|_{2\to\infty}\max_{t\in[m]}\|\bB_t\|_2 
\lesssim {n^{1/2}\rho_n}\{\tau^2 + (mn\rho_n)^{1/2}\tau\}.
\end{align*}
with probability at least $1 - O(e^{-\tau^2})$. The proof is completed by a union bound over $i\in [n]$ and setting $\tau \asymp \sqrt{\log (m + n)}$.
\end{proof}
Next, we present a row-wise concentration bound of $\calH(\bE\bE\transpose)\bY$, where $\bY\in\mathbb{O}(n, d)$ is a deterministic $n\times d$ matrix with orthonormal columns. Its proof relies on the decoupling inequality from \cite{10.1214/aop/1176988291} and a conditioning argument. 
\begin{lemma}\label{lemma:decoupling_inequality}
Let $\bE_1,\ldots,\bE_m$ be the matrices described in Section \ref{sub:JSE}. Then, there exists a numerical constant $C > 0$, such that for any matrices $\bX\in\mathbb{R}^{n\times d_1},\bY\in\mathbb{R}^{n\times d_2}$, and any $\tau > 0$,
$
\prob\{\|\bX\transpose\calH(\bE\bE\transpose)\bY\|_2 > \tau\}\leq C
\prob\{\|\bX\transpose\calH(\bE\bar{\bE}\transpose)\bY\|_2 > 
\tau/C\}
$
where $\bar{\bE} = [\bar{\bE}_1,\ldots,\bar{\bE}_m]$ is an independent copy of $\bE$. 
Furthermore, if $\bY\in\mathbb{O}(n, d)$ and $m^{1/2}n\rho_n = \Omega(\log(m + n))$, then
\begin{align*}
\|\be_i\transpose\calH(\bE\bE\transpose)\bY\|_2 = \Optilde\{\|\bY\|_{2\to\infty}m^{1/2}n\rho_n\log(m + n)\}.
\end{align*}
\end{lemma}

\begin{proof}
Let $E_{tij}$ denote the $(i, j)$th entry of $\bE_t$, $\bar{E}_{tij}$ denote the $(i, j)$th entry of $\bar{\bE}_t$, and write $\bX = [\bx_1,\ldots,\bx_n]\transpose$, $\bY = [\by_1,\ldots,\by_n]\transpose$, where $\bx_i\in\mathbb{R}^{d_1}$ and $\by_j\in\mathbb{R}^{d_2}$. The key idea is to apply the decoupling inequality in \cite{10.1214/aop/1176988291}. 
By definition, 
\begin{align*}
\bX\transpose\calH(\bE\bE\transpose)\bY
& = \sum_{i,k,j,l\in[n],t\in[m]}
E_{tik}E_{tjl}\bx_i\by_j\transpose\{\mathbbm{1}(i\neq j)\mathbbm{1}(i\leq k,j\leq l)\mathbbm{1}(k = l)
 + 
\mathbbm{1}(i\neq j)\mathbbm{1}(i > k,j\leq l)\mathbbm{1}(k = l)\}\\
&\quad + \sum_{i,k,j,l\in[n],t\in[m]}
E_{tik}E_{tlj}\bx_i\by_j\transpose\{\mathbbm{1}(i\neq j)\mathbbm{1}(i\leq k,j > l)\mathbbm{1}(k = l)
 + 
 \mathbbm{1}(i\neq j)\mathbbm{1}(i > k,j > l)\mathbbm{1}(k = l)\}\\
& = \sum_{(i, k),(j, l)\in\calI}\bF_{(i, k)(j, l)}(\calE_{(i, j)},\calE_{(k, l)}),
\end{align*}
where $\calI = \{(i, k)\in[n]\times [n]:1\leq i\leq k\leq n\}$, and for any $(i, k), (j, l)\in\calI$, $\calE_{(i, k)} = [E_{1ik},\ldots,E_{mik}]$,
\begin{align*}
\bF_{(i, k)(j, l)}(\calE_{(i, j)},\calE_{(k, l)})
& = \sum_{t = 1}^mE_{tik}E_{tjl}\bx_i\by_j\transpose\mathbbm{1}(i\neq j)\mathbbm{1}(k = l)
 + \sum_{t = 1}^mE_{tik}E_{tjl}\bx_k\by_j\transpose\mathbbm{1}(k\neq j)\mathbbm{1}(i \neq k)\mathbbm{1}(i = l)\\
&\quad + \sum_{t = 1}^mE_{tik}E_{tjl}\bx_i\by_l\transpose\mathbbm{1}(i\neq l)\mathbbm{1}(j \neq l)\mathbbm{1}(k = j)
\\&\quad
 + \sum_{t = 1}^mE_{tik}E_{tjl}\bx_k\by_l\transpose\mathbbm{1}(k\neq l)\mathbbm{1}(i \neq k,j \neq l)\mathbbm{1}(i = j).
\end{align*}
Clearly, $\bF_{(i, k)(j, l)}(\calE_{(i, j)},\calE_{(k, l)}) = 0$ when $(i, k) = (j, l)$, so that
\[
\bX\transpose\calH(\bE\bE\transpose)\bY = \sum_{(i, k),(j, l)\in\calI,(i, k)\neq (j, l)}\bF_{(i, k)(j, l)}(\calE_{(i, j)},\calE_{(k, l)})
\]
This expression enables us to apply the decoupling technique in \cite{10.1214/aop/1176988291}. 
Let $\bar{\bE} = [\bar{\bE}_1,\ldots,\bar{\bE}_m]$ be an independent copy of $\bE$, where $\bar{\bE}_t = [E_{tij}]_{n\times n}$, $t\in [m]$. Then
\[
\sum_{(i, k),(j, l)\in\calI,(i, k)\neq (j, l)}\bF_{(i, k)(j, l)}(\calE_{(i, j)},\bar{\calE}_{(k, l)}) = \sum_{t = 1}^m\sum_{i\neq j}\sum_{l = 1}^nE_{til}\bar{E}_{tjl}\bx_i\by_j\transpose,
\]
where $\bar{\calE}_{tik} = [\bar{E}_{1ik},\ldots,\bar{E}_{mik}]\transpose$. 
By Theorem 1 in \cite{10.1214/aop/1176988291}, there exists a constant $C > 0$, such that for any $\tau > 0$,
\begin{align*}
\prob\left\{\|\bX\transpose\calH(\bE\bE\transpose)\bY\|_2 > \tau\right\}\leq C\prob\bigg(\bigg\|\sum_{t = 1}^m\sum_{i,j\in[n],i\neq j}\sum_{l = 1}^nE_{til}\bar{E}_{tjl}\bx_i\by_j\transpose\bigg\|_2 > \frac{\tau}{C}\bigg) = C\prob\bigg\{\|\bX\transpose\calH(\bE\bar{\bE}\transpose)\bY\|_2 > \frac{\tau}{C}\bigg\}.
\end{align*}
This completes the first assertion. For the second assertion, we apply the first assertion with  $\bX = \be_i$ to obtain
\begin{align*}
\prob\left\{\left\|\be_i\transpose\calH(\bE\bE\transpose)\bY\right\|_2 > \tau\right\}\leq C
\prob\left\{\left\|\be_i\transpose\calH(\bE\bar{\bE}\transpose)\bY\right\|_2 > \frac{\tau}{C}\right\}
= \prob\bigg\{\bigg\|\sum_{t = 1}^m\sum_{j\in[n]\backslash\{i\}}\sum_{l = 1}^nE_{til}\bar{E}_{tjl}\by_j\bigg\|_2 > \frac{\tau}{C}\bigg\},
\end{align*}
where $\bar{\bE} = [\bar{\bE}_1,\ldots,\bar{\bE}_m]$ is an independent copy of $\bar{\bE}$ and $\bar{E}_{tjl}$ denotes the $(j, l)$th entry of $\bar{\bE}_t$. By Bernstein's inequality and a conditioning argument, for any $\tau > 0$, we have
\begin{align*}
\prob\bigg\{\bigg\|\sum_{t = 1}^m\sum_{j\in[n]\backslash\{i\}}\sum_{l = 1}^nE_{til}\bar{E}_{tjl}\by_j\bigg\|_2 > 2b_{in}(\bar{\bE})\tau^2 + 2v_{in}(\bar{\bE})\tau\mathrel{\bigg|}\bar{\bE}\bigg\}\leq 2e^{-\tau^2},
\end{align*}
where $v_{in}^2(\bar{\bE}) := \sum_{t = 1}^m\sum_{l = 1}^n\sigma_{til}^2\|\sum_{j = 1}^n\bar{E}_{tjl}\by_j\mathbbm{1}(j\neq i)\|_2^2$ and $b_{in}(\bar{\bE}) := \max_{t\in[m],i,l\in[n]}\|\sum_{j = 1}^n\bar{E}_{tjl}\by_j\mathbbm{1}(j\neq i)\|_2$.
Next, we consider the error bounds for $v_{in}(\bar{\bE})$ and $b_{in}(\bar{\bE})$. By Bernstein's inequality again and a union bound over $t\in[m]$ and $l\in[n]$, for fixed $t\in[m]$, $l\in[n]$, $b_{in}(\bar{\bE}) = \Optilde\{\rho_n^{1/2}\log^{1/2}(m + n) + \|\bY\|_{2\to\infty}\log (m + n)\}$. For $v_{in}(\bar{\bE})$, by Hanson-Wright inequality for bounded random variables (see Theorem 3 in \cite{bellec2019concentration})
\begin{align*}
v_{in}^2(\bar{\bE})& \leq 2\rho_n\sum_{t = 1}^m\sum_{l = 1}^n\bigg\|\sum_{j\leq l}\bar{E}_{tjl}\by_j\mathbbm{1}(j\neq i)\bigg\|_2^2 + 2\rho_n\sum_{t = 1}^m\sum_{l = 1}^n\bigg\|\sum_{j > l}\bar{E}_{tjl}\by_j\mathbbm{1}(j\neq i)\bigg\|_2^2\\
& \leq 2\rho_n\sum_{t = 1}^m\sum_{l = 1}^n\sum_{j_1,j_2\leq l}(\bar{E}_{tj_1l}\bar{E}_{tj_2l} - \expect \bar{E}_{tj_1l}\bar{E}_{tj_2l})\by_{j_1}\transpose\by_{j_2}\mathbbm{1}(j_1\neq i,j_2\neq i) + 2\rho_n^2\sum_{t = 1}^m\sum_{l = 1}^n\sum_{j\leq l}\|\by_j\|_2^2\\
&\quad + 2\rho_n\sum_{t = 1}^m\sum_{l = 1}^n\sum_{j_1,j_2 > l}(\bar{E}_{tj_1l}\bar{E}_{tj_2l} - \expect \bar{E}_{tj_1l}\bar{E}_{tj_2l})\by_{j_1}\transpose\by_{j_2}\mathbbm{1}(j_1\neq i,j_2\neq i) + 2\rho_n^2\sum_{t = 1}^m\sum_{l = 1}^n\sum_{j > l}\|\by_j\|_2^2\\
& = O(mn\rho_n^2) + \Optilde\left\{\rho_n\log(m + n) + m^{1/2}n^{1/2}\rho_n^{3/2}\log^{1/2}(m + n)\right\} = \Optilde(mn\rho_n^2).
\end{align*}
By a union bound over $t\in[m]$ and $l\in[n]$, for any $c > 0$, there exists a $c$-dependent constant $C_c > 0$, such that the event $\calE_{in}(c)=\left\{
\bar{\bE}:v_{in}(\bar{\bE})\leq C_c\eps_{v_{in}},b_{in}(\bar{\bE})\leq C_c\eps_{b_{in}}
\right\}$
occurs with probability at least $1 - O((m + n)^{-c})$, where
$\eps_{v_{in}}:= m^{1/2}n^{1/2}\rho_n$ and
$\eps_{b_{in}}:= \rho_n^{1/2}\log^{1/2} (m + n) + \|\bY\|_{2\to\infty}\log (m + n)$.
Namely, for $\tau = c^{1/2}\log^{1/2}(m + n)$
\begin{align*}
&\prob\bigg(\bigg\|\sum_{t = 1}^m\sum_{j\in[n]\backslash\{i\}}\sum_{l = 1}^nE_{til}\bar{E}_{tjl}\by_j\bigg\|_2 > 2\eps_{b_{in}}c\log(m + n) + 2\eps_{v_{in}}c^{1/2}\log^{1/2}(m + n)\bigg)\\
&\quad\leq \expect_{\bar{\bE}}\bigg[
\prob\bigg\{\bigg\|\sum_{t = 1}^m\sum_{j\in[n]\backslash\{i\}}\sum_{l = 1}^nE_{til}\bar{E}_{tjl}\by_j\bigg\|_2 > 2\eps_{b_{in}}c\log(m + n) + 2\eps_{v_{in}}c^{1/2}\log^{1/2}(m + n)\mathrel{\bigg|}\bar{\bE}\bigg\}\mathbbm{1}_{\calE_{in}(c)}
\bigg]
\\&\quad\quad
 + \prob\{\calE_{in}^c(c)\}\\
&\quad = O((m + n)^{-c}).
\end{align*}
This implies that $\|\be_i\transpose\calH(\bE\bE\transpose)\bY\|_2 = \Optilde\{\|\bY\|_{2\to\infty}m^{1/2}n\rho_n\log(m + n)\}$,
where we have used the fact that $\|\bY\|_{2\to\infty}\geq \sqrt{d/n}$ for any $\bY\in\mathbb{O}(n, d)$. The proof is thus completed.
\end{proof}
Now, we dive into a collection of slightly more sophisticated concentration results regarding the spectral norm concentration of $\bE\bE\transpose$. The proofs of these results rely on modifications of Theorem 4 and Theorem 5 in \cite{doi:10.1080/01621459.2022.2054817}.
\begin{lemma}
\label{lemma:spectral_norm_concentration_Hollow_EEtranspose}
Let $(\bE_t)_{t = 1}^m$ be the matrices described in Section \ref{sub:JSE}, $\max_{t\in[m],i,j\in[n]}\expect E_{tij}^2\leq \rho_n$, and suppose $m^{1/2}n\rho_n = \Omega(\log(m + n))$.
Then
$\|\calH(\bE\bE\transpose)\|_2 = \Optilde\{m^{1/2}n\rho_n\log(m + n)\}$. 
\end{lemma}
\begin{proof}
The proof breaks down into two regimes: $n\rho_n = \Omega(\log(m + n))$ and $n\rho_n = O(\log(m + n))$. If $n\rho_n = \Omega(\log(m + n))$, then we apply Theorem 4 in \cite{doi:10.1080/01621459.2022.2054817} with $\bG_l = \eye_n$ for all $l\in[m]$, $\nu_1,\nu_2,\nu_2' = O(\rho_n)$, $R_1, R_2, R_2' = O(1)$, $\sigma_1 = \sqrt{m}$, $\sigma_2 = \sigma_3 = 1$, and $\sigma'_2 = \sqrt{mn}$ to obtain
$\|\calH(\bE\bE\transpose)\|_2 = \Optilde\{m^{1/2}n\rho_n\log(m + n)\}$. The rest of the proof focuses on the regime where $n\rho_n = O(\log(m + n))$. 
Let $P_{tij}^*$ be the success probability corresponding to $E_{tij}$ such that $E_{tij} = A_{tij}^* - P_{tij}^*$, where $A_{tij}^*\sim\mathrm{Bernoulli}(P_{tij}^*)$ for all $t\in[m],i,j\in[n]$, $i\leq j$, $A_{tij}^* = A_{tji}^*$, and $P^*_{tij} = P^*_{tji}$ if $i > j$. 
We modify the proof of Lemma C.1 and Theorem 5 in \cite{doi:10.1080/01621459.2022.2054817} as follows. Let $\bA_t^* = [A_{tij}^*]_{n\times n}$ and $\bP_t^* = [P_{tij}^*]_{n\times n}$ for all $t\in[m]$. We then modify Lemma C.1 in \cite{doi:10.1080/01621459.2022.2054817} to obtain: i) $\max_{t\in[m],i\in[n]}\sum_{j = 1}^nA_{tij}^* = \Optilde\{\log(m + n)\}$; ii) $\max_{i\in[n]}\sum_{t = 1}^n\sum_{j = 1}^nA_{tij}^* = \Optilde(mn\rho_n)$; iii) $\sum_{i = 1}^n\sum_{t = 1}^n\sum_{j = 1}^nA_{tij}^* = \Optilde(mn^2\rho_n)$ (by Bernstein's inequality and a union bound over $t\in[m]$ and $i\in[n]$). For $\|\sum_{t = 1}^m\bA_t^{*2}\|_2$, we write it as $\|\sum_{t = 1}^m\bA_t^{*2}\|_2\leq \max_{i\in[n]}\sum_{t = 1}^m\sum_{j = 1}^nA_{tij}^* + \|\calH(\bA^*\bA^{*\mathrm{T}})\|_2 = \Optilde(mn\rho_n) + \|\calH(\bA^*\bA^{*\mathrm{T}})\|_2$. By the decoupling inequality (Theorem 1 in \cite{10.1214/aop/1176988291}), it is sufficient to consider $\|\calH(\bA^*\bar{\bA}^{*\mathrm{T}})\|_2$, where $\bar{\bA}^* = [\bar{\bA}_1^*,\ldots,\bar{\bA}_m^*]$ is an independent copy of $\bA^*$ and $\bar{\bA}_t^* = [\bar{A}_{tij}^*]_{n\times n}$. By Perron-Frobenius theorem, every non-negative matrix has a non-negative eigenvector with the corresponding eigenvalue being the spectral radius. This implies that $\|\calH(\bA^*\bar{\bA}^{*\mathrm{T}})\|_2\leq \max_{i\in[n]}\sum_{t = 1}^m\sum_{j = 1}^n\sum_{l = 1}^nA_{tij}^*\bar{A}_{tlj}^*$. By Bernstein's inequality, results i), ii), and iii), and a union bound over $i\in[n]$,
\begin{align*}
\|\calH(\bA^*\bar{\bA}^{*\mathrm{T}})\|_2&\leq \max_{i\in[n]}\sum_{t = 1}^m\sum_{j = 1}^n\sum_{l = 1}^nA_{tij}^*\bar{A}_{tlj}^* = \max_{i\in[n]}\sum_{t = 1}^m\sum_{j = 1}^nE_{tij}\bigg(\sum_{l = 1}^n\bar{A}_{tlj}^*\bigg) + \max_{i\in[n]}\sum_{t = 1}^m\sum_{j = 1}^nP_{tij}\bigg(\sum_{l = 1}^n\bar{A}_{tlj}^*\bigg)\\
& = \Optilde\bigg[\max_{t\in[m],j\in[n]}\bigg(\sum_{l = 1}^n\bar{A}_{tlj}^*\bigg)\log(m + n) + \bigg\{\sum_{t = 1}^m\sum_{j = 1}^n\bigg(\sum_{l = 1}^n\bar{A}_{tlj}^*\bigg)^2\bigg\}^{1/2}\rho_n^{1/2}\log^{1/2}(m + n)\bigg]\\
&\quad + \Optilde(mn^2\rho_n^2)\\
& = \Optilde\bigg\{\max_{t\in[m],j\in[n]}\bigg(\sum_{l = 1}^n\bar{A}_{tlj}^*\bigg)\log(m + n) + \bigg(\sum_{t = 1}^m\sum_{j = 1}^n\sum_{l = 1}^n\bar{A}_{tlj}^*\bigg)^{1/2}\rho_n^{1/2}\log(m + n)\bigg\}\\
&\quad + \Optilde(mn^2\rho_n^2)\\
& = \Optilde(mn^2\rho_n^2).
\end{align*}
Therefore, we conclude that $\|\bA^*\bA^{*\mathrm{T}}\|_2 = \Optilde(mn\rho_n + mn^2\rho_n^2)$. We now turn our attention to $\|\calH(\bE\bE\transpose)\|_2$. Again, by Theorem 1 in \cite{10.1214/aop/1176988291}, it is sufficient to consider $\|\calH(\bE\bar{\bE}\transpose)\|_2$, where $\bar{\bE} = [\bar{\bE}_1,\ldots,\bar{\bE}_m]$ is an independent copy of $\bE$ and $\bar{\bE}_t = [\bar{E}_{tij}]_{n\times n}$. By definition, $\|\calH(\bE\bar{\bE}\transpose)\|_2\leq \max_{i\in[n]}|\sum_{t = 1}^m\sum_{j = 1}^nE_{tij}\bar{E}_{tij}| + \|\sum_{t = 1}^m\bE_t\bar{\bA}_t^*\|_2 + \|\sum_{t = 1}^m\bE_t\bP_t^*\|_2$. By Bernstein's inequality, Theorem 3 in \cite{doi:10.1080/01621459.2022.2054817}, and union bound over $i\in[n]$, we obtain $\max_{i\in[n]}|\sum_{t = 1}^m\sum_{j = 1}^nE_{tij}\bar{E}_{tij}| = \Optilde\{m^{1/2}n\rho_n\log^{1/2}(m + n)\}$ and $\|\sum_{t = 1}^m\bE_t\bP_t^*\|_2 = \Optilde\{m^{1/2}(n\rho_n)^{3/2}\log^{1/2}(m + n)\}$. By Theorem 3 in \cite{doi:10.1080/01621459.2022.2054817} again, a conditioning argument, together with results i)--iii) and the bound for $\|\bA\bA\transpose\|_2$, $\max_{t\in[m]}\|\bar{\bA}_t^*\|_{2\to\infty} = (\max_{t\in[m],i\in[n]}\sum_{j = 1}^n\bar{A}_{tij}^*)^{1/2} = \Optilde\{\log^{1/2}(m + n)\}$,
\begin{align*}
\bigg\|\sum_{t = 1}^m\bE_t\bar{\bA}_t^*\bigg\|_2
& = \Optilde\bigg\{\max_{t\in[m]}\|\bar{\bA}_t^*\|_{2\to\infty}\log(m + n) + \|\bar{\bA}^*\bar{\bA}^{*\mathrm{T}}\|_2^{1/2}(n\rho_n)^{1/2}\log^{1/2}(m + n)\bigg\}\\
& = \Optilde\{m^{1/2}n\rho_n\log(m + n)\}.
\end{align*}
The proof is completed by combining the above concentration bounds.
\end{proof}

The following lemma characterizes the noise level of the COSIE model by providing a sharp error bound on the spectral norm of the oracle noise matrix $\calH(\bA\bA\transpose) - \bP\bP\transpose - \bM$. Let 
\begin{align}\label{eqn:zeta_op_noise_level}
\zeta_{\mathsf{op}} = m^{1/2}n\rho_n\log(m + n) + m^{1/2}(n\rho_n)^{3/2}\log^{1/2}(m + n). 
\end{align}
It turns out that $\zeta_{\mathsf{op}} / (mn^2\rho_n^2) = \Theta(\varepsilon_n^{(\mathsf{op})})$ and it can be viewed as the inverse signal-to-noise ratio. 

\begin{lemma}
\label{lemma:spectral_norm_concentration_noise}
Suppose Assumption \ref{assumption:eigenvector_delocalization} hold, $(\bE_t)_{t = 1}^m$ are the matrices in Section \ref{sub:JSE}, $\max_{t\in[m],i,j\in[n]}\expect E_{tij}^2\leq \rho_n$, and $m^{1/2}n\rho_n = \Omega(\log(m + n))$. 
Then
$\|\calH(\bA\bA\transpose) - \bP\bP\transpose - \bM\|_2 = \Optilde(\zeta_{\mathsf{op}})$.
\end{lemma}

\begin{proof}
By definition,
\begin{align*}
\|\calH(\bA\bA\transpose) - \bP\bP\transpose - \bM\|_2
\leq \|\calH(\bE\bE\transpose)\|_2 + 2\|\sum_{t = 1}^m\bE_t\bP_t\|_2 + 2\max_{i\in [n]}|\sum_{t = 1}^m\sum_{j = 1}^nP_{tij}E_{tij}|.
\end{align*}
By Lemma \ref{lemma:spectral_norm_concentration_Hollow_EEtranspose}, we know that
$\|\calH(\bE\bE\transpose)\|_2 = \Optilde\left\{m^{1/2}n\rho_n\log(m + n)\right\}$.
By Bernstein's inequality and a union bound over $i\in[n]$, we have
$\max_{i\in[n]}|\sum_{t = 1}^m\sum_{j = 1}^nP_{tij}E_{tij}|
 = \Optilde\{(mn)^{1/2}\rho_n^{3/2}\log^{1/2}(m + n)\}$.
By Theorem 3 in \cite{doi:10.1080/01621459.2022.2054817}, we obtain $\|\sum_{t = 1}^m\bE_t\bP_t\|_2 = \Optilde\{m^{1/2}(n\rho_n)^{3/2}\log^{1/2} (m + n) + n^{1/2}\rho_n\log(m + n)\}$.
The proof is completed by combining the above results.
\end{proof}

\begin{lemma}
\label{lemma:concentration_noise_quadratic_form}
Suppose Assumption \ref{assumption:eigenvector_delocalization} hold, $(\bE_t)_{t = 1}^m$ are the matrices in Section \ref{sub:JSE}, $\max_{t\in[m],i,j\in[n]}\expect E_{tij}^2\leq \rho_n$, and $m^{1/2}n\rho_n = \Omega(\log(m + n))$. Then for any fixed $\bX,\bY\in\mathbb{O}(n, d)$ with $\|\bX\|_{2\to\infty}\vee\|\bY\|_{2\to\infty} = O(n^{-1/2})$, 
\[
\|\bX\transpose\{\calH(\bA\bA\transpose) - \bP\bP\transpose - \bM\}\bY\|_2 = \Optilde\{m^{1/2}n\rho_n^{3/2}\log^{1/2}(m + n)\}.
\]
\end{lemma}

\begin{proof}
Let $\bX = [\bx_{1},\ldots,\bx_{n}]\transpose$ and $\bY = [\by_{1},\ldots,\by_{n}]\transpose$. By triangle inequality, $\|\bX\transpose\{\calH(\bA\bA\transpose) - \bP\bP\transpose - \bM\}\bY\|_2$ is upper bounded by
\begin{align*}
\|\bX\transpose\calH(\bE\bE\transpose)\bY\|_2 + \bigg\|\sum_{t = 1}^m\bX\transpose\bP_t\bE_t\bY\bigg\|_2
 + \bigg\|\sum_{t = 1}^m\bX\transpose\bE_t\bP_t\bY\bigg\|_2
 + \bigg\|\sum_{t = 1}^m\sum_{i,j = 1}^n\bx_i\by_i\transpose 2P_{tij}E_{tij}\bigg\|_2.
\end{align*}
By Assumption \ref{assumption:eigenvector_delocalization}, we have $\max_{t\in[m]}\|\bP_t\bX\|_{2\to\infty} = O(n^{1/2}\rho_n)$, $\max_{t\in[m]}\|\bP_t\bY\|_{2\to\infty} = O(n^{1/2}\rho_n)$.
Then for the second and third term, we apply Lemma \ref{lemma:quadratic_form_Et} to obtain
\begin{align*}
&\bigg\|\sum_{t = 1}^m\bX\transpose\bP_t\bE_t\bY\bigg\| + \bigg\|\sum_{t = 1}^m\bX\transpose\bE_t\bP_t\bY\bigg\|_2
 = \Optilde\left\{m^{1/2}n\rho_n^{3/2}\log^{1/2}(m + n)\right\}.
\end{align*}
The last term can be bounded by Bernstein's inequality as well:
\begin{align*}
\bigg\|\sum_{i = 1}^n\sum_{t = 1}^m\sum_{j = 1}^n\bx_i\by_i\transpose P_{tij}E_{tij}\bigg\|
 = \Optilde\left\{(mn)^{1/2}\rho_n^{3/2}\log^{1/2}(m + n)\right\}.
\end{align*}
It is now sufficient to work with the first term. By Lemma \ref{lemma:decoupling_inequality}, there exists a constant $C > 0$, such that for any $\tau > 0$,
\begin{align*}
\prob\left\{\|\bX\transpose\calH(\bE\bE\transpose)\bY\|_2 > \tau\right\}\leq C\prob\bigg(\bigg\|\sum_{t = 1}^m\sum_{i,j\in[n],i\neq j}\sum_{l = 1}^nE_{til}\bar{E}_{tjl}\bx_i\by_j\transpose\bigg\|_2 > \frac{\tau}{C}\bigg).
\end{align*}
By Bernstein's equality and a conditioning argument,
\begin{align*}
\prob\bigg\{\bigg\|\sum_{t = 1}^m\sum_{i,j\in[n],i\neq j}\sum_{l = 1}^nE_{til}\bar{E}_{tjl}\bx_i\by_j\transpose\bigg\|_2 > 2b_n(\bar{\bE})\tau^2 + 2v_n(\bar{\bE})\tau\mathrel{\bigg|}\bar{\bE}\bigg\}\leq 2e^{-\tau^2},
\end{align*}
where 
\begin{align*}
v_n^2(\bar{\bE}) := \sum_{t = 1}^m\sum_{i = 1}^n\sum_{l = 1}^n\sigma_{til}^2\bigg\|\sum_{j = 1}^n\bar{E}_{tjl}\bx_i\by_j\transpose\mathbbm{1}(j\neq i)\bigg\|_2^2,\;
b_n(\bar{\bE}) := \max_{t\in[m],i,l\in[n]}\bigg\|\sum_{j = 1}^n\bar{E}_{tjl}\bx_i\by_j\transpose\mathbbm{1}(j\neq i)\bigg\|_2.
\end{align*}
Next, we consider the error bounds for $v_n$ and $b_n$. By Bernstein's inequality again and a union bound over $t\in[m]$, $i,l\in[n]$, we have $b_n(\bar{\bE}) = \Optilde\{\rho_n^{1/2}\|\bX\|_{2\to\infty}\log^{1/2} (m + n) + \|\bX\|_{2\to\infty} \|\bY\|_{2\to\infty}\log (m + n)\}$. For $v_n$, by Hanson-Wright inequality for bounded random variables (see Theorem 3 in \cite{bellec2019concentration}),
\begin{align*}
v_n^2(\bar{\bE}) &\leq 2\rho_n\sum_{t = 1}^m\sum_{i = 1}^n\sum_{l = 1}^n\|\bx_i\|_2^2\bigg\|\sum_{j\leq l}^n\bar{E}_{tjl}\by_j\mathbbm{1}(j\neq i)\bigg\|_2^2 + 2\rho_n\sum_{t = 1}^m\sum_{i = 1}^n\sum_{l = 1}^n\|\bx_i\|_2^2\bigg\|\sum_{j > l}^n\bar{E}_{tjl}\by_j\mathbbm{1}(j\neq i)\bigg\|_2^2\\
&\leq 2\rho_n\sum_{t = 1}^m\sum_{i = 1}^n\sum_{l = 1}^n\|\bx_i\|_2^2\sum_{j_1,j_2\leq l}^n(\bar{E}_{tj_1l}\bar{E}_{tj_2l} - \expect \bar{E}_{tj_1l}\bar{E}_{tj_2l})\by_{j_1}\transpose\by_{j_2}\mathbbm{1}(j_1,j_2\neq i)
\\&\quad
 + 2\rho_n\sum_{t = 1}^m\sum_{i = 1}^n\sum_{l = 1}^n\|\bx_i\|_2^2\sum_{j_1,j_2 > l}^n(\bar{E}_{tj_1l}\bar{E}_{tj_2l} - \expect \bar{E}_{tj_1l}\bar{E}_{tj_2l})\by_{j_1}\transpose\by_{j_2}\mathbbm{1}(j_1,j_2\neq i)\\
&\quad + 2\rho_n^2\sum_{t = 1}^m\sum_{i = 1}^n\sum_{l = 1}^n\sum_{j > l}\|\bx_i\|_2^2\|\by_j\|_2^2 + 2\rho_n^2\sum_{t = 1}^m\sum_{i = 1}^n\sum_{l = 1}^n\sum_{j\leq l}\|\bx_i\|_2^2\|\by_j\|_2^2\\
& = \Optilde(mn\rho_n^2).
\end{align*}
Then for any $c > 0$, there exists a $c$-dependent constant $C_c > 0$, such that 
$\calE_{3n}(c):=\{
\bar{\bE}:v_n(\bar{\bE})\leq C_c\eps_{v_n},b_n(\bar{\bE})\leq C_c\eps_{b_n}
\}$
occurs with probability at least $1 - O((m + n)^{-c})$, where
$\eps_{v_n}:= m^{1/2}n^{1/2}\rho_n$,
$\eps_{b_n}:= \rho_n^{1/2}n^{-1/2}\log^{1/2} (m + n) + n^{-1}\log (m + n)$.
Namely, for $\tau = c^{1/2}\log^{1/2}(m + n)$,
\begin{align*}
&\prob\bigg\{\bigg\|\sum_{t = 1}^m\sum_{i,j\in[n],i\neq j}\sum_{l = 1}^nE_{til}\bar{E}_{tjl}\bx_i\by_j\transpose\bigg\|_2 > 2\eps_{b_n}c\log(m + n) + 2\eps_{v_n}c^{1/2}\log^{1/2}(m + n)\bigg\}\\
&\quad\leq \expect_{\bar{\bE}}\bigg[
\prob\bigg\{\bigg\|\sum_{t = 1}^m\sum_{i,j\in[n],i\neq j}\sum_{l = 1}^nE_{til}\bar{E}_{tjl}\bx_i\by_j\transpose\bigg\|_2 > 2\eps_{b_n}c\log(m + n) + 2\eps_{v_n}c^{1/2}\log^{1/2}(m + n)\mathrel{\bigg|}\bar{\bE}\bigg\}\mathbbm{1}_{\calE_{3n}(c)}
\bigg]
\\&\quad\quad
 + \prob\{\calE_{3n}^c(c)\}\\
&\quad\leq \expect_{\bar{\bE}}\left[
\prob\left\{\left\|\sum_{t = 1}^m\sum_{i,j\in[n],i\neq j}\sum_{l = 1}^nE_{til}\bar{E}_{tjl}\bx_i\by_j\transpose\right\|_2 > 2b_nc\log(m + n) + 2v_nc^{1/2}\log^{1/2}(m + n)\mathrel{\bigg|}\bar{\bE}\right\}\mathbbm{1}_{\calE_{3n}(c)}
\right]\\
&\quad\quad + O((m + n)^{-c})\\
&\quad = O((m + n)^{-c}).
\end{align*}
Namely,
$\|\bX\transpose\calH(\bE\bE\transpose)\bY\|_2 
 = \Optilde\{m^{1/2}n^{1/2}\rho_n\log^{1/2}(m + n)\}$.
Combining the above error bounds completes the proof.
\end{proof}

\begin{lemma}
\label{lemma:rowwise_concentration_noise}
Suppose Assumption \ref{assumption:eigenvector_delocalization} hold, $(\bE_t)_{t = 1}^m$ are the matrices in Section \ref{sub:JSE}, $\max_{t\in[m],i,j\in[n]}\expect E_{tij}^2\leq \rho_n$, and $m^{1/2}n\rho_n = \Omega(\log(m + n))$. 
Then 
$\|\{\calH(\bA\bA\transpose) - \bP\bP\transpose - \bM\}\bV\|_{2\to\infty}
 = \Optilde(\zeta_{\mathsf{op}}/\sqrt{n})$, where $\zeta_{\mathsf{op}}$ is defined in \eqref{eqn:zeta_op_noise_level}. 
\end{lemma}

\begin{proof}
By triangle inequality, Bernstein's inequality, Theorem 3 in \cite{doi:10.1080/01621459.2022.2054817}, and Lemma \ref{lemma:rowwise_concentration_linear_term}, for any fixed $i\in[n]$,
\begin{align*}
&\left\|\be_i\transpose\{\calH(\bA\bA\transpose) - \bP\bP\transpose - \bM\}\bV\right\|_2\\
&\quad \leq 
\|\be_i\transpose\calH(\bE\bE\transpose)\bV\|_2 + \|\bU\|_{2\to\infty}\bigg\|\sum_{t = 1}^m\bP_t\bE_t\bigg\|_2 + \bigg\|\sum_{t = 1}^m\be_i\transpose\bE_t\bU\bB_t\bigg\|_2
  + \bigg|\sum_{t = 1}^m\sum_{j = 1}^n2P_{tij}E_{tij}\bigg|\|\bV\|_{2\to\infty}\\
&\quad = \|\be_i\transpose\calH(\bE\bE\transpose)\bV\|_2 + \|\bU\|_{2\to\infty}\Optilde\left\{m^{1/2}(n\rho_n)^{3/2}\log^{1/2} (m + n) + n^{1/2}\rho_n\log(m + n)\right\}
\\&\quad\quad
 + \|\bU\|_{2\to\infty}\Optilde\left\{m^{1/2}(n\rho_n)^{3/2}\log^{1/2} (m + n)\right\}
 + \|\bU\|_{2\to\infty}\Optilde\left\{(mn)^{1/2}\rho_n^{3/2}\log^{1/2}(m + n)\right\}
 \\
&\quad = \|\be_i\transpose\calH(\bE\bE\transpose)\bV\|_2 + \|\bU\|_{2\to\infty}\Optilde\left\{m^{1/2}(n\rho_n)^{3/2}\log^{1/2}(m + n)\right\}.
\end{align*}
By Lemma \ref{lemma:decoupling_inequality}, we have 
$\|\be_i\transpose\calH(\bE\bE\transpose)\bV\|_2 
 = \Optilde\{(mn)^{1/2}\rho_n\log(m + n)\}$.
Combining the above error bounds and applying a union bound over $i\in[n]$ complete the proof.
\end{proof}

\section{Simple Remainder Analyses}
\label{sec:simple_remainder_analyses}

This section presents some simple analyses of remainders $\bR_4^{(RS)}$--$\bR_7^{(RS)}$, which are quite straightforward by leveraging the concentration results obtained in Section \ref{sec:preliminary_results}. 

\begin{lemma}\label{lemma:remainder_II}
Suppose Assumptions \ref{assumption:eigenvector_delocalization} and \ref{assumption:condition_number} hold. If $\|\bM - \widehat{\bM}_{RS}\|_2\leq (1/4)\lambda_d(\bP\bP\transpose)$, then
\begin{align*}
&\|\sin\Theta(\widehat{\bU}_{RS}, \bV)\|_2
 = \Optilde
\bigg(\frac{\zeta_{\mathsf{op}}}{mn^2\rho_n^2}\bigg) + O\bigg(\frac{1}{mn^2\rho_n^2}\bigg)\|\bM - \widehat{\bM}_{RS}\|_2,\quad\|\widehat{\bS}_{RS}^{-1}\|_2 = \Optilde\bigg(\frac{1}{mn^2\rho_n^2}\bigg),\\
&\|\bS\bV\transpose\widehat{\bU}_{RS} - \bV\transpose\widehat{\bU}_{RS}\widehat{\bS}_{RS}\|_2
 = \Optilde\bigg(\frac{\zeta_{\mathsf{op}}}{\sqrt{n}} + \frac{\zeta_{\mathsf{op}}^2}{mn^2\rho_n^2}
  \bigg)
 + \Optilde(1)\|\bM - \widehat{\bM}_{RS}\|_2,
\\
&\|\bR_4^{(RS)}\|_{2\to\infty} = \Optilde\bigg(\frac{\zeta_{\mathsf{op}}}{mn^3\rho_n^2} + 
\frac{\zeta_{\mathsf{op}}^2}{m^2n^{9/2}\rho_n^4}
 \bigg)
 + 
\Optilde\bigg(\frac{1}{mn^{5/2}\rho_n^2}\bigg)\|\bM - \widehat{\bM}_{RS}\|_2,
\end{align*}
where $\zeta_{\mathsf{op}}$ is defined in \eqref{eqn:zeta_op_noise_level}. 
\end{lemma}

\begin{proof}
By triangle inequality and Lemma \ref{lemma:spectral_norm_concentration_noise},
\begin{align*}
\|\calH(\bA\bA\transpose) - \bP\bP\transpose - \widehat{\bM}_{RS}\|_2 &\leq 
\|\calH(\bA\bA\transpose) - \bP\bP\transpose - \bM\|_2 + \|\widehat{\bM}_{RS} - \bM\|_2 = 
\Optilde(\zeta_{\mathsf{op}}) + \|\bM - \widehat{\bM}_{RS}\|_2.
\end{align*} 
Note that $m^{1/2}n\rho_n = \omega(\log(m + n))$ implies 
\begin{align}\label{eqn:SNR}
\frac{\zeta_{\mathsf{op}}}{\lambda_d(\bP\bP\transpose)} = o(1).
\end{align}
This also entails that there exists a $c$-dependent constant $N_c > 0$, such that for all $n\geq N_c$,
\begin{align}\label{eqn:spectral_norm_noise_bound}
\|\calH(\bA\bA\transpose) - \bP\bP\transpose - \widehat{\bM}_{RS}\|_2\leq \frac{1}{2}\lambda_d(\bP\bP\transpose)
\end{align}
with probability at least $1 - O((m + n)^{-c})$. 
Then by Davis-Kahan theorem \cite{doi:10.1137/0707001}, 
\begin{align*}
\|\sin\Theta(\widehat{\bU}_{RS}, \bV)\|_2
 = \Optilde
\left(\frac{\zeta_{\mathsf{op}}}{mn^2\rho_n^2}\right) + O\left(\frac{1}{mn^2\rho_n^2}\right)\|\bM - \widehat{\bM}_{RS}\|_2,
\end{align*}
which establishes the first assertion. The second assertion follows directly from Weyls' inequality, \eqref{eqn:spectral_norm_noise_bound}, and Assumption \ref{assumption:condition_number}. 
For the third assertion, recall that $\bV\bS = \bP\bP\transpose\bV$ and $\widehat{\bU}_{RS}\widehat{\bS}_{RS} = \{\calH(\bA\bA\transpose) - \widehat{\bM}_{RS}\}\widehat{\bU}_{RS}$. Then
\begin{align*}
\bS\bV\transpose\widehat{\bU}_{RS} - \bV\transpose\widehat{\bU}_{RS}\widehat{\bS}_{RS}
& = \bV\transpose\bP\bP\transpose\widehat{\bU}_{RS} - \bV\transpose\{\calH(\bA\bA\transpose) - \widehat{\bM}_{RS}\}\widehat{\bU}_{RS}\\
& = -\bV\transpose\{\calH(\bA\bA\transpose) - \bP\bP\transpose - \bM\}\widehat{\bU}_{RS}
 - \bV\transpose(\bM - \widehat{\bM}_{RS})\widehat{\bU}_{RS}\\
& = -\bV\transpose\{\calH(\bA\bA\transpose) - \bP\bP\transpose - \bM\}\bV\bW_{RS}\transpose
\\&\quad
 -\bV\transpose\{\calH(\bA\bA\transpose) - \bP\bP\transpose - \bM\}(\widehat{\bU}_{RS} - \bV\bW_{RS}\transpose)
\\&\quad
 - \bV\transpose(\bM - \widehat{\bM}_{RS})\widehat{\bU}_{RS}.
\end{align*}
By Lemma \ref{lemma:concentration_noise_quadratic_form}, we have
$\|\bV\transpose\{\calH(\bA\bA\transpose) - \bP\bP\transpose - \bM\}\bV\bW_{RS}\transpose\|_2
 = \Optilde(n^{-1/2}\zeta_{\mathsf{op}})$.
By Lemma \ref{lemma:spectral_norm_concentration_noise} and the first assertion,
\begin{align*}
\left\|\bV\transpose\{\calH(\bA\bA\transpose) - \bP\bP\transpose - \bM\}(\widehat{\bU}_{RS} - \bV\bW_{RS}\transpose)\right\|_2
& = \Optilde\bigg(\frac{\zeta_{\mathsf{op}}^2}{mn^2\rho_n^2}\bigg) + 
\Optilde\bigg(\frac{\zeta_{\mathsf{op}}}{mn^2\rho_n^2}\bigg)\|\bM - \widehat{\bM}_{RS}\|_2.
\end{align*}
Hence, we conclude that
\begin{align*}
&\|\bS\bV\transpose\widehat{\bU}_{RS} - \bV\transpose\widehat{\bU}_{RS}\widehat{\bS}_{RS}\|_2\\
&\quad\leq \|\bV\transpose\{\calH(\bA\bA\transpose) - \bP\bP\transpose - \bM\}\bV\|_2
 + \|\bV\transpose\{\calH(\bA\bA\transpose) - \bP\bP\transpose - \bM\}(\widehat{\bU}_{RS} - \bV\bW_{RS}\transpose)\|_2
 + \|\bM - \widehat{\bM}_{RS}\|_2\\
&\quad = \Optilde\bigg(\frac{1}{\sqrt{n}}\zeta_{\mathsf{op}} + \frac{\zeta_{\mathsf{op}}^2}{mn^2\rho_n^2}\bigg)
 + \Optilde(1)\|\bM - \widehat{\bM}_{RS}\|_2,
\end{align*}
and hence,
\begin{align*}
\|\bR_4^{(RS)}\|_{2\to\infty} &\leq \|\bU\|_{2\to\infty}\|\bS\bV\transpose\widehat{\bU}_{RS} - \bV\transpose\widehat{\bU}_{RS}\widehat{\bS}_{RS}\|_2\|\widehat{\bS}^{-1}_{RS}\|_2
\\&
 = 
\Optilde\bigg(\frac{\zeta_{\mathsf{op}}}{mn^3\rho_n^2} + 
\frac{\zeta_{\mathsf{op}}^2}{m^2n^{9/2}\rho_n^4}\bigg)
 + 
\Optilde\bigg(\frac{1}{mn^{5/2}\rho_n^2}\bigg)\|\bM - \widehat{\bM}_{RS}\|_2.
\end{align*}
The proof is thus completed.
\end{proof}

\begin{lemma}\label{lemma:remainder_IV}
Suppose Assumption \ref{assumption:eigenvector_delocalization} and \ref{assumption:condition_number} hold. If $\|\bM - \widehat{\bM}_{RS}\|_2\leq (1/4)\lambda_d(\bP\bP\transpose)$, then
\begin{align*}
\|\bR_6^{(RS)}\|_{2\to\infty} + \|\bR_7^{(RS)}\|_{2\to\infty} = \Optilde\bigg(\frac{\zeta_{\mathsf{op}}^2}{m^2n^{9/2}\rho_n^4} + \frac{\zeta_{\mathsf{op}}}{mn^3\rho_n^2}\bigg)
 +  \Optilde\bigg(\frac{1}{mn^{5/2}\rho_n^2}\bigg)\|\bM - \widehat{\bM}_{RS}\|_2.
\end{align*}
\end{lemma}

\begin{proof}
Recall that $\bW_{RS}\transpose = \mathrm{sgn}(\bU\transpose\widehat{\bU}_{RS})$. By Lemma \ref{lemma:spectral_norm_concentration_noise}, Davis-Kahan theorem, and Lemma \ref{lemma:remainder_II}, we have
\begin{align*}
&\|\bW_{RS}\transpose\widehat{\bS}_{RS}^{-1} - \bS^{-1}\bW_{RS}\transpose\|_2\\
&\quad \leq 
\|\bS^{-1}\|_2\|\bS\bW_{RS}\transpose - \bW_{RS}\transpose\widehat{\bS}_{RS}\|_2\|\widehat{\bS}_{RS}^{-1}\|_2\\
&\quad \leq \|\bS^{-1}\|_2\|\bS\|_2\|\bW_{RS}\transpose - \bV\transpose\widehat{\bU}_{RS}\|_2\|\widehat{\bS}_{RS}^{-1}\|_2 + \|\bS^{-1}\|_2\|\bS\bV\transpose\widehat{\bU}_{RS} - \bV\transpose\widehat{\bU}_{RS}\widehat{\bS}_{RS}\|_2\|\widehat{\bS}_{RS}^{-1}\|_2\\
&\quad\quad + \|\bS^{-1}\|_2\|\bW_{RS}\transpose - \bV\transpose\widehat{\bU}_{RS}\|_2\|\widehat{\bS}_{RS}\|_2\|\widehat{\bS}_{RS}^{-1}\|_2\\
&\quad = 
\Optilde\bigg(\frac{1}{mn^2\rho_n^2}\bigg)\|\sin\Theta(\widehat{\bU}_{RS}, \bV)\|_2^2
+ \Optilde\bigg(\frac{1}{m^2n^4\rho_n^4}\bigg)\|\bS\bV\transpose\widehat{\bU}_{RS} - \bV\transpose\widehat{\bU}_{RS}\widehat{\bS}_{RS}\|_2\\
&\quad = \Optilde\bigg(\frac{\zeta_{\mathsf{op}}^2}{m^3n^6\rho_n^6} + \frac{\zeta_{\mathsf{op}}}{m^2n^{9/2}\rho_n^4}\bigg)
 + \Optilde\bigg(\frac{1}{m^2n^4\rho_n^4}\bigg)\|\bM - \widehat{\bM}_{RS}\|_2,
\end{align*}
where we have used the condition $\|\bM - \widehat{\bM}_{RS}\|_2 = O(\lambda_d(\bP\bP\transpose)) = O(mn^2\rho_n^2)$. 
Then by Lemma \ref{lemma:rowwise_concentration_noise} and a union bound over $i\in [n]$, we have
\begin{align*}
\|\bR_6^{(RS)}\|_{2\to\infty}& = \|\{\calH(\bA\bA\transpose) - \bP\bP\transpose - \bM\}\bV\|_{2\to\infty}\|\bW_{RS}\transpose\widehat{\bS}_{RS}^{-1} - \bS^{-1}\bW_{RS}\transpose\|_2\\
& = \Optilde\bigg(\frac{\zeta_{\mathsf{op}}^2}{m^2n^{9/2}\rho_n^4} + \frac{\zeta_{\mathsf{op}}}{mn^3\rho_n^2}\bigg)
 +  \Optilde\bigg(\frac{1}{mn^{5/2}\rho_n^2}\bigg)\|\bM - \widehat{\bM}_{RS}\|_2.
\end{align*}
For the last assertion, we directly obtain
\begin{align*}
\|\bR_7^{(RS)}\|_{2\to\infty}&\leq \|\bM - \widehat{\bM}_{RS}\|_{\infty}\|\bV\|_{2\to\infty}\|\bW_{RS}\transpose\widehat{\bS}_{RS}^{-1} - \bS^{-1}\bW_{RS}\transpose\|_2 = \Optilde\left(\frac{1}{mn^{5/2}\rho_n^2}\right)\|\bM - \widehat{\bM}_{RS}\|_2.
\end{align*}
The proof is thus completed.
\end{proof}

\begin{lemma}\label{lemma:remainder_III}
Suppose Assumption \ref{assumption:eigenvector_delocalization}--\ref{assumption:condition_number}  hold, $m^{1/2}n\rho_n = \omega(\log(m + n))$, and $\|\bM - \widehat{\bM}_{RS}\|_2\leq (1/4)\lambda_d(\bP\bP\transpose)$. Then
\[
\|\bR_2^{(RS)}\|_{2\to\infty} + \|\bR_5^{(RS)}\|_{2\to\infty} = \Optilde\bigg(\frac{\zeta_{\mathsf{op}}^2}{m^2n^4\rho_n^4}\bigg)\|\bU\|_{2\to\infty} + O\bigg(\frac{1}{mn^2\rho_n^2}\bigg)\|\bM - \widehat{\bM}_{RS}\|_2\|\bU\|_{2\to\infty}.
\]
\end{lemma}

\begin{proof}
The lemma follows from $\|\bV(\bV\transpose\widehat{\bU}_{RS} - \bW_{RS}\transpose)\bW_{RS}\|_{2\to\infty}
\leq \|\bV\|_{2\to\infty}\|\sin\Theta(\widehat{\bU}_{RS}, \bV)\|_2^2$,
Assumption \ref{assumption:eigenvector_delocalization}, and Lemma \ref{lemma:remainder_II}. 
\end{proof}

\section{Leave-One-Out and Leave-Two-Out Analyses}
\label{sec:LOO_analysis}

In this section, we elaborate on the decoupling arguments based on the delicate leave-one-out and leave-two-out analyses for $\bR^{(RS)}_1$ and establish the corresponding sharp error bounds. Before proceeding to the proofs, we first introduce the notions of the leave-one-out and leave-two-out matrices. For each $t\in [m]$ and $i\in [n]$, let $\bA_t^{(i)} = [A_{tab}^{(i)}]_{n\times n}$ be the $i$th leave-one-out version of $\bA_t$ defined as follows: 
\[
A_{tab}^{(i)} = \left\{
\begin{aligned}
&A_{tab},&\quad&\mbox{if }a\neq i\mbox{ and }b\neq i,\\
&\expect A_{tab},&\quad&\mbox{if }a = i\mbox{ or }b = i,
\end{aligned}
\right.\quad a,b\in [n]
\]
In other words, $\bA_t^{(i)}$ is constructed by replacing the $i$th row and the $i$th column of $\bA_t$ with their expected values. Let $\bE^{(i)}_t = \bA^{(i)}_t - \bP_t$, $\bA^{(i)} = [\bA_1^{(i)},\ldots,\bA_m^{(i)}]$, and $\bE^{(i)} = [\bE_1^{(i)},\ldots,\bE_m^{(i)}]$. For any $r\in[R]$ and $s\in [S]$, set $\widehat{\bU}_{0s}^{(i)} = \zero_{n\times d}$, $\widehat{\bM}_{r0}^{(i)} = \zero_{n\times n}$, and
\begin{align*}
\widehat{\bM}_{rs}^{(i)} = -\sum_{k = 1}^n\be_k\be_k\transpose(\widehat{\bU}_{(r - 1)S}^{(i)})(\widehat{\bU}_{(r - 1)S}^{(i)})\transpose\left\{\calH((\bA^{(i)})(\bA^{(i)})\transpose) - \widehat{\bM}_{r(s - 1)}^{(i)}\right\}(\widehat{\bU}_{(r - 1)S}^{(i)})(\widehat{\bU}_{(r - 1)S}^{(i)})\transpose\be_k\be_k\transpose.
\end{align*}
Let $\widehat{\bU}^{(i)}_{rS} = \texttt{eigs}(\calH(\bA^{(i)}(\bA^{(i)})\transpose) - \widehat{\bM}_{rS}^{(i)}; d)$ with the corresponding eigenvalues encoded in the diagonal matrix $\widehat{\bS}^{(i)}_{rS}$, where
\[
\widehat{\bS}^{(i)}_{rS} = \mathrm{diag}\left[
\lambda_1\{\calH(\bA^{(i)}(\bA^{(i)})\transpose) - \widehat{\bM}_{rS}^{(i)}\},\ldots,\lambda_d\{\calH(\bA^{(i)}(\bA^{(i)})\transpose) - \widehat{\bM}_{rS}^{(i)}\}\right].
\]
Define $\bH^{(i)}_{rS} = (\widehat{\bU}^{(i)}_{rS})\transpose\bV$. Note that $\widehat{\bU}_{rS}^{(i)}$ is a function of $\bA^{(i)}$ and $\widehat{\bM}_{rS}^{(i)}$, and $\bM_{rS}^{(i)}$ is a function of $\bA^{(i)}$ and $\widehat{\bU}_{(r - 1)S}^{(i)}$. Since $\widehat{\bU}_{0s}^{(i)} = \zero_{n\times d}$ for all $s\in[S]$, it follows that $\widehat{\bU}_{rS}^{(i)}$ is also a function of $\bA^{(i)}$, so that $\widehat{\bU}_{rS}^{(i)}$ is independent of $\be_i\transpose\bE = (E_{tij}:t\in[m],j\in[n])$. This independent structure is the key to the decoupling arguments. 

Similarly, for any fixed $t\in[m]$, $(i, j)\in [n]$, $i\neq j$, define the $(i, j)$th leave-two-out version $\bA_t^{(i, j)} = [A_{tab}^{(i, j)}]_{n\times n}$ of $\bA_t$ as
\[
A_{tab}^{(i, j)} = \left\{
\begin{aligned}
&A_{tab},&\quad&\mbox{if }a\notin \{i, j\}\mbox{ and }b\notin \{i, j\},\\
&\expect A_{tab},&\quad&\mbox{if }a \in \{i, j\}\mbox{ or }b \in \{i, j\}.
\end{aligned}
\right.,\quad a,b\in [n]
\]
Rather than setting one row and one column of $\bA_t$ as their expected values, the leave-two-out version $\bA_t^{(i, j)}$ converts its two designated rows and two designated columns to their expected values. Let $\bE_t^{(i, j)} = \bA_t^{(i, j)} - \bP_t$, $\bA^{(i, j)} = [\bA_1^{(i, j)}, \ldots, \bA_m^{(i, j)}]$, and $\bE^{(i, j)} = [\bE_1^{(i, j)},\ldots,\bE_m^{(i, j)}]$. 
For any $r\in[R]$ and $s\in[S]$, set $\widehat{\bU}_{0s}^{(i, j)} = \zero_{n\times d}$, $\widehat{\bM}_{r0}^{(i, j)} = \zero_{n\times n}$, and
\begin{align*}
\widehat{\bM}_{rs}^{(i, j)} = -\sum_{k = 1}^n\be_k\be_k\transpose(\widehat{\bU}_{(r - 1)S}^{(i, j)})(\widehat{\bU}_{(r - 1)S}^{(i, j)})\transpose\left\{\calH((\bA^{(i, j)})(\bA^{(i, j)})\transpose) - \widehat{\bM}_{r(s - 1)}^{(i, j)}\right\}(\widehat{\bU}_{(r - 1)S}^{(i, j)})(\widehat{\bU}_{(r - 1)S}^{(i, j)})\transpose\be_k\be_k\transpose.
\end{align*}
Let $\widehat{\bU}^{(i, j)}_{rS} = \texttt{eigs}(\calH(\bA^{(i, j)}(\bA^{(i, j)})\transpose) - \widehat{\bM}_{rS}^{(i, j)}; d)$ with the diagonal matrix of the associated eigenvalues
\[
\widehat{\bS}^{(i, j)}_{rS} = \mathrm{diag}\left[
\lambda_1\{\calH(\bA^{(i, j)}(\bA^{(i, j)})\transpose) - \widehat{\bM}_{rS}^{(i, j)}\},\ldots,\lambda_d\{\calH(\bA^{(i, j)}(\bA^{(i, j)})\transpose) - \widehat{\bM}_{rS}^{(i, j)}\}\right].
\]
and let $\bH^{(i, j)}_{rS} = (\widehat{\bU}_r^{(i, j)})\transpose\bV$. Note that by a similar reasoning, $\widehat{\bU}_{rS}^{(i, j)}$ is also independent of $(E_{tia}, E_{tja}:t\in[m],a\in[n])$. 

\subsection{Preliminary Lemmas for Leave-One-Out Matrices}
\label{sub:preliminary_lemmas_LOO_matrices}
We first collect several concentration inequalities regarding certain quadratic functions of $\bE$. These results are non-trivial, and our proofs rely on delicate analyses of the higher-order moments of these polynomials of random variables.  
\begin{lemma}\label{lemma:LOO_preliminary}
Let $(\bE_t)_{t = 1}^m$ be the matrices defined in Section \ref{sub:JSE} and suppose Assumptions \ref{assumption:condition_number} holds. Let $(\bX^{(i)})_{i = 1}^n$ be a collection of $n\times d$ random matrices indexed by $i\in[n]$ such that $\bX^{(i)}\neq \zero_{n\times d}$ with probability one, and $\be_i\transpose\bE = (E_{tij}:t\in[m],j\in[n])$ and $\bX^{(i)}$ are independent for each $i\in[n]$. Then
\begin{align*}
\max_{i\in[n]}\frac{\sum_{a = 1}^n\|\sum_{t = 1}^m\sum_{j\in[n]\backslash\{a\}}E_{tia}E_{tij}\be_j\transpose\bX^{(i)}\|_2^2}{\|\bX^{(i)}\|_{2\to\infty}^2} = \Optilde\left\{mn^2\rho_n^2(\log n)^2\right\}.
\end{align*}
\end{lemma}

\begin{proof}
Let $\bx_j^{(i)}$ be the $j$th row of $\bX^{(i)}$
and $p\in\mathbb{N}_+$ be a positive integer to be determined later. 
Expanding the quantity of interest directly yields
\begin{align*}
&\sum_{a = 1}^n\bigg\|\sum_{t = 1}^m\sum_{j\in[n]\backslash\{a\}}E_{tia}E_{tij}\be_j\transpose\bX^{(i)}\bigg\|_2^2\\
&\quad = \sum_{j_1 = 1}^n\sum_{t = 1}^m\sum_{s = 1}^m\sum_{j_2\in[n]\backslash\{j_1\}}\sum_{j_3\in[n]\backslash\{j_1\}}E_{tij_1}E_{sij_1}E_{tij_2}E_{sij_3}\langle \bx_{j_2}^{(i)},\bx_{j_3}^{(i)}\rangle\\
&\quad = \sum_{t = 1}^m\sum_{j_1 = 1}^n\sum_{j_2\in[n]\backslash\{j_1\}}
E_{tij_1}^2E_{tij_2}^2\|\bx_{j_2}^{(i)}\|_2^2 + \sum_{t = 1}^n\sum_{j_1\neq j_2\neq j_3}E_{tij_1}^2E_{tij_2}E_{tij_3}\langle\bx_{j_2}^{(i)},\bx_{j_3}^{(i)}\rangle\\
&\quad\quad + \sum_{t_1,t_2\in[m],t_1\neq t_2}\sum_{j_1 = 1}^n\sum_{j_2,j_3\in[n]\backslash\{j_1\}}E_{t_1ij_1}E_{t_1ij_2}E_{t_2ij_1}E_{t_2ij_3}\langle\bx_{j_2}^{(i)},\bx_{j_3}^{(i)}\rangle,
\end{align*}
where the summation over $\{j_1\neq j_2\neq j_3\}$ means the summation over $\{(j_1,j_2,j_3)\in[n]^3:j_1\neq j_2,j_1\neq j_3,j_2\neq j_3\}$. Let $\bar{\bE}$ and $\tilde{\bE}$ be independent copies of $\bE$. For the first term, by the decouping inequality \cite{10.1214/aop/1176988291}, it is sufficient to consider the first term with $E_{tij_2}$ replaced by $\bar{E}_{tij_2}$. If $n\rho_n\geq 1$, then by Bernstein's inequality, a condition argument, and a union bound over $t\in[m]$, we have
\begin{align*}
\sum_{t = 1}^m\sum_{j_1 = 1}^n\sum_{j_2\in[n]\backslash\{j_1\}}
\frac{E_{tij_1}^2\bar{E}_{tij_2}^2\|\bx_{j_2}^{(i)}\|_2^2}{\|\bX^{(i)}\|_{2\to\infty}^2}
&\leq \sum_{t = 1}^m\sum_{j_1 = 1}^n(E_{tij_1}^2 - \sigma_{tij_1}^2)\bigg(\max_{t\in[m]}\sum_{j_2 = 1}^n\bar{E}_{tij_2}^2\bigg) + mn\rho_n\bigg(\max_{t\in[m]}\sum_{j_2 = 1}^n\bar{E}_{tij_2}^2\bigg)\\
& \leq \Optilde(mn\rho_n)\bigg\{\max_{t\in[m]}\sum_{j_2 = 1}^n(\bar{E}_{tij_2}^2 - \sigma_{tij_2}^2) + n\rho_n\bigg\} = \Optilde(mn^2\rho_n^2\log n).
\end{align*}
If $n\rho_n\leq 1$, then similarly, we have
\begin{align*}
\sum_{t = 1}^m\sum_{j_1 = 1}^n\sum_{j_2\in[n]\backslash\{j_1\}}
\frac{E_{tij_1}^2\bar{E}_{tij_2}^2\|\bx_{j_2}^{(i)}\|_2^2}{\|\bX^{(i)}\|_{2\to\infty}^2}
&\leq \sum_{t = 1}^m\sum_{j_1 = 1}^n(E_{tij_1}^2 - \sigma_{tij_1}^2)\sum_{j_2 = 1}^n\bar{E}_{tij_2}^2 + \rho_n\bigg(\sum_{t = 1}^m\sum_{j_1 = 1}^n\sum_{j_2 = 1}^n\bar{E}_{tij_2}^2\bigg)\\
& \leq \Optilde\bigg\{mn^2\rho_n^2 + (\log n)^2 + (n\rho_n)^{1/2}\log n\bigg(\sum_{t = 1}^m\sum_{j = 1}^n\bar{E}_{tij}^2\bigg)^{1/2}\bigg\}\\
& = \Optilde(mn^2\rho_n^2\log n).
\end{align*}
Similarly, applying the decoupling inequality \cite{10.1214/aop/1176988291} to the second term, we see that it is sufficient to consider the second term with $E_{tij_2}$ and $E_{tij_3}$ replaced by $\bar{E}_{tij_2}$ and $\tilde{E}_{tij_3}$, respectively. 
If $n\rho_n = \Omega(\log n)$, then by Bernstein's inequality, a condition argument, and union bounds, we have
\begin{align*}
&\sum_{t = 1}^m\sum_{j_1\neq j_2\neq j_3}\frac{E_{tij_1}^2\bar{E}_{tij_2}\tilde{E}_{tij_3}\langle\bx_{j_2}^{(i)},\bx_{j_3}^{(i)}\rangle}{\|\bX^{(i)}\|_{2\to\infty}^2}\\
&\quad = \sum_{t = 1}^m\sum_{j_1\neq j_2\neq j_3}\frac{(E_{tij_1}^2 - \sigma_{tij_1}^2)\bar{E}_{tij_2}\tilde{E}_{tij_3}\langle\bx_{j_2}^{(i)},\bx_{j_3}^{(i)}\rangle}{\|\bX^{(i)}\|_{2\to\infty}^2}
+ \sum_{t = 1}^m\sum_{j_1\neq j_2\neq j_3}\frac{\bar{E}_{tij_2}\tilde{E}_{tij_3}\langle\bx_{j_2}^{(i)},\bx_{j_3}^{(i)}\rangle\sigma_{tij_1}^2}{\|\bX^{(i)}\|_{2\to\infty}^2}\\
&\quad = \Optilde(mn\rho_n)\bigg(\max_{t\in[m]}\max_{j_1\in[n]}
\bigg\|\sum_{j_2\in[n]\backslash\{j_1\}}\frac{\bar{E}_{tij_2}\bx_{j_2}^{(i)}}{\|\bX^{(i)}\|_{2\to\infty}}\bigg\|_2\bigg)\bigg(\max_{t\in[m]}\max_{j_1,j_2\in[n],j_1\neq j_2}\bigg\|\sum_{j_3\in[n]\backslash\{j_1,j_2\}}\frac{\tilde{E}_{tij_3}\bx_{j_3}^{(i)}}{\|\bX^{(i)}\|_{2\to\infty}}\bigg\|_2\bigg)
\\&\quad
 = \Optilde(mn^2\rho_n^2\log n).
\end{align*}
If $n\rho_n = O(\log n)$, then by Bernstein's inequality, a conditioning argument, and a union bound over $t\in[m],j_1\in[n]$, we have
\begin{align*}
&\max_{t\in[m],j_2\in[n]}\bigg|\sum_{j_1,j_3\in[n]\backslash\{j_2\},j_1\neq j_3}
\frac{\widetilde{E}_{tij_3}\langle\bx_{j_2}^{(i)},\bx_{j_3}^{(i)}\rangle\sigma_{tij_1}^2}{\|\bX^{(i)}\|_{2\to\infty}^2}
\bigg| = \Optilde\{(\log n)^2\}
,\\
&\max_{t\in[m],j_1\in[n]}
\bigg|\sum_{j_2,j_3\in[n]\backslash\{j_1\},j_2\neq j_3}\frac{\bar{E}_{tij_2}\tilde{E}_{tij_3}\langle\bx_{j_2}^{(i)},\bx_{j_3}^{(i)}\rangle}{\|\bX^{(i)}\|_{2\to\infty}^2}\bigg| = \Optilde\{(\log n)^2\},\\
&\sum_{t = 1}^m\sum_{j_1\neq j_2\neq j_3}|\widetilde{E}_{tij_3}|
 \leq n^2\sum_{t = 1}^m\sum_{j_3 = 1}^n(|\widetilde{E}_{tij_3}| - \expect|\widetilde{E}_{tij_3}|) + n^2\sum_{t = 1}^m\sum_{j_3 = 1}^n\expect|\widetilde{E}_{tij_3}| = \Optilde(mn^3\rho_n),\\
&\sum_{t = 1}^m\sum_{j_1\neq j_2\neq j_3}\frac{\bar{E}_{tij_2}\tilde{E}_{tij_3}\langle\bx_{j_2}^{(i)},\bx_{j_3}^{(i)}\rangle\sigma_{tij_1}^2}{\|\bX^{(i)}\|_{2\to\infty}^2} = \Optilde\{(\log n)^3 + (\rho_n\log n)^{1/2}\log n(mn^3\rho_n^2)^{1/2}\} = \Optilde(mn^2\rho_n^2\log n).
\end{align*}
Also, a similar argument shows that when $n\rho_n = O(\log n)$, $\max_{t\in[m]}\sum_{j_3 = 1}^n|\widetilde{E}_{tij_3}| = \Optilde(\log n)$, so that 
\begin{align*}
&\sum_{t = 1}^m\sum_{j_1\neq j_2\neq j_3}|\bar{E}_{tij_2}\widetilde{E}_{tij_3}|\leq n\sum_{t = 1}^m\sum_{j_2 = 1}^n(|\bar{E}_{tij_2}| - \expect|\bar{E}_{tij_2}|)\sum_{j_3 = 1}^n|\widetilde{E}_{tij_3}| + n\sum_{t = 1}^m\sum_{j_2 = 1}^n\expect|\bar{E}_{tij_2}|\sum_{j_3 = 1}^n|\widetilde{E}_{tij_3}|\\
&\quad = \Optilde(mn^2\rho_n\log n)\\
&\sum_{t = 1}^m\sum_{j_1\neq j_2\neq j_3}\frac{(E_{tij_1}^2 - \sigma_{tij_1}^2)\bar{E}_{tij_2}\tilde{E}_{tij_3}\langle\bx_{j_2}^{(i)},\bx_{j_3}^{(i)}\rangle}{\|\bX^{(i)}\|_{2\to\infty}^2}\\
&\quad = \Optilde\bigg\{
\max_{t\in[m],j_1\in[n]}
\bigg|\sum_{j_2,j_3\in[n]\backslash\{j_1\},j_2\neq j_3}\frac{\bar{E}_{tij_2}\tilde{E}_{tij_3}\langle\bx_{j_2}^{(i)},\bx_{j_3}^{(i)}\rangle}{\|\bX^{(i)}\|_{2\to\infty}^2}\bigg|\log n\\
&\quad\quad\quad\quad + (\rho_n\log n)^{1/2}
\max_{t\in[m],j_1\in[n]}
\bigg|\sum_{j_2,j_3\in[n]\backslash\{j_1\},j_2\neq j_3}\frac{\bar{E}_{tij_2}\tilde{E}_{tij_3}\langle\bx_{j_2}^{(i)},\bx_{j_3}^{(i)}\rangle}{\|\bX^{(i)}\|_{2\to\infty}^2}\bigg|^{1/2}\bigg(\sum_{t = 1}^m\sum_{j_1\neq j_2\neq j_3}|\bar{E}_{tij_2}\widetilde{E}_{tij_3}|\bigg)^{1/2}
\bigg\}\\
&\quad = \Optilde\{(\log n)^3 + (\rho_n\log n)^{1/2}(mn^2\rho_n\log n)^{1/2}\log n\} = \Optilde(mn^2\rho_n^2\log n).
\end{align*}
Therefore, when $n\rho_n = O(\log n)$, we further obtain
\begin{align*}
&\sum_{t = 1}^m\sum_{j_1\neq j_2\neq j_3}\frac{E_{tij_1}^2\bar{E}_{tij_2}\tilde{E}_{tij_3}\langle\bx_{j_2}^{(i)},\bx_{j_3}^{(i)}\rangle}{\|\bX^{(i)}\|_{2\to\infty}^2}\\
&\quad = \sum_{t = 1}^m\sum_{j_1\neq j_2\neq j_3}\frac{(E_{tij_1}^2 - \sigma_{tij_1}^2)\bar{E}_{tij_2}\tilde{E}_{tij_3}\langle\bx_{j_2}^{(i)},\bx_{j_3}^{(i)}\rangle}{\|\bX^{(i)}\|_{2\to\infty}^2}
+ \sum_{t = 1}^m\sum_{j_1\neq j_2\neq j_3}\frac{\bar{E}_{tij_2}\tilde{E}_{tij_3}\langle\bx_{j_2}^{(i)},\bx_{j_3}^{(i)}\rangle\sigma_{tij_1}^2}{\|\bX^{(i)}\|_{2\to\infty}^2}\\
&\quad = \Optilde(mn^2\rho_n^2\log n).
\end{align*}
It is now sufficient to derive the $\Optilde(\cdot)$ bound for the third term. We achieve this by a higher-order moment bound and Markov's inequality. Let $p = O(\log(m + n))$ be a positive integer to be determined later. 
We expand the $p$th moment of the third term and compute
\begin{align*}
&\expect\bigg[\bigg\{ \sum_{t_1,t_2\in[m],t_1\neq t_2}\sum_{j_1 = 1}^n\sum_{j_2,j_3\in[n]\backslash\{j_1\}}E_{t_1ij_1}E_{t_1ij_2}E_{t_2ij_1}E_{t_2ij_3}\langle\bx_{j_2}^{(i)},\bx_{j_3}^{(i)}\rangle\bigg\}^p\mathrel{\bigg|}\bX^{(i)}\bigg]\\
&\quad = 
\sum_{a_1,\ldots,a_p\in[n]}\sum_{\substack{t_1,\ldots,t_p\in[m]\\s_1\in[m]\backslash\{t_1\},\ldots,s_p\in[m]\backslash\{t_p\}}}
\sum_{\substack{j_1\in[n]\backslash\{a_1\},\ldots,j_{p}\in[n]\backslash\{a_p\}\\
l_1\in[n]\backslash\{a_1\},\ldots,l_{p}\in[n]\backslash\{a_p\}
}}\expect\bigg(\prod_{k = 1}^{p}E_{t_kia_k}E_{s_kia_k}E_{t_kij_k}E_{s_kil_k}\bigg)\prod_{k = 1}^p\langle \bx_{j_k}^{(i)},\bx_{l_k}^{(i)}\rangle
\\
&\quad \leq 
\sum_{a_1,\ldots,a_p\in[n]}\sum_{\substack{t_1,\ldots,t_p\in[m]\\s_1\in[m]\backslash\{t_1\},\ldots,s_p\in[m]\backslash\{t_p\}}}
\sum_{\substack{j_1\in[n]\backslash\{a_1\},\ldots,j_{p}\in[n]\backslash\{a_p\}\\
l_1\in[n]\backslash\{a_1\},\ldots,l_{p}\in[n]\backslash\{a_p\}
}}\bigg|\expect\bigg(\prod_{k = 1}^{p}E_{t_kia_k}E_{s_kia_k}E_{t_kij_k}E_{s_kil_k}\bigg)\bigg|
\|\bX^{(i)}\|_{2\to\infty}^{2p}.
\end{align*}
Relabeling $t_{p + k} = s_k$, $j_{p + k} = l_k$ and set $a_{p + k} = a_k$ for $k \in [p]$, we then re-write the right-hand side of the above inequality as
\begin{align*}
\sum_{\substack{a_1,\ldots,a_{2p}\in[n]\\a_{p + 1} = a_1,\ldots,a_{2p} = a_p}}\sum_{\substack{t_1,\ldots,t_{2p}\in[m]\\t_{p + 1}\neq t_1,\ldots,t_{2p}\neq t_p}}
\sum_{\substack{j_1\in[n]\backslash\{a_1\},\ldots,j_{2p}\in[n]\backslash\{a_{2p}\}
}}\bigg|\expect\bigg(\prod_{k = 1}^{2p}E_{t_kia_k}E_{t_kij_k}\bigg)\bigg|
\|\bX^{(i)}\|_{2\to\infty}^{2p}
.
\end{align*}
Now let $T(t_1,\ldots,t_{2p})$ be the number of unique elements among $\{t_1,\ldots,t_{2p}\}$. For $T = T(t_1,\ldots,t_{2p})$, let $t_1^*,\ldots,t_N^*$ be these unique elements and denote $\alpha_r = \sum_{k = 1}^{2p}\mathbbm{1}(t_k = t_r^*)$. In other words, $\alpha_r$ keeps track of the number of times $t_r^*$ appearing in the sequence $(t_k)_{k = 1}^{2p}$. Clearly, $\alpha_1 + \ldots + \alpha_T = 2p$ and $\alpha_1,\ldots,\alpha_T\geq 1$. Furthermore, in order that the expected value
$\expect(\prod_{k = 1}^{2p}E_{t_kia_k}E_{t_kij_k})$
is nonzero, we must have
$\min_{r\in[T]}\alpha_r\geq 2$. Indeed, if there exists some $r\in[T]$ such that $\alpha_r < 2$, then we know that $\alpha_r = 1$ and there exists a unique $k'\in[2p]$ such that $t_{k'} = t_r^*$, so that
\[
\expect\bigg(\prod_{k = 1}^{2p}E_{t_kia_k}E_{t_kij_k}\bigg)
= \expect\bigg(\prod_{k\in[2p]\backslash\{k'\}}E_{t_kia_k}E_{t_kij_k}\bigg)\expect(E_{t_{k'}ia_{k'}}E_{t_{k'}ij_{k'}}) = 0
\]
by the independence of $\bE_1,\ldots,\bE_t$ and the fact that $\expect(E_{t_{k'}ia_{k'}}E_{t_{k'}ij_{k'}}) = 0$ (since $a_{k'}\neq j_{k'}$). Therefore, we obtain $\min_{r\in[T]}\alpha_r\geq 2$, and hence, $T\leq p$. Also, by the constraint $t_{p + k}\neq t_k$ for all $k\in[p]$, we have $T\geq 2$.

Next, let $J(a_1,\ldots,a_{2p},j_1,\ldots,j_{2p})$ be the number of unique elements among $\{a_1,\ldots,a_{2p},j_1,\ldots,j_{2p}\}$, and for $J = J(a_1,\ldots,a_{2p},j_1,\ldots,j_{2p})$, let $j_1^*,\ldots,j_J^*$ be these unique elements. 
This allows us to re-write the expected value in the summand of interest as
\begin{align*}
\expect\bigg(\prod_{k = 1}^{2p}E_{t_kia_k}E_{t_kij_k}\bigg)
&
 = \expect\bigg[\prod_{r = 1}^T\prod_{q = 1}^J
E_{t_r^*ij_q^*}^{\sum_{k = 1}^{2p}\{\mathbbm{1}(t_k = t_r^*,a_k = j_q^*) + \mathbbm{1}(t_k = t_r^*,j_k = j_q^*)\}}\bigg]
 = \prod_{r = 1}^T\prod_{q = 1}^J\expect\left( E_{t_r^*ij_q^*}^{\beta_{rq}}\right),
\end{align*}
where $\beta_{rq} = \sum_{k = 1}^{2p}\{\mathbbm{1}(t_k = t_r^*,a_k = j_q^*) + \mathbbm{1}(t_k = t_r^*,j_k = j_q^*)\}$. Note that by construction, the expected value
 $\expect (E_{t_r^*ij_q^*}^{\beta_{rq}})$ is nonzero only if $\beta_{rq}\geq 2$, and we also have
\[
\sum_{r = 1}^T\sum_{q = 1}^J\beta_{rq} = \sum_{k = 1}^{2p}\sum_{r = 1}^T\sum_{q = 1}^J\mathbbm{1}(t_k = t_r^*,a_k = j_q^*) + \sum_{k = 1}^{2p}\sum_{r = 1}^T\sum_{q = 1}^J\mathbbm{1}(t_k = t_r^*,j_k = j_q^*) = 4p
\]
by definition of $\beta_{rq}$'s. 
This entails that $TJ\leq 2p$, and hence, $J\leq p$ because $T\geq 2$. Also, by the constraint $j_k\neq a_k$ for all $k\in[p]$, we naturally have $J\geq 2$. 

Returning to the summation, we first note that $\expect|E_{tij}|^k\leq \rho_n$ for all $k$ and for all $t\in[m],i,j\in[n]$. This enables us to derive
\begin{align*}
&\frac{1}{\|\bX^{(i)}\|_{2\to\infty}^{2p}}\expect\bigg[\bigg\{ \sum_{t_1,t_2\in[m],t_1\neq t_2}\sum_{j_1,j_2,j_3\in[n]:j_1\neq j_2, j_1\neq j_3}E_{t_1ij_1}E_{t_1ij_2}E_{t_2ij_1}E_{t_2ij_3}\langle\bx_{j_2}^{(i)},\bx_{j_3}^{(i)}\rangle\bigg\}^p\mathrel{\bigg|}\bX^{(i)}\bigg]\\
&\quad \leq 
\sum_{\substack{a_1,\ldots,a_{2p}\in[n]\\a_{p + 1} = a_1,\ldots,a_{2p} = a_p}}\sum_{\substack{t_1,\ldots,t_{2p}\in[m]}}
\sum_{\substack{j_1\in[n]\backslash\{a_1\},\ldots,j_{2p}\in[n]\backslash\{a_{2p}\}}
}\bigg|\expect\bigg(\prod_{k = 1}^{2p}E_{t_kia_k}E_{t_kij_k}\bigg)\bigg|\\
&\quad \leq \sum_{T = 2}^p\sum_{J = 2}^p\mathbbm{1}(TJ\leq 2p)\sum_{t_1^*,\ldots,t_T^*\in[m]}\sum_{j_1^*,\ldots,j_J^*\in[n]}\sum_{\substack{t_1,\ldots,t_{2p}\in[m]\\\{t_1,\ldots,t_{2p}\} =\{t_1^*,\ldots,t_T^*\} }}\sum_{\substack{a_1,\ldots,a_{2p},j_1,\ldots,j_{2p}\in[n]\\ \{a_1,\ldots,a_{2p},j_1,\ldots,j_{2p}\} = \{j_1^*,\ldots,j_J^*\}\\a_{p + 1} = a_1,\ldots,a_{2p} = a_p }}\prod_{r = 1}^T\prod_{q = 1}^J\expect\left|E_{t_r^*ij_q^*}^{\beta_{rq}}\right|\\
&\quad \leq \sum_{T = 2}^p\sum_{J = 2}^p\mathbbm{1}(TJ\leq 2p)\sum_{t_1^*,\ldots,t_T^*\in[m]}\sum_{j_1^*,\ldots,j_J^*\in[n]}\sum_{\substack{t_1,\ldots,t_{2p}\in[m]\\\{t_1,\ldots,t_{2p}\} =\{t_1^*,\ldots,t_T^*\} }}\sum_{\substack{a_1,\ldots,a_{2p},j_1,\ldots,j_{2p}\in[n]\\ \{a_1,\ldots,a_{2p},j_1,\ldots,j_{2p}\} = \{j_1^*,\ldots,j_J^*\}\\a_{p + 1} = a_1,\ldots,a_{2p} = a_p }}\prod_{r = 1}^T\prod_{q = 1}^J\rho_n\\
&\quad \leq \sum_{T = 2}^p\sum_{J = 2}^p\sum_{t_1^*,\ldots,t_T^*\in[m]}\sum_{j_1^*,\ldots,j_J^*\in[n]}\sum_{\substack{t_1,\ldots,t_{2p}\in[m]\\\{t_1,\ldots,t_{2p}\} =\{t_1^*,\ldots,t_T^*\} }}\sum_{\substack{a_1,\ldots,a_{2p},j_1,\ldots,j_{2p}\in[n]\\ \{a_1,\ldots,a_{2p},j_1,\ldots,j_{2p}\} = \{j_1^*,\ldots,j_J^*\}\\a_{p + 1} = a_1,\ldots,a_{2p} = a_p}}\rho_n^{TJ}\mathbbm{1}(TJ\leq 2p)\\
&\quad \leq \sum_{T = 2}^{p}\sum_{J = 2}^{p}\sum_{t_1^*,\ldots,t_T^*\in[m]}\sum_{j_1^*,\ldots,j_J^*\in[n]}\mathbbm{1}(TJ\leq 2p)
T^{2p}J^{3p}
\rho_n^{TJ}\\
&\quad = 
\sum_{T = 2}^{p}\sum_{J = 2}^{p}m^Tn^{TJ}\rho_n^{TJ}(TJ)^{2p}\frac{J^{p}}{n^{(T - 1)J}}\mathbbm{1}(TJ\leq 2p)
 \leq \sum_{T = 2}^{p}\sum_{J = 2}^{p}m^{TJ/2}(n\rho_n)^{TJ}(2p)^{2p}\bigg(\frac{J^p}{n^J}\bigg)\mathbbm{1}(TJ\leq 2p)
\\&\quad
\leq C^pp^{2p}\bigg\{\sum_{T = 2}^{p}\sum_{J = 2}^p(m^{1/2}n\rho_n)^{TJ}\mathbbm{1}(TJ\leq 2p)\bigg\}
\leq C^pp^{2p + 2}(m^{1/2}n\rho_n)^{2p},
\end{align*}
where we have used the inequality $x^p/n^x\leq (p/\log n)^p\leq C^p$ for some constant $C > 0$ for all $x > 0$ and the condition that $m^{1/2}n\rho_n\geq 1$, provided that $p$ is selected such that $p\leq C\log n$ for some constant $C > 0$.
Then by Markov's inequality, for any $K > 0$ and even $p$,
\begin{align*}
&\prob\bigg[\bigg|\sum_{t_1,t_2\in[m],t_1\neq t_2}\sum_{j_1\neq j_2\neq j_3}E_{t_1ij_1}E_{t_1ij_2}E_{t_2ij_1}E_{t_2ij_3}\langle\bx_{j_2}^{(i)},\bx_{j_3}^{(i)}\rangle\bigg| > Kmn^2\rho_n^2\log^2(m + n)\|\bX^{(i)}\|_{2\to\infty}^2\bigg]\\
&\quad = \expect\bigg\{\prob\bigg[\bigg|\sum_{t_1,t_2\in[m],t_1\neq t_2}\sum_{j_1\neq j_2\neq j_3}E_{t_1ij_1}E_{t_1ij_2}E_{t_2ij_1}E_{t_2ij_3}\langle\bx_{j_2}^{(i)},\bx_{j_3}^{(i)}\rangle\bigg| > Kmn^2\rho_n^2\log^2(m + n)\|\bX^{(i)}\|_{2\to\infty}^2\mathrel{\bigg|}\bX^{(i)}\bigg]\bigg\}\\
&\quad\leq 2\exp\left\{p\log C + (2p + 2)\log p - p\log K - 2p\log\log(m + n)\right\}.
\end{align*}
Then by picking a sufficiently large $K$ and $p = 2\lfloor(1/2)\log(m + n)\rfloor$, we see that
\begin{align*}
&\prob\bigg\{\bigg|\sum_{t_1,t_2\in[m],t_1\neq t_2}\sum_{j_1\neq j_2\neq j_3}E_{t_1ij_1}E_{t_1ij_2}E_{t_2ij_1}E_{t_2ij_3}\langle\bx_{j_2}^{(i)},\bx_{j_3}^{(i)}\rangle\bigg| > Kmn^2\rho_n^2\log^2(m + n)\|\bX^{(i)}\|_{2\to\infty}^2\bigg\}\\
&\quad\leq O((m + n)^{-c}).
\end{align*}
Therefore, a union bound over $i\in[n]$ entails that
\begin{align*}
\max_{i\in[n]}\frac{\sum_{a = 1}^n\|\sum_{t = 1}^m\sum_{j\in[n]\backslash\{a\}}E_{tia}E_{tij}\be_j\transpose\bX^{(i)}\|_2^2}{\|\bX^{(i)}\|_{2\to\infty}^2} = \Optilde\left\{mn^2\rho_n^2\log^2(m + n)\right\}.
\end{align*}
The proof is thus completed because $\log(m + n) = \Theta(\log n)$.
\end{proof}

\begin{lemma}\label{lemma:LOO_preliminary_II}
Let $(\bE_t)_{t = 1}^m$ be the matrices defined in Section \ref{sub:JSE} and suppose Assumption \ref{assumption:condition_number} holds. Then
\begin{align*}
\sum_{a \in[n]\backslash\{i\}}\bigg|\sum_{t = 1}^m\sum_{j = 1}^nE_{taj}E_{tij}\bigg|^2 = \Optilde\left\{mn^2\rho_n^2(\log n)^2\right\}.
\end{align*}
\end{lemma}

\begin{proof}
The proof idea is similar to that of Lemma \ref{lemma:LOO_preliminary}, and indeed, is slightly more straightforward. 
Let $p$ be a positive integer to be determined later. We focus on bounding $p$th moment of the quantity of interest.
Write
\begin{align*}
\expect\bigg[\bigg(\sum_{a \in[n]\backslash\{i\}}\bigg|\sum_{t = 1}^m\sum_{j = 1}^nE_{taj}E_{tij}\bigg|^2\bigg)^p\bigg]
 = \sum_{a_1,\ldots,a_p\in[n]\backslash\{i\}}\sum_{\substack{t_1,\ldots,t_p\in[m]\\s_1,\ldots,s_p\in[m]}}\sum_{\substack{j_1,\ldots,j_p\in[n]\\l_1,\ldots,l_p\in[n]}}\expect\bigg(
\prod_{k = 1}^pE_{t_ka_kj_k}E_{t_kij_k}E_{s_ka_kl_k}E_{s_kil_k}\bigg).
\end{align*}
Relabel $a_{p + k} = a_k$, $t_{p + k} = s_k$, $j_{p + k} = l_k$ for all $k\in[p]$. This allows us to write the expected value in the summation above as
$\expect(\prod_{k = 1}^{2p}E_{t_ka_kj_k}E_{t_kij_k})$.
Let $A(a_1,\ldots,a_{2p})$ denote the number of unique elements in $\{a_1,\ldots,a_{2p}\}$, $T(t_1, \ldots, t_{2p})$ denote the number of unique elements in $\{t_1,\ldots,t_{2p}\}$, and $J(j_1,\ldots,j_{2p})$ denote the number of unique elements in $\{j_1,\ldots,j_{2p}\}$. 
Note that $A(a_1,\ldots,a_{2p})\leq p$ because of the constraint $a_{p + k} = a_k$ for all $k\in[p]$. 
For $A = A(a_1,\ldots,a_{2p})$, $T = T(t_1,\ldots,t_{2p})$, $J = J(j_1,\ldots,j_{2p})$, let $a_1^*,\ldots,a_A^*$ be the unique elements among $\{a_1,\ldots,a_p\}$, $t_1^*,\ldots,t_T^*$ be the unique elements among $\{t_1,\ldots,t_{2p}\}$, and $j_1^*,\ldots,j_J^*$ be the unique elements among $\{j_1,\ldots,j_{2p}\}$. For each $v\in[A]$, $r\in[T]$, $q\in[J]$, define $\alpha_r = \sum_{k = 1}^{2p}\mathbbm{1}(t_k = t_r^*)$, $\beta_{rq} = \sum_{k = 1}^{2p}\mathbbm{1}(t_k = t_r^*, j_k = j_q^*)$, and $\gamma_{rqv} = \sum_{k = 1}^{2p}\mathbbm{1}(t_k = t_r^*, j_k = j_q^*, a_k = a_v^*)$. Then the expected value in the summand can be re-written as
\begin{align*}
\bigg|\expect\bigg(\prod_{k = 1}^{2p}E_{t_ka_kj_k}E_{t_kij_k}\bigg)\bigg|
& = \bigg|\expect\bigg\{\prod_{r = 1}^T\prod_{q = 1}^J\prod_{v = 1}^A\prod_{k = 1}^{2p}E_{t_r^*a_v^*j_q^*}^{\mathbbm{1}(t_k = t_r^*,j_k = j_q^*,a_k = a_v^*)}E_{t_r^*ij_q^*}^{\mathbbm{1}(t_k = t_r^*,j_k = j_q^*,a_k = a_v^*)}\bigg\}\bigg|\\
& = \bigg|\expect\bigg\{\prod_{r = 1}^T\prod_{q = 1}^J\prod_{v = 1}^AE_{t_r^*a_v^*j_q^*}^{\sum_{k = 1}^{2p}\mathbbm{1}(t_k = t_r^*,j_k = j_q^*,a_k = a_v^*)}E_{t_r^*ij_q^*}^{\sum_{k = 1}^{2p}\mathbbm{1}(t_k = t_r^*,j_k = j_q^*,a_k = a_v^*)}\bigg\}\bigg|\\
& = \bigg|\expect\bigg\{\prod_{r = 1}^T\prod_{q = 1}^J\prod_{v = 1}^AE_{t_r^*a_v^*j_q^*}^{\gamma_{rqv}}E_{t_r^*ij_q^*}^{\gamma_{rqv}}\bigg\}\bigg|
 = \bigg|\expect\bigg\{\prod_{r = 1}^T\prod_{q = 1}^J\bigg(\prod_{v = 1}^AE_{t_r^*a_v^*j_q^*}^{\gamma_{rqv}}\bigg)\bigg(E_{t_r^*ij_q^*}^{\sum_{v = 1}^A\gamma_{rqv}}\bigg)\bigg\}\bigg|\\
& = \bigg|\expect\bigg\{\prod_{r = 1}^T\prod_{q = 1}^J\bigg(\prod_{v = 1}^AE_{t_r^*a_v^*j_q^*}^{\gamma_{rqv}}\bigg)\bigg(E_{t_r^*ij_q^*}^{\beta_{rq}}\bigg)\bigg\}\bigg|
 = \prod_{r = 1}^T\prod_{q = 1}^J\bigg|\prod_{v = 1}^A\expect\bigg(E_{t_r^*a_v^*j_q^*}^{\gamma_{rqv}}\bigg)\bigg|\bigg|\expect\bigg(E_{t_r^*ij_q^*}^{\beta_{rq}}\bigg)\bigg|.
\end{align*}
Note that in order for the above expected value to be nonzero, it is necessary that $\beta_{rq}\geq 2$ and $\gamma_{rqv}\geq 2$ for all $r\in[T],q\in[J],v\in[A]$. Since $\sum_{r = 1}^T\sum_{q = 1}^J\beta_{rq} = 2p$ and $\sum_{r = 1}^T\sum_{q = 1}^J\sum_{v = 1}^A\gamma_{rqv} = 2p$, it is also necessary that $TJA\leq p$ and hence, $T\leq p$ and $J\leq p$. Observe that $\expect|E_{tij}|^k\leq \rho_n$ for all $t\in[m],i,j\in[n],k\geq 2$ and 
$\expect(\prod_{k = 1}^{2p}E_{t_ka_kj_k}E_{t_kij_k})|
\leq \prod_{r = 1}^T\prod_{q = 1}^J\rho_n^{A + 1}$.
Returning to the sum of interest, we can write it alternatively as follows:
\begin{align*}
&
\expect\bigg\{\bigg(\sum_{a \in[n]\backslash\{i\}}\bigg|\sum_{t = 1}^m\sum_{j = 1}^nE_{taj}E_{tij}\bigg|^2\bigg)^p\bigg\}
\\
&\quad \leq
\sum_{T = 1}^p\sum_{J = 1}^p\sum_{A = 1}^p\mathbbm{1}(TJA\leq p)\sum_{\substack{(a_v^*)_{v = 1}^A\subset[n]\backslash\{i\}\\
(t_r^*)_{r = 1}^T\subset[m]\\
(j_q^*)_{q = 1}^J\subset[n]}}
\sum_{
\substack{
a_1,\ldots,a_{2p}:\{a_1,\ldots,a_{2p}\} = \{a_1^*,\ldots,a_A^*\}\\
t_1,\ldots,t_{2p}:\{t_1,\ldots,t_{2p}\} = \{t_1^*,\ldots,t_T^*\}\\
j_1,\ldots,j_{2p}:\{j_1,\ldots,j_{2p}\} = \{j_1^*,\ldots,j_J^*\}
}}
\bigg|\expect\bigg(
\prod_{k = 1}^pE_{t_ka_kj_k}E_{t_kij_k}E_{s_kal_k}E_{s_kil_k}\bigg)\bigg|\\
&\quad \leq
\sum_{T = 1}^p\sum_{J = 1}^p\sum_{A = 1}^p\mathbbm{1}(TJA\leq p)\sum_{\substack{(a_v^*)_{v = 1}^A\subset[n]\backslash\{i\}\\
(t_r^*)_{r = 1}^T\subset[m]\\
(j_q^*)_{q = 1}^J\subset[n]}}
\sum_{
\substack{
a_1,\ldots,a_{2p}:\{a_1,\ldots,a_{2p}\} = \{a_1^*,\ldots,a_A^*\}\\
t_1,\ldots,t_{2p}:\{t_1,\ldots,t_{2p}\} = \{t_1^*,\ldots,t_T^*\}\\
j_1,\ldots,j_{2p}:\{j_1,\ldots,j_{2p}\} = \{j_1^*,\ldots,j_J^*\}
}}
\rho_n^{(A + 1)TJ}
\\
&\quad\leq p^{2p}\sum_{T = 1}^p\sum_{J = 1}^p\sum_{A = 1}^p\mathbbm{1}(TJA\leq p)m^Tn^{A + J}\rho_n^{(A + 1)TJ}
\leq p^{2p}\sum_{T = 1}^p\sum_{J = 1}^p\sum_{A = 1}^p\mathbbm{1}(TJA\leq p)(m^{1/2}n\rho_n)^{(A + 1)TJ}\\
&\quad\leq p^{2p + 3}(m^{1/2}n\rho_n)^{2p}.
\end{align*}
Then by Markov's inequality, for any $K > 0$,
\begin{align*}
\prob\bigg\{\sum_{a\in[n]\backslash\{i\}}\bigg|\sum_{t = 1}^m\sum_{j = 1}^nE_{taj}E_{tij}\bigg|^2 > Kmn^2\rho_n^2\log^2(m + n)\bigg\}
&\leq \exp\left\{(2p + 3)\log p - p\log K - 2p\log\log(m + n)\right\}. 
\end{align*}
Then for any $c > 0$, with $p = \lfloor\log(m + n)\rfloor$ and $K = e^{c + 3}$, we see immediately that
\[
\prob\bigg\{\sum_{a\in[n]\backslash\{i\}}\bigg|\sum_{t = 1}^m\sum_{j = 1}^nE_{taj}E_{tij}\bigg|^2 > e^{c + 3}mn^2\rho_n^2\log^2(m + n)\bigg\}\leq O((m + n)^{-c}).
\]
The proof is thus completed because $\log(m + n) = \Theta(\log n)$. 
\end{proof}
\begin{lemma}\label{lemma:LOO_stage_I}
Suppose Assumption \ref{assumption:eigenvector_delocalization} and \ref{assumption:condition_number} hold. Let $(\bX^{(i)})_{i = 1}^n$ be random matrices in $\mathbb{O}(n, d)$ such that $\bX^{(i)}\in\mathbb{R}^{n\times d}$ is independent of $\be_i\transpose\bE = (E_{tij}:t\in[m],j\in[n])$  for each $i\in [n]$. Then
\begin{align*}
\max_{i\in[n]}\frac{\|\{\calH((\bA^{(i)})(\bA^{(i)})\transpose) - \calH(\bA\bA\transpose)\}\bX^{(i)}\|_2}{\|\bX^{(i)}\|_{2\to\infty}} = \Optilde(\zeta_{\mathsf{op}}).
\end{align*}
\end{lemma}

\begin{proof}
By definition, 
\begin{align}
\label{eqn:LOO_stage_I}
&\left\{\calH(\bA\bA\transpose) - \calH((\bA^{(i)})(\bA^{(i)})\transpose)\right\}\bX^{(i)}
\\&\quad
 = \sum_{t = 1}^m\calH((\bA_t - \bA_t^{(i)})\bA_t)\bX^{(i)}
 + \sum_{t = 1}^m\calH(\bA_t(\bA_t - \bA_t^{(i)}))\bX^{(i)} + \sum_{t = 1}^m\calH((\bA_t - \bA_t^{(i)})^2)\bX^{(i)}.\nonumber
\end{align}
Below, we work with the three terms on the right-hand side of \eqref{eqn:LOO_stage_I} separately. 

\noindent
$\blacksquare$ \textbf{The first term in \eqref{eqn:LOO_stage_I}.}
For each $j\in [n]$, by definition, we have
\begin{align*}
\calH((\bA_t - \bA_t^{(i)})\bA_t)\be_j
& = \begin{bmatrix}
E_{ti1}A_{tij}&\cdots&E_{ti(i - 1)}A_{tij}&\be_i\transpose\bE_t\bA_t\be_j&E_{ti(i + 1)}A_{tij}&\cdots&E_{tin}A_{tij}
\end{bmatrix}\transpose
 - E_{tij}A_{tij}\be_j
\end{align*}
if $j\neq i$ and $\calH((\bA_t - \bA_t^{(i)})\bA_t)\be_i
 = [E_{ti1}A_{tii},\cdots,E_{ti(i - 1)}A_{tii},0,E_{ti(i + 1)}A_{tii},\cdots,E_{tin}A_{tii}]\transpose$.
Let $\bx_j^{(i)}$ denote the transpose of the $j$th row of $\bX^{(i)}$. Then for any $a\in [n]$, $a\neq i$, we have
\begin{align*}
\be_a\transpose\sum_{t = 1}^m\calH((\bA_t - \bA_t^{(i)})\bA_t)\bX^{(i)}
& = \sum_{t = 1}^m\sum_{j\in[n]\backslash\{i\}}(E_{tia}A_{tij} - E_{tij}A_{tij}\be_a\transpose\be_j)({\bx}_j^{(i)})\transpose + \sum_{t = 1}^mE_{tia}A_{tii}(\bx_i^{(i)})\transpose\\
& = \sum_{t = 1}^m\sum_{j = 1}^nE_{tia}A_{tij}(\bx_j^{(i)})\transpose - E_{tia}A_{ia}(\bx_a^{(i)})\transpose
 = \sum_{t = 1}^m\sum_{j \in[n]\backslash\{a\}}E_{tia}A_{tij}(\bx_j^{(i)})\transpose.
\end{align*}
This implies that
\begin{align*}
&\sum_{a\in[n]\backslash\{i\}}\bigg\|\be_a\transpose\sum_{t = 1}^m\calH((\bA_t - \bA_t^{(i)})\bA_t)\bX^{(i)}\bigg\|_2^2\\
&\quad\lesssim \sum_{a \in[n]\backslash\{i\}}\bigg\|\sum_{t = 1}^m\sum_{j\in[n]\backslash\{a\}}E_{tia}E_{tij}(\bx_j^{(i)})\transpose\bigg\|_2^2 + \sum_{a\in[n]\backslash\{i\}}\bigg\|\sum_{t = 1}^m\sum_{j\in[n]\backslash\{a\}}E_{tia}P_{tij}(\bx_j^{(i)})\transpose\bigg\|_2^2.
\end{align*}
For the second term, by Bernstein's inequality, we write
\begin{align*}
&\sum_{a\in[n]\backslash\{i\}}\bigg\|\sum_{t = 1}^m\sum_{j\in[n]\backslash\{a\}}E_{tia}P_{tij}(\bx_j^{(i)})\transpose\bigg\|_2^2\\
&\quad = \sum_{a\in[n]\backslash\{i\}}\sum_{t,s\in[m]}\sum_{j,l\in[n]\backslash\{a\}}E_{tia}E_{sia}P_{tij}P_{sil}\langle\bx_j^{(a)},\bx_l^{(a)}\rangle\\
&\quad = \sum_{a\in[n]\backslash\{i\}}\sum_{t = 1}^m\sum_{j,l\in[n]\backslash\{a\}}E_{tia}^2P_{tij}P_{til}\langle\bx_j^{(a)},\bx_l^{(a)}\rangle + \sum_{a\in[n]\backslash\{i\}}\sum_{t,s\in[m],t\neq s}\sum_{j,l\in[n]\backslash\{a\}}E_{tia}E_{sia}P_{tij}P_{sil}\langle\bx_j^{(a)},\bx_l^{(a)}\rangle\\
&\quad = \sum_{a\in[n]\backslash\{i\}}\sum_{t = 1}^m(E_{tia}^2 - \sigma_{tia}^2)\bigg(\sum_{j,l\in[n]\backslash\{a\}}P_{tij}P_{til}\langle\bx_j^{(a)},\bx_l^{(a)}\rangle\bigg)
+ \sum_{a\in[n]\backslash\{i\}}\sum_{t = 1}^m\sigma_{tia}^2\bigg(\sum_{j,l\in[n]\backslash\{a\}}P_{tij}P_{til}\langle\bx_j^{(a)},\bx_l^{(a)}\rangle\bigg)\\
&\quad\quad + \sum_{a\in[n]\backslash\{i\}}\sum_{t,s\in[m],t\neq s}\sum_{j,l\in[n]\backslash\{a\}}E_{tia}E_{sia}P_{tij}P_{sil}\langle\bx_j^{(a)},\bx_l^{(a)}\rangle\\
&\quad = \Optilde(mn^3\rho_n^3)\|\bX^{(i)}\|_{2\to\infty}^2 + \sum_{a\in[n]\backslash\{i\}}\sum_{t,s\in[m],t\neq s}\sum_{j,l\in[n]\backslash\{a\}}E_{tia}E_{sia}P_{tij}P_{sil}\langle\bx_j^{(a)},\bx_l^{(a)}\rangle.
\end{align*}
The concentration bound for the second term above can be obtained by applying a decoupling argument. Specifically, let $\{\bar{E}_{sia}:s\in[m],a\in[n]\}$ be an independent copy of $\{E_{sia}:s\in[m],a\in[n]\}$. By a conditioning argument and Bernstein's inequality, we obtain
\begin{align*}
&\sum_{a\in[n]\backslash\{i\}}\sum_{t,s\in[m],t\neq s}\sum_{j,l\in[n]\backslash\{a\}}E_{tia}\bar{E}_{sia}P_{tij}P_{sil}\langle\bx_j^{(a)},\bx_l^{(a)}\rangle\\
&\quad = \Optilde\bigg\{(mn\rho_n)^{1/2}\log^{1/2}(m + n)\max_{a\in[n],t\in[m]}\bigg|\sum_{s\in[m]\backslash\{a\}}\sum_{j,l\in[n]\backslash\{a\}}\bar{E}_{sia}P_{tij}P_{sil}\langle\bx_j^{(a)},\bx_l^{(a)}\rangle\bigg|\bigg\}\\
&\quad = \Optilde\bigg[(mn\rho_n)^{1/2}(\log n)^{1/2}\{\log n + (m\rho_n)^{1/2}(\log n)^{1/2}\}n^2\rho_n^2\|\bX^{(i)}\|_{2\to\infty}^2\bigg]\\
&\quad = \Optilde\{mn^3\rho_n^3\log n\|\bX^{(i)}\|_{2\to\infty}^2\}.
\end{align*}
Then by the decoupling inequality  \cite{10.1214/aop/1176988291}, we obtain
\[
\sum_{a\in[n]\backslash\{i\}}\sum_{t,s\in[m],t\neq s}\sum_{j,l\in[n]\backslash\{a\}}E_{tia}E_{sia}P_{tij}P_{sil}\langle\bx_j^{(a)},\bx_l^{(a)}\rangle = \Optilde\{mn^3\rho_n^3\log n\|\bX^{(i)}\|_{2\to\infty}^2\}.
\]
Therefore, 
\begin{align*}
\max_{i\in[n]}\frac{2\sum_{a\in[n]\backslash\{i\}}\|\sum_{t = 1}^m\sum_{j\in[n]\backslash\{a\}}E_{tia}P_{tij}(\bx_j^{(i)})\transpose\|_2^2}{\|\bX^{(i)}\|_{2\to\infty}^2}
= \Optilde\left\{m^{1/2}(n\rho_n)^{3/2}(\log n)^{1/2}\right\}^2.
\end{align*}
For the first term, Lemma \ref{lemma:LOO_preliminary} yields
\begin{align*}
\max_{i\in[n]}\frac{\sum_{a\in[n]}\|\sum_{t = 1}^m\sum_{j\in[n]\backslash\{a\}}E_{tia}E_{tij}(\bx_j^{(i)})\transpose\|_2^2}{\|\bX^{(i)}\|_{2\to\infty}^2}
 = \Optilde\left\{mn^2\rho_n^2(\log n)^2\right\}.
\end{align*} 
We then proceed to combine the above results and compute
\begin{equation}
\label{eqn:LOO_term_II}
\begin{aligned}
&\max_{i\in[n]}\frac{\sum_{a\in[n]\backslash\{i\}}\|\be_a\transpose\sum_{t = 1}^m\calH((\bA_t - \bA_t^{(i)})\bA_t)\bX^{(i)}\|_2^2}{\|\bX^{(i)}\|_{2\to\infty}^2}
 = \Optilde(\zeta_{\mathsf{op}}^2).
\end{aligned}
\end{equation}
For $a = i$, by definition, we have
\begin{equation}
\label{eqn:LOO_term_I}
\begin{aligned}
\be_i\transpose\sum_{t = 1}^m\calH((\bA_t - \bA_t^{(i)})\bA_t)\bX^{(i)}
& = \sum_{t = 1}^m\sum_{j\in[n]\backslash\{i\}}\sum_{l = 1}^nE_{til}E_{tlj}(\bx_j^{(i)})\transpose + \sum_{t = 1}^m\sum_{j\in[n]\backslash\{i\}}\sum_{l = 1}^nE_{til}P_{tlj}(\bx_j^{(i)})\transpose.
\end{aligned}
\end{equation}
For the first term in \eqref{eqn:LOO_term_I}, note that 
$\sum_{t = 1}^m\sum_{j\in[n]\backslash\{i\}}\sum_{l = 1}^nE_{til}E_{tlj}(\bx_j^{(i)})\transpose = \be_i\transpose\calH(\bE\bE\transpose)\bX^{(i)}$.
Then by Lemma \ref{lemma:decoupling_inequality}
\begin{equation}
\label{eqn:LOO_term_I_1}
\begin{aligned}
\max_{i\in[n]}\frac{\|\sum_{t = 1}^m\sum_{j\in[n]\backslash\{i\}}\sum_{l = 1}^nE_{til}E_{tlj}(\bx_j^{(i)})\transpose\|_2}{\|\bX^{(i)}\|_{2\to\infty}} = \Optilde(m^{1/2}n\rho_n\log n).
\end{aligned}
\end{equation}
For the second term in \eqref{eqn:LOO_term_I}, by Bernstein's inequality,
\begin{equation}
\label{eqn:LOO_term_I_2}
\begin{aligned}
\sum_{t = 1}^m\sum_{\substack{j\in[n]\backslash\{i\}}}\sum_{l = 1}^n
E_{til}P_{tlj}(\bx_j^{(i)})\transpose
& = \max_{t\in[m],l\in[n]}\bigg\|\sum_{j = 1}^nP_{tlj}(\bx_j^{(i)})\transpose\bigg\|_2\Optilde\left\{(mn\rho_n)^{1/2}(\log n)^{1/2}\right\}
\\
& \leq \|\bX\|_{2\to\infty}\Optilde\left\{m^{1/2}(n\rho_n)^{3/2}(\log n)^{1/2}\right\}.
\end{aligned}
\end{equation}
Combining \eqref{eqn:LOO_term_I_1} and \eqref{eqn:LOO_term_I_2} leads to
\begin{align*}
&\max_{i\in[n]}\frac{\|\be_i\transpose\sum_{t = 1}^m\calH((\bA_t - \bA_t^{(i)})\bA_t)\bX^{(i)}\|_2}{\|\bX^{(i)}\|_{2\to\infty}} = \Optilde(\zeta_{\mathsf{op}})
\end{align*}
by \eqref{eqn:LOO_term_I}. Combining the above result with \eqref{eqn:LOO_term_II}, we conclude that
\begin{align*}
\max_{i\in[n]}\frac{\|\sum_{t = 1}^m\calH((\bA_t - \bA_t^{(i)})\bA_t)\bX^{(i)}\|_2}{\|\bX^{(i)}\|_{2\to\infty}} = \Optilde(\zeta_{\mathsf{op}}).
\end{align*}

\noindent$\blacksquare$ \textbf{The second term in \eqref{eqn:LOO_stage_I}.} By definition, for any $j\in[n]$, $\bA_t(\bA_t - \bA_t^{(i)})\be_j = E_{tij}\bA_t\be_i$ if $j\neq i$ and $\bA_t(\bA_t - \bA_t^{(i)})\be_i = \bA_t\bE_t\be_i$.
Then for any $a\in[n]$, 
\begin{align*}
&\be_a\transpose\sum_{t = 1}^m\calH(\bA_t(\bA_t - \bA_t^{(i)}))\bX^{(i)}\\
&\quad = \sum_{t = 1}^m\sum_{j\in[n]\backslash\{i\}}\be_a\transpose\calH(\bA_t(\bA_t - \bA_t^{(i)}))\be_j(\bx_j^{(i)})\transpose + \sum_{t = 1}^m\be_a\transpose\calH(\bA_t(\bA_t - \bA_t^{(i)}))\be_i(\bx_i^{(i)})\transpose\\
&\quad = \sum_{t = 1}^m\sum_{\substack{j = 1\\j\neq i}}^nE_{tia}E_{tij}(\bx_j^{(i)})\transpose\mathbbm{1}(a\neq j) + \sum_{t = 1}^m\sum_{\substack{j = 1\\j\neq i}}^nP_{tia}E_{tij}(\bx_j^{(i)})\transpose\mathbbm{1}(a\neq j)
 + \sum_{t = 1}^m\be_a\transpose\bA_t\bE_t\be_i
(\bx_i^{(i)})\transpose\mathbbm{1}(a\neq i)\\
&\quad = \sum_{t = 1}^m\sum_{\substack{j = 1\\j\neq i}}^nE_{tia}E_{tij}(\bx_j^{(i)})\transpose\mathbbm{1}(a\neq j) + \sum_{t = 1}^m\sum_{\substack{j = 1\\j\neq i}}^nP_{tia}E_{tij}(\bx_j^{(i)})\transpose\mathbbm{1}(a\neq j)\\
&\quad\quad + \sum_{t = 1}^m\sum_{j = 1}^nE_{taj}E_{tij}(\bx_i^{(i)})\transpose\mathbbm{1}(a\neq i) + \sum_{t = 1}^m\sum_{j = 1}^nP_{taj}E_{tij}(\bx_i^{(i)})\transpose\mathbbm{1}(a\neq i).
\end{align*}
We next proceed to write
\begin{align*}
&\bigg\|\sum_{t = 1}^m\calH(\bA_t(\bA_t - \bA_t^{(i)}))\bX^{(i)}\bigg\|_2\\
&\quad\lesssim \bigg\{\sum_{a = 1}^n\bigg\|\sum_{t = 1}^m\sum_{j\in[n]\backslash\{i,a\}}E_{tia}E_{tij}\bx_j^{(i)}\bigg\|_2^2\bigg\}^{1/2} + \sqrt{n}\max_{a\in[n]}\bigg\|\sum_{t = 1}^m\sum_{j\in[n]\backslash\{i,a\}}P_{tia}E_{tij}\bx_j^{(i)}\bigg\|_2\\
&\quad\quad + \|\bX^{(i)}\|_{2\to\infty}\bigg[\bigg\{\sum_{a \in[n]\backslash\{i\}}\bigg|\sum_{t = 1}^m\sum_{j = 1}^nE_{taj}E_{tij}\bigg|_2^2\bigg\}^{1/2} + \sqrt{n}\max_{a \in[n]}\bigg|\sum_{t = 1}^m\sum_{j = 1}^nP_{taj}E_{tij}\bigg|_2\bigg].
\end{align*}
By Lemma \ref{lemma:LOO_preliminary}, the first term satisfies
\begin{align*}
\max_{i\in[n]}\frac{\{\sum_{a = 1}^n\|\sum_{t = 1}^m\sum_{j\in[n]\backslash\{i,a\}}E_{tia}E_{tij}\bx_j^{(i)}\|_2^2\}^{1/2}}{\|\bX^{(i)}\|_{2\to\infty}} = \Optilde(m^{1/2}n\rho_n\log n).
\end{align*}
By Bernstein's inequality and a union bound over $i,a\in[n]$, the second term and the fourth term satisfy
\begin{align*}
&\max_{i\in[n]}\frac{\sqrt{n}\max_{a\in[n]}\|\sum_{t = 1}^m\sum_{j\in[n]\backslash\{i,a\}}P_{tia}E_{tij}\bx_j^{(i)}\|_2}{\|\bX^{(i)}\|_{2\to\infty}}
 = \Optilde\left\{m^{1/2}n\rho_n^{3/2}(\log n)^{1/2}\right\},\\
&\max_{i\in[n]}\frac{\|\bX^{(i)}\|_{2\to\infty}\sqrt{n}\max_{a \in[n]\backslash\{i\}}|\sum_{t = 1}^m\sum_{j = 1}^nP_{taj}E_{tij}|}{\|\bX^{(i)}\|_{2\to\infty}}
 = \Optilde\left\{m^{1/2}n\rho_n^{3/2}(\log n)^{1/2}\right\}.
\end{align*}
For the third term, we apply Lemma \ref{lemma:LOO_preliminary_II} together with a union bound over $i\in[n]$ to obtain
\begin{align*}
&\max_{i\in[n]}\frac{\|\bX^{(i)}\|_{2\to\infty}\{\sum_{a\in[n]\backslash\{i\}}\sum_{t = 1}^m|\sum_{t = 1}^m\sum_{j = 1}^nE_{taj}E_{tij}|^2\}^{1/2}}{\|\bX^{(i)}\|_{2\to\infty}}
 = \Optilde(m^{1/2}n\rho_n\log n).
\end{align*}
Combining the above concentration results yields
\begin{align*}
\max_{i\in[n]}\frac{\|\sum_{t = 1}^m\calH(\bA_t(\bA_t - \bA_t^{(i)}))\bX^{(i)}\|_2}{\|\bX^{(i)}\|_{2\to\infty}}
& = \Optilde(\zeta_{\mathsf{op}}).
\end{align*}

\noindent
$\blacksquare$ \textbf{The third term in \eqref{eqn:LOO_stage_I}.} By definition, 
\begin{align*}
\calH((\bA_t - \bA_t^{(i)})^2)\be_j = [
E_{ti1}E_{tij},\cdots,E_{ti(j - 1)}E_{tij},0,E_{ti(j + 1)}E_{tij},\cdots, E_{tin}E_{tij}]\transpose 
\end{align*}
if $j\neq i$, and 
\begin{align*}
\calH((\bA_t - \bA_t^{(i)})^2)\be_i = 
[E_{ti1}E_{tii}, \cdots, E_{ti(i - 1)}E_{tii}, 0, E_{ti(i + 1)}E_{tii}, \cdots, E_{tin}E_{tii}]\transpose.
\end{align*} 
Therefore, for any $a\in[n]$,
\begin{align*}
\be_a\transpose\sum_{t = 1}^m\calH((\bA_t - \bA_t^{(i)})^2)\bX^{(i)}
& = \sum_{t = 1}^m\sum_{j = 1}^nE_{tia}E_{tij}\mathbbm{1}(a\neq j)(\bx_j^{(i)})\transpose.
\end{align*}
This further entails that
\begin{align*}
\|\sum_{t = 1}^m\calH((\bA_t - \bA_t^{(i)})^2)\bX^{(i)}\|_2
 = \bigg\{\sum_{a\in[n]\backslash\{i\}}\|\sum_{t = 1}^m\sum_{j\in[n]\backslash\{a\}}^nE_{tia}E_{tij}\bx_j^{(i)}\|_2^2\bigg\}^{1/2}
,
\end{align*}
and by Lemma \ref{lemma:LOO_preliminary}, we immediately obtain that
\begin{align*}
\max_{i\in[n]}\frac{\|\sum_{t = 1}^m\calH((\bA_t - \bA_t^{(i)})^2)\bX^{(i)}\|_2}{\|\bX^{(i)}\|_{2\to\infty}} = \Optilde(m^{1/2}n\rho_n\log n).
\end{align*}
Combining the concentration results for the three terms in \eqref{eqn:LOO_stage_I}
completes the proof. 
\end{proof}

\subsection{Concentration Bounds for Leave-One-Out Matrices}
\label{sub:concentration_bounds_LOO_matrices}
We next apply the preliminary concentration results in Section \ref{sub:preliminary_lemmas_LOO_matrices} to obtain sharp leave-one-out error bounds. The proofs of these results primarily rely recursive error bounds and induction. 

\begin{lemma}\label{lemma:LOO_stage_II}
Suppose Assumptions \ref{assumption:eigenvector_delocalization}--\ref{assumption:condition_number} hold. Further assume that, for any $c > 0$, there exists a $c$-dependent constant $N_c > 0$, such that for all $n\geq N_c$, with probability at least $1 - O(n^{-c})$, $\max_{s\in[S],r\in[R]}\|\bM - \widehat{\bM}_{rs}\|_2\leq \lambda_d(\bP\bP\transpose)/8$ and $\max_{i\in[n]}\max_{s\in[S],r\in[R]}\|\bM - \widehat{\bM}_{rs}^{(i)}\|_2\leq \lambda_d(\bP\bP\transpose)/8$.
If $R,S = O(1)$, then
\begin{align*}
\max_{i\in[n]}\|\widehat{\bU}^{(i)}_{RS}\bH^{(i)}_{RS} - \widehat{\bU}_{RS}\bH_{RS}\|_2 & = \Optilde\left(\frac{\zeta_{\mathsf{op}}}{mn^2\rho_n^2}\right)\max_{i\in[n]}\|\widehat{\bU}_{RS}^{(i)}\|_{2\to\infty}.
\end{align*}
\end{lemma}

\begin{proof}
By equation \eqref{eqn:spectral_norm_noise_bound} in the proof of Lemma \ref{lemma:remainder_II} and Weyl's inequality, $\lambda_d\{\calH((\bA^{(i)})(\bA^{(i)})\transpose) - \widehat{\bM}_{RS}^{(i)}\}\geq \lambda_d(\bP\bP\transpose)/2$ and $\lambda_{d + 1}\{\calH(\bA\bA\transpose) - \widehat{\bM}_{RS}\}\leq (1/4)\lambda_d(\bP\bP\transpose)$
with probability at least $1 - O(n^{-c})$. 
This entails that 
\[
\lambda_d\{\calH((\bA^{(i)})(\bA^{(i)})\transpose) - \widehat{\bM}_{RS}^{(i)}\} - \lambda_{d + 1}\{\calH(\bA\bA\transpose) - \widehat{\bM}_{RS}\}\geq\lambda_d(\bP\bP\transpose)/4
\]
with probability at least $1 - O(n^{-c})$. By Davis-Kahan theorem (in the form of Theorem VII 3.4 in \cite{bhatia2013matrix}), we have
\begin{align*}
\|\widehat{\bU}^{(i)}_{RS}\bH^{(i)}_{RS} - \widehat{\bU}_{RS}\bH_{RS}\|_2 
&\leq \frac{\|\{\calH((\bA^{(i)})(\bA^{(i)})\transpose) - \calH(\bA\bA\transpose) - \widehat{\bM}_{RS}^{(i)} + \widehat{\bM}_{RS}\}\widehat{\bU}_{RS}^{(i)}\|_2}{\lambda_d\{\calH((\bA^{(i)})(\bA^{(i)})\transpose) - \widehat{\bM}_{RS}^{(i)}\} - \lambda_{d + 1}\{\calH(\bA\bA\transpose) - \widehat{\bM}_{RS}\}}\\
&\leq \frac{4\|\{\calH((\bA^{(i)})(\bA^{(i)})\transpose) - \calH(\bA\bA\transpose)\}\widehat{\bU}_{RS}^{(i)}\|_2}{m\Delta_n^2}
  + \frac{4\|\widehat{\bM}_{RS} - \widehat{\bM}_{RS}^{(i)}\|_2\|\widehat{\bU}_{RS}^{(i)}\|_{2\to\infty}}{m\Delta_n^2},
\end{align*}
where we have used the fact that $\widehat{\bM}_{RS}$ and $\widehat{\bM}_{RS}^{(i)}$ are diagonal matrices. For the second term, by the definition of $\widehat{\bM}_{Rs}$, $\widehat{\bM}_{Rs}^{(i)}$, and triangle inequality, we have, for any $s\in[S]$,
\begin{align*}
\|\widehat{\bM}_{Rs} - \widehat{\bM}_{Rs}^{(i)}\|_2
&\leq \|\widehat{\bU}_{(R - 1)S}\widehat{\bU}_{(R - 1)S}\transpose - (\widehat{\bU}_{(R - 1)S}^{(i)})(\widehat{\bU}_{(R - 1)S}^{(i)})\transpose\|_2\\
&\quad\times \{\|\calH(\bA\bA\transpose) - \bM\|_2 + \|\calH((\bA^{(i)})(\bA^{(i)})\transpose) - \bM\|_2 + \|\widehat{\bM}_{R(s - 1)} - \bM\|_2 + \|\widehat{\bM}_{R(s - 1)}^{(i)} - \bM\|_2\}
\\&\quad
 + \|\{\calH((\bA^{(i)})(\bA^{(i)})\transpose) - \calH(\bA\bA\transpose)\}\widehat{\bU}_{(R - 1)S}^{(i)}\|_2 + \|\widehat{\bM}_{R(s - 1)} - \widehat{\bM}_{R(s - 1)}^{(i)}\|_2.
\end{align*}
By Lemma \ref{lemma:spectral_norm_concentration_noise}, we know that
$\|\calH(\bA\bA\transpose) - \bM\|_2 = \Optilde(mn^2\rho_n^2)$ and $\|\{\calH((\bA^{(i)})(\bA^{(i)})\transpose) - \bM\|_2 = \Optilde(mn^2\rho_n^2)$ because $\bA^{(i)} = \bP + \bE^{(i)}$ and $\bE^{(i)}$ also satisfies Assumption \ref{assumption:condition_number}.
Note that $\widehat{\bM}_{R0} = \widehat{\bM}_{R0}^{(i)} = \zero_{n\times n}$. Then by induction over $s$, we obtain
\begin{align*}
\|\widehat{\bM}_{RS} - \widehat{\bM}_{RS}^{(i)}\|_2
&\leq \Optilde(mn^2\rho_n^2)\|\widehat{\bU}_{(R - 1)S}\widehat{\bU}_{(R - 1)S}\transpose - (\widehat{\bU}_{(R - 1)S}^{(i)})(\widehat{\bU}_{(R - 1)S}^{(i)})\transpose\|_2\\
&\quad + S\|\{\calH((\bA^{(i)})(\bA^{(i)})\transpose) - \calH(\bA\bA\transpose)\}\widehat{\bU}_{(R - 1)S}^{(i)}\|_2.
\end{align*}
This entails that
\begin{align*}
\|\widehat{\bU}^{(i)}_{RS}\bH^{(i)}_{RS} - \widehat{\bU}_{RS}\bH_{RS}\|_2
&\leq \frac{4\|\{\calH((\bA^{(i)})(\bA^{(i)})\transpose) - \calH(\bA\bA\transpose)\}\widehat{\bU}_{RS}^{(i)}\|_2}{m\Delta_n^2}  + \frac{4\|\widehat{\bM}_{RS} - \widehat{\bM}_{RS}^{(i)}\|_2\|\widehat{\bU}_{RS}^{(i)}\|_{2\to\infty}}{m\Delta_n^2}\\
&\leq \frac{4\|\{\calH((\bA^{(i)})(\bA^{(i)})\transpose) - \calH(\bA\bA\transpose)\}\widehat{\bU}_{RS}^{(i)}\|_2}{m\Delta_n^2}
\\&\quad
 + \Optilde(1)\|\widehat{\bU}_{(R - 1)S}\widehat{\bU}_{(R - 1)S}\transpose - (\widehat{\bU}_{(R - 1)S}^{(i)})(\widehat{\bU}_{(R - 1)S}^{(i)})\transpose\|_2\|\widehat{\bU}_{RS}^{(i)}\|_{2\to\infty}\\
&\quad + \frac{S\|\{\calH((\bA^{(i)})(\bA^{(i)})\transpose) - \calH(\bA\bA\transpose)\}\widehat{\bU}_{(R - 1)S}^{(i)}\|_2}{m\Delta_n^2}\|\widehat{\bU}_{RS}^{(i)}\|_{2\to\infty}.
\end{align*}
Observe that $\widehat{\bU}_{RS}^{(i)}$ is independent of $\be_i\transpose\bE = (E_{tij}:t\in[m],j\in[n])$. Then by Lemma \ref{lemma:LOO_stage_I}, for any $c > 0$, there exists a constant $C_c > 0$ such that for sufficiently large $n$, with probability at least $1 - O(n^{-c})$,
\begin{align*}
&\max_{i\in[n]}\|\widehat{\bU}^{(i)}_{RS}\bH^{(i)}_{RS} - \widehat{\bU}_{RS}\bH_{RS}\|_2\\
&\quad\leq \Optilde\bigg(\frac{\zeta_{\mathsf{op}}}{mn^2\rho_n^2}\bigg)
\max_{i\in[n]}\|\widehat{\bU}_{RS}^{(i)}\|_{2\to\infty}
+ \Optilde\bigg(\frac{\zeta_{\mathsf{op}}}{mn^2\rho_n^2}\bigg)
\max_{i\in[n]}\|\widehat{\bU}_{(R - 1)S}^{(i)}\|_{2\to\infty}
\|\widehat{\bU}_{RS}^{(i)}\|_{2\to\infty}
\\&\quad\quad
 + \Optilde(1)\max_{i\in[n]}\|\widehat{\bU}_{(R - 1)S}\widehat{\bU}_{(R - 1)S}\transpose - (\widehat{\bU}_{(R - 1)S}^{(i)})(\widehat{\bU}_{(R - 1)S}^{(i)})\transpose\|_2\max_{i\in[n]}\|\widehat{\bU}_{RS}^{(i)}\|_2\\
&\quad \leq \Optilde\bigg(\frac{\zeta_{\mathsf{op}}}{mn^2\rho_n^2}\bigg)\max_{i\in[n]}\|\widehat{\bU}_{RS}^{(i)}\|_{2\to\infty}
 + \Optilde(1)\max_{i\in[n]}\left\|\widehat{\bU}_{(R - 1)S}^{(i)}\bH_{(R - 1)S}^{(i)} - \widehat{\bU}_{(R - 1)S}\bH_{(R - 1)S}\right\|_2\max_{i\in[n]}\|\widehat{\bU}_{RS}^{(i)}\|_{2\to\infty}.
\end{align*}
The remaining proof is completed by induction over $R$ and observe that $\widehat{\bU}_{0S} = \widehat{\bU}_{0S}^{(i)} = \zero_{n\times d}$ for all $i\in[n]$. 
\end{proof}

\begin{lemma}\label{lemma:LOO_Utilde_two_to_infinity_norm}
Suppose Assumptions \ref{assumption:eigenvector_delocalization} and \ref{assumption:condition_number} hold. Further assume that for any $c > 0$, there exists a $c$-dependent constant $N_c$, such that for any $n \geq N_c$, with probability at least $1 - O(n^{-c})$, $\max_{r\in[R],s\in[S]}\|\bM - \widehat{\bM}_{rs}\|_2\leq \lambda_d(\bP\bP\transpose)/20$ and $\max_{i\in[n]}\max_{r\in[R],s\in[S]}\|\bM - \widehat{\bM}_{rs}^{(i)}\|_2\leq \lambda_d(\bP\bP\transpose)/20$.
If $R,S = O(1)$, then for any $c > 0$, there exist a $c$-dependent constants $N_c > 0$, such that for any $n\geq N_c$, 
$\max_{i\in[n]}\|(\bH^{(i)}_{RS})^{-1}\|_2\leq 2$ and 
$\max_{i\in[n]}\|\widehat{\bU}^{(i)}_{RS}\|_{2\to\infty} \leq 4\|\widehat{\bU}_{RS}\|_{2\to\infty}$
with probability at least $1 - O(n^{-c})$. Consequently,
\begin{align*}
\max_{i\in[n]}\|\widehat{\bU}^{(i)}_{RS}\bH^{(i)}_{RS} - \widehat{\bU}_{RS}\bH_{RS}\|_2 & = \Optilde\left(\frac{\zeta_{\mathsf{op}}}{mn^2\rho_n^2}\right)\|\widehat{\bU}_{RS}\|_{2\to\infty},\;
\max_{i\in[n]}\|\widehat{\bU}^{(i)}_{RS}\bH^{(i)}_{RS} - \bV\|_{2\to\infty} = \Optilde(1)\|\widehat{\bU}_{RS}\|_{2\to\infty}.
\end{align*}
\end{lemma}

\begin{proof}
The ``consequently'' part follows from the second assertion, Lemma \ref{lemma:LOO_stage_II}, and the triangle inequality that
\[
\max_{i\in[n]}\|\widehat{\bU}^{(i)}_{RS}\bH^{(i)}_{RS} - \bV\|_{2\to\infty}
\leq \max_{i\in[n]}\|\widehat{\bU}^{(i)}_{RS}\bH^{(i)}_{RS} - \widehat{\bU}_{RS}\bH_{RS}\|_2 + \|\widehat{\bU}_{RS}\|_{2\to\infty} + \|\bV\|_{2\to\infty}.
\] 
It is thus sufficient to establish the first and second assertions. 
By Lemma \ref{lemma:spectral_norm_concentration_noise} and a union bound over $i\in[n]$, there exists some $c$-dependent constant $N_c$, such that
\begin{align*}
\max_{i\in[n]}\frac{\|\calH((\bA^{(i)})(\bA^{(i)})\transpose) - \bP\bP\transpose - \widehat{\bM}_{RS}^{(i)}\|_2}{\lambda_d(\bP\bP\transpose)}
&\leq \max_{i\in[n]}\frac{\|\calH((\bA^{(i)})(\bA^{(i)})\transpose) - \bP\bP\transpose - \bM\|_2}{m\Delta_n^2} + \max_{i\in[n]}\frac{\|\bM - \widehat{\bM}_{RS}^{(i)}\|_2}{m\Delta_n^2}\\
&\leq 
\frac{1}{2}
\end{align*}
with probability at least $1 - O(n^{-c})$ for all $n\geq N_c$. By Lemma 2 in \cite{abbe-fan-wang-zhong-2020}, 
\begin{align}\label{eqn:Hri_error_bound}
\prob\left\{\max_{i\in[n]}\|(\bH^{(i)}_{RS})^{-1}\|_2\leq 2\right\}\geq 1 - O(n^{-c})\quad\mbox{for all }n\geq N_c.
\end{align}
By Lemma \ref{lemma:LOO_stage_II}, there exists a $c$-dependent constant $C_c > 0$, such that
\begin{align*}
\max_{i\in[n]}\|\widehat{\bU}_{RS}^{(i)}\|_{2\to\infty}
&\leq \max_{i\in[n]}\|\widehat{\bU}_{RS}^{(i)}\bH^{(i)}_{RS}\|_{2\to\infty}\|(\bH^{(i)}_{RS})^{-1}\|_2
\leq 2\max_{i\in[n]}\|\widehat{\bU}_{RS}^{(i)}\bH^{(i)}_{RS}\|_{2\to\infty}
\\&
\leq 2\max_{i\in[n]}\|\widehat{\bU}_{RS}^{(i)}\bH^{(i)}_{RS} - \widehat{\bU}_{RS}\bH_{RS}\|_{2\to\infty} + 2\|\widehat{\bU}_{RS}\|_{2\to\infty}
\\&
\leq 2\|\widehat{\bU}_{RS}\|_{2\to\infty} + \frac{2C_c\zeta_{\mathsf{op}}}{m\Delta_n^2}
\max_{i\in[n]}\|\widehat{\bU}_{RS}^{(i)}\|_{2\to\infty}
\end{align*}
with probability at least $1 - O(n^{-c})$ whenever $n\geq N_c$. Namely, 
\begin{align*}
\frac{1}{2}\max_{i\in[n]}\|\widehat{\bU}_{RS}^{(i)}\|_{2\to\infty}
&\leq \left(1 - \frac{2C_c\zeta_{\mathsf{op}}}{m\Delta_n^2}\right)\max_{i\in[n]}\|\widehat{\bU}_{RS}^{(i)}\|_{2\to\infty}
\leq 2\|\widehat{\bU}_{RS}\|_{2\to\infty},
\end{align*}
and hence, $\max_{i\in[n]}\|\widehat{\bU}_{RS}^{(i)}\|_{2\to\infty}
\leq 4\|\widehat{\bU}_{RS}\|_{2\to\infty}$
with probability at least $1 - O(n^{-c})$ for all $n\geq N_c$.
\end{proof}

\begin{lemma}\label{lemma:LOO_Rowwise_concentration}
Suppose the conditions of Lemma \ref{lemma:LOO_Utilde_two_to_infinity_norm} hold. Further assume that, for any $c > 0$, there exists a $c$-dependent constant $N_c > 0$, such that for any $n > N_c$, 
$\max_{i,j\in[n],i\neq j,r\in[R],s\in[S]}\|\bM - \widehat{\bM}_{rs}^{(i,j)}\|_2 \leq \lambda_d(\bP\bP\transpose)/20$
 with probability at least $1 - O(n^{-c})$. Then 
\begin{align*}
&\max_{i\in[n]}\|\be_i\transpose\calH(\bE\bE\transpose)(\widehat{\bU}_{RS}^{(i)}\bH_{RS}^{(i)} - \bV)\|_2
 =  \Optilde\bigg(q_n\|\widehat{\bU}_{RS}\|_{2\to\infty} + \frac{\zeta_{\mathsf{op}}\log n}{m^{1/2}n^{3/2}\rho_n}\bigg)
 + \Optilde\bigg(\frac{\log n}{m^{1/2}n^{3/2}\rho_n}\bigg)
\|\bM - \widehat{\bM}_{RS}\|_2
,\\
 &\max_{i\in[n]}\frac{\|\be_i\transpose\{\calH(\bA\bA\transpose) - \bP\bP\transpose - \bM\}(\widehat{\bU}_{RS}^{(i)}\bH_{RS}^{(i)} - \bV)\|_2}{m\Delta_n^2}
\\&\quad
 = \Optilde\bigg\{\frac{\log n}{m^{1/2}n\rho_n} + \frac{(\log n)^{1/2}}{(mn\rho_n)^{1/2}}\bigg\}\|\widehat{\bU}_{RS}\|_{2\to\infty}
 + \Optilde\bigg(\frac{\log n}{m^{3/2}n^{7/2}\rho_n^3}\bigg)\|\bM - \widehat{\bM}_{RS}\|_2,
\end{align*}
where $q_n
 = (\log n)^2 
+ \zeta_{\mathsf{op}}\log n(n\rho_n\vee\log n)^{1/2}/(mn^2\rho_n^2)
$.
\end{lemma}

\begin{proof}
For the first claim, a simple algebra shows that
\[
\be_i\transpose\calH(\bE\bE\transpose)
 = \sum_{t = 1}^m\be_i\transpose\bE_t\bE_t^{(i)} + 
\sum_{t = 1}^m E_{tii}(\be_i\transpose\bE_t - E_{tii}\be_i\transpose)
\]
and
\begin{equation}
\begin{aligned}
\label{eqn:LOO_rowwise_concentration}
\be_i\transpose\calH(\bE\bE\transpose)(\widehat{\bU}_{RS}^{(i)}\bH_{RS}^{(i)} - \bV)
 = \sum_{t = 1}^m\be_i\transpose\bE_t\bE_t^{(i)}(\widehat{\bU}_{RS}^{(i)}\bH_{RS}^{(i)} - \bV)
 + \sum_{t = 1}^mE_{tii}(\be_i\transpose\bE_t - E_{tii}\be_i\transpose)
(\widehat{\bU}_{RS}^{(i)}\bH_{RS}^{(i)} - \bV).
\end{aligned}
\end{equation}
Below, we analyze the two terms on the right-hand side of \eqref{eqn:LOO_rowwise_concentration} separately. 

\noindent
$\blacksquare$
For the first term in \eqref{eqn:LOO_rowwise_concentration}, observe that $(\be_i\transpose\bE_t)_{t = 1}^m$ and $\{\bE_t^{(i)}(\widehat{\bU}_{RS}^{(i)}\bH_{RS}^{(i)} - \bV)\}_{t = 1}^m$ are independent. Also, note that by Lemma \ref{lemma:LOO_Utilde_two_to_infinity_norm}, Lemma \ref{lemma:spectral_norm_concentration_noise}, Davis-Kahan theorem, and a union bound over $i\in[n]$, we have
\begin{align*}
\max_{i\in[n]}\|\widehat{\bU}^{(i)}_{RS}\bH^{(i)}_{RS} - \bV\|_{2}
& \leq \max_{i\in[n]}\|\widehat{\bU}^{(i)}_{RS}\bH^{(i)}_{RS} - \widehat{\bU}_{RS}\bH_{RS}\|_2 + \|\widehat{\bU}_{RS}\bH_{RS} - \bV\|_2
\\&
 = 
\Optilde\left(\frac{\zeta_{\mathsf{op}}}{mn^2\rho_n^2}\right) + O\left(\frac{1}{mn^2\rho_n^2}\right)\|\bM - \widehat{\bM}_{RS}\|_2.
\end{align*}
By Lemma \ref{lemma:norm_concentration_Et} and Lemma \ref{lemma:spectral_norm_concentration_Hollow_EEtranspose}, we have
\begin{align*}
\|\bE\bE\transpose\|_2\leq \|\calH(\bE\bE\transpose)\|_2 + \max_{i\in[n]}\sum_{t = 1}^m\sum_{j = 1}^nE_{tij}^2 = \Optilde(m^{1/2}n\rho_n\log n + mn\rho_n)\leq \Optilde(mn\rho_n\log n).
\end{align*} 
By a union bound, $\max_{i\in[n]}\|(\bE^{(i)})(\bE^{(i)})\transpose\|_2 = \Optilde(mn\rho_n\log n)$.
Then by Bernstein's inequality, Lemma \ref{lemma:norm_concentration_Et}, and a union bound over $i\in[n]$,
\begin{align*}
&
\max_{i\in[n]}\bigg\|\sum_{t = 1}^m\be_i\transpose\bE_t\bE_t^{(i)}(\widehat{\bU}_{RS}^{(i)}\bH_{RS}^{(i)} - \bV)\bigg\|_2
\\
&
\quad
 = \max_{i\in[n]}\max_{t\in[m]}\left\|\bE_t^{(i)}(\widehat{\bU}_{RS}^{(i)}\bH_{RS}^{(i)} - \bV)\right\|_{2\to\infty}\Optilde(\log n)\\
&\quad\quad + \max_{i\in[n]}\bigg\{\sum_{t = 1}^m\sum_{j = 1}^n\|\be_j\transpose\bE_t^{(i)}(\widehat{\bU}_{RS}^{(i)}\bH_{RS}^{(i)} - \bV)\|_2\bigg\}^{1/2}\Optilde\left\{\rho_n^{1/2}(\log n)^{1/2}\right\}
 \\
&
\quad
 = \max_{i\in[n]}\max_{t\in[m]}\left\|\bE_t^{(i)}(\widehat{\bU}_{RS}^{(i)}\bH_{RS}^{(i)} - \bV)\right\|_{2\to\infty}\Optilde(\log n)
 + \max_{i\in[n]}
 \|(\bE^{(i)})\transpose(\widehat{\bU}_{RS}^{(i)}\bH_{RS}^{(i)} - \bV)\|_{\mathrm{F}}
 \Optilde\left\{\rho_n^{1/2}(\log n)^{1/2}\right\}
  \\
&
\quad
 = \max_{i\in[n]}\max_{t\in[m]}\left\|\bE_t^{(i)}(\widehat{\bU}_{RS}^{(i)}\bH_{RS}^{(i)} - \bV)\right\|_{2\to\infty}\Optilde(\log n)\\
&\quad\quad + \max_{i\in[n]}
 \|(\bE^{(i)})\transpose(\bE^{(i)})\transpose\|_2^{1/2}\|\widehat{\bU}_{RS}^{(i)}\bH_{RS}^{(i)} - \bV\|_{\mathrm{F}}
 \Optilde\left\{\rho_n^{1/2}(\log n)^{1/2}\right\}
\\
&
\quad
 = \max_{i\in[n]}\max_{t\in[m]}\left\|\bE_t^{(i)}(\widehat{\bU}_{RS}^{(i)}\bH_{RS}^{(i)} - \bV)\right\|_{2\to\infty}\Optilde(\log n)
 + 
\Optilde
\bigg(\frac{\zeta_{\mathsf{op}}\log n}{m^{1/2}n^{3/2}\rho_n}
 \bigg)
 + \Optilde\bigg(
\frac{\log n}{m^{1/2}n^{3/2}\rho_n}\bigg)
\|\bM - \widehat{\bM}_{RS}\|_2.
\end{align*}
The analysis of $\max_{i\in[n]}\max_{t\in[m]}\|\bE_t^{(i)}(\widehat{\bU}_{RS}^{(i)}\bH_{RS}^{(i)} - \bV)\|_{2\to\infty}$ requires the introduction of the leave-two-out matrices. Write
\begin{align*}
&\max_{i\in[n]}\max_{t\in[m]}\|\bE_t^{(i)}(\widehat{\bU}_{RS}^{(i)}\bH_{RS}^{(i)} - \bV)\|_{2\to\infty}
 = \max_{i\in[n]}\max_{t\in[m]}\max_{j\in[n]\backslash\{i\}}
\|\be_j\transpose\bE_t^{(i)}(\widehat{\bU}_{RS}^{(i)}\bH_{RS}^{(i)} - \bV)\|_2\\
&\quad\leq \max_{i\in[n]}\max_{t\in[m]}\max_{j\in[n]\backslash\{i\}}\|\be_j\transpose\bE_t^{(i)}(\widehat{\bU}_{RS}^{(i)}\bH_{RS}^{(i)} - \widehat{\bU}_{RS}^{(i, j)}\bH_{RS}^{(i, j)})\|_2
 + \max_{i\in[n]}\max_{t\in[m]}\max_{j\in[n]\backslash\{i\}}\|\be_j\transpose\bE_t^{(i)}(\widehat{\bU}_{RS}^{(i, j)}\bH_{RS}^{(i, j)} - \bV)\|_2\\
&\quad\leq \max_{t\in[m]}\max_{i,j\in[n],i\neq j}\|\bE_t^{(i)}\|_2\|\widehat{\bU}_{RS}^{(i)}\bH_{RS}^{(i)} - \widehat{\bU}_{RS}^{(i, j)}\bH_{RS}^{(i, j)}\|_2
 + \max_{i\in[n]}\max_{t\in[m]}\max_{j\in[n]\backslash\{i\}}\|\be_j\transpose\bE_t^{(i)}(\widehat{\bU}_{RS}^{(i, j)}\bH_{RS}^{(i, j)} - \bV)\|_2
\end{align*}
Since $\widehat{\bU}_{RS}^{(i, j)}$ and $\bH_{RS}^{(i, j)}$ are the leave-one-out versions of $\widehat{\bU}_{RS}^{(i)}$ and $\bH_{RS}^{(i)}$, respectively, and $\bE^{(i)}$ also satisfies Assumptions \ref{assumption:eigenvector_delocalization}--\ref{assumption:condition_number}, 
then by Lemma \ref{lemma:LOO_Utilde_two_to_infinity_norm}, we have
\begin{align*}
\max_{i,j\in[n],i\neq j}\|\widehat{\bU}_{RS}^{(i)}\bH_{RS}^{(i)} - \widehat{\bU}_{RS}^{(i, j)}\bH_{RS}^{(i, j)}\|_2
 = \Optilde\{\zeta_{\mathsf{op}}(mn^2\rho_n^2)^{-1}\}\|\widehat{\bU}_{RS}\|_{2\to\infty}
 ,
\end{align*}
so that the first term satisfies
\begin{align*}
\max_{t\in[m]}\max_{i,j\in[n],i\neq j}\|\bE_t^{(i)}\|_2\|\widehat{\bU}_{RS}^{(i)}\bH_{RS}^{(i)} - \widehat{\bU}_{RS}^{(i, j)}\bH_{RS}^{(i, j)}\|_2 = \Optilde\bigg\{\zeta_{\mathsf{op}}\frac{(n\rho_n\vee\log n)^{1/2}}{mn^2\rho_n^2}
\bigg\}\|\widehat{\bU}_{RS}\|_{2\to\infty}
\end{align*}
by Lemma \ref{lemma:norm_concentration_Et} and a union bound over $t\in[m]$. 
Note here we also used 
\[
\max_{i,j\in[n],i\neq j}\max_{r\in[R],s\in[S]}\|\bM - \widehat{\bM}_{rs}^{(i, j)}\|_2 \leq \lambda_d(\bP\bP\transpose)/20
\]
 with probability at least $1 - O(n^{-c})$.
For the second term, observe that by Lemma \ref{lemma:LOO_Utilde_two_to_infinity_norm}, Davis-Kahan theorem, Lemma \ref{lemma:spectral_norm_concentration_noise}, and a union bound over $i,j\in[n],i\neq j$, we have
\begin{align*}
\max_{i,j\in[n],i\neq j}\|\widehat{\bU}^{(i, j)}_{RS}\bH^{(i, j)}_{RS} - \bV\|_{2}
&\leq\max_{i,j\in[n],i\neq j}\|\widehat{\bU}^{(i, j)}_{RS}\bH^{(i, j)}_{RS} - \widehat{\bU}^{(i)}_{RS}\bH^{(i)}_{RS}\|_{2} + \max_{i\in[n]}\|\widehat{\bU}^{(i)}_{RS}\bH^{(i)}_{RS} - \bV\|_{2}\\
& = \Optilde\bigg(\frac{\zeta_{\mathsf{op}}}{mn^2\rho_n^2}\bigg) + O\left(\frac{1}{mn^2\rho_n^2}\right)\|\bM - \widehat{\bM}_{RS}\|_2,\\
\max_{i,j\in[n],i\neq j}\|\widehat{\bU}^{(i, j)}_{RS}\bH^{(i, j)}_{RS} - \bV\|_{2\to\infty}
& = \Optilde(\|\widehat{\bU}_{RS}\|_{2\to\infty}).
\end{align*}
Also, note that $(\be_j\transpose\bE_t^{(i)})_{t = 1}^m$ and $\widehat{\bU}_{RS}^{(i, j)}\bH_{RS}^{(i, j)} - \bV$ are independent so that the second term can be bounded using a union bound over $i,j\in[n],t\in[m]$, Bernstein's inequality, and Lemma \ref{lemma:LOO_Utilde_two_to_infinity_norm}:
\begin{align*}
&\max_{i\in[n]}\max_{t\in[m]}\max_{j\in[n]\backslash\{i\}}\|\be_j\transpose\bE_t^{(i)}(\widehat{\bU}_{RS}^{(i, j)}\bH_{RS}^{(i, j)} - \bV)\|_2\\
&\quad = \Optilde\bigg\{\log n\max_{i,j\in[n]:i\neq j}\|\widehat{\bU}_{RS}^{(i, j)}\bH_{RS}^{(i, j)} - \bV\|_{2\to\infty} + \rho_n^{1/2}(\log n)^{1/2}\max_{i,j\in[n],i\neq j}\|\widehat{\bU}^{(i, j)}_{RS}\bH^{(i, j)}_{RS} - \bV\|_{2}\bigg\}\\
&\quad\leq  \Optilde(\log n)\|\widehat{\bU}_{RS}\|_{2\to\infty}
 + \Optilde\bigg\{\frac{\zeta_{\mathsf{op}}(\log n)^{1/2}}{mn^2\rho_n^{3/2}}\bigg\}
 + \Optilde\bigg\{\frac{(\log n)^{1/2}}{mn^2\rho_n^{3/2}}\bigg\}\|\bM - \widehat{\bM}_{RS}\|_2.
\end{align*}
Combining the two pieces above together, we obtain
\begin{align*}
&\max_{i\in[n]}\max_{t\in[m]}\|\bE_t^{(i)}(\widehat{\bU}_{RS}^{(i)}\bH_{RS}^{(i)} - \bV)\|_{2\to\infty}\\
&\quad\leq \max_{t\in[m]}\max_{i,j\in[n],i\neq j}\|\bE_t^{(i)}\|_2\|\widehat{\bU}_{RS}^{(i)}\bH_{RS}^{(i)} - \widehat{\bU}_{RS}^{(i, j)}\bH_{RS}^{(i, j)}\|_2
 + \max_{i\in[n]}\max_{t\in[m]}\max_{j\in[n]\backslash\{i\}}\|\be_j\transpose\bE_t^{(i)}(\widehat{\bU}_{RS}^{(i, j)}\bH_{RS}^{(i, j)} - \bV)\|_2\\
&\quad = \Optilde\bigg\{\zeta_{\mathsf{op}}\frac{(n\rho_n \vee\log n)^{1/2}}{mn^2\rho_n^2} + \log n
\bigg\}\|\widehat{\bU}_{RS}\|_{2\to\infty} 
 + \Optilde\bigg\{\frac{\zeta_{\mathsf{op}}(\log n)^{1/2}}{mn^2\rho_n^{3/2}}\bigg\}
 + \Optilde\bigg\{\frac{(\log n)^{1/2}}{mn^2\rho_n^{3/2}}\bigg\}\|\bM - \widehat{\bM}_{RS}\|_2.
\end{align*}
Hence, we further obtain
\begin{align*}
&\max_{i\in[n]}\bigg\|\sum_{t = 1}^m\be_i\transpose\bE_t\bE_t^{(i)}(\widehat{\bU}_{RS}^{(i)}\bH_{RS}^{(i)} - \bV)\bigg\|_2\\
&\quad = 
\Optilde\bigg\{
(\log n)^2
+ \zeta_{\mathsf{op}}\frac{(n\rho_n\vee\log n)^{1/2}}{mn^2\rho_n^2}\log n
\bigg\}\|\widehat{\bU}_{RS}\|_{2\to\infty}
+ \Optilde\bigg(
\frac{\zeta_{\mathsf{op}}\log n}{m^{1/2}n^{3/2}\rho_n}\bigg)\\
&\quad\quad
 + \Optilde\bigg(
\frac{\log n}{m^{1/2}n^{3/2}\rho_n}\bigg)
\|\bM - \widehat{\bM}_{RS}\|_2. 
\end{align*}


\noindent$\blacksquare$ For the second term in \eqref{eqn:LOO_rowwise_concentration}, we can rewrite it as $\sum_{t = 1}^m\sum_{j\in[n]\backslash\{i\}}E_{tii}E_{tij}\be_j\transpose(\widehat{\bU}_{RS}^{(i)}\bH_{RS}^{(i)} - \bV)$.
Denote by $\widehat{\bV}_{RS}^{(i)} = \widehat{\bU}_{RS}^{(i)}\bH_{RS}^{(i)} - \bV$ and $[\widehat{\bV}_{RS}^{(i)}]_{jk}$ the $(j, k)$th entry of $\widehat{\bV}_{RS}^{(i)}$. By Bernstein's inequality and a conditioning argument,
\begin{align*}
&\bigg\|\sum_{t = 1}^mE_{tii}(\be_i\transpose\bE_t - E_{tii}\be_i\transpose)\widehat{\bV}_{RS}^{(i)}\bigg\|_2^2\\
&\quad = \Optilde\bigg\{\|\widehat{\bV}_{RS}^{(i)}\|_{2\to\infty}\log n + \bigg(\sum_{t = 1}^m\sum_{j = 1}^nE_{tii}^2\|\be_j\transpose\widehat{\bV}_{RS}^{(i)}\|_2^2\bigg)^{1/2}\rho_n^{1/2}(\log n)^{1/2}\bigg\}\\
&\quad = \Optilde\bigg[\|\widehat{\bV}_{RS}^{(i)}\|_{2\to\infty}\log n + \bigg\{\sum_{t = 1}^m(E_{tii}^2 - \sigma_{tii}^2) + \sum_{t = 1}^m\sigma_{tii}^2\bigg\}^{1/2}\|\widehat{\bV}_{RS}^{(i)}\|_{\mathrm{F}}\rho_n^{1/2}(\log n)^{1/2}\bigg]\\
&\quad = \Optilde\bigg\{\|\widehat{\bV}_{RS}^{(i)}\|_{2\to\infty}\log n + \|\widehat{\bV}_{RS}^{(i)}\|_{\mathrm{F}}m^{1/2}\rho_n(\log n)^{1/2} + \|\widehat{\bV}_{RS}^{(i)}\|_{\mathrm{F}}\rho_n^{1/2}\log n\bigg\}\\
&\quad = \Optilde\bigg(\|\widehat{\bU}_{RS}\|_{2\to\infty}\log n + \frac{\zeta_{\mathsf{op}}\log n}{m^{1/2}n^{3/2}\rho_n}\bigg) + \Optilde\bigg(\frac{\log n}{m^{1/2}n^{3/2}\rho_n}\bigg)\|\bM - \widehat{\bM}_{RS}\|_2.
\end{align*}
Combining the concentration bounds for the two terms in \eqref{eqn:LOO_rowwise_concentration} completes the proof of the first assertion. For the second assertion, by triangle inequality, we have
\begin{align*}
\max_{i\in[n]}\frac{\|\be_i\transpose\{\calH(\bA\bA\transpose) - \bP\bP\transpose - \bM\}(\widehat{\bU}_{RS}^{(i)}\bH_{RS}^{(i)} - \bV)\|_2}{m\Delta_n^2}
&\leq q_1 + q_2 + q_3 + q_4,
\end{align*}
where
\begin{equation}
\label{eqn:q1_q2_q3_q4}
\begin{aligned}
q_1 & = \max_{i\in[n]}\frac{\|\be_i\transpose\calH(\bE\bE\transpose)(\widehat{\bU}_{RS}^{(i)}\bH_{RS}^{(i)} - \bV)\|_2}{m\Delta_n^2},\;
q_2 = \frac{1}{m\Delta_n^2}\bigg\|\sum_{t = 1}^m\bP_t\bE_t\bigg\|_{2\to\infty}\max_{i\in[n]}\|\widehat{\bU}_{RS}^{(i)}\bH_{RS}^{(i)} - \bV\|_2,\\
q_3 & = \frac{1}{m\Delta_n^2}\max_{i\in[n]}\bigg\|\be_i\transpose\sum_{t = 1}^m\bE_t\bP_t(\widehat{\bU}_{RS}^{(i)}\bH_{RS}^{(i)} - \bV)\bigg\|_2,\;
q_4 = \frac{1}{m\Delta_n^2}\max_{i\in[n]}\bigg\|2\sum_{t = 1}^m\sum_{j = 1}^nP_{tij}E_{tij}\be_i\transpose(\widehat{\bU}_{RS}^{(i)}\bH_{RS}^{(i)} - \bV)\bigg\|_2.
\end{aligned}
\end{equation}
By the first assertion, we have
\begin{align*}
q_1 & = \Optilde\bigg\{\frac{(\log n)^2}{mn^2\rho_n^2} + \frac{\log n}{mn^2\rho_n^2}(n\rho_n\vee\log n)^{1/2}\bigg\}\|\widehat{\bU}_{RS}\|_{2\to\infty} + \Optilde\bigg(\frac{\log n}{m^{1/2}n\rho_n}\times\frac{1}{\sqrt{n}}\bigg)\\
&\quad + \Optilde\bigg(\frac{\log n}{m^{3/2}n^{7/2}\rho_n^3}\bigg)
\|\bM - \widehat{\bM}_{RS}\|_2\\
& = \Optilde\bigg(\frac{\log n}{m^{1/2}n\rho_n}\|\widehat{\bU}_{RS}\|_{2\to\infty}\bigg)
 + \Optilde\bigg(\frac{\log n}{m^{3/2}n^{7/2}\rho_n^3}\bigg)
\|\bM - \widehat{\bM}_{RS}\|_2.
\end{align*}
For $q_2$, by Bernstein's inequality, we see that $\|\sum_{t = 1}^m\bE_t\bP_t\|_2 = 
\Optilde\{m^{1/2}(n\rho_n)^{3/2}(\log n)^{1/2}\}$. 
Then it follows from Lemma \ref{lemma:LOO_Utilde_two_to_infinity_norm} that 
\begin{align*}
q_2 &\leq \frac{\sqrt{n}}{m\Delta_n^2}\left\|\sum_{t = 1}^m\bP_t\bE_t\right\|_{2\to\infty}\max_{i\in[n]}\|\widehat{\bU}_{RS}^{(i)}\bH_{RS}^{(i)} - \bV\|_{2\to\infty}
\leq \frac{d^{1/2}\mu^{1/2}}{m\Delta_n^2}\left\|\sum_{t = 1}^m\bB_t\bU\transpose\bE_t\right\|_2\max_{i\in[n]}\|\widehat{\bU}_{RS}^{(i)}\bH_{RS}^{(i)} - \bV\|_{2\to\infty} 
\\&
= \frac{d^{1/2}\mu^{1/2}}{m\Delta_n^2}\left\|\sum_{t = 1}^m\bE_t\bP_t\right\|_2\max_{i\in[n]}\|\widehat{\bU}_{RS}^{(i)}\bH_{RS}^{(i)} - \bV\|_{2\to\infty} = \Optilde\left\{\frac{(\log n)^{1/2}}{(mn\rho_n)^{1/2}}\right\}\|\widehat{\bU}_{RS}\|_{2\to\infty}.
\end{align*} 
For $q_3$, by Bernstein's inequality and Lemma \ref{lemma:LOO_Utilde_two_to_infinity_norm}, we obtain
\begin{align*}
q_3
& = \Optilde\left\{\frac{(n\rho_n\log n)^{1/2}}{m^{1/2}\Delta_n^2}\right\}
\max_{t\in[m]}\|\bP_t\|_\infty\|\widehat{\bU}_{RS}^{(i)}\bH_{RS}^{(i)} - \bV\|_{2\to\infty} = \Optilde\left\{\frac{(\log n)^{1/2}}{(mn\rho_n)^{1/2}}\right\}\|\widehat{\bU}_{RS}\|_{2\to\infty}
\end{align*}
For $q_4$, by Bernstein's inequality, Lemma \ref{lemma:LOO_Utilde_two_to_infinity_norm}, and Assumption \ref{assumption:eigenvector_delocalization}, we also have
\begin{align*}
q_4& = \Optilde\left\{\frac{(\log n)^{1/2}}{m^{1/2}(n\rho_n)^{3/2}}\right\}\|\bP\|_{\max}\max_{i\in[n]}\|\widehat{\bU}_{RS}^{(i)}\bH_{RS}^{(i)} - \bV\|_{2\to\infty} = \Optilde\left\{\frac{\rho_n(\log n)^{1/2}}{m^{1/2}(n\rho_n)^{3/2}}\right\}\|\widehat{\bU}_{RS}\|_{2\to\infty}.
\end{align*}
The proof is thus completed by combining the above error bounds for $q_1$, $q_2$, $q_3$, and $q_4$.
\end{proof}

\begin{lemma}\label{lemma:Utilde_two_to_infinity_norm}
Suppose the conditions of Lemma \ref{lemma:LOO_Rowwise_concentration} hold.
Then for any $c > 0$, there exists a $c$-dependent constant $N_c > 0$, such that 
$\|\bH_{RS}^{-1}\|_2\leq 2$ and $\|\widehat{\bU}_{RS}\|_{2\to\infty}\leq C\kappa^2n^{-1/2}$ with probability at least $1 - O(n^{-c})$ for all $n\geq N_c$.
\end{lemma}

\begin{proof}
For the first assertion, note that $\|\calH(\bA\bA\transpose) - \bP\bP\transpose - \widehat{\bM}_{RS}\|_2/\lambda_d(\bP\bP\transpose)\leq \Optilde\{{\zeta_{\mathsf{op}}}/(m\Delta_n^2)\} + 1/10$
with probability at least $1 - O(n^{-c})$. Then Lemma 2 in \cite{abbe-fan-wang-zhong-2020} yields that for any $c > 0$, there exists a $c$-dependent constant $N_c > 0$, such that
\begin{align}\label{eqn:Hr_error_bound}
\prob\left(\|\bH_{RS}^{-1}\|_2\leq 2\right)\geq 1 -O(n^{-c})\quad\mbox{for all }n\geq N_c.
\end{align}  
Now we focus on the second assertion.
By Lemma \ref{lemma:spectral_norm_concentration_noise}, there exists a $c$-dependent constant $N_c > 0$, such that
\begin{align*}
\frac{\|\calH(\bA\bA\transpose) - \bP\bP\transpose - \widehat{\bM}_{RS}\|_2}{\lambda_d(\bP\bP\transpose)}&\leq \frac{C_c\zeta_{\mathsf{op}}}{\lambda_d(\bP\bP\transpose)} + \frac{\|\bM - \widehat{\bM}_{RS}\|_2}{\lambda_d(\bP\bP\transpose)}\leq \frac{1}{10}
\end{align*}
with probability $1 - O(n^{-c})$ for any $n\geq N_c$. Then by Lemma 1 in \cite{abbe-fan-wang-zhong-2020}, Lemma \ref{lemma:spectral_norm_concentration_noise}, Davis-Kahan theorem, and Lemma \ref{lemma:rowwise_concentration_noise}, there exists a $c$-dependent constant $C_c > 0$, such that for any $i\in[n]$, 
\begin{align*}
\|\be_i\transpose\widehat{\bU}_{RS}\bH_{RS}\|_2
&\leq \frac{2}{\lambda_d(\bP\bP\transpose)}[\|\be_i\transpose\{\calH(\bA\bA\transpose) - \widehat{\bM}_{RS}\}\bV\|_2 + \|\{\calH(\bA\bA\transpose) - \widehat{\bM}_{RS}\}(\widehat{\bU}_{RS}\bH_{RS} - \bV)\|_2]\\
&\leq \bigg(\frac{2C_c\zeta_{\mathsf{op}}}{m\Delta_n^2} + 2\kappa^2 + \frac{1}{10}\bigg)\|\bU\|_{2\to\infty}
 + \frac{2}{\lambda_d(\bP\bP\transpose)}\|\be_i\transpose\{\calH(\bA\bA\transpose) - \bP\bP\transpose - \bM\}(\widehat{\bU}_{RS}\bH_{RS} - \bV)\|_2\\
&\quad + \bigg(2\kappa^2\|\bU\|_{2\to\infty}\|\widehat{\bU}_{RS}\bH_{RS} - \bV\|_2 + \frac{1}{5}\|\widehat{\bU}_{RS}\|_{2\to\infty} + \frac{1}{5}\|\bU\|_{2\to\infty}\bigg)\\
&\leq \frac{2}{m\Delta_n^2}\|\be_i\transpose\{\calH(\bA\bA\transpose) - \bP\bP\transpose - \bM\}(\widehat{\bU}_{RS}\bH_{RS} - \bV)\|_2 + C\kappa^2\|\bU\|_{2\to\infty} + \frac{1}{5}\|\widehat{\bU}_{RS}\|_{2\to\infty}
\end{align*}
with probability at least $1 - O(n^{-c})$ for all $n\geq N_c$, where $C > 0$ is a constant not depending on $c$. 
It follows that
\begin{align*}
\|\widehat{\bU}_{RS}\|_{2\to\infty}
& = \max_{i\in[n]}\|\be_i\transpose\widehat{\bU}_{RS}\bH_{RS}\bH_{RS}^{-1}\|_2
\leq \max_{i\in[n]}\|\be_i\transpose\widehat{\bU}_{RS}\bH_{RS}\|_2\|\bH_{RS}^{-1}\|_2
\leq 2\max_{i\in[n]}\|\be_i\transpose\widehat{\bU}_{RS}\bH_{RS}\|_2\\
&\leq \frac{4}{m\Delta_n^2}\max_{i\in[n]}\|\be_i\transpose\{\calH(\bA\bA\transpose) - \bP\bP\transpose - \bM\}(\widehat{\bU}_{RS}\bH_{RS} - \bV)\|_2
 + C\kappa^2\|\bU\|_{2\to\infty} + \frac{2}{5}\|\widehat{\bU}_{RS}\|_{2\to\infty}.
\end{align*}
Hence,
\begin{align*}
\|\widehat{\bU}_{RS}\|_{2\to\infty}\leq C\kappa^2\|\bU\|_{2\to\infty} + \frac{C}{m\Delta_n^2}\max_{i\in[n]}\|\be_i\transpose\{\calH(\bA\bA\transpose) - \bP\bP\transpose - \bM\}(\widehat{\bU}_{RS}\bH_{RS} - \bV)\|_2
\end{align*}
with probability at least $1 - O(n^{-c})$ for all $n\geq N_c$, where $C > 0$ is a constant not depending on $c$. It is sufficient to focus on the last term on the right-hand side above. We invoke the leave-one-out matrices, Lemma \ref{lemma:spectral_norm_concentration_noise}, and Lemma \ref{lemma:LOO_Utilde_two_to_infinity_norm} to write
\begin{align*}
&\frac{C}{m\Delta_n^2}\max_{i\in[n]}\|\be_i\transpose\{\calH(\bA\bA\transpose) - \bP\bP\transpose - \bM\}(\widehat{\bU}_{RS}\bH_{RS} - \bV)\|_2\\
&\quad\leq \frac{C}{m\Delta_n^2}\max_{i\in[n]}\|\be_i\transpose\{\calH(\bA\bA\transpose) - \bP\bP\transpose - \bM\}(\widehat{\bU}_{RS}\bH_{RS} - \widehat{\bU}_{RS}^{(i)}\bH^{(i)}_{RS})\|_2 
\\&\quad\quad
 + \frac{C}{m\Delta_n^2}\max_{i\in[n]}\|\be_i\transpose\{\calH(\bA\bA\transpose) - \bP\bP\transpose - \bM\}(\widehat{\bU}_{RS}^{(i)}\bH_{RS}^{(i)} - \bV)\|_2\\
&\quad\leq \frac{C\|\calH(\bA\bA\transpose) - \bP\bP\transpose - \bM\|_2}{m\Delta_n^2}\max_{i\in[n]}\|\widehat{\bU}_{RS}\bH_{RS} - \widehat{\bU}_{RS}^{(i)}\bH^{(i)}_{RS}\|_2\\
&\quad\quad + \max_{i\in[n]}\frac{C\|\be_i\transpose\{\calH(\bA\bA\transpose) - \bP\bP\transpose - \bM\}(\widehat{\bU}_{RS}^{(i)}\bH_{RS}^{(i)} - \bV)\|_2}{m\Delta_n^2}\\
&\quad\leq \frac{1}{4}\|\widehat{\bU}_{RS}\|_{2\to\infty} + \max_{i\in[n]}\frac{C\|\be_i\transpose\{\calH(\bA\bA\transpose) - \bP\bP\transpose - \bM\}(\widehat{\bU}_{RS}^{(i)}\bH_{RS}^{(i)} - \bV)\|_2}{mn^2\rho_n^2}.
\end{align*}
with probability at least $1 - O(n^{-c})$ for all $n\geq N_c$. For the last term, by Lemma \ref{lemma:LOO_Rowwise_concentration}, there exist $c$-dependent constants $C_c, N_c > 0$, such that
\begin{align*}
&\max_{i\in[n]}\frac{C\|\be_i\transpose\{\calH(\bA\bA\transpose) - \bP\bP\transpose - \bM\}(\widehat{\bU}_{RS}^{(i)}\bH_{RS}^{(i)} - \bV)\|_2}{mn^2\rho_n^2}
\\&\quad
 \leq \bigg\{\frac{C_c(\log n)^{1/2}}{(mn\rho_n)^{1/2}} + \frac{C_c\log n}{m^{1/2}n\rho_n}\bigg\}\|\widehat{\bU}_{RS}\|_{2\to\infty} + C_c\bigg(\frac{\log n}{m^{3/2}n^{7/2}\rho_n^3}\bigg)\|\bM - \widehat{\bM}_{RS}\|_2\leq \frac{1}{4}\|\widehat{\bU}_{RS}\|_{2\to\infty}
\end{align*}
with probability at least $1 - O(n^{-c})$ for all $n\geq N_c$, and hence, we conclude that
$\|\widehat{\bU}_{RS}\|_{2\to\infty}\leq C\kappa^2n^{-1/2}$
with probability at least $1 - O(n^{-c})$ for all $n\geq N_c$. 
The proof is thus completed. 
\end{proof}

The following lemma seems to be quite similar to Lemma \ref{lemma:Utilde_two_to_infinity_norm}. Nevertheless, it is worth remarking that the conditions are weaker. The key difference is that, unlike Lemma \ref{lemma:Utilde_two_to_infinity_norm}, Lemma \ref{lemma:Utilde_two_to_infinity_norm_II} below no longer requires the condition on $\|\bM - \widehat{\bM}_{rs}\|_2$, $\|\widehat{\bM} - \widehat{\bM}_{rs}^{(i)}\|_2$, and $\|\bM  - \widehat{\bM}_{rs}^{(i, j)}\|_2$ for all $r\in[R]$, $s\in[S]$, $i,j\in[n]$, $i\neq j$, but directly justifies these conditions instead.

\begin{lemma}\label{lemma:Utilde_two_to_infinity_norm_II}
Suppose Assumptions \ref{assumption:eigenvector_delocalization}--\ref{assumption:condition_number} hold.
If $R, S = O(1)$, 
then for any $c > 0$, there exists a $c$-dependent constant $N_c > 0$, such that 
\[
\max\bigg\{\max_{r\in[R],s\in[S]}\|\bM - \widehat{\bM}_{rs}\|_2,\max_{i\in[n],r\in[R],s\in[S]}\|\bM - \widehat{\bM}_{rs}^{(i)}\|_2,\max_{i,j\in[n],i\neq j}\max_{r\in[R],s\in[S]}\|\bM - \widehat{\bM}_{rs}^{(i,j)}\|_2\bigg\} \leq \frac{\lambda_d(\bP\bP\transpose)}{20}
\]
with probability at least $1 - O(n^{-c})$ for all $n\geq N_c$. Consequently,
\begin{align*}
\|\bH_{RS}^{-1}\|_2\leq 2,\quad
\max_{i\in[n]}\|(\bH_{RS}^{(i)})^{-1}\|_2\leq 2,\quad
\|\widehat{\bU}_{RS}\|_{2\to\infty}\leq C\kappa^2n^{-1/2},\quad
\max_{i\in[n]}\|\widehat{\bU}_{RS}^{(i)}\|_{2\to\infty}\leq C\kappa^2n^{-1/2}
\end{align*}
with probability at least $1 - O(n^{-c})$ for all $n\geq N_c$, where $C > 0$ is a constant not depending on $c$.
Furthermore, 
\begin{align*}
\max_{i\in[n]}\|\widehat{\bU}_{RS}^{(i)}\bH_{RS}^{(i)} - \widehat{\bU}_{RS}\bH_{RS}\|_2
  = \Optilde\left(\frac{\zeta_{\mathsf{op}}}{mn^{5/2}\rho_n^2}\right),\quad
\max_{i\in[n]}\|\widehat{\bU}_{RS}^{(i)}\bH_{RS}^{(i)} - \bV\|_{2\to\infty} = \Optilde(n^{-1/2})_,
 \end{align*}
and
\begin{align*}
&\max_{i\in[n]}\|\be_i\transpose\calH(\bE\bE\transpose)(\widehat{\bU}_{RS}^{(i)}\bH_{RS}^{(i)} - \bV)\|_2
 =  \Optilde\bigg(\frac{q_n}{\sqrt{n}} + \frac{\zeta_{\mathsf{op}}\log n}{m^{1/2}n^{3/2}\rho_n}\bigg)
 + \Optilde\bigg(\frac{\log n}{m^{1/2}n^{3/2}\rho_n}\bigg)
\|\bM - \widehat{\bM}_{RS}\|_2
,\\
 &\max_{i\in[n]}\frac{\|\be_i\transpose\{\calH(\bA\bA\transpose) - \bP\bP\transpose - \bM\}(\widehat{\bU}_{RS}^{(i)}\bH_{RS}^{(i)} - \bV)\|_2}{m\Delta_n^2}
\\&\quad
 = \Optilde\bigg\{\frac{\log n}{m^{1/2}n^{3/2}\rho_n} + \frac{(\log n)^{1/2}}{m^{1/2}n\rho_n^{1/2}}\bigg\}
 + \Optilde\bigg(\frac{\log n}{m^{3/2}n^{7/2}\rho_n^3}\bigg)\|\bM - \widehat{\bM}_{RS}\|_2,
\end{align*}
where $q_n = (\log n)^2 + \zeta_{\mathsf{op}}\log n(n\rho_n \vee\log n)^{1/2}\}/(mn^2\rho_n^2)$.
\end{lemma}

\begin{proof}
By Lemma \ref{lemma:LOO_stage_II}, Lemma \ref{lemma:LOO_Utilde_two_to_infinity_norm}, Lemma \ref{lemma:LOO_Rowwise_concentration}, and Lemma \ref{lemma:Utilde_two_to_infinity_norm},
it is sufficient to establish the part before ``consequently''. 
We prove these error bounds by induction. When $R = 1$, we immediately have $\widehat{\bM}_{1s} = \widehat{\bM}_{1s}^{(i)} = \zero_{n\times n}$, so that 
\begin{align*}
&\max_{s\in[S]}\|\bM - \widehat{\bM}_{1s}\|_2 = \max_{i\in[n],s\in[S]}\|\bM - \widehat{\bM}_{1s}^{(i)}\|_2 = \max_{i,j\in[n],i\neq j}\max_{s\in[S]}\|\bM - \widehat{\bM}_{1s}^{(i, j)}\|_2 = \|\bM\|_2\leq mn\rho_n^2 = o(\lambda_d(\bP\bP\transpose))
\end{align*}
with probability one. Now assume for any $c > 0$, there exists a $c$-dependent constant $N_c > 0$, such that
\begin{align*}
\max\left\{\max_{r\in[R],s\in[S]}\|\bM - \widehat{\bM}_{rs}\|_2,\max_{i\in[n]}\max_{r\in[R],s\in[S]}\|\bM - \widehat{\bM}_{rs}^{(i)}\|_2,
\max_{i,j\in[n],i\neq j}\max_{r\in[R],s\in[S]}\|\bM - \widehat{\bM}_{rs}^{(i, j)}\|_2\right\} \leq \frac{\lambda_d(\bP\bP\transpose)}{20}
\end{align*}
with probability at least $1 - O(n^{-c})$ whenever $n\geq N_c$. By Lemma \ref{lemma:Utilde_two_to_infinity_norm},
$\|\widehat{\bU}_{RS}\|_{2\to\infty} = \Optilde(n^{-1/2})$. Furthermore, by Lemma \ref{lemma:LOO_Utilde_two_to_infinity_norm}, we also know that $\max_{i\in[n]}\|\widehat{\bU}_{RS}^{(i)}\|_{2\to\infty} = \Optilde(n^{-1/2})$ and $\max_{i,j\in[n],i\neq j}\|\widehat{\bU}_{RS}^{(i,j)}\|_{2\to\infty} = \Optilde(n^{-1/2})$. 
In addition, by Lemma \ref{lemma:spectral_norm_concentration_noise} and union bounds over $i,j\in[n]$, 
\begin{align*}
&\|\calH(\bA\bA\transpose)\|_2
\leq \|\calH(\bA\bA\transpose) - \bP\bP\transpose - \bM\|_2 + \|\bP\bP\transpose\|_2 + \|\bM\|_2 = \Optilde(\zeta_{\mathsf{op}}) + 2\kappa^2 m\Delta_n^2 = \Optilde(\kappa^2 m\Delta_n^2),\\
&\max_{i\in[n]}\|\calH((\bA^{(i)})(\bA^{(i)})\transpose)\|_2 = \Optilde(\kappa^2 m\Delta_n^2),\quad
\max_{i,j\in[n],i\neq j}\|\calH((\bA^{(i, j)})(\bA^{(i, j)})\transpose)\|_2 = \Optilde(\kappa^2 m\Delta_n^2).
\end{align*}
Then, we work with $\widehat{\bM}_{(R + 1)s}$, $\widehat{\bM}_{(R + 1)s}^{(i)}$, and $\widehat{\bM}_{(R + 1)s}^{(i, j)}$ for $s \in [S]$. Recall that 
$\widehat{\bM}_{(R + 1)0} = \widehat{\bM}_{(R + 1)0}^{(i)} = \widehat{\bM}_{(R + 1)0}^{(i, j)} = \zero_{n\times n}$
by definition, and we have the recursive relation
\begin{align*}
\|\widehat{\bM}_{(R + 1)s}\|_2
&\leq \{\|\calH(\bA\bA\transpose)\|_2 + \|\widehat{\bM}_{(R + 1)(s - 1)}\|_2\}\|\widehat{\bU}_{RS}\|_{2\to\infty}^2
 = \Optilde(mn\rho_n^2) + \Optilde\bigg(\frac{\|\widehat{\bM}_{(R + 1)(s - 1)}\|_2}{n}\bigg),\\
\max_{i\in[n]}\|\widehat{\bM}_{(R + 1)s}^{(i)}\|_2
&\leq \max_{i\in[n]}\left\{\|\calH((\bA^{(i)})(\bA^{(i)})\transpose)\|_2 + \|\widehat{\bM}_{(R + 1)(s - 1)}^{(i)}\|_2\right\}\|\widehat{\bU}_{RS}^{(i)}\|_{2\to\infty}^2\\
& = \Optilde(mn\rho_n^2) + \Optilde\bigg(\frac{\|\widehat{\bM}_{(R + 1)(s - 1)}\|_2}{n}\bigg),\\
\max_{i,j\in[n],i\neq j}\|\widehat{\bM}_{(R + 1)s}^{(i, j)}\|_2
&\leq \max_{i,j\in[n],i\neq j}\left\{\|\calH(\bA^{(i, j)}(\bA^{(i, j)})\transpose)\|_2 + \|\widehat{\bM}^{(i, j)}_{(R + 1)(s - 1)}\|_2\right\}\|\widehat{\bU}_{RS}^{(i, j)}\|_{2\to\infty}^2\\
& = \Optilde(mn\rho_n^2) + \Optilde\bigg(\frac{\|\widehat{\bM}_{(R + 1)(s - 1)}\|_2}{n}\bigg).
\end{align*}
Then by induction, we immediately obtain, for all $s \in [S]$,
\begin{equation}
\label{eqn:Mtilde_spectral_norm}
\begin{aligned}
\|\widehat{\bM}_{(R + 1)s}\|_2 = \Optilde\left(mn\rho_n^2\right),\;
\max_{i\in[n]}\|\widehat{\bM}_{(R + 1)s}^{(i)}\|_2 = \Optilde\left(mn\rho_n^2\right),\;
&\max_{i,j\in[n],i\neq j}\|\widehat{\bM}_{(R + 1)s}^{(i, j)}\|_2 = \Optilde\left(mn\rho_n^2\right).
\end{aligned}
\end{equation}
Therefore, for the case of $R + 1$, there exist $c$-dependent constants $C_c, N_c > 0$, such that
\begin{align*}
\max_{s\in[S]}\|\bM - \widehat{\bM}_{(R + 1)s}\|_2
&\leq \|\bM\|_2 + \max_{s\in[S]}\|\widehat{\bM}_{(R + 1)s}\|_2\leq C_cmn\rho_n^2
\leq \frac{\lambda_d(\bP\bP\transpose)}{20},\\
\max_{i\in[n]}\max_{s\in[S]}\|\bM - \widehat{\bM}_{(R + 1)s}^{(i)}\|_2
&\leq \|\bM\|_2 + \max_{i\in[n]}\max_{s\in[S]}\|\widehat{\bM}_{(R + 1)s}^{(i)}\|_2\leq C_cmn\rho_n^2
\leq \frac{\lambda_d(\bP\bP\transpose)}{20},\\
 \max_{i,j\in[n],i\neq j}\max_{s\in[S]}\|\bM - \widehat{\bM}_{(R + 1)s}^{(i, j)}\|_2
&\leq \|\bM\|_2 + \max_{i,j\in[n],i\neq j}\max_{s\in[S]}\|\widehat{\bM}_{(R + 1)s}^{(i, j)}\|_2\leq  C_cmn\rho_n^2
\leq \frac{\lambda_d(\bP\bP\transpose)}{20}
\end{align*}
with probability at least $1 - O(n^{-c})$ whenever $n\geq N_c$. The proof is thus completed.
\end{proof}

\subsection{Joint proof of Lemma \ref{lemma:recursive_two_to_infinity_error_bound} and Lemma \ref{lemma:R1_remainder}}
\label{sub:proof_of_perturbation_theorem}

We first establish the second assertion of Lemma \ref{lemma:recursive_two_to_infinity_error_bound}. For any $s\in [S]$, by definition, triangle inequality, and Lemma \ref{lemma:Utilde_two_to_infinity_norm_II}, we have
\begin{align*}
\|\bM - \widehat{\bM}_{Rs}\|_2 
&\leq \|\widehat{\bU}_{(R - 1)S}\bW_{(R - 1)S} - \bV\|_{2\to\infty}\left\{\|\calH(\bA\bA\transpose)\|_2 + \|\widehat{\bM}_{R(s - 1)}\|_2\right\}\|\widehat{\bU}_{(R - 1)S}\|_{2\to\infty}\\
&\quad + \|\bU\|_{2\to\infty}\|\widehat{\bU}_{(R - 1)S}\bW_{(R - 1)S} - \bV\|_2\left\{\|\calH(\bA\bA\transpose)\|_2 + \|\widehat{\bM}_{R(s - 1)}\|_2\right\}\|\widehat{\bU}_{(R - 1)S}\|_{2\to\infty}\\
&\quad + \|\bU\|_{2\to\infty}\|\calH(\bA\bA\transpose) - \bP\bP\transpose - \widehat{\bM}_{R(s - 1)}\|_2\|\widehat{\bU}_{(R - 1)S}\bW_{(R - 1)S} - \bV\|_2\|\widehat{\bU}_{(R - 1)S}\|_{2\to\infty}\\
&\quad + \|\bU\|_{2\to\infty}\|\bV\transpose\{\calH(\bA\bA\transpose) - \bP\bP\transpose - \bM\}\bV\|_2\|\widehat{\bU}_{(R - 1)S}\|_{2\to\infty}\\
&\quad + \|\bU\|_{2\to\infty}\|\bV\transpose(\widehat{\bM}_{R(s - 1)} - \bM)\bV\|_2\|\widehat{\bU}_{(R - 1)S}\|_{2\to\infty}\\
&\quad + \|\bU\|_{2\to\infty}\|\bP\bP\transpose\|_2\|\widehat{\bU}_{(R - 1)S}\bW_{(R - 1)S} - \bV\|_2\|\widetilde{\bU}_{(R - 1)S}\|_{2\to\infty}\\
&\quad + \|\bU\|_{2\to\infty}\|\bP\bP\transpose\|_2\|\widehat{\bU}_{(R - 1)S}\bW_{(R - 1)S} - \bV\|_{2\to\infty}.
\end{align*}
By Lemma \ref{lemma:spectral_norm_concentration_noise} and \eqref{eqn:Mtilde_spectral_norm}, we know that $\|\calH(\bA\bA\transpose)\|_2 = \Optilde(mn^2\rho_n^2)$ and $\|\widehat{\bM}_{Rs}\|_2 = \Optilde(mn^2\rho_n^2)$ for any $s\in [S]$. Furthermore, Lemma \ref{lemma:spectral_norm_concentration_noise} entails that 
$\|\calH(\bA\bA\transpose) - \bP\bP\transpose - \widehat{\bM}_{Rs}\|_2
\leq \|\calH(\bA\bA\transpose) - \bP\bP\transpose - \bM\|_2 + \|\bM\|_2 + \|\widehat{\bM}_{Rs}\|_2
 = \Optilde\left(\zeta_{\mathsf{op}} + mn\rho_n^2\right)$
for any $s\in[S]$. 
In addition, Lemma \ref{lemma:concentration_noise_quadratic_form} implies that
\begin{align*}
\|\bV\transpose\{\calH(\bA\bA\transpose) - \bP\bP\transpose - \widehat{\bM}_{Rs}\}\bV\|_2
&\leq \|\bV\transpose\{\calH(\bA\bA\transpose) - \bP\bP\transpose - \bM\}\bV\|_2 + \|\bM - \widehat{\bM}_{Rs}\|_2
\\&
 = \Optilde\bigg(\frac{\zeta_{\mathsf{op}}}{\sqrt{n}}\bigg) + \|\widehat{\bM}_{Rs} - \bM\|_2
\end{align*}
for any $s\in[S]$. 
It follows from Lemma \ref{lemma:Utilde_two_to_infinity_norm_II} that
\begin{align*}
\|\bM - \widehat{\bM}_{Rs}\|_2
&\leq \Optilde(mn^{3/2}\rho_n^2)\|\widehat{\bU}_{(R - 1)S}\bW_{(R - 1)S} - \bV\|_{2\to\infty} + \Optilde(mn\rho_n^2)\|\widehat{\bU}_{(R - 1)S}\bW_{(R - 1)S} - \bV\|_2
\\&\quad
 + \Optilde\bigg(\frac{\zeta_{\mathsf{op}}}{n^{3/2}}\bigg) + \Optilde\bigg(\frac{1}{n}\bigg)\|\bM - \widehat{\bM}_{R(s - 1)}\|_2.
\end{align*}
By induction over $s$ and the assumption $s\leq S = O(1)$, we have
\begin{align*}
\|\bM - \widehat{\bM}_{RS}\|_2
&\leq \Optilde(mn^{3/2}\rho_n^2)\|\widehat{\bU}_{(R - 1)S}\bW_{(R - 1)S} - \bV\|_{2\to\infty} + \Optilde\bigg(\frac{mn^2\rho_n^2}{n^{S + 1}} + \frac{\zeta_{\mathsf{op}}}{n^{3/2}}\bigg).
\end{align*} 
The proof of the second assertion of Lemma \ref{lemma:recursive_two_to_infinity_error_bound} is therefore completed by dividing both sides of the inequality by $mn^{5/2}\rho_n^2$.

\noindent Next, we work with $\|\bR_1^{(RS)}\|_{2\to\infty}$.
By Lemma \ref{lemma:rowwise_concentration_noise}, we directly have
\begin{align*}
\left\|\{\calH(\bA\bA\transpose) - \bP\bP\transpose - \bM\}\bV\bS^{-1}\right\|_{2\to\infty}
 = \Optilde\bigg(\frac{\zeta_{\mathsf{op}}}{mn^{5/2}\rho_n^2}
  \bigg).
\end{align*}
By Lemma \ref{lemma:remainder_II}, Lemma \ref{lemma:remainder_IV}, Lemma \ref{lemma:remainder_III}, and Lemma \ref{lemma:Utilde_two_to_infinity_norm_II} we have
\begin{align*}
\|\bR_2^{(RS)} + \bR_4^{(RS)} + \bR_5^{(RS)} + \bR_6^{(RS)} + \bR_7^{(RS)}\|_{2\to\infty}
& = \Optilde\bigg(\frac{\zeta_{\mathsf{op}}}{mn^{5/2}\rho_n^2}
 \bigg)
 + \Optilde\bigg(\frac{1}{mn^{5/2}\rho_n^2}\bigg)\|\bM - \widehat{\bM}_{RS}\|_2.
\end{align*}
For $\bR_3^{(RS)}$, by Lemma \ref{lemma:Utilde_two_to_infinity_norm_II}, we also have
\begin{align*}
\|\bR_3^{(RS)}\|_{2\to\infty}\leq \|\bM - \widehat{\bM}_{RS}\|_\infty(\|\widehat{\bU}_{RS}\|_{2\to\infty} + \|\bV\|_{2\to\infty})\|\widehat{\bS}_{RS}^{-1}\|_2 = \Optilde\left(\frac{1}{mn^{5/2}\rho_n^2}\right)\|\bM - \widehat{\bM}_{RS}\|_2.
\end{align*}
It is therefore sufficient to show that 
\begin{align*}
\|\bR_1^{(RS)}\|_{2\to\infty}
& = \Optilde\bigg(
\frac{\zeta_{\mathsf{op}}}{mn^{5/2}\rho_n^2}\bigg)
  + \Optilde\bigg(\frac{1}{mn^{5/2}\rho_n^2}\bigg)\|\bM - \widehat{\bM}_{RS}\|_2
\end{align*}
because by the second assertion of Lemma \ref{lemma:recursive_two_to_infinity_error_bound}, we have
\begin{align*}
\Optilde\bigg(\frac{1}{mn^{5/2}\rho_n^2}\bigg)\|\bM - \widehat{\bM}_{RS}\|_2
& = \Optilde\left(\frac{1}{n}\right)\|\widehat{\bU}_{(R - 1)S}\bW_{(R - 1)S} - \bV\|_{2\to\infty}
 + \Optilde\bigg(\frac{1}{n^{S + 3/2}} + \frac{\zeta_{\mathsf{op}}}{mn^{5/2}\rho_n^2}\bigg).
\end{align*}
By definition of $\bR_1^{(RS)}$, we have
\begin{align*}
\max_{i\in[n]}\|\be_i\transpose\bR_1^{(RS)}\|_2& = \max_{i\in[n]}\left\|\be_i\transpose\{\calH(\bA\bA\transpose) - \bP\bP\transpose - \bM\}(\widehat{\bU}_{RS}\bW_{RS} - \bV)\widehat{\bS}^{-1}_{RS}\right\|_2
\leq r_1^{(1)} + r_1^{(2)} + r_1^{(3)} + r_1^{(4)} + r_1^{(5)},
\end{align*}
where
\begin{equation}\label{eqn:r11_r12_r13_r14}
\begin{aligned}
r_1^{(1)}& = \max_{i\in[n]}\left\|\be_i\transpose\{\calH(\bA\bA\transpose) - \bP\bP\transpose - \bM\}(\widehat{\bU}_{RS}\bH_{RS} - \widehat{\bU}_{RS}^{(i)}\bH_{RS}^{(i)})\bH_{RS}^{-1}(\bW_{RS} - \bH_{RS})\widehat{\bS}^{-1}_{RS}\right\|_2,\\
r_{1}^{(2)}& = \max_{i\in[n]}\left\|\be_i\transpose\{\calH(\bA\bA\transpose) - \bP\bP\transpose - \bM\}(\widehat{\bU}_{RS}^{(i)}\bH_{RS}^{(i)} - \bV)\bH_{RS}^{-1}(\bW_{RS} - \bH_{RS})\widehat{\bS}^{-1}_{RS}\right\|_2,\\
r_{1}^{(3)}& = \max_{i\in[n]}\left\|\be_i\transpose\{\calH(\bA\bA\transpose) - \bP\bP\transpose - \bM\}\bV\bH_{RS}^{-1}(\bW_{RS} - \bH_{RS})\widehat{\bS}^{-1}_{RS}\right\|_2,\\
r_{1}^{(4)}& = \max_{i\in[n]}\left\|\be_i\transpose\{\calH(\bA\bA\transpose) - \bP\bP\transpose - \bM\}(\widehat{\bU}_{RS}\bH_{RS} - \widehat{\bU}_{RS}^{(i)}\bH_{RS}^{(i)})\widehat{\bS}^{-1}_{RS}\right\|_2,\\
r_{1}^{(5)}& = \max_{i\in[n]}\left\|\be_i\transpose\{\calH(\bA\bA\transpose) - \bP\bP\transpose - \bM\}(\widehat{\bU}_{RS}^{(i)}\bH_{RS}^{(i)} - \bV)\widehat{\bS}^{-1}_{RS}\right\|_2.
\end{aligned}
\end{equation}
By Lemma \ref{lemma:Utilde_two_to_infinity_norm_II}, we know that for any $c > 0$, there exists a $c$-dependent constant $N_c > 0$ such that $\|\bH_{RS}^{-1}\|_2\leq 2$ and $\max_{i\in[n]}\|(\bH^{(i)}_{RS})^{-1}\|_2\leq 2$ with probability at least $1 - O(n^{-c})$ for all $n\geq N_c$. Then for $r_{1}^{(1)}$ and $r_1^{(4)}$, by Lemma \ref{lemma:spectral_norm_concentration_noise} and Lemma \ref{lemma:Utilde_two_to_infinity_norm_II}, 
$
r_1^{(1)}\vee r_1^{(4)} \leq 2\|\calH(\bA\bA\transpose) - \bP\bP\transpose - \bM\|_2\max_{i\in[n]}\|\widehat{\bU}_{RS}\bH_{RS} - \widehat{\bU}_{RS}^{(i)}\bH_{RS}^{(i)}\|_2\|\widehat{\bS}_{RS}^{-1}\|_2 = \Optilde\{\zeta_{\mathsf{op}}(mn^{5/2}\rho_n^2)^{-1}\}
$.
For $r_1^{(2)}$ and $r_1^{(5)}$, by Lemma \ref{lemma:Utilde_two_to_infinity_norm_II}, we have
\begin{align*}
r_1^{(2)}\vee r_1^{(5)} 
&\lesssim \max_{i\in[n]}\|\be_i\transpose\{\calH(\bA\bA\transpose) - \bP\bP\transpose - \bM)(\widehat{\bU}_{RS}^{(i)}\bH_{RS}^{(i)} - \bV)\|_2\|\widehat{\bS}_{RS}^{-1}\|_2\\
& =  \Optilde\bigg\{\frac{\log n}{m^{1/2}n^{3/2}\rho_n} + \frac{(\log n)^{1/2}}{m^{1/2}n\rho_n^{1/2}}\bigg\}
  + \Optilde\left(\frac{1}{mn^{5/2}\rho_n^2}\right)\|\bM - \widehat{\bM}_{RS}\|_2.
\end{align*}
For $r_1^{(3)}$, by Lemma \ref{lemma:rowwise_concentration_noise}, Lemma \ref{lemma:Utilde_two_to_infinity_norm_II}, and Lemma \ref{lemma:remainder_II}, we have
\begin{align*}
r_1^{(3)}&\leq 2\max_{i\in[n]}\|\be_i\transpose\{\calH(\bA\bA\transpose) - \bP\bP\transpose - \bM\}\bV\|_2\|\widehat{\bS}^{-1}_{RS}\|_2 = \Optilde\bigg(\frac{\zeta_{\mathsf{op}}}{mn^{5/2}\rho_n^2}\bigg).
\end{align*}
Combining the error bounds for $r_1^{(1)}$ through $r_1^{(5)}$ yields
\begin{align*}
\max_{i\in[n]}\|\be_i\transpose\bR_1^{(RS)}\|_2
& = \Optilde\bigg\{
\frac{\log n}{m^{1/2}n^{3/2}\rho_n} + \frac{(\log n)^{1/2}}{m^{1/2}n\rho_n^{1/2}}\bigg\}
  + \Optilde\left(\frac{1}{mn^{5/2}\rho_n^2}\right)\|\bM - \widehat{\bM}_{RS}\|_2.
\end{align*}
This completes the proof of Lemma \ref{lemma:R1_remainder}. The proof of the first assertion of Lemma \ref{lemma:recursive_two_to_infinity_error_bound} is then completed by combining the error bounds for $\|\bR_1^{(RS)}\|_{2\to\infty}$ through $\|\bR_7^{(RS)}\|_{2\to\infty}$.

\section{Proof of Exact Recovery in MLSBM (Theorem \ref{thm:spectral_clustering_MLSBM})}
\label{sec:proof_of_exact_recovery_MLSBM}

Let $\bZ$ be an $n\times K$ community assignment matrix whose $\tau(i)$th entry of the $i$th row is $1$ for all $i\in[n]$, and the remaining entries are zero. Let $n_k = \sum_{i = 1}^n\mathbbm{1}\{\tau(i) = k\}$, \emph{i.e.}, the number of vertices in the $k$th community, and denote by $n_{\min} = \min_{k\in[K]}n_k$ and $n_{\max} = \max_{k\in[K]}n_k$. Then we take $\bU = \bZ(\bZ\transpose\bZ)^{-1/2}$ so that $\bU\in\mathbb{O}(n, d)$ and $\bB_t = (\bZ\transpose\bZ)^{1/2}\bC_t(\bZ\transpose\bZ)^{1/2}$ for all $t\in[m]$, and it is clear that $\mathrm{MLSBM}(\tau;\bC_1,\ldots,\bC_m) = \mathrm{COSIE}(\bU;\bB_1,\ldots,\bB_m)$. Since $n_k\asymp n$, then $\sigma_k(\bB_t) = \Theta(n\rho_n)$ for all $k\in[K]$, $t\in[m]$, and $\|\bU\|_{2\to\infty}\leq \|\bZ\|_{2\to\infty}\|(\bZ\transpose\bZ)^{-1/2}\|_2 = O(n^{-1/2})$, so that the conditions of Theorem \ref{thm:entrywise_eigenvector_perturbation_bound} are satisfied. Furthermore, if $(\bmu_k^*)_{k = 1}^K$ are the $K$ unique rows of $\bU$, then we also have $\min_{k\neq l}\|\bmu_k^* - \bmu_l^*\|_2\geq c_0n_{\max}^{-1/2}$ for some constant $c_0 \in (0, 1]$. 
By Theorem \ref{thm:entrywise_eigenvector_perturbation_bound}, there exists some $\bW\in\mathbb{O}(K)$ such that $\|\widehat{\bU}\bW - \bU\|_{2\to\infty} = \Optilde\{(\varepsilon_n^{(\mathsf{op})} + \varepsilon_n^{(\mathsf{bc})})/\sqrt{n}\}$ and $\|\widehat{\bU}\bW - \bU\|_{\mathrm{F}} = \Optilde(\varepsilon_n^{(\mathsf{op})} + \varepsilon_n^{(\mathsf{bc})})$. Let $\delta_k = \min_{l\in[K]\backslash\{k\}}\|\bmu_k^* - \bmu_l^*\|_2$. Clearly, $\delta_k\geq n_k^{-1/2}$ for all $k\in[K]$ and $\delta_k\leq (2/n_{\min})^{1/2}$. Because $\varepsilon_n^{(\mathsf{op})} + \varepsilon_n^{(\mathsf{bc})} = o(1)$, Theorem \ref{thm:entrywise_eigenvector_perturbation_bound} implies that for any $c > 0$ and the given threshold $\eps > 0$, there exists a constant $N_{c} > 0$, such that 
\[
\prob\bigg(\|\widehat{\bU}\bW - \bU\|_{\mathrm{F}}^2\leq \frac{c_0^2\beta n_{\min}}{512n}\bigg)\geq 1 - O(n^{-c})\quad\mbox{for all }n\geq N_{c},
\]
where $\beta:=n_{\min}/n_{\max}$. Define $\calS_k = \{i\in[n]:\tau(i) = k,\|\be_i\transpose(\widehat{\bL}\bW - \bU)\|_2\geq \delta_k/2\}$ and denote by $|\calS_k|$ the number of elements in $\calS_k$. By Lemma 5.3 in \cite{lei2015}, there exists a permutation $\sigma\in\calP_K$, where $\calP_K$ denotes the collection of all permutations over $[K]$, such that 
\[
\frac{1}{n}\sum_{i = 1}^n\mathbbm{1}\{\widehat{\tau}(i)\neq \sigma(\tau(i))\}\leq \frac{1}{n}\sum_{k = 1}^K|\calS_k|\leq \sum_{k = 1}^K\frac{|\calS_k|}{n_k}\leq \sum_{k = 1}^K|\calS_k|\delta_k^2\leq 16\|\widehat{\bU}\bW - \bU\|_{\mathrm{F}}^2\leq \frac{c_0^2\beta n_{\min}}{64n}
\]
with probability at least $1 - O(n^{-c})$ for all $n\geq N_{c,\eps}$. Let $\calT_k = \{i\in[n]:\tau(i) = k\}\backslash\calS_k$. Then Lemma 5.3 in \cite{lei2015} further implies that $\widehat{\tau}(i) = \sigma(\tau(i))$ for all $i\in\bigcup_{k = 1}^K\calT_k$, so that $|\calT_k| = n_k - |\calS_k|\geq n_k - \sum_{l = 1}^K|\calS_l|\geq 63n_k/64$. Now let $\calC_k = \{i\in[n]:\tau(i) = k\}$ and $\widehat{\calC}_k = \{i\in[n]:\widehat{\tau}(i) = k\}$. Clearly, $\calT_{\sigma^{-1}(k)}\subset \{i\in[n]:\sigma(\tau(i)) = \widehat{\tau}(i) = k\}\subset\widehat{\calC}_k$, so that one also obtains $|\widehat{\calC}_k|\geq|\calT_{\sigma^{-1}(k)}|\geq n_{\min} - \sum_{l = 1}^K|\calS_l|\geq (1 - \beta/64)n_{\min}$ and $|\widehat{\calC}_k\backslash\calC_{\sigma^{-1}(k)}|\leq c_0^2\beta n_{\min}/64$.

Note that by the definition of $\widehat{\bL}$, if we take $(\widehat{\bmu}_k)_{k = 1}^K$ as the $K$ unique rows of $\widehat{\bL}$, then given $\widehat{\tau}(\cdot)$, the $K$-means clustering problem is equivalent to the following minimization problem
\[
\min_{(\widehat{\bmu}_k)_{k = 1}^K}\sum_{i = 1}^n\|\widehat{\bmu}_{\widehat{\tau}(i)} - \bU\transpose\be_i\|_{\mathrm{F}}^2 = \min_{(\widehat{\bmu}_k)_{k = 1}^K}\sum_{k = 1}^K\sum_{i\in[n]:\widehat{\tau}(i) = k}\|\widehat{\bmu}_{k} - \bU\transpose\be_i\|_{\mathrm{F}}^2,
\]
the solution of which are given by $\widehat{\bmu}_k = |\widehat{\calC}_k|^{-1}\sum_{i\in\widehat{\calC}_k}\widehat{\bU}\transpose\be_i$, $k\in[K]$. It follows that for all $k\in[K]$,
\begin{align*}
\|\bW\transpose\widehat{\bmu}_k - \bmu_{\sigma^{-1}(k)}^*\|_2
&\leq \bigg\|\frac{1}{|\widehat{\calC}_k|}\sum_{i\in\widehat{\calC}_k}(\bW\transpose\widehat{\bU}\transpose\be_i - \bU\transpose\be_i)\bigg\|_2 + \bigg\|\frac{1}{|\widehat{\calC}_k|}\sum_{i\in\widehat{\calC}_k}(\bU\transpose\be_i - \bmu_{\sigma^{-1}(k)}^*)\bigg\|_2\\
& = \bigg\|\frac{1}{|\widehat{\calC}_k|}\sum_{i\in\widehat{\calC}_k}(\bW\transpose\widehat{\bU}\transpose\be_i - \bU\transpose\be_i)\bigg\|_2 + \bigg\|\frac{1}{|\widehat{\calC}_k|}\sum_{i\in\widehat{\calC}_k\backslash\calC_{\sigma^{-1}(k)}}(\bU\transpose\be_i - \bmu_{\sigma^{-1}(k)}^*)\bigg\|_2\\
&\leq \frac{1}{|\widehat{\calC}_k|^{1/2}}\|\widehat{\bU}\bW - \bU\|_{\mathrm{F}} + \frac{\sqrt{2}|\widehat{\calC}_k\backslash\calC_{\sigma^{-1}(k)}|}{n_{\min}^{1/2}|\widehat{\calC}_k|}
\leq \bigg\{\frac{c_0^2\beta}{512n(1 - \beta/64)}\bigg\}^{1/2} + \frac{c_0^2\sqrt{2\beta}}{63n_{\max}^{1/2}}\leq \frac{c_0}{8n_{\max}^{1/2}}
\end{align*}
with probability at least $1 - O(n^{-c})$ for all $n\geq N_c$ because $\beta < 1$. Now we claim that for any $i\in[n]$, if $\|\be_i\transpose(\widehat{\bU}\bW - \bU)\|_2\leq c_0/(4n_{\max}^{1/2})$, then $\widehat{\tau}(i) = \sigma(\tau(i))$. Suppose $\tau(i) = k$, and we will argue that $\widehat{\tau}(i) = \sigma(k)$. First note that $\|\be_i\transpose(\widehat{\bU}\bW - \bU)\|_2\leq c_0/(4n_{\max}^{1/2})$ implies
\begin{align*}
\|\bW\transpose\widehat{\bU}\transpose\be_i - \bW\transpose\widehat{\bmu}_{\sigma(k)}\|_2&\leq 
\|\bW\transpose\widehat{\bU}\transpose\be_i - \bU\transpose\be_i\|_2 + \|\bmu_k^* - \bW\transpose\widehat{\bmu}_{\sigma(k)}\|_2\leq \frac{3c_0}{8n_{\max}^{1/2}},\\
\|\bW\transpose\widehat{\bU}\transpose\be_i - \bW\transpose\widehat{\bmu}_{\sigma(l)}\|_2&\geq
\|{\bmu}^*_l - {\bmu}^*_k\|_2 -  
\|\bW\transpose\widehat{\bU}\transpose\be_i - \bU\transpose\be_i\|_2 - \|\bmu_l^* - \bW\transpose\widehat{\bmu}_{\sigma(l)}\|_2\geq \frac{5c_0}{8n_{\max}^{1/2}}.
\end{align*}
Now we argue that $\widehat{\tau}(i) = \sigma(k)$ by contradiction. Indeed, assume otherwise, so that $\widehat{\tau}(i) = \sigma(l)$ for some $l\in[K]\backslash\{k\}$. Then
\[
\|\widehat{\bU} - \widehat{\bL}\|_{\mathrm{F}}^2 = \sum_{j\in[n]\backslash\{i\}}\|\be_j\transpose(\widehat{\bU} - \widehat{\bL})\|_2^2 + \|\bW\transpose\widehat{\bU}\transpose\be_i - \bW\transpose\widehat{\bmu}_{\sigma(l)}\|_2^2 > \sum_{j\in[n]\backslash\{i\}}\|\be_j\transpose(\widehat{\bU} - \widehat{\bL})\|_2^2 + \|\bW\transpose\widehat{\bU}\transpose\be_i - \bW\transpose\widehat{\bmu}_{\sigma(k)}\|_2^2,
\]
which implies that replacing $\widehat{\tau}(i) = \sigma(l)$ with $\widehat{\tau}(i) = \sigma(k)$ strictly decreases the objective function without violating the feasibility because $|\widehat{\calC}_l|\geq 2$ for all $l\in[K]$. This violates the optimality of $\widehat{\bL}$. Hence, it must be the case that $\widehat{\tau}(i) = \sigma(k)$. Therefore,
\begin{align*}
\prob\bigg\{\text{There exists }\sigma\in\calP_K\text{ such that }\widehat{\tau}(i)
& = \sigma(\tau(i))\text{ for all }i\in[n]\bigg\}\geq \prob\left(\|\widehat{\bU}\bW - \bU\|_{2\to\infty}\leq \frac{c_0}{4n_{\max}^{1/2}}\right)\\
&\geq 1 - O(n^{-c})
\end{align*}
for all $n\geq N_c$, where $N_c$ is a constant depending on the given power $c > 0$. The proof is therefore completed.

\section{Proof of Entrywise Eigenvector Limit Theorem (Theorem \ref{thm:Rowwise_CLT})}
\label{sec:proof_of_eigenvector_CLT}

\subsection{Proof of Lemma \ref{lemma:martingale_construction}}
\label{sub:proof_sketch_CLT}

By definition, we have $\bV\bS^{-1}\bQ_1\transpose = \bU(\bB\bB\transpose)^{-1}$ and $\widehat{\bU}\mathrm{sgn}(\widehat{\bU}\transpose\bU) - \bU = \{\widehat{\bU}\mathrm{sgn}(\widehat{\bU}\transpose\bV) - \bV)\}\bQ_1\transpose$.
Let $\widehat{\bU}$ denote $\widehat{\bU}_{RS}$, $\bW$ denote $\bW_{RS}$, and $\bT$ denote $\bT^{(RS)}$. Recall the decomposition \eqref{eqn:eigenvector_decomposition_main}:
$\widehat{\bU}\bW - \bU = \{\calH(\bA\bA\transpose) - \bP\bP\transpose - \bM\}\bV\bS^{-1} + \bT$. It turns out that $\|\bT\|_{2\to\infty}$ is negligible compared to the first term $\be_i\transpose\{\calH(\bA\bA\transpose) - \bP\bP\transpose - \bM\}\bV\bS^{-1}$ for each fixed $i\in[n]$. For the leading term, by Bernstein's inequality, we have
$\be_i\transpose\left\{\calH(\bA\bA\transpose) - \bP\bP\transpose - \bM\right\}\bV\bS^{-1}\bQ_1\transpose\bGamma_i^{-1/2}\bz = X_{in}(\bz) + o_p(1)$, where $\bz\in\mathbb{R}^d$ is a unit vector, and
\begin{align}\label{eqn:CLT_dominating_term}
X_{in}(\bz) & = \be_i\transpose\calH(\bE\bE\transpose)\bV\bS^{-1}\bQ_1\transpose\bGamma_i^{-1/2}\bz + \sum_{t = 1}^m\sum_{j = 1}^nE_{tij}(\be_j\transpose\bP_t\bV\bS^{-1}\bQ_1\transpose\bGamma_i^{-1/2}\bz).
\end{align}
We first argue that $X_{in}(\bz) = \sum_{t = 1}^m\sum_{\substack{j_1,j_2\in[n],j_1 \leq j_2}}E_{tj_1j_2}\{b_{tij_1j_2}(\bz) + c_{tij_1j_2}(\bz)\}$, where $b_{tij_1j_2}(\bz)$ and $c_{tij_1j_2}(\bz)$ are defined in \eqref{eqn:btj1j2} and \eqref{eqn:ctj1j2}. For notational convenience, below in this section, we will suppress the dependence on $\bz$ from $b_{tij_1j_2}$, $c_{tij_1j_2}$, $\gamma_{ij}$, and $\xi_{tij}$. The dependence on $\bz$ will be used later in the proof of hypothesis testing for membership profiles in MLMM. 
Write
\begin{align*}
X_{in}(\bz)
& = \sum_{t = 1}^m\sum_{j_1 = 1}^n\sum_{j_2 = 1}^n\sum_{j_3 = 1}^ne_{ij_1}E_{tj_1j_2}E_{tj_2j_3}\gamma_{ij_3}\mathbbm{1}(j_3\neq i) + \sum_{t = 1}^m\sum_{j_1 = 1}^n\sum_{j_2 = 1}^ne_{ij_1}E_{tj_1j_2}\xi_{tij_2}\\
& = \sum_{t = 1}^m\sum_{j_1,j_2,j_3\in[n]:j_1 > j_3}e_{ij_1}E_{tj_1j_2}E_{tj_2j_3}\gamma_{ij_3}\mathbbm{1}(j_3\neq i) + \sum_{t = 1}^m\sum_{j_1,j_2,j_3\in[n]:j_1 < j_3}e_{ij_1}E_{tj_1j_2}E_{tj_2j_3}\gamma_{ij_3}\mathbbm{1}(j_3\neq i)\\
&\quad + \sum_{t = 1}^m\sum_{j_1,j_2\in[n]}e_{ij_1}E_{tj_1j_2}^2\gamma_{ij_1}\mathbbm{1}(j_1\neq i) + \sum_{t = 1}^m\sum_{j_1 = 1}^n\sum_{j_2 = 1}^ne_{ij_1}E_{tj_1j_2}\xi_{tij_2}\\
& = \sum_{t = 1}^m\sum_{j_1,j_2,j_3\in[n]:j_1 > j_3}E_{tj_1j_2}E_{tj_2j_3}\left\{e_{ij_1}
\gamma_{ij_3}\mathbbm{1}(j_3\neq i) + 
e_{ij_3}\gamma_{ij_1}\mathbbm{1}(j_1\neq i)\right\}
 + \sum_{t = 1}^m\sum_{j_1 = 1}^n\sum_{j_2 = 1}^ne_{ij_1}E_{tj_1j_2}\xi_{tij_2}.
\end{align*}
Breaking the first summation into three parts $\{j_1 < j_2\}$, $\{j_1 = j_2\}$, and $\{j_1 > j_2\}$, we further obtain
\begin{align*}
X_{in}(\bz)
& = \sum_{t = 1}^m\sum_{\substack{j_1,j_2,j_3\in[n]\\
j_1 > j_3, j_1 < j_2}}E_{tj_1j_2}E_{tj_2j_3}\left\{e_{ij_1}
\gamma_{ij_3}\mathbbm{1}(j_3\neq i) + 
e_{ij_3}\gamma_{ij_1}\mathbbm{1}(j_1\neq i)\right\}
\\&\quad
 + \sum_{t = 1}^m\sum_{\substack{j_1,j_2,j_3\in[n]\\
j_1 > j_3, j_1 > j_2}}E_{tj_1j_2}E_{tj_2j_3}\left\{e_{ij_1}
\gamma_{ij_3}\mathbbm{1}(j_3\neq i) + 
e_{ij_3}\gamma_{ij_1}\mathbbm{1}(j_1\neq i)\right\}\\
&\quad + \sum_{t = 1}^m\sum_{j_1,j_3\in[n]:j_1 > j_3}E_{tj_1j_1}E_{tj_1j_3}\left\{e_{ij_1}
\gamma_{ij_3}\mathbbm{1}(j_3\neq i) + 
e_{ij_3}\gamma_{ij_1}\mathbbm{1}(j_1\neq i)\right\}
 + \sum_{t = 1}^m\sum_{j_1 = 1}^n\sum_{j_2 = 1}^ne_{ij_1}E_{tj_1j_2}\xi_{tij_2}.
\end{align*}
Now switching the roles of $j_1$ and $j_2$ in the second summation above and rearranging complete the proof that $X_{in}(\bz) = \sum_{t = 1}^m\sum_{j_1,j_2\in[n],j_1\leq j_2}E_{tj_1j_2}\{b_{tij_1j_2}(\bz) + c_{tij_1j_2}(\bz)\}$.
We first argue that $(t, j_1, j_2)\mapsto \alpha(t, j_1, j_2)$ is one-to-one. Assume otherwise. Then there exists some $(t', j_1', j_2')\in[m]\times [n]\times [n]$, $j_1'\leq j_2'$, such that
\begin{align}\label{eqn:martingale_mapping_rule}
\frac{1}{2}(t - 1)n(n + 1) + j_1 + \frac{1}{2}j_2(j_2 - 1) = 
\frac{1}{2}(t' - 1)n(n + 1) + j_1' + \frac{1}{2}j_2'(j_2' - 1).
\end{align}
This equation forces $t = t'$ because if not, then without loss of generality, we may assume that $t \leq t' - 1$, in which case
\begin{align*}
\frac{1}{2}(t - 1)n(n + 1) + j_1 + \frac{1}{2}j_2(j_2 - 1)
&\leq \frac{1}{2}(t - 1)n(n + 1) + n + \frac{1}{2}n(n - 1) = \frac{1}{2}tn(n + 1)\\
&\leq \frac{1}{2}(t' - 1)n(n + 1) < \frac{1}{2}(t' - 1)n(n + 1) + j_1' + \frac{1}{2}j_2'(j_2' - 1),
\end{align*}
contradicting with \eqref{eqn:martingale_mapping_rule}. Therefore, \eqref{eqn:martingale_mapping_rule} reduces to
\begin{align}\label{eqn:martingale_mapping_rule_II}
j_1 + \frac{1}{2}j_2(j_2 - 1) = 
j_1' + \frac{1}{2}j_2'(j_2' - 1).
\end{align}
Now we claim that \eqref{eqn:martingale_mapping_rule_II} implies that $j_1 = j_1'$ and $j_2 = j_2'$. Assume otherwise. Then it follows that $j_1 \neq j_1'$ and $j_2\neq j_2'$, and without loss of generality, assume that $j_2 \leq j_2' - 1$. Then \eqref{eqn:martingale_mapping_rule_II} yields that
$(1/2)j_2'(j_2' - 1) - (1/2)j_2(j_2 - 1) = j_1 - j_1'\leq j_1 - 1$.
On the other hand, $j_2 \leq j_2' - 1$ implies
$(1/2)j_2'(j_2' - 1) - (1/2)j_2(j_2 - 1)\geq (1/2)(j_2 + 1)j_2 - (1/2)j_2(j_2 - 1) = j_2\geq j_1 > j_1 - 1$,
which contradicts with the previous inequality. Hence, we conclude that $j_1 = j_1'$ and $j_2 = j_2'$. 

The relabeling scheme $(t, j_1, j_2)\mapsto \alpha = \alpha(t, j_1, j_2)$ enables us to rewrite
\[
X_{in}(\bz) = \sum_{t = 1}^m\sum_{\substack{j_1,j_2\in[n],j_1\leq j_2}}E_{tj_1j_2}(b_{tij_1j_2} + c_{tij_1j_2}) = \sum_{t = 1}^m\sum_{\substack{j_1,j_2\in[n],j_1\leq j_2}}Y_{n\alpha(t, j_1, j_2)} = \sum_{\alpha = 1}^{N_n}Y_{n\alpha},
\]
where we set $Y_{n\alpha} = Y_{n\alpha(t, j_1, j_2)} = E_{tj_1j_2}(b_{tij_1j_2} + c_{tij_1j_2})$ and $N_n = mn(n + 1)/2$. Clearly, the nested $\sigma$-fields $\calF_{n0}\subset\calF_{n1}\subset\ldots\subset\calF_{nN_n}$ form a filtration. For any $\alpha \in [N_n]$, the random variable
$Z_{n\alpha} := \sum_{\beta = 1}^\alpha Y_{n\beta}$ is $\calF_{n\alpha}$-measurable. Indeed, first observe that $E_{tj_1j_2}$ is $\calF_{n\alpha}$-measurable for any $(t, j_1, j_2)$ satisfying $\alpha(t,j_1, j_2)\leq\alpha$. Also, for any $t\in[m], j_1, j_2\in[n]$, $j_1\leq j_2$ with $\alpha(t, j_1, j_2)\leq \alpha$,
$b_{tij_1j_2}(\bz)$
is a function of $(E_{tj_3j_2})_{j_3 = 1}^{j_1 - 1}$ and $(E_{tj_1j_3})_{j_3 = 1}^{j_2 - 1}$, and these random variables are $\calF_{n(\alpha - 1)}$-measurable because
\[
\alpha(t, j_3, j_2) = \frac{1}{2}(t - 1)n(n + 1) + j_3 + \frac{1}{2}j_2(j_2 - 1)\leq \frac{1}{2}(t - 1)n(n + 1) + j_1 - 1 + \frac{1}{2}j_2(j_2 - 1) = \alpha(t, j_1, j_2) - 1\leq\alpha - 1,
\]
for any $j_3 \leq j_1 - 1$, 
\[
\alpha(t, j_1, j_3) = \frac{1}{2}(t - 1)n(n + 1) + j_1 + \frac{1}{2}j_3(j_3 - 1)\leq\frac{1}{2}(t - 1)n(n + 1) + j_1 + \frac{1}{2}(j_2 - 1)(j_2 - 2)\leq \alpha(t, j_1, j_2) - 1\leq\alpha - 1
\]
if $j_1 \leq j_3$, and
\begin{align*}
\alpha(t, j_3, j_1) &\leq\frac{1}{2}(t - 1)n(n + 1) + j_1 - 1 + \frac{1}{2}j_1(j_1 - 1)\leq \frac{1}{2}(t - 1)n(n + 1) + j_1 - 1 + \frac{1}{2}j_2(j_2 - 1)\\
& = \alpha(t, j_1, j_2) - 1<\alpha
\end{align*}
if $j_1 > j_3$.
Note that $c_{tij_1j_2}\in\calF_{n(\alpha - 1)}$ because it is a constant. These observations imply that $Y_{n\alpha}$ is $\calF_{n\alpha}$-measurable and $b_{tij_1j_2} + c_{tij_1j_2}$ is $\calF_{n(\alpha - 1)}$-measurable if $\alpha(t, j_1, j_2)\leq \alpha$. Therefore, with $t(\alpha)\in[m],j_1(\alpha),j_2(\alpha)\in[n]$ satisfying $\alpha(t(\alpha),j_1(\alpha),j_2(\alpha)) = \alpha$, we further obtain 
\begin{align*}
\expect\left(Z_{n\alpha}\mid\calF_{n(\alpha - 1)}\right)
& = Z_{n(\alpha - 1)} + \expect\left\{E_{t(\alpha)j_1(\alpha)j_2(\alpha)}(b_{t(\alpha)ij_1(\alpha)j_2(\alpha)} + c_{t(\alpha)ij_1(\alpha)j_2(\alpha)})\mid\calF_{n(\alpha - 1)}\right\} = Z_{n(\alpha - 1)},
\end{align*}
where we have used the facts that $E_{t(\alpha)j_1(\alpha)j_2(\alpha)}$ is a mean-zero independent random variable independent of the $\sigma$-field $\calF_{n(\alpha - 1)}$ and that $b_{t(\alpha)ij_1(\alpha)j_2(\alpha)} + c_{t(\alpha)ij_1(\alpha)j_2(\alpha)}$ is $\calF_{n(\alpha - 1)}$-measurable. Therefore, we see that $(Z_{n\alpha})_{\alpha = 1}^N$ forms a martingale with martingale difference sequence $(Y_{n\alpha})_{\alpha = 1}^N$. The proof is completed.

\subsection{Martingale Moment Bounds}
\label{sub:martingale_moment_bound}
We first establish several martingale moment bounds that facilitate the application of the martingale central limit theorem. 
\begin{lemma}\label{lemma:btj1j2_fourth_moment}
Suppose Assumptions \ref{assumption:eigenvector_delocalization}--\ref{assumption:condition_number} hold. Then for any $t\in[m]$ and $j_1,j_2\in[n]$, $j_1\leq j_2$,
\begin{align*}
\expect\{b_{tij_1j_2}^4(\bz)\}\lesssim (n\rho_ne_{ij_1} + n^2\rho_n^2e_{ij_1} + n\rho_ne_{ij_2} + n^2\rho_n^2e_{ij_2} + \rho_n)\max_{\|\bz\|_2\leq 1,j_1,j_2\in[n]}|\gamma_{j_1j_2}(\bz)|^4,
\end{align*}
where $b_{tij_1j_2}(\bz)$ is defined in \eqref{eqn:btj1j2}.
\end{lemma}

\begin{proof}
For convenience, we let $\gamma_{\max}$ denote $\max_{j_1,j_2\in[n]}|\gamma_{j_1j_2}|$ and suppress the dependence on $\bz$ in this proof. 
Without loss of generality, we may assume that $j_1\neq j_2$. By definition and Young's inequality for product, 
$\expect b_{tij_1j_2}^4\lesssim \delta_{tj_1j_2}^{(1)} + \delta_{tj_1j_2}^{(4)}$,
where
\begin{align*}
\delta_{tj_1j_2}^{(1)}
& = \expect\bigg[\sum_{a = 1}^{j_1 - 1}E_{taj_2}\left\{e_{ij_1}\gamma_{ia}\mathbbm{1}(a\neq i) + e_{ia}\gamma_{ij_1}\mathbbm{1}(j_1\neq i)\right\}\bigg]^4,\\
\delta_{tj_1j_2}^{(2)}
& = \expect\bigg[\sum_{b = 1}^{j_2 - 1}E_{tbj_1}\left\{e_{ij_2}\gamma_{ib}\mathbbm{1}(b\neq i) + e_{ib}\gamma_{ij_2}\mathbbm{1}(j_2\neq i)\right\}\bigg]^4.
\end{align*}
For $\delta_{tj_1j_2}^{(1)}$, we have
\begin{align*}
\delta_{tj_1j_2}^{(1)}
&\lesssim \sum_{a = 1}^n(e_{ij_1} + e_{ia})\rho_n\gamma_{\max}^4 + \sum_{a_1,a_2\in[n]}(e_{ij_1} + e_{ia_1})(e_{ij_1} + e_{ia_2})\rho_n^2\gamma_{\max}^4\lesssim (n\rho_ne_{ij_1} + n^2\rho_n^2e_{ij_1} + \rho_n)\gamma_{\max}^4
\end{align*}
because the remaining terms in the expansion have zero expected values.
Similarly,
$\delta_{tj_1j_2}^{(2)}\lesssim (n\rho_ne_{ij_2} + n^2\rho_n^2e_{ij_2} + \rho_n)\gamma_{\max}^4$.
Combining the above upper bounds completes the proof.
\end{proof}

\begin{lemma}\label{lemma:btj1j2_cross_fourth_moment}
Suppose Assumption \ref{assumption:eigenvector_delocalization}--\ref{assumption:condition_number} hold. For any fixed $i_1,i_2\in[n]$, denote by $\bar{b}_{tj_1j_2} = b_{ti_1j_1j_2}(\bz_1) + b_{ti_2j_1j_2}(\bz_2)$, where $b_{tij_1j_2}(\bz)$ is defined by \eqref{eqn:btj1j2} and $\bz_1,\bz_2\in\mathbb{R}^d$ with $\|\bz_1\|_2,\|\bz_2\|\leq 1$. Then
\begin{align*}
\sum_{t = 1}^m\sum_{\substack{j_1,j_2,j_3,j_4\in[n]\\j_1\leq j_2,j_3\leq j_4\\\alpha(t, j_1, j_2) < \alpha(t, j_3, j_4)}}\expect(\bar{b}_{tj_1j_2}^2\bar{b}_{tj_3j_4}^2)\sigma_{tj_1j_2}^2\sigma_{tj_3j_4}^2\lesssim (mn^3\rho_n^3 + mn^4\rho_n^4)\max_{\|\bz\|_2\leq 1,j_1,j_2\in[n]}|\gamma_{j_1j_2}(\bz)|^4,
\end{align*}
where $\alpha(t, j_1, j_2)$ is the relabeling function defined in \eqref{eqn:relabeling_function} and $\gamma_{j_1j_2}(\bz)$ is defined in \eqref{eqn:gamma_xi}.
\end{lemma}

\begin{proof}
By definition and Cauchy-Schwarz inequality, for any $j_1,j_2,j_3,j_4\in[n]$,
\begin{align*}
\bar{b}_{tj_1j_2}^2\bar{b}_{tj_3j_4}^2
& \leq 4\bigg\{\bigg(\sum_{a = 1}^{j_1 - 1}E_{tj_2a}\varpi_{i_1i_2j_1a}\bigg)^2
+ 
\bigg(\sum_{a = 1}^{j_2 - 1}E_{tj_1a}\varpi_{i_1i_2j_2a}\bigg)^2\bigg\}\\
&\quad\times\bigg\{\bigg(\sum_{b = 1}^{j_3 - 1}E_{tj_4b}\varpi_{i_1i_2j_3b}\bigg)^2
 + 
\bigg(\sum_{b = 1}^{j_4 - 1}E_{tj_3b}\varpi_{i_1i_2j_4b}\bigg)^2\bigg\}\\
& = 4\vartheta_{tj_1j_2j_3j_4}^{(1)} + 4\vartheta_{tj_1j_2j_3j_4}^{(2)} + 4\vartheta_{tj_1j_2j_3j_4}^{(3)} + 4\vartheta_{tj_1j_2j_3j_4}^{(4)},
\end{align*}
where
\begin{align*}
\vartheta_{tj_1j_2j_3j_4}^{(1)}
& = \bigg(\sum_{a = 1}^{j_1 - 1}E_{tj_2a}\varpi_{i_1i_2j_1a}\bigg)^2
\bigg(\sum_{b = 1}^{j_3 - 1}E_{tj_4b}\varpi_{i_1i_2j_3b}\bigg)^2,\\
\vartheta_{tj_1j_2j_3j_4}^{(2)}
& = \bigg(\sum_{a = 1}^{j_1 - 1}E_{tj_2a}\varpi_{i_1i_2j_1a}\bigg)^2
\bigg(\sum_{b = 1}^{j_4 - 1}E_{tj_3b}\varpi_{i_1i_2j_4b}\bigg)^2,\\
\vartheta_{tj_1j_2j_3j_4}^{(3)}
& = \bigg(\sum_{a = 1}^{j_2 - 1}E_{tj_1a}\varpi_{i_1i_2j_2a}\bigg)^2
\bigg(\sum_{b = 1}^{j_3 - 1}E_{tj_4b}\varpi_{i_1i_2j_3b}\bigg)^2,\\
\vartheta_{tj_1j_2j_3j_4}^{(4)}
& = \bigg(\sum_{a = 1}^{j_2 - 1}E_{tj_1a}\varpi_{i_1i_2j_2a}\bigg)^2\bigg(\sum_{b = 1}^{j_4 - 1}E_{tj_3b}\varpi_{i_1i_2j_4b}\bigg)^2.
\end{align*}
and
\begin{align*}
\varpi_{i_1i_2j_1a}& = e_{i_1j_1}\gamma_{i_1a}(\bz_1)\mathbbm{1}(a\neq i_1) + 
e_{i_1a}\gamma_{i_1j_1}(\bz_1)\mathbbm{1}(j_1\neq i_1) + 
e_{i_2j_1}\gamma_{i_2a}(\bz_2)\mathbbm{1}(a\neq i_2) + 
e_{i_2a}\gamma_{i_2j_1}(\bz_2)\mathbbm{1}(j_1\neq i_2),\\
\varpi_{i_1i_2j_2a}& = e_{i_1j_2}\gamma_{i_1a}(\bz_1)\mathbbm{1}(a\neq i_1) + 
e_{i_1a}\gamma_{i_1j_2}(\bz_1)\mathbbm{1}(j_2\neq i_1)
+
e_{i_2j_2}\gamma_{i_2a}(\bz_2)\mathbbm{1}(a\neq i_2) + 
e_{i_2a}\gamma_{i_2j_2}(\bz_2)\mathbbm{1}(j_2\neq i_2),
\\
\varpi_{i_1i_2j_3b}& = e_{i_1j_3}\gamma_{i_1b}(\bz_1)\mathbbm{1}(b\neq i_1) + 
e_{i_1b}\gamma_{i_1j_3}(\bz_1)\mathbbm{1}(j_3\neq i_1) + 
e_{i_2j_3}\gamma_{i_2b}(\bz_1)\mathbbm{1}(b\neq i_2) + 
e_{i_2b}\gamma_{i_2j_3}(\bz_1)\mathbbm{1}(j_3\neq i_2),
\\
\varpi_{i_1i_2j_4b}& = e_{i_1j_4}\gamma_{i_1b}(\bz_1)\mathbbm{1}(b\neq i_1) + 
e_{i_1b}\gamma_{i_1j_4}(\bz_1)\mathbbm{1}(j_4\neq i_1) + 
e_{i_2j_4}\gamma_{i_2b}(\bz_2)\mathbbm{1}(b\neq i_2) + 
e_{i_2b}\gamma_{i_2j_4}(\bz_2)\mathbbm{1}(j_4\neq i_2).
\end{align*}
The analyses of $\vartheta_{tj_1j_2j_3j_4}^{(1)}$, $\vartheta_{tj_1j_2j_3j_4}^{(2)}$, $\vartheta_{tj_1j_2j_3j_4}^{(3)}$, and $\vartheta_{tj_1j_2j_3j_4}^{(4)}$ are almost identical because $\vartheta_{tj_1j_2j_3j_4}^{(1)} = \vartheta_{tj_1j_2j_4j_3}^{(2)}$, $\vartheta_{tj_1j_2j_3j_4}^{(3)} = \vartheta_{tj_2j_1j_3j_4}^{(2)}$, and $\vartheta_{tj_1j_2j_3j_4}^{(4)} = \vartheta_{tj_2j_1j_4j_3}^{(2)}$. We only present the analysis of $\vartheta_{tj_1j_2j_3j_4}^{(2)}$ here. For convenience, we let $\gamma_{\max}$ denote $\max_{\|\bz\|_2\leq1,j_1,j_2\in[n]}|\gamma_{j_1j_2}(\bz)|$ and $\bar{e}_{i_1i_2j} = e_{i_1j} + e_{i_2j}$. Write
\begin{align*}
  \expect \vartheta_{tj_1j_2j_3j_4}^{(2)}
  & \leq \sum_{a_1,a_2,b_1,b_2\in[n]}|\expect (E_{ta_1j_2}E_{ta_2j_2}E_{tb_1j_3}E_{tb_2j_3})|(\bar{e}_{i_1i_2j_1} + \bar{e}_{i_1i_2a_1})(\bar{e}_{i_1i_2j_1} + \bar{e}_{i_1i_2a_2})\\
  &\quad\quad\times(\bar{e}_{i_1i_2j_4} + \bar{e}_{i_1i_2b_1})(\bar{e}_{i_1i_2j_4} + \bar{e}_{i_1i_2b_2})\mathbbm{1}(a_1,a_2 < j_1,b_1,b_2 < j_4)\gamma_{\max}^4.
  \end{align*}
  Note $\expect (E_{ta_1j_2}E_{ta_2j_2}E_{tb_1j_3}E_{tb_2j_3})\neq 0$ only if the number of random variables in $\{E_{ta_1j_2},E_{ta_2j_2},E_{tb_1j_3},E_{tb_2j_3}\}$ is $1$ or $2$. 
  \begin{enumerate}
    \item If this number is $1$, then one must have $a_1 = a_2$ and $b_1 = b_2$, so that either $j_2 = j_3$, $a_1 = a_2 = b_1 = b_2$ or $a_1 = a_2 = j_3$, $b_1 = b_2 = j_2$. 
    \item If this number is $2$, then there are two cases:
    \begin{itemize}
      \item If $j_2 \neq j_3$, then one of the following cases must occur: i) $a_1 = a_2$, $b_1 = b_2$; ii) $\{a_1,j_2\} = \{b_1,j_3\}$, $\{a_2, j_2\} = \{b_2,j_3\}$, implying that $a_1 = j_3,b_1 = j_2$, $b_2 = j_2$, $a_2 = j_3$; iii) $\{a_1,j_2\} = \{b_2,j_3\}$, $\{a_2,j_2\} = \{b_1,j_3\}$, implying that $a_1 = j_3,b_2 = j_2$, $a_2 = j_3,b_1 = j_2$. However, ii) and iii) occurs exactly when $E_{ta_1j_2} = E_{ta_2j_2} = E_{tb_1j_3} = E_{tb_2j_3} = E_{t_3j_2}$, so it is sufficient to only consider i).
      \item If $j_2 = j_3$, then one of the following three cases must occur: i) $a_1 = a_2\neq b_1 = b_2$; ii) $a_1 = b_1\neq a_2 = b_2$; iii) $a_1 = b_2\neq a_2 = b_1$.
    \end{itemize}
  \end{enumerate}
  Therefore, we are able to further upper bound $\expect \vartheta_{tj_1j_2j_3j_4}^{(2)}$ as follows:
  \begin{align*}
  \expect \vartheta_{tj_1j_2j_3j_4}^{(2)}
  &\lesssim \rho_n(\bar{e}_{i_1i_2j_1} + \bar{e}_{i_1i_2j_3})^2(\bar{e}_{i_1i_2j_4} + \bar{e}_{i_1i_2j_2})^2\gamma_{\max}^4
  \\&\quad
   + \sum_{a = 1}^n\rho_n(\bar{e}_{i_1i_2j_1} + \bar{e}_{i_1i_2a})^2(\bar{e}_{i_1i_2j_4} + \bar{e}_{i_1i_2a})^2\mathbbm{1}(j_2 = j_3)\gamma_{\max}^4
  \\&\quad
   + \sum_{a_1,b_1\in[n]}\rho_n^2(\bar{e}_{i_1i_2j_1} + \bar{e}_{i_1i_2a_1})^2(\bar{e}_{i_1i_2j_4} + \bar{e}_{i_1i_2b_1})^2\gamma_{\max}^4\\
  &\quad + \sum_{a_1,a_2\in[n]}\rho_n^2(\bar{e}_{i_1i_2j_1} + \bar{e}_{i_1i_2a_1})(\bar{e}_{i_1i_2j_1} + \bar{e}_{i_1i_2a_2})
  (\bar{e}_{i_1i_2j_4} + \bar{e}_{i_1i_2a_1})(\bar{e}_{i_1i_2j_4} + \bar{e}_{i_1i_2a_2})\gamma_{\max}^4\\
  &\lesssim (\bar{e}_{i_1i_2j_1}\bar{e}_{i_1i_2j_2} + \bar{e}_{i_1i_2j_3}\bar{e}_{i_1i_2j_2} + \bar{e}_{i_1i_2j_1}\bar{e}_{i_1i_2j_4} + \bar{e}_{i_1i_2j_3}\bar{e}_{i_1i_2j_4})\rho_n\gamma_{\max}^4\\
  &\quad + (n\bar{e}_{i_1i_2j_1}\bar{e}_{i_1i_2j_4} + \bar{e}_{i_1i_2j_1} + \bar{e}_{i_1i_2j_4} + 1)\mathbbm{1}(j_2 = j_3)\rho_n\gamma_{\max}^4
  \\&\quad
   + (n^2\bar{e}_{i_1i_2j_1}\bar{e}_{i_1i_2j_4} + n\bar{e}_{i_1i_2j_1} + n\bar{e}_{i_1i_2j_4} + 1)\rho_n^2\gamma_{\max}^4.
  \end{align*}
  This entails that $\sum_{t = 1}^m\sum_{j_1,j_2,j_3,j_4\in[n]}\expect \vartheta_{tj_1j_2j_3j_4}^{(2)}\sigma_{tj_1j_2}^2\sigma_{tj_3j_4}^2\lesssim (mn^3\rho_n^3 + mn^4\rho_n^4)\gamma_{\max}^4$. Therefore, 
  \[
  \sum_{t = 1}^m\sum_{j_1,j_2,j_3,j_4\in[n]}\left(\expect \vartheta_{tj_1j_2j_3j_4}^{(1)} + \expect \vartheta_{tj_1j_2j_3j_4}^{(3)} + \expect \vartheta_{tj_1j_2j_3j_4}^{(2)} + \expect \vartheta_{tj_1j_2j_3j_4}^{(4)}\right)\sigma_{tj_1j_2}^2\sigma_{tj_3j_4}^2\lesssim (mn^3\rho_n^3 + mn^4\rho_n^4)\gamma_{\max}^4
  \]
  because $\vartheta_{tj_1j_2j_3j_4}^{(1)} = \vartheta_{tj_1j_2j_4j_3}^{(2)}$, $\vartheta_{tj_1j_2j_3j_4}^{(3)} = \vartheta_{tj_2j_1j_3j_4}^{(2)}$, and $\vartheta_{tj_1j_2j_3j_4}^{(4)} = \vartheta_{tj_2j_1j_4j_3}^{(2)}$. The proof is thus completed. 
\end{proof}

\subsection{Proof of Theorem \ref{thm:Rowwise_CLT}}
\label{sub:proof_of_eigenvector_CLT}
This subsection completes the proof of Theorem \ref{thm:Rowwise_CLT}. We first argue that the remainder $\bT$ has uniformly negligible maximum row norms, and then invoke the martingale central limit theorem to establish the desired asymptotic normality. Note that Lemma \ref{lemma:Rowwise_CLT_remainder} is essentially the same as Lemma \ref{lemma:sharpened_remainder}. 
\begin{lemma}\label{lemma:Rowwise_CLT_remainder}
Suppose Assumptions \ref{assumption:eigenvector_delocalization}--\ref{assumption:condition_number} hold, $m^{1/2}(n\rho_n)^{3/2} = \omega(\theta_n(\log n)^2)$, $mn\rho_n = \omega(\theta_n(\log n)^2)$, and $m^{1/2}n\rho_n = \omega(\theta_n(\log n)^{3/2})$, where $\theta_n = (n\rho_n)^{1/2}\wedge 1$. Then
\begin{align*}
&\max_{i\in[n]}\frac{\|\be_i\transpose\{\calH(\bA\bA\transpose) - \bP\bP\transpose - \bM\}(\widehat{\bU}_{RS}^{(i)}\bH_{RS}^{(i)} - \bV)\|_2}{\lambda_d(\bP\bP\transpose)} = \optilde\bigg(\frac{1}{m^{1/2}n\rho_n^{1/2}\theta_n}\bigg) + \Optilde\bigg(\frac{1}{n^{S + 3/2}} + \frac{1}{n^{R + 1/2}}\bigg),\\
&\|\bT^{(RS)}\|_{2\to\infty} = \optilde\bigg(\frac{1}{m^{1/2}n\rho_n^{1/2}\theta_n}\bigg) + \Optilde\bigg(\frac{1}{n^{S + 3/2}} + \frac{1}{n^{R + 1/2}}\bigg).
\end{align*}
\end{lemma}

\begin{proof}
The proof is similar to that of Lemma \ref{lemma:LOO_Rowwise_concentration}, except that we take advantage of the error bound established in Theorem \ref{thm:entrywise_eigenvector_perturbation_bound} to obtain refinement. 
By Theorem \ref{thm:entrywise_eigenvector_perturbation_bound}, we immediately obtain
$\|\widehat{\bU}\bW - \bV\|_2 = \Optilde(\varepsilon_n^{(\mathsf{op})} + \varepsilon_n^{(\mathsf{bc})})$.
By Lemma \ref{lemma:Utilde_two_to_infinity_norm_II}, $\max_{i\in[n]}\|\widehat{\bU}^{(i)}\bH^{(i)} - \bV\|_2 = \Optilde(\varepsilon_n^{(\mathsf{op})} + \varepsilon_n^{(\mathsf{bc})})$, where $\widehat{\bU}^{(i)}$ denotes $\widehat{\bU}_{RS}^{(i)}$ and $\bW^{(i)} = \mathrm{sgn}((\widehat{\bU}^{(i)})\transpose\bV)$.
By triangle inequality, we have
\begin{align}\label{eqn:LOO_spectral_norm_error_bound}
\max_{i\in[n]}\frac{\|\be_i\transpose\{\calH(\bA\bA\transpose) - \bP\bP\transpose - \bM\}(\widehat{\bU}^{(i)}\bH^{(i)} - \bV)\|_2}{\lambda_d(\bP\bP\transpose)}\leq q_1 + q_2 + q_3 + q_4,
\end{align}
where $q_1,q_2,q_3,q_4$ are defined in \eqref{eqn:q1_q2_q3_q4}. Observe that
\begin{align*}
&\frac{q_n}{mn^{5/2}\rho_n^2}
 = \frac{1}{m^{1/2}n\rho_n^{1/2}\theta_n}\bigg\{\frac{\theta_n(\log n)^2}{m^{1/2}(n\rho_n)^{3/2}}\bigg\} = o\bigg(\frac{1}{m^{1/2}n\rho_n^{1/2}\theta_n}\bigg),\\
&\frac{\zeta_{\mathsf{op}}\log n}{m^{3/2}n^{7/2}\rho_n^3}
 = \frac{1}{m^{1/2}n\rho_n^{1/2}\theta_n}\bigg\{
\frac{\theta_n(\log n)^2}{m^{1/2}(n\rho_n)^{3/2}} + 
\frac{\theta_n(\log n)^{3/2}}{m^{1/2}n\rho_n}\bigg\} = o\bigg(\frac{1}{m^{1/2}n\rho_n^{1/2}\theta_n}\bigg).
\end{align*}
Then by Lemma \ref{lemma:Utilde_two_to_infinity_norm_II}, Lemma \ref{lemma:recursive_two_to_infinity_error_bound}, and Theorem \ref{thm:entrywise_eigenvector_perturbation_bound}, we have
\begin{align*}
q_1 & = \optilde\bigg(\frac{1}{m^{1/2}n\rho_n^{1/2}\theta_n}\bigg)
 + \Optilde\bigg(\frac{\log n}{m^{3/2}n^{7/2}\rho_n^3}\bigg)
\|\bM - \widehat{\bM}_{RS}\|_2
 = \optilde\bigg(\frac{1}{m^{1/2}n\rho_n^{1/2}\theta_n}\bigg) + \Optilde\bigg(\frac{1}{n^{S + 3/2}} + \frac{1}{n^{R + 1/2}}\bigg).
\end{align*}
For $q_2$, we apply the above spectral norm error bound and Lemma \ref{lemma:rowwise_concentration_linear_term} to obtain
\begin{align*}
q_2& = \frac{1}{m\Delta_n^2}\bigg\|\sum_{t = 1}^m\bP_t\bE_t\bigg\|_{2\to\infty}\max_{i\in[n]}\|\widehat{\bU}^{(i)}\bH^{(i)} - \bV\|_2
\leq \frac{1}{m\Delta_n^2}\|\bU\|_{2\to\infty}\bigg\|\sum_{t = 1}^m\bB_t\bU\transpose\bE_t\bigg\|_2\max_{i\in[n]}\|\widehat{\bU}^{(i)}\bH^{(i)} - \bV\|_2\\
& = \frac{(\mu d)^{1/2}}{mn^{1/2}\Delta_n^2}\bigg\|\sum_{t = 1}^m\bE_t\bU\bB_t\bigg\|_2\max_{i\in[n]}\|\widehat{\bU}^{(i)}\bH^{(i)} - \bV\|_2 = \Optilde\bigg\{\frac{(\log n)^{1/2}}{m^{1/2}n\rho_n^{1/2}}(\varepsilon_n^{(\mathsf{op})} + \varepsilon_n^{(\mathsf{bc})})\bigg\} = \optilde\bigg(\frac{1}{m^{1/2}n\rho_n^{1/2}\theta_n}\bigg).
\end{align*}
By Bernstein's inequality and \eqref{eqn:LOO_spectral_norm_error_bound}, we obtain similarly that
\begin{align*}
q_3&\leq \Optilde\bigg\{\frac{(\log n)^{1/2}}{m^{1/2}(n\rho_n)^{3/2}}\bigg\}
\|\bP\|_{2\to\infty}\max_{i\in[n]}\|\widehat{\bU}^{(i)}\bH^{(i)} - \bV\|_2 = \Optilde\bigg\{\frac{(\log n)^{1/2}}{m^{1/2}n\rho_n^{1/2}}(\varepsilon_n^{(\mathsf{op})} + \varepsilon_n^{(\mathsf{bc})})\bigg\}\\
& = \optilde\bigg(\frac{1}{m^{1/2}n\rho_n^{1/2}\theta_n}\bigg),\\
q_4&\leq \Optilde\bigg\{\frac{(\log n)^{1/2}}{m^{1/2}(n\rho_n)^{3/2}}\bigg\}\|\bP\|_{\max}\max_{i\in[n]}\|\widehat{\bU}^{(i)}\bH^{(i)} - \bV\|_2 = \Optilde\bigg\{\frac{(\log n)^{1/2}}{m^{1/2}n^{3/2}\rho_n^{1/2}}(\varepsilon_n^{(\mathsf{op})} + \varepsilon_n^{(\mathsf{bc})})\bigg\}\\
& = \optilde\bigg(\frac{1}{m^{1/2}n\rho_n^{1/2}\theta_n}\bigg).
\end{align*}
The proof of the first assertion is then completed by combining the above error bounds for $q_1$, $q_2$, $q_3$, and $q_4$. We next work with $\|\bT^{(RS)}\|_{2\to\infty}$. 
For $\bR_2^{(RS)}$, we immediately have 
\begin{align*}
\|\bR_2^{(RS)}\|_{2\to\infty}
&\leq \|\bM - \widehat{\bM}_{RS}\|_\infty\|\bV\|_{2\to\infty}\|\bS^{-1}\|_2
= \|\bM - \widehat{\bM}_{RS}\|_2\|\bV\|_{2\to\infty}\|\bS^{-1}\|_2\\
& = O\left(\frac{1}{mn^{5/2}\rho_n^2}\right)\|\bM - \widehat{\bM}_{RS}\|_2.
\end{align*}
For $\bR_3^{(RS)}$, by Lemma \ref{lemma:Utilde_two_to_infinity_norm_II}, we have a similar error bound:
\begin{align*}
\|\bR_3^{(RS)}\|_{2\to\infty}
&\leq \|\bM - \widehat{\bM}_{RS}\|_2(\|\widehat{\bU}\bW\|_{2\to\infty} + \|\bV\|_{2\to\infty})\|\widehat{\bS}_{RS}^{-1}\|_2 = \Optilde\left(\frac{1}{mn^{5/2}\rho_n^2}\right)\|\bM - \widehat{\bM}_{RS}\|_2.
\end{align*}
By Lemma \ref{lemma:remainder_II}, Lemma \ref{lemma:remainder_IV}, and \ref{lemma:remainder_III}, we obtain
\begin{align*}
\|\bR_4^{(RS)} + \bR_5^{(RS)} + \bR_6^{(RS)} + \bR_7^{(RS)}\|_{2\to\infty}& = \optilde\bigg(\frac{1}{m^{1/2}n\rho_n^{1/2}\theta_n}\bigg) + \Optilde\bigg(\frac{1}{mn^{5/2}\rho_n^2}\bigg)\|\bM - \widehat{\bM}_{RS}\|_2.
\end{align*}
Now we work with $\|\bR_1^{(RS)}\|_{2\to\infty}$. Recall in the proof of Theorem \ref{thm:entrywise_eigenvector_perturbation_bound},
$\|\bR_1^{(RS)}\|_{2\to\infty} = r_1^{(1)} + r_1^{(2)} + r_1^{(3)} + r_1^{(4)} + r_1^{(5)}$,
where $r_1^{(1)}$, $r_1^{(2)}$, $r_1^{(3)}$, $r_1^{(4)}$, and $r_1^{(5)}$ are defined in \eqref{eqn:r11_r12_r13_r14}. By Lemma \ref{lemma:spectral_norm_concentration_noise} and Lemma \ref{lemma:Utilde_two_to_infinity_norm_II}, we have
\begin{align*}
r_1^{(1)}\vee r_1^{(4)} & \leq 2\|\calH(\bA\bA\transpose) - \bP\bP\transpose - \bM\|_2\max_{i\in[n]}\|\widehat{\bU}\bH - \widehat{\bU}^{(i)}\bH^{(i)}\|_2\|\widehat{\bS}_{RS}^{-1}\|_2 = \Optilde\bigg(\frac{\zeta_{\mathsf{op}}^2}{m^2n^{9/2}\rho_n^4}\bigg)\\
& = \optilde\bigg(\frac{1}{m^{1/2}n\rho_n^{1/2}\theta_n}\bigg).
\end{align*}
For $r_1^{(2)}$ and $r_1^{(5)}$, instead of using Lemma \ref{lemma:Utilde_two_to_infinity_norm_II}, we invoke the refined error bound in the first assertion and obtain
\begin{align*}
r_1^{(2)}\vee r_1^{(5)} 
&\lesssim \max_{i\in[n]}\|\be_i\transpose\{\calH(\bA\bA\transpose) - \bP\bP\transpose - \bM)(\widehat{\bU}^{(i)}\bH^{(i)} - \bV)\|_2\|\widehat{\bS}_{RS}^{-1}\|_2\\
& = \optilde\bigg(\frac{1}{m^{1/2}n\rho_n^{1/2}\theta_n}\bigg) + \Optilde\bigg(\frac{1}{n^{S + 3/2}} + \frac{1}{n^{R + 1/2}}\bigg).
\end{align*}
For $r_1^{(3)}$, by Lemma \ref{lemma:rowwise_concentration_noise}, Lemma \ref{lemma:Utilde_two_to_infinity_norm_II}, and Lemma \ref{lemma:remainder_II}, we have
\begin{align*}
r_1^{(3)}&\leq 2\max_{i\in[n]}\|\be_i\transpose\{\calH(\bA\bA\transpose) - \bP\bP\transpose - \bM\}\bV\|_2\|\sin\Theta(\widehat{\bU}_{RS},\bV)\|_2^2\|\widehat{\bS}^{-1}_{RS}\|_2
\\&
 = \optilde\bigg(\frac{1}{m^{1/2}n\rho_n^{1/2}\theta_n}\bigg) + \Optilde\bigg(\frac{1}{mn^{5/2}\rho_n^2}\bigg)\|\bM - \widehat{\bM}_{RS}\|_2.
\end{align*}
Combining the error bounds for $r_1^{(1)}$ through $r_1^{(5)}$ yields
\begin{align*}
\|\bR_1^{(RS)}\|_{2\to\infty}
& = \optilde\bigg(\frac{1}{m^{1/2}n\rho_n^{1/2}\theta_n}\bigg) + \Optilde\bigg(\frac{1}{n^{S + 3/2}} + \frac{1}{n^{R + 1/2}}\bigg) + \Optilde\bigg(\frac{1}{mn^{5/2}\rho_n^2}\bigg)\|\bM - \widehat{\bM}_{RS}\|_2.
\end{align*}
Hence, we combine the error bounds for $\|\bR_1\|_{2\to\infty}$ through $\|\bR_7\|_{2\to\infty}$ to obtain
\begin{align*}
\|\bT^{(RS)}\|_{2\to\infty}
&\leq \optilde\bigg(\frac{1}{m^{1/2}n\rho_n^{1/2}\theta_n}\bigg) + \Optilde\bigg(\frac{1}{n^{S + 3/2}} + \frac{1}{n^{R + 1/2}}\bigg) + \Optilde\bigg(\frac{1}{mn^{5/2}\rho_n^2}\bigg)\|\bM - \widehat{\bM}_{RS}\|_2.
\end{align*}
By Lemma \ref{lemma:recursive_two_to_infinity_error_bound} and Theorem \ref{thm:entrywise_eigenvector_perturbation_bound}, we have
\begin{align*}
\Optilde\bigg(\frac{1}{mn^{5/2}\rho_n^2}\bigg)\|\widehat{\bM}_{RS} - \bM\|_2
& = \Optilde\bigg(\frac{\varepsilon_n^{(\mathsf{op})}}{n^{3/2}} + \frac{\varepsilon_n^{(\mathsf{bc})}}{n^{1/2}}\bigg) = \optilde\bigg(\frac{1}{m^{1/2}n\rho_n^{1/2}\theta_n}\bigg) + \Optilde\bigg(\frac{1}{n^{S + 3/2}} + \frac{1}{n^{R + 1/2}}\bigg).
\end{align*}
The proof is completed by combining the above error bounds.
\end{proof}

\begin{proof}[\bf Proof of Theorem \ref{thm:Rowwise_CLT}]
By Assumption \ref{assumption:condition_number} and the condition of Theorem \ref{thm:Rowwise_CLT}, we know that $\lambda_k(\bB\bB\transpose) = \lambda_k(\bS) = \Theta(mn^2\rho_n^2)$, $\lambda_k(\bF_i) = \Theta(mn^2\rho_n^3)$, $\lambda_k(\bG_i) = \Theta(mn\rho_n^2)$ for any $k\in[d]$, $i\in[n]$, so that $\lambda_k(\bGamma_i) = \Theta(1/(mn^2\rho_n\theta_n^2))$ and $\|\bGamma_i^{-1/2}\|_2 = \Omega(m^{1/2}n\rho_n^{1/2}\theta_n)$, where $\theta_n^2 = (n\rho_n)\wedge 1$.
By Lemma \ref{lemma:Rowwise_CLT_remainder}, we have 
\[
\be_i\transpose(\widehat{\bU}\bW - \bV)\bQ_1\transpose\bGamma_i^{-1/2} = \be_i\transpose\left\{\calH(\bA\bA\transpose) - \bP\bP\transpose - \bM\right\}\bV\bS^{-1}\bQ_1\transpose\bGamma_i^{-1/2} + \optilde(1).
\]
It is sufficient to show
$\be_i\transpose\left\{\calH(\bA\bA\transpose) - \bP\bP\transpose - \bM\right\}\bV\bS^{-1}\bQ_1\transpose\bGamma_i^{-1/2}\bz\overset{\calL}{\to}\mathrm{N}(0, 1)$.
By definition, we have
\begin{align*}
&\be_i\transpose\left\{\calH(\bA\bA\transpose) - \bP\bP\transpose - \bM\right\}\bV\bS^{-1}\bQ_1\transpose\bGamma_i^{-1/2}\bz\\
&\quad = \be_i\transpose\calH(\bE\bE\transpose)\bV\bS^{-1}\bQ_1\transpose\bGamma_i^{-1/2}\bz + \sum_{t = 1}^m\sum_{j = 1}^nE_{tij}(\be_j\transpose\bP_t\bV\bS^{-1}\bQ_1\transpose\bGamma_i^{-1/2}\bz)\\
&\quad\quad + \sum_{t = 1}^m\be_i\transpose\bP_t\bE_t\bV\bS^{-1}\bQ_1\transpose\bGamma_i^{-1/2}\bz - 2\sum_{t = 1}^m\sum_{j = 1}^nP_{tij}E_{tij}(\be_i\transpose\bV\bS^{-1}\bQ_1\transpose\bGamma_i^{-1/2}\bz).
\end{align*}
By Lemma \ref{lemma:quadratic_form_Et}, the third term satisfies
\begin{align*}
\bigg|\sum_{t = 1}^m\be_i\transpose\bP_t\bE_t\bV\bS^{-1}\bQ_1\transpose\bGamma_i^{-1/2}\bz\bigg|&\leq \|\bU\|_{2\to\infty}\bigg\|\sum_{t = 1}^m\bB_t\bU\transpose\bE_t\bV\bS^{-1}\bQ_1\transpose\bGamma_i^{-1/2}\bz\bigg\|_2
 = \Optilde\left\{\frac{(\log n)^{1/2}}{n^{1/2}}\right\} = \optilde(1).
\end{align*}
By Bernstein's inequality, 
\[
2\sum_{t = 1}^m\sum_{j = 1}^nP_{tij}E_{tij}(\be_i\transpose\bV\bS^{-1}\bQ_1\transpose\bGamma_i^{-1/2}\bz) = \optilde\bigg\{\frac{(\log n)^{1/2}}{n}\bigg\} = \optilde(1)
.
\] 
These two concentration results imply
$\be_i\transpose\left\{\calH(\bA\bA\transpose) - \bP\bP\transpose - \bM\right\}\bV\bS^{-1}\bQ_1\transpose\bGamma_i^{-1/2}\bz = X_{in}(\bz) + \optilde(1)$,
where $X_{in}(\bz)$ is defined in \eqref{eqn:CLT_dominating_term}. As introduced in Lemma \ref{lemma:martingale_construction}, we rewrite $X_{in}$ as a mean-zero martingale and apply the martingale central limit theorem, which we review here (see, for example, Theorem 35.12 in \cite{billingsley1995probability}). Suppose that for each $n$, $(X_{n\alpha})_{\alpha = 0}^{N_n}$ is a martingale with respect to the filtration $(\calF_{n\alpha})_{\alpha = 0}^{N_n} = (\sigma(X_{n0},\ldots,X_{n\alpha}))_{\alpha = 0}^{N_n}$ with $X_{n0} = 0$, $\calF_{n0} = \{\varnothing, \Omega\}$, and  martingale difference sequence $(Y_{n\alpha})_{\alpha = 1}^{N_n} = (X_{n\alpha} - X_{n(\alpha - 1)})_{\alpha = 1}^{N_n}$. If the following conditions are satisfied:
\begin{enumerate}
   \item[(a)] $\sum_{\alpha = 1}^{N_n}\expect (Y_{n\alpha}^2\mid F_{n(\alpha - 1)})\overset{\prob}{\to} \sigma^2$ for some constant $\sigma^2 > 0$;
   \item[(b)] $\sum_{\alpha = 1}^{N_n}\expect \{Y_{n\alpha}^2\mathbbm{1}(|Y_{n\alpha}| \geq \eps)\}\to 0$ for any $\eps > 0$.
\end{enumerate} 
Then $\sum_{\alpha = 1}^{N_n}Y_{n\alpha}\overset{\calL}{\to}\mathrm{N}(0, \sigma^2)$. Specialized to our setup in Lemma \ref{lemma:martingale_construction}, given the one-to-one relabeling function $\alpha = \alpha(t, j_1, j_2)$ defined in \eqref{eqn:relabeling_function}, we will verify the above conditions with $N_n = (1/2)mn(n + 1)$, $Y_{n\alpha} = E_{tj_1j_2}(b_{tij_1j_2} + c_{tij_1j_2})$, $\calF_{n\alpha} = \sigma(\{E_{tj_1j_2}:\alpha(t, j_1, j_2)\leq \alpha\})$, $Y_{n0} = 0$, and $X_{in}(\bz) = \sum_{\alpha = 1}^{N_n}Y_{n\alpha}$, where $b_{tij_1j_2} = b_{tij_1j_2}(\bz)$ is defined in \eqref{eqn:btj1j2} and $c_{tij_1j_2} = c_{tij_1j_2}(\bz)$ is defined in \eqref{eqn:ctj1j2}. In particular, we have $\max_{j\in[n]}|\gamma_{ij}| = O\{\theta_n/(m^{1/2}n^{3/2}\rho_n^{3/2})\}$ and $\max_{t\in[m],j\in[n]}|\xi_{tij}| = O\{\theta_n/(mn\rho_n)^{1/2}\}$.

\noindent $\blacksquare$ \textbf{Condition (a) for martingale central limit theorem.}
For any $\alpha\in[N_n]$, let $t(\alpha)\in[m],j_1(\alpha),j_2(\alpha)$, $j_1(\alpha)\leq j_2(\alpha)$ be the unique indices such that $\alpha(t(\alpha),j_1(\alpha),j_2(\alpha)) = \alpha$. This enables us to write
\begin{equation}
\label{eqn:martingale_CLT_condition_a}
\begin{aligned}
&\sum_{\alpha = 1}^{N_n}\{\expect(Y_{n\alpha}^2\mid\calF_{n(\alpha - 1)}) - \expect(Y_{n\alpha}^2)\} = \sum_{t = 1}^m\sum_{j_1\leq j_2}(b_{tij_1j_2}^2 - \expect b_{tij_1j_2}^2)\sigma_{tj_1j_2}^2 + 2\sum_{t = 1}^m\sum_{j_1\leq j_2}b_{tij_1j_2}c_{tij_1j_2}\sigma_{tj_1j_2}^2.
\end{aligned}
\end{equation}
We work with the two terms separately. For the second term in \eqref{eqn:martingale_CLT_condition_a}, we have
\begin{align*}
&2\sum_{t = 1}^m\sum_{j_1,j_2\in[n]:j_1\leq j_2}b_{tij_1j_2}c_{tij_1j_2}\sigma_{tj_1j_2}^2\\
&\quad = 2\sum_{t = 1}^m\sum_{j_1,j_2\in[n]:j_1 \leq j_2}\iota_{j_1j_2}\sum_{j_3 = 1}^{j_1 - 1}E_{tj_2j_3}\{e_{ij_1}\gamma_{ij_3}\mathbbm{1}(j_3\neq i) + e_{ij_3}\gamma_{ij_1}\mathbbm{1}(j_1\neq i)\}c_{tij_1j_2}\sigma_{tj_1j_2}^2\\
&\quad\quad + 2\sum_{t = 1}^m\sum_{j_1,j_2\in[n]:j_1 \leq j_2}\iota_{j_1j_2}\sum_{j_3 = 1}^{j_2 - 1}E_{tj_1j_3}\{e_{ij_2}\gamma_{ij_3}\mathbbm{1}(j_3\neq i) + e_{ij_3}\gamma_{ij_2}\mathbbm{1}(j_2\neq i)\}c_{tij_1j_2}\sigma_{tj_1j_2}^2.
\end{align*}
The two terms above are sums of independent mean-zero random variables with variances $O(\theta_n^4(mn^2\rho_n)^{-1})$.
Therefore, the second term in \eqref{eqn:martingale_CLT_condition_a} converges to $0$ in probability by Chebyshev's inequality. For the first term \eqref{eqn:martingale_CLT_condition_a}, we compute the second moment
\begin{align*}
&\expect\bigg\{\sum_{t = 1}^m\sum_{j_1,j_2\in[n]:j_1\leq j_2}(b_{tij_1j_2}^2 - \expect b_{tij_1j_2}^2)\sigma_{tj_1j_2}^2\bigg\}^2 \leq \sum_{t = 1}^m\sum_{\substack{j_1,j_2,j_3,j_4\in[n],j_1\leq j_2, j_3\leq j_4}}\expect(b_{tij_1j_2}^2b_{tij_3j_4}^2)\sigma_{tj_1j_2}^2\sigma_{tj_3j_4}^2\\
&\quad = 2\sum_{t = 1}^m\sum_{\substack{j_1\leq j_2, j_3\leq j_4,\alpha(t, j_1, j_2) < \alpha(t, j_3, j_4)}}\expect(b_{tij_1j_2}^2b_{tij_3j_4}^2)\sigma_{tj_1j_2}^2\sigma_{tj_3j_4}^2 + \sum_{t = 1}^m\sum_{j_1,j_2\in[n],j_1\leq j_2}\expect(b_{tij_1j_2}^4)\sigma_{t_1j_1j_2}^4.
\end{align*}
By Lemma \ref{lemma:btj1j2_cross_fourth_moment} with $i_1 = i_2 = i$, the first term above converges to $0$. 
By Lemma \ref{lemma:btj1j2_fourth_moment}, the second term above also converges to $0$.
This shows that the second moment of the first term in \eqref{eqn:martingale_CLT_condition_a} goes to $0$, and hence, the first term in \eqref{eqn:martingale_CLT_condition_a} is $o_p(1)$ by Chebyshev's inequality. Therefore, we have shown that \eqref{eqn:martingale_CLT_condition_a} converges to $0$ in probability.
Next, we show that $\sum_{\alpha = 1}^{N_n}\expect Y_{n\alpha}^2\to \|\bz\|_2^2$ for any deterministic vector $\bz\in\mathbb{R}^d$. To this end, first observe that for any $j_1,j_2\in[n]$, $j_1 < j_2$, $(E_{tj_2j_3})_{j_3 = 1}^{j_1 - 1}$ and $(E_{tj_1j_3})_{j_3 = 1}^{j_2 - 1}$ form a collection of independent mean-zero random variables, in which case
\begin{align*}
\expect b_{tij_1j_2}^2 &= \sum_{j_3 = 1}^{j_1 - 1}\sigma_{tj_2j_3}^2\{e_{ij_1}\gamma_{ij_3}\mathbbm{1}(j_3\neq i) + e_{ij_3}\gamma_{ij_1}\mathbbm{1}(j_1\neq i)\}^2
 + \sum_{j_3 = 1}^{j_2 - 1}\sigma_{tj_1j_3}^2\{e_{ij_2}\gamma_{ij_3}\mathbbm{1}(j_3\neq i) + e_{ij_3}\gamma_{ij_2}\mathbbm{1}(j_2\neq i)\}^2,
\end{align*}
and for the case where $j_1 = j_2$, we also have
$\expect b_{tij_1j_1}^2 = \sum_{j_3 = 1}^{j_1 - 1}\sigma_{tj_1j_3}^2\{e_{ij_1}\gamma_{ij_3}\mathbbm{1}(j_3\neq i) + e_{ij_3}\gamma_{ij_1}\mathbbm{1}(j_1\neq i)\}^2$.
Therefore, we proceed to compute
\begin{align*}
\sum_{\alpha = 1}^{N_n}\expect Y_{n\alpha}^2
& = \sum_{t = 1}^m\sum_{j_1 = 1}^{n - 1}\sum_{j_2 = j_1 + 1}^n\sum_{j_3 = 1}^{j_1 - 1}\sigma_{tj_1j_2}^2\sigma_{tj_2j_3}^2\{e_{ij_1}\gamma_{ij_3}\mathbbm{1}(j_3\neq i) + e_{ij_3}\gamma_{ij_1}\mathbbm{1}(j_1\neq i)\}^2\\
&\quad + \sum_{t = 1}^m\sum_{j_1 = 1}^{n - 1}\sum_{j_2 = j_1 + 1}^n\sum_{j_3 = 1}^{j_2 - 1}\sigma_{tj_1j_2}^2\sigma_{tj_1j_3}^2\{e_{ij_2}\gamma_{ij_3}\mathbbm{1}(j_3\neq i) + e_{ij_3}\gamma_{ij_2}\mathbbm{1}(j_2\neq i)\}^2\\
&\quad + \sum_{t = 1}^m\sum_{j_1 = 1}^n\sum_{j_3 = 1}^{j_1 - 1}\sigma_{tj_1j_1}^2\sigma_{tj_1j_3}^2\{e_{ij_1}\gamma_{ij_3}\mathbbm{1}(j_3\neq i) + e_{ij_3}\gamma_{ij_1}\mathbbm{1}(j_1\neq i)\}^2\\
&\quad + \sum_{t = 1}^m\sum_{j_1 = 1}^{n - 1}\sum_{j_2 = j_1 + 1}^n\sigma_{tj_1j_2}^2(e_{ij_1}\xi_{tij_2} + e_{ij_2}\xi_{tij_1})^2 + \sum_{t = 1}^m\sum_{j_1 = 1}^n\sigma_{tj_1j_1}^2e_{ij_1}\xi_{tij_1}^2
\end{align*}
Using the fact that $e_{ij_1}e_{ij_2} = 0$ when $j_1 < j_2$, switching the roles between $j_1$ and $j_2$ in the second summation, and rearranging, we further obtain
\begin{align*}
\sum_{\alpha = 1}^{N_n}\expect Y_{n\alpha}^2
& = \sum_{t = 1}^m\sum_{j_2 = 1}^n\sum_{j_3 = 1}^n\sigma_{tij_2}^2\sigma_{tj_2j_3}^2\gamma_{ij_3}^2 + \sum_{t = 1}^m\sum_{j_2 = 1}^n\sigma_{tij_2}^2\xi_{tij_2}^2 + O\left(\frac{\theta_n^2}{n^2\rho_n}\right)\to\|\bz\|_2^2,
\end{align*}
thereby establishing condition (a) for martingale central limit theorem.

\noindent
\textbf{Condition (b) for martingale central limit theorem}. It is sufficient to show that $\sum_{\alpha = 1}^{N_n}\expect Y_{n\alpha}^4\to 0$. By definition, we have
\begin{align*}
\sum_{\alpha = 1}^{N_n}\expect Y_{n\alpha}^4
& = \sum_{t = 1}^m\sum_{j_1\leq j_2}(\expect E_{tj_1j_2}^4)(\expect b_{tij_1j_2}^4 + 4\expect b_{tij_1j_2}^3c_{tij_1j_2} + 6\expect b_{tij_1j_2}^2c_{tij_1j_2}^2 + c_{tij_1j_2}^4)\\
&\lesssim \rho_n\sum_{t = 1}^m\sum_{j_1\leq j_2}(\expect b_{tij_1j_2}^4 + c_{tij_1j_2}^4).
\end{align*}
It follows from Lemma \ref{lemma:btj1j2_fourth_moment} that $\sum_{\alpha = 1}^{N_n}\expect Y_{n\alpha}^4 = O\{\theta_n^4(mn^4\rho_n^4)^{-1} + \theta_n^4(mn^3\rho_n^3)^{-1} + \theta_n^4(mn\rho_n)^{-1}\}\to 0$. 
Therefore, condition (b) for martingale central limit theorem also holds, so that
$X_{in}(\bz)\overset{\calL}{\to}\mathrm{N}(0, 1)$. 
The proof is thereby completed.
\end{proof}

\section{Proof of Theorem \ref{thm:testing_membership_profiles}}
\label{sec:proof_of_testing_membership_profiles}
Let $\bU = \bZ(\bZ\transpose\bZ)^{-1/2}$ and $\bB_t = (\bZ\transpose\bZ)^{1/2}\bC_t(\bZ\transpose\bZ)^{1/2}$. By Rayleigh-Ritz theorem, 
\begin{align*}
\min_{t\in[m]}\sigma_d(\bB_t)& = \min_{t\in[m]}\lambda_d^{1/2}(\bB_t^2)
 = \min_{t\in[m]}\min_{\|\bx\|_2 = 1}(\bx\transpose\bB_t^2\bx)^{1/2}\\
 &\geq \min_{t\in[m]}\min_{\|\bx\|_2 = 1}\|(\bZ\transpose\bZ)^{1/2}\bx\|_2\lambda_d^{1/2}(\bZ\transpose\bZ)\lambda_d(\bC_t) = \Omega(n\rho_n).
\end{align*}
Clearly, $\max_{t\in[m]}\sigma_1(\bB_t)/\sigma_d(\bB_t) = O(1)$ and $\|\bU\|_{2\to\infty}\leq \|\bZ\|_{2\to\infty}\lambda_d^{-1/2}(\bZ\transpose\bZ) = O(n^{-1/2})$. Given the conditions of Theorem \ref{thm:testing_membership_profiles}, it is straightforward to obtain $\lambda_k(\bF_i) = \Theta(mn^2\rho_n^3)$ and $\lambda_k(\bG_i) = \Theta(mn\rho_n^2)$.
Hence, the conditions of Theorem \ref{thm:Rowwise_CLT} hold, implying that for any fixed vector $\bz\in\mathbb{R}^d$,
$\bz\transpose\bGamma_i^{-1/2}(\bQ_1\bW\transpose\widehat{\btheta}_i - \btheta_i) = X_{in}(\bz) + \optilde(1)$,
where $\widehat{\btheta}_i$ and $\btheta_i$ denote the $i$th row of $\widehat{\bU}$ and $\bU$, respectively, $X_{in}(\bz)$ is defined in \eqref{eqn:CLT_dominating_term}, and $\bW = \mathrm{sgn}(\widehat{\bU}\transpose\bV)$. 

\noindent
$\blacksquare$ \textbf{Asymptotic normality of $\widehat{\btheta}_{i_1} - \widehat{\btheta}_{i_2}$.} This part is similar to the proof of Theorem \ref{thm:Rowwise_CLT}. Recall the notations  in Lemma \ref{lemma:martingale_construction}. For any deterministic vectors $\bz_1,\bz_2\in\mathbb{R}^d$, let $v_{i_1i_2n} = X_{i_1n}(\bz_1) + X_{i_2n}(\bz_2)$. Note that 
\[
v_{i_1i_2n}
 = \sum_{t = 1}^m\sum_{\substack{j_1 \leq j_2}}E_{tj_1j_2}\{b_{ti_1j_1j_2}(\bz_1) + b_{ti_2j_1j_2}(\bz_2) + c_{ti_1j_1j_2}(\bz_1) + c_{ti_2j_1j_2}(\bz_2)\}
\]
and $b_{ti_1j_1j_2}(\bz_1) + b_{ti_2j_1j_2}(\bz_2) + c_{ti_1j_1j_2}(\bz_1) + c_{ti_2j_1j_2}(\bz_2)$ is $\calF_{n(\alpha(t, j_1, j_2) - 1)}$ measurable. For any $\alpha\in N_n = (1/2)mn(n + 1)$, let $(t(\alpha),j_1(\alpha),j_2(\alpha))\in[m]\times[n]\times[n]$ be the unique indices such that $\alpha(t(\alpha),j_1(\alpha),j_2(\alpha)) = \alpha$. Let 
 \[
 Y_{n\alpha} = E_{t(\alpha)j_1(\alpha)j_2(\alpha)}\{b_{t(\alpha)i_1j_1(\alpha)j_2(\alpha)}(\bz_1) + b_{t(\alpha)i_2j_1(\alpha)j_2(\alpha)}(\bz_2) + c_{t(\alpha)i_1j_1(\alpha)j_2(\alpha)}(\bz_1) + c_{t(\alpha)i_2j_1(\alpha)j_2(\alpha)}(\bz_2)\}
 \]
 and $\calF_{n\alpha} = \sigma(\{E_{tj_1j_2}:t\in[m],j_1,j_2\in[n],j_1\leq j_2,\alpha(t, j_1, j_2)\leq \alpha\})$, where $\alpha(\cdot)$ is the relabeling function defined in \eqref{eqn:relabeling_function}. Set $\calF_{n0} = \varnothing$. Clearly, $v_{i_1i_2n} = \sum_{\alpha = 1}^{N_n}Y_{n\alpha}$, where $(Y_{n\alpha})_{\alpha = 1}^{N_n}$ forms a martingale difference sequence with regard to the filtration $(\calF_{n\alpha})_{\alpha = 0}^{N_n - 1}$. We now argue the asymptotic normality of $v_{i_1i_2n}$ by applying the martingale central limit theorem.
 We first verify condition (a) for martingale central limit theorem.
Write
\begin{equation}
\label{eqn:martingale_CLT_condition_a_membership_profiles_testing}
\begin{aligned}
&\sum_{\alpha = 1}^{N_n}\expect(Y_{n\alpha}^2\mid\calF_{n(\alpha - 1)}) - \sum_{\alpha = 1}^{N_n}\expect(Y_{n\alpha}^2)\\
&\quad = \sum_{t = 1}^m\sum_{j_1,j_2\in[n]:j_1\leq j_2}
\left(\bar{b}_{tj_1j_2}^2 - \expect\bar{b}_{tj_1j_2}^2\right)
\sigma_{tj_1j_2}^2
\\&\quad\quad
 + 2\sum_{i,i'\in\{i_1,i_2\}}\sum_{\bz,\bz'\in\{\bz_1,\bz_2\}}\sum_{t = 1}^m\sum_{j_1,j_2\in[n]:j_1\leq j_2}b_{tij_1j_2}(\bz)c_{ti'j_1j_2}(\bz')\sigma_{tj_1j_2}^2,
\end{aligned}
\end{equation}
where $\bar{b}_{tj_1j_2} := b_{ti_1j_1j_2}(\bz_1) + b_{ti_2j_1j_2}(\bz_2)$.
We work with the two terms on the right-hand side of \eqref{eqn:martingale_CLT_condition_a_membership_profiles_testing} separately. For the second term in \eqref{eqn:martingale_CLT_condition_a_membership_profiles_testing}, for any $i,i'\in\{i_1,i_2\}$ and any $\bz,\bz'\in\{\bz_1,\bz_2\}$, we have
\begin{align*}
&2\sum_{t = 1}^m\sum_{j_1,j_2\in[n]:j_1\leq j_2}b_{tij_1j_2}(\bz)c_{ti'j_1j_2}(\bz')\sigma_{tj_1j_2}^2\\
&\quad = 2\sum_{t = 1}^m\sum_{j_1,j_2\in[n]:j_1 \leq j_2}\sum_{j_3 = 1}^{j_1 - 1}E_{tj_2j_3}\{e_{ij_1}\gamma_{ij_3}(\bz)\mathbbm{1}(j_3\neq i) + e_{ij_3}\gamma_{ij_1}(\bz)\mathbbm{1}(j_1\neq i)\}c_{ti'j_1j_2}(\bz')\sigma_{tj_1j_2}^2\\
&\quad\quad + 2\sum_{t = 1}^m\sum_{j_1,j_2\in[n]:j_1 < j_2}\sum_{j_3 = 1}^{j_2 - 1}E_{tj_1j_3}\{e_{ij_2}\gamma_{ij_3}(\bz)\mathbbm{1}(j_3\neq i) + e_{ij_3}\gamma_{ij_2}(\bz)\mathbbm{1}(j_2\neq i)\}c_{ti'j_1j_2}(\bz')\sigma_{tj_1j_2}^2.
\end{align*}
The two terms above are sums of independent mean-zero random variables with variances $O(\theta_n^4(mn^2\rho_n)^{-1})$. Therefore, the second term in \eqref{eqn:martingale_CLT_condition_a_membership_profiles_testing} is $o_p(1)$ by Chebyshev's inequality. For the first on the right-hand side of \eqref{eqn:martingale_CLT_condition_a_membership_profiles_testing}, we compute the second moment
\begin{align*}
&\expect\bigg\{\sum_{t = 1}^m\sum_{j_1,j_2\in[n]:j_1\leq j_2}\left(\bar{b}_{tj_1j_2}^2 - \expect \bar{b}_{tj_1j_2}^2\right)\sigma_{tj_1j_2}^2\bigg\}^2
\leq \sum_{t = 1}^m\sum_{\substack{j_1,j_2,j_3,j_4\in[n],j_1\leq j_2, j_3\leq j_4}}\expect\bar{b}_{tj_1j_2}^2\bar{b}_{tj_3j_4}^2\sigma_{tj_1j_2}^2\sigma_{tj_3j_4}^2\\
&\quad = 2\sum_{t = 1}^m\sum_{\substack{j_1\leq j_2, j_3\leq j_4,\alpha(t, j_1, j_2) < \alpha(t, j_3, j_4)}}\expect\bar{b}_{tj_1j_2}^2\bar{b}_{tj_3j_4}^2\sigma_{tj_1j_2}^2\sigma_{tj_3j_4}^2 + \sum_{t = 1}^m\sum_{j_1,j_2\in[n],j_1\leq j_2}\expect\bar{b}_{tj_1j_2}^4\sigma_{t_1j_1j_2}^4.
\end{align*}
By Lemma \ref{lemma:btj1j2_cross_fourth_moment} and Lemma \ref{lemma:btj1j2_fourth_moment}, the two terms above converge to $0$. 
This shows that the second moments of the first and second terms in \eqref{eqn:martingale_CLT_condition_a_membership_profiles_testing} go to $0$, and hence, the first and second terms on the right-hand side of \eqref{eqn:martingale_CLT_condition_a_membership_profiles_testing} is $o_p(1)$ by Chebyshev's inequality. Therefore, we have shown that \eqref{eqn:martingale_CLT_condition_a_membership_profiles_testing} is $o_p(1)$.
Next, we show that $\sum_{\alpha = 1}^{N_n}\expect Y_{n\alpha}^2\to \|\bz_1\|_2^2 + \|\bz_2\|_2^2$ for any deterministic vectors $\bz_1,\bz_2\in\mathbb{R}^d$. To this end, denote by $\gamma_{\max} = \max_{i,j\in[n]}\{|\gamma_{ij}(\bz_1)|\vee|\gamma_{ij}(\bz_2)|\}$, $\xi_{\max} = \max_{t\in[m],i,j\in[n]}\{|\xi_{tij}(\bz_1)|\vee|\xi_{tij}(\bz_2)|\}$, and observe that 
\begin{align*}
&\sum_{t = 1}^m\sum_{j_1\leq j_2}\sigma_{tj_1j_2}^2\expect\{b_{ti_1j_1j_2}(\bz_1)b_{ti_2j_1j_2}(\bz_2)\}\\
&\quad\leq \sum_{t = 1}^m\sum_{j_1\leq j_2}\sum_{a,b\in[n]}\{|\expect E_{tj_2a}E_{tj_2b}|(e_{i_1j_1} + e_{i_1a})(e_{i_2j_1} + e_{i_2b})
 + |\expect E_{tj_2a}E_{tj_1b}|(e_{i_1j_1} + e_{i_1a})(e_{i_2j_2} + e_{i_2b})\}\rho_n\gamma_{\max}^2\\
&\quad\quad + \sum_{t = 1}^m\sum_{j_1\leq j_2}\sum_{a,b\in[n]}\{|\expect E_{tj_1a}E_{tj_2b}|(e_{i_1j_2} + e_{i_1a})(e_{i_2j_1} + e_{i_2b}) + |\expect E_{tj_1a}E_{tj_1b}|(e_{i_1j_2} + e_{i_1a})(e_{i_2j_2} + e_{i_2b})\}\rho_n\gamma_{\max}^2\\
&\quad\lesssim \sum_{t = 1}^m\sum_{j_1,j_2\in[n]}\frac{(e_{i_1j_1} + e_{i_1j_2} + e_{i_2j_1} + e_{i_2j_2})\theta_n^2}{mn^3\rho_n^2} = O\bigg(\frac{\theta_n^2}{n^2\rho_n}\bigg),\\
&\sum_{t = 1}^m\sum_{j_1\leq j_2}\sigma_{tj_1j_2}^2c_{ti_1j_1j_2}(\bz_1)c_{ti_2j_1j_2}(\bz_2)
\leq \sum_{t = 1}^m\sum_{j_1\leq j_2}\rho_n(e_{i_1j_1} + e_{i_1j_2})(e_{i_2j_1} + e_{i_2j_2})\xi_{\max}^2 = O\bigg(\frac{\theta_n^2}{n}\bigg),
\end{align*}
where we have used the fact that $e_{i_1j}e_{i_2j} = 0$ for any $j\in[n]$ since $i_1\neq i_2$. 
Therefore, by the proof of Theorem \ref{thm:Rowwise_CLT}, we obtain
\begin{align*}
\sum_{\alpha = 1}^{N_n}\expect Y_{n\alpha}^2
& = \sum_{t = 1}^m\sum_{j_1\leq j_2}\sigma_{tj_1j_2}^2\{\expect b_{ti_1j_1j_2}^2(\bz_1) + c_{ti_1j_1j_2}^2(\bz_1)\}
 + \sum_{t = 1}^m\sum_{j_1\leq j_2}\sigma_{tj_1j_2}^2\{\expect b_{ti_2j_1j_2}^2(\bz_2) + c_{ti_2j_1j_2}^2(\bz_2)\}\\
&\quad + 2\sum_{t = 1}^m\sum_{j_1\leq j_2}\sigma_{tj_1j_2}^2\expect\{b_{ti_1j_1j_2}(\bz_1)b_{ti_2j_1j_2}(\bz_2)\} + 2\sum_{t = 1}^m\sum_{j_1\leq j_2}\sigma_{tj_1j_2}^2c_{ti_1j_1j_2}(\bz_1)c_{ti_2j_1j_2}(\bz_2)\}\\
&
 = \|\bz_1\|_2^2 + \|\bz_2\|_2^2 + o(1).
\end{align*}
This establishes condition (a) for martingale central limit theorem with $\sigma^2 = \|\bz_1\|_2^2 + \|\bz_2\|_2^2$. For condition (b), it is sufficient to show that $\sum_{\alpha = 1}^{N_n}\expect Y_{n\alpha}^4\to 0$. By definition, we have
\begin{align*}
\sum_{\alpha = 1}^{N_n}\expect Y_{n\alpha}^4
& = \sum_{t = 1}^m\sum_{j_1,j_2\in[n]:j_1\leq j_2}(\expect E_{tj_1j_2}^4)\expect(\bar{b}_{tj_1j_2} + \bar{c}_{tij_1j_2})^4\\
&\lesssim \rho_n\sum_{t = 1}^m\sum_{j_1,j_2\in[n]:j_1\leq j_2}\{\expect b_{ti_1j_1j_2}^4(\bz_1) + \expect b_{ti_2j_1j_2}^4(\bz_2) + c_{ti_1j_1j_2}^4(\bz_1) + c_{ti_2j_1j_2}^4(\bz_2)\}.
\end{align*}
It follows from Lemma \ref{lemma:btj1j2_fourth_moment} that $\sum_{\alpha = 1}^{N_n}\expect Y_{n\alpha}^4 = O(\theta_n^4(mn^4\rho_n^4)^{-1} + \theta_n^4(mn^3\rho_n^3)^{-1} + (mn\rho_n)^{-1})\to 0$. Therefore, condition (b) for martingale central limit theorem also holds, so that
$v_{i_1i_2n}\overset{\calL}{\to}\mathrm{N}(0, \|\bz_1\|_2^2 + \|\bz_2\|_2^2)$. By Cram\'er-Wold device and Theorem \ref{thm:Rowwise_CLT}, this further implies that $[\bGamma_{i_1}^{-1/2}\{\be_{i_1}\transpose(\widehat{\bU}\bW - \bV)\bQ_1\transpose\}\transpose,
\bGamma_{i_2}^{-1/2}\{\be_{i_2}\transpose(\widehat{\bU}\bW - \bV)\bQ_1\transpose\}\transpose]\overset{\calL}{\to}\mathrm{N}_{2d}(\zero_{2d}, \eye_{2d})$.
Namely, under the null hypothesis $H_0$, we have $(\bGamma_{i_1} + \bGamma_{i_2})^{-1/2}\bQ_1\bW\transpose(\widehat{\btheta}_{i_1} - \widehat{\btheta}_{i_2})\overset{\calL}{\to}\mathrm{N}_{d}(\zero_{d}, \eye_{d})$ and $\|\widehat{\btheta}_{i_1} - \widehat{\btheta}_{i_2}\|_2 = O_p((m^{1/2}n\rho_n^{1/2}\theta_n)^{-1})$. Similarly, under $H_A$, we have $(\bGamma_{i_1} + \bGamma_{i_2})^{-1/2}\{\bQ_1\bW\transpose(\widehat{\btheta}_{i_1} - \widehat{\btheta}_{i_2}) - (\btheta_{i_1} - \btheta_{i_2})\}\overset{\calL}{\to}\mathrm{N}_{d}(\zero_{d}, \eye_{d})$ and $\|\bQ_1\bW\transpose(\widehat{\btheta}_{i_1} - \widehat{\btheta}_{i_2}) - (\btheta_{i_1} - \btheta_{i_2})\|_2 = O_p((m^{1/2}n\rho_n^{1/2}\theta_n)^{-1})$.

\noindent$\blacksquare$ \textbf{Convergence of $\widehat{\bGamma}_i$}. Denote by $\widehat{\bW} = \bW\bQ_1\transpose$ By Theorem \ref{thm:entrywise_eigenvector_perturbation_bound}, we know that 
\[
\|\widehat{\bU}\widehat{\bW} - \bU\|_{2\to\infty} = \Optilde\bigg\{\frac{1}{\sqrt{n}}(\varepsilon_n^{(\mathsf{op})} + \varepsilon_n^{(\mathsf{bc})})\bigg\},\quad\|\widehat{\bU}\|_{2\to\infty} = \Optilde(n^{-1/2}),\quad\|\widehat{\bU}\widehat{\bW} - \bU\|_2 = \Optilde(\varepsilon_n^{(\mathsf{op})} + \varepsilon_n^{(\mathsf{bc})}).
\]
By Lemma \ref{lemma:norm_concentration_Et} and a union bound over $t\in[m]$, we know that $\max_{t\in[m]}\|\bE_t\|_2 = \Optilde\{(\log n)^{1/2} + (n\rho_n)^{1/2}\}$ and $\max_{t\in[m]}\|\bA_t\|_2 = \Optilde\{n\rho_n + (\log n)^{1/2} + (n\rho_n)^{1/2}\}$. 
Then we obtain
\begin{align*}
\max_{t\in[m]}\|\widehat{\bP}_t - \bP_t\|_{\max}
&\leq \max_{t\in[m]}\|\widehat{\bU}\widehat{\bW} - \bU\|_{2\to\infty}\|\bA_t\|_2\|\widehat{\bU}\|_{2\to\infty}
 + \max_{t\in[m]}\|\bU\|_{2\to\infty}\|\widehat{\bU}\widehat{\bW} - \bU\|_2\|\bA_t\|_2\|\widehat{\bU}\|_{2\to\infty}\\
&\quad + \max_{t\in[m]}\|\bU\|_{2\to\infty}\|\bE_t\|_2\|\widehat{\bU}\widehat{\bW} - \bU\|_2\|\widehat{\bU}\|_{2\to\infty} + \max_{t\in[m]}\|\bU\|_{2\to\infty}\|\bU\transpose\bE_t\bU\|_2\|\|\widehat{\bU}\|_{2\to\infty}\\
&\quad + \max_{t\in[m]}\|\bU\|_{2\to\infty}\|\bP_t\|_2\|\widehat{\bU}\widehat{\bW} - \bU\|_2\|\widehat{\bU}\|_{2\to\infty} + \max_{t\in[m]}\|\bU\|_{2\to\infty}\|\bB_t\|_2\|\widehat{\bU}\widehat{\bW} - \bU\|_{2\to\infty}\\
& = \optilde(\rho_n)
 .
\end{align*}
Denote by $P_{tij} = \bz_i\transpose\bC_t\bz_j$. Note that $\max_{t\in[m]}\|\bP_t\|_{\max}\leq \|\bZ\|_{2\to\infty}\max_{t\in[m]}\|\bC_t\|_2 = O(\rho_n)$.
This further implies the following estimation error bounds for $\widehat{P}_{tij}$'s: and $\widehat{\sigma}_{tij}$'s:
\begin{align*}
\min_{t\in[m],i,j\in[n]}|\widehat{P}_{tij}|&\geq \min_{t\in[m],i,j\in[n]}P_{tij} - \max_{t\in[m]}\|\widehat{\bP}_t - \bP_t\|_{\max} = \Omega(\rho_n) - |\optilde(\rho_n)|,\\
\max_{t\in[m],i,j\in[n]}|\widehat{P}_{tij}|&\leq \max_{t\in[m],i,j\in[n]}P_{tij} + \max_{t\in[m]}\|\widehat{\bP}_t - \bP_t\|_{\max} = O(\rho_n) + |\optilde(\rho_n)|,\\
\max_{t\in[m],i,j\in[n]}|\widehat{\sigma}_{tij}^2 - \sigma_{tij}^2|
&\leq \max_{t\in[m],i,j\in[n]}|\widehat{P}_{tij} - P_{tij}|\left(2 + |\widehat{P}_{tij}|\right) = \optilde(\rho_n)\\
\max_{t\in[m],i,j,l\in[n]}|\widehat{\sigma}_{tij}^2\widehat{\sigma}_{tjl}^2 - \sigma_{tij}^2\sigma_{tjl}^2|
&\leq  \max_{t\in[m],i,j,l\in[n]}\left\{|(\widehat{\sigma}_{tij}^2 - \sigma_{tij}^2)\widehat{\sigma}_{tjl}^2| + |\sigma_{tij}^2(\widehat{\sigma}_{tjl}^2 - \sigma_{tjl}^2)|\right\} = \optilde(\rho_n^2).
\end{align*}
Similarly, we also obtain the following estimation error bounds for $\widehat{\bB}_t$ and $\widehat{\bB}$:
\begin{align*}
\max_{t\in[m]}\|\widehat{\bW}\transpose\widehat{\bB}_t\widehat{\bW} - \bB_t\|_{\mathrm{F}}
&\leq d^{1/2}\|\widehat{\bU}\widehat{\bW} - \bU\|_{2\to\infty}\max_{t\in[m]}(\|\bA_t\|_2\|\widehat{\bU}\|_{2\to\infty} + \|\bP_t\|_2\|\bU\|_{2\to\infty})
\\&\quad
 + d^{1/2}\max_{t\in[m]}\|\bU\transpose\bE_t\bU\|_2
 + d^{1/2}\max_{t\in[m]}\|\bU\|_{2\to\infty}\|\bE_t\|_2\|\widehat{\bU}\widehat{\bW} - \bU\|_{2\to\infty} = \optilde(n\rho_n),
 \\
\max_{t\in[m]}\|\widehat{\bB}_t\|_2
&\leq \max_{t\in[m]}\|\widehat{\bW}\transpose\widehat{\bB}_t\widehat{\bW} - \bB_t\|_2 + \max_{t\in[m]}\|\bB_t\|_2 = O(n\rho_n) + \optilde(n\rho_n),\\
\max_{t\in[m]}\|\widehat{\bW}\transpose\widehat{\bB}_t^2\widehat{\bW} - \bB_t^2\|_2
&\leq \max_{t\in[m]}\|(\widehat{\bW}\transpose\widehat{\bB}_t\widehat{\bW} - \bB_t)\widehat{\bW}\widehat{\bB}_t\widehat{\bW}\|_2 + \max_{t\in[m]}\|\bB_t(\widehat{\bW}\transpose\widehat{\bB}_t\widehat{\bW} - \bB_t)\|_2 = \optilde\{(n\rho_n)^2\},\\
\min_{t\in[m]}\sigma_d(\widehat{\bB}_t)
&\geq \min_{t\in[m]}\{\lambda_d(\bB_t^2) - \|\widehat{\bW}\transpose\widehat{\bB}_t^2\widehat{\bW} - \lambda_d(\bB_t^2)\|_2\}^{1/2} = \Omega(n\rho_n) - |\optilde(n\rho_n)|,\\
\|\widehat{\bW}\transpose\widehat{\bB}\widehat{\bB}\transpose\widehat{\bW} - \bB\bB\transpose\|_2
&\leq \sum_{t = 1}^m\|\widehat{\bW}\transpose\widehat{\bB}_t^2\widehat{\bW} - \bB_t^2\| = \optilde\{m(n\rho_n)^2\}\\
\|\widehat{\bW}\transpose\widehat{\bB}\widehat{\bB}\transpose\widehat{\bW}\|_2&
\leq \|\widehat{\bW}\transpose\widehat{\bB}\widehat{\bB}\transpose\widehat{\bW} - \bB\bB\transpose\|_2 + \|\bB\bB\transpose\|_2 = O(mn^2\rho_n^2) + \optilde(mn^2\rho_n^2),\\
\lambda_d(\widehat{\bB}\widehat{\bB}\transpose)
&\geq \lambda_d(\bB\bB\transpose) - \|\widehat{\bW}\transpose\widehat{\bB}\widehat{\bB}\transpose\widehat{\bW} - \bB\bB\transpose\|_2 = \Omega(mn^2\rho_n^2) - |\optilde(mn^2\rho_n^2)|.
\end{align*}
Therefore, we can further obtain the following estimation error bounds for $\widehat{\bF}_i$, $\widehat{\bG}_i$:
\begin{align*}
\|\widehat{\bW}\transpose\widehat{\bF}_i\widehat{\bW} - \bF_i\|_2
&\leq\sum_{t = 1}^m\|\widehat{\bW}\transpose\widehat{\bB}_t\widehat{\bW} - \bB_t\|_2\max_{j\in[n]}\widehat{P}_{tij}\|\widehat{\bB}_t\|_2
 + \sum_{t = 1}^m\|\bB_t\|_2\|\widehat{\bU}\widehat{\bW} - \bU\|_2\max_{j\in[n]}\widehat{P}_{tij}\|\widehat{\bB}_t\|_2\\
&\quad + \sum_{t = 1}^m\|\bB_t\|_2\max_{j\in[n]}|\widehat{\sigma}_{tij}^2 - \sigma_{tij}^2|\|\widehat{\bB}_t\|_2 + \sum_{t = 1}^m\|\bB_t\|_2\max_{j\in[n]}\sigma_{tij}^2\|\widehat{\bU}\widehat{\bW} - \bU\|_2\|\widehat{\bB}_t\|_2
\\&\quad
 + \sum_{t = 1}^m\|\bB_t\|_2\max_{j\in[n]}\sigma_{tij}^2\|\widehat{\bW}\transpose\widehat{\bB}_t\widehat{\bW} - \bB_t\|_2
 = \optilde(mn^2\rho_n^3),
\end{align*}
so that $\|\widehat{\bF}_i\|_2 \leq \|\widehat{\bW}\transpose\widehat{\bF}_i\widehat{\bW} - \bF_i\|_2 + \|\bF_i\|_2 = O(mn^2\rho_n^3) + \optilde(mn^2\rho_n^3)$
and
\begin{align*}
&\|\widehat{\bW}\transpose\widehat{\bG}_i\widehat{\bW} - \bG_i\|\\
&\quad\leq \|\widehat{\bU}\widehat{\bW} - \bU\|_2\max_{l\in[n]}\sum_{t = 1}^m\sum_{j = 1}^n\widehat{\sigma}_{til}^2\widehat{\sigma}_{tjl}^2
 + \max_{l\in[n]}\sum_{t = 1}^m\sum_{j = 1}^n|\widehat{\sigma}_{til}^2\widehat{\sigma}_{tjl}^2 - \sigma^2_{til}\sigma^2_{tjl}|
 + \max_{l\in[n]}\sum_{t = 1}^m\sum_{j = 1}^n\sigma^2_{til}\sigma^2_{tjl}\|\widehat{\bU}\widehat{\bW} - \bU\|_2\\
 & = \optilde(mn\rho_n^2),
\end{align*}
so that $\|\widehat{\bG}_i\|_2\leq \|\widehat{\bW}\transpose\widehat{\bG}_i\widehat{\bW} - \bG_i\|_2 + \|\bG_i\| = O(mn\rho_n^2) + \optilde(mn\rho_n^2)$.
Also, observe that
\begin{align*}
\|(\widehat{\bW}\transpose\widehat{\bB}\widehat{\bB}\transpose\widehat{\bW})^{-1} - (\bB\bB\transpose)^{-1}\|_2
&\leq \|(\widehat{\bW}\transpose\widehat{\bB}\widehat{\bB}\transpose\widehat{\bW})^{-1}\|_2\|_2\|\widehat{\bW}\transpose\widehat{\bB}\widehat{\bB}\transpose\widehat{\bW} - \bB\bB\transpose\|_2\|(\bB\bB\transpose)^{-1}\|_2\\
& = \optilde\left((mn^2\rho_n^2)^{-1}\right).
\end{align*}
Hence, we can finally provide the error bound for $\widehat{\bGamma}_i$:
\begin{align*}
\|\widehat{\bW}\transpose\widehat{\bGamma}_i\widehat{\bW} - \bGamma_i\|
&\leq \|(\widehat{\bW}\transpose\widehat{\bB}\widehat{\bB}\transpose\widehat{\bW})^{-1} - (\bB\bB\transpose)^{-1}\|_2(\|\widehat{\bF}_i\|_2 + \|\widehat{\bG}_i\|_2)\|(\widehat{\bB}\widehat{\bB}\transpose)^{-1}\|_2\\
&\quad + \|(\bB\bB\transpose)^{-1}\|_2(\|\widehat{\bW}\transpose\widehat{\bF}_i\widehat{\bW} - \bF_i\|_2 + \|\widehat{\bW}\transpose\widehat{\bG}_i\widehat{\bW} - \bG_i\|_2)\|(\widehat{\bW}\transpose\widehat{\bB}\widehat{\bB}\transpose\widehat{\bW})^{-1}\|_2\\
&\quad + \|(\bB\bB\transpose)^{-1}\|_2(\|\bF_i\|_2 + \|\bG_i\|_2)\|(\widehat{\bW}\transpose\widehat{\bB}\widehat{\bB}\transpose\widehat{\bW})^{-1} - (\bB\bB\transpose)^{-1}\|_2 = \optilde\left(\frac{1}{mn^2\rho_n\theta_n^2}\right).
\end{align*}
In particular, this implies that $\|\widehat{\bGamma}_i\|_2 = \Optilde((mn^2\rho_n\theta_n^2)^{-1})$ and $\lambda_d(\widehat{\bGamma}_i)\geq \Omega((mn^2\rho_n\theta_n^2)^{-1}) - |\optilde((mn^2\rho_n\theta_n^2)^{-1})|$.

\noindent$\blacksquare$ \textbf{Asymptotic null distribution of $T_{i_1i_2}$.}
By definition, we have
\begin{align*}
T_{i_1i_2}& = (\widehat{\btheta}_{i_1} - \widehat{\btheta}_{i_2})\transpose(\widehat{\bW}\bGamma_{i_1}\widehat{\bW}\transpose + \widehat{\bW}{\bGamma}_{i_2}\widehat{\bW}\transpose)^{-1}(\widehat{\btheta}_{i_1} - \widehat{\btheta}_{i_2})\\
&\quad + (\widehat{\btheta}_{i_1} - \widehat{\btheta}_{i_2})\transpose\widehat{\bW}\{(\widehat{\bW}\transpose\widehat{\bGamma}_{i_1}\widehat{\bW} + \widehat{\bW}\transpose\widehat{\bGamma}_{i_2}\widehat{\bW})^{-1} - ({\bGamma}_{i_1} + {\bGamma}_{i_2})^{-1}\}\widehat{\bW}\transpose(\widehat{\btheta}_{i_1} - \widehat{\btheta}_{i_2})
\end{align*}
Observe that
\begin{align*}
&|(\widehat{\btheta}_{i_1} - \widehat{\btheta}_{i_2})\transpose\widehat{\bW}\{(\widehat{\bW}\transpose\widehat{\bGamma}_{i_1}\widehat{\bW} + \widehat{\bW}\transpose\widehat{\bGamma}_{i_2}\widehat{\bW})^{-1} - (\bGamma_{i_1} + \bGamma_{i_2})^{-1}\}\widehat{\bW}\transpose(\widehat{\btheta}_{i_1} - \widehat{\btheta}_{i_2})|\\
&\quad \leq\|\widehat{\btheta}_{i_1} - \widehat{\btheta}_{i_2}\|_2^2\|(\widehat{\bW}\transpose\widehat{\bGamma}_{i_1}\widehat{\bW} + \widehat{\bW}\transpose\widehat{\bGamma}_{i_2}\widehat{\bW})^{-1}\|_2(\|\widehat{\bW}\transpose\widehat{\bGamma}_{i_1}\widehat{\bW} - \bGamma_{i_1}\|_2 + \|\widehat{\bW}\transpose\widehat{\bGamma}_{i_2}\widehat{\bW} - \bGamma_{i_2}\|_2)\|(\bGamma_{i_1} + \bGamma_{i_2})^{-1}\|_2\\
&\quad = o_p(1).
\end{align*}
It follows from Slutsky's theorem that
$T_{i_1i_2} = (\widehat{\btheta}_{i_1} - \widehat{\btheta}_{i_2})\transpose(\widehat{\bW}{\bGamma}_{i_1}\widehat{\bW}\transpose + \widehat{\bW}{\bGamma}_{i_2}\widehat{\bW}\transpose)^{-1}(\widehat{\btheta}_{i_1} - \widehat{\btheta}_{i_2}) + o_p(1)\overset{\calL}{\to}\chi^2_d$.

\noindent$\blacksquare$ \textbf{Asymptotic distribution of $T_{i_1i_2}$ under the alternative.} 
Let $M_n = (mn^2\rho_n)^{1/2}\theta_n\|\btheta_{i_1} - \btheta_{i_2}\|_2$. Clearly, we know that $M_n\to\infty$ because $\|\btheta_{i_1} - \btheta_{i_2}\|_2 = \Theta(n^{-1/2}\|\bz_{i_1} - \bz_{i_2}\|_2)$, and 
\begin{align*}\prob\{(mn^2\rho_n)^{1/2}\theta_n\|\widehat{\bW}\transpose(\widehat{\btheta}_{i_1} - \widehat{\btheta}_{i_2}) - (\btheta_{i_1} - \btheta_{i_2})\|_2\geq M_n/2\}\to 0,
\end{align*} 
so that
\begin{align*}
\prob\{(mn^2\rho_n)^{1/2}\theta_n\|\widehat{\btheta}_{i_1} - \widehat{\btheta}_{i_2}\|_2 > M_n/2\}
&\geq \prob\{(mn^2\rho_n)^{1/2}\theta_n\|\widehat{\bW}\transpose(\widehat{\btheta}_{i_1} - \widehat{\btheta}_{i_2}) - (\btheta_{i_1} - \btheta_{i_2})\|_2 \leq M_n/2\}\to 1
\end{align*}
and there exists some constant $c_0 > 0$, such that
\begin{align*}
\prob(T_{i_1i_2}\geq C)&\geq \prob\{T_{i_1i_2}\geq M_n^2/(4c_0)\} = \prob\{(\widehat{\btheta}_{i_1} - \widehat{\btheta}_{i_2})\transpose(\widehat{\bGamma}_{i_1} + \widehat{\bGamma}_{i_2})^{-1}(\widehat{\btheta}_{i_1} - \widehat{\btheta}_{i_2})\geq M_n^2/(4c_0)\}\\
&\geq \prob[\lambda_d\{(\widehat{\bGamma}_{i_1} + \widehat{\bGamma}_{i_2})^{-1}\}\|\widehat{\btheta}_{i_1} - \widehat{\btheta}_{i_2}\|_2^2\geq M_n^2/(4c_0),\lambda_d\{(\widehat{\bGamma}_{i_1} + \widehat{\bGamma}_{i_2})^{-1}\}\geq mn^2\rho_n\theta_n^2/c_0]\\
&\geq \prob[mn^2\rho_n\theta_n^2\|\widehat{\btheta}_{i_1} - \widehat{\btheta}_{i_2}\|_2^2\geq M_n^2/4,\lambda_d\{(\widehat{\bGamma}_{i_1} + \widehat{\bGamma}_{i_2})^{-1}\}\geq mn^2\rho_n\theta_n^2/c_0]\to 1.
\end{align*}
On the other hand, under the alternative $H_A:\bz_{i_1}\neq\bz_{i_2}$ but $(\bGamma_{i_1} + \bGamma_{i_2})^{-1/2}(\bZ\transpose\bZ)^{-1/2}(\bz_{i_1} - \bz_{i_2})\to\bmu$ for some $\bmu\in\mathbb{R}^d$, we have $(\bGamma_{i_1} + \bGamma_{i_2})^{-1/2}\{\widehat{\bW}\transpose(\widehat{\btheta}_{i_1} - \widehat{\btheta}_{i_2})\}\overset{\calL}{\to}\mathrm{N}_d(\bmu, \eye_d)$.
Similarly, we also have $\|\widehat{\btheta}_{i_1} - \widehat{\btheta}_{i_2}\|_2^2 = O_p((mn^2\rho_n\theta_n^2)^{-1})$. It follows that
\begin{align*}
T_{i_1i_2}
& = (\widehat{\btheta}_{i_1} - \widehat{\btheta}_{i_2})\transpose(\widehat{\bW}\bGamma_{i_1}\widehat{\bW}\transpose + \widehat{\bW}{\bGamma}_{i_2}\widehat{\bW}\transpose)^{-1}(\widehat{\btheta}_{i_1} - \widehat{\btheta}_{i_2})\\
&\quad + (\widehat{\btheta}_{i_1} - \widehat{\btheta}_{i_2})\transpose\widehat{\bW}\{(\widehat{\bW}\transpose\widehat{\bGamma}_{i_1}\widehat{\bW} + \widehat{\bW}\transpose\widehat{\bGamma}_{i_2}\widehat{\bW})^{-1} - ({\bGamma}_{i_1} + {\bGamma}_{i_2})^{-1}\}\widehat{\bW}\transpose(\widehat{\btheta}_{i_1} - \widehat{\btheta}_{i_2})\\
& = (\widehat{\btheta}_{i_1} - \widehat{\btheta}_{i_2})\transpose(\widehat{\bW}\bGamma_{i_1}\widehat{\bW}\transpose + \widehat{\bW}{\bGamma}_{i_2}\widehat{\bW}\transpose)^{-1}(\widehat{\btheta}_{i_1} - \widehat{\btheta}_{i_2}) + o_p(1)\overset{\calL}{\to}\chi_d^2(\|\bmu\|_2^2).
\end{align*}
The proof is therefore completed.
}
 
\bibliographystyle{acm}
\bibliography{reference,reference1,reference2,reference_RMT,reference_MG}

\end{document}